\documentclass[11pt]{amsart}
\usepackage[margin=0.8in]{geometry} 
\usepackage{amsthm, amsrefs, amsmath,amsfonts,amssymb,euscript,hyperref,graphics,color,slashed,mathrsfs}
\usepackage{graphicx}
\usepackage[toc]{appendix}
\usepackage{comment}
\usepackage{import}
\usepackage{tikz}
\usepackage{latexsym}
\usepackage{enumerate,enumitem}

\def\Ntop{{N_{_{\rm top}}}}
\def\Ninf{N_{_{\infty}}}

\def\Xh{{\widehat{X}}}
\def\Th{{\widehat{T}}}

\def\Xr{\mathring{X}}
\def\Tr{\mathring{T}}
\def\Trh{\widehat{\mathring{T}}}

\def\Zr{{\mathring{Z}}}
\def\Lr{{\mathring{L}}}
\def\Lbr{{\mathring{\underline{L}}}}

\def\zr{{\mathring{z}}}
\def\yr{{\mathring{y}}}

\def\kappar{{\mathring{\kappa}}}
\def\mur{{\mathring{\mu}}}
\def\etar{{\mathring{\eta}}}
\def\deltar{{\mathring{\delta}}}
\def\deltasr{{\mathring{\slashed{\delta}}}}
\def\Cb{{\underline{C}}}

\def\zetar{{\mathring{\zeta}}}
\def\chir{{\mathring{\chi}}}
\def\chibr{{\mathring{\underline{\chi}}}}

\def\E{{\mathcal E}}
\def\Eb{{\underline{\mathcal E}}}

\def\F{{\mathcal F}}
\def\Fb{{\underline{\mathcal F}}}

\def\ub{\underline{u}}
\def\wb{\underline{w}}

\def\chib{\underline{\chi}}

\def\Lb{\underline{L}}
\def\Hb{\underline{H}}

\def\Xih{\widehat{{\Xi}}}

\newtheorem*{TheoremDataExistence}{Theorem 1}
\newtheorem*{CorollaryExistenceOnDdelta}{Corollary 1}
\newtheorem*{TheoremExistenceRarefactionWaves}{Theorem 2}
\newtheorem*{TheoremStructuralStability}{Theorem 3}

\newtheorem{theorem}{Theorem}[section]
\newtheorem{lemma}[theorem]{Lemma}
\newtheorem{proposition}[theorem]{Proposition}
\newtheorem{corollary}[theorem]{Corollary}
\newtheorem{definition}[theorem]{Definition}
\newtheorem{remark}[theorem]{Remark}

\numberwithin{equation}{section}

\begin{document}
\title[Multi-dimensional rarefaction waves II]{On the stability of multi-dimensional rarefaction waves II: existence of solutions and applications to Riemann problem}

\author{Tian-Wen LUO and Pin YU}

\address{School of Mathematical Sciences, South China Normal University, Guangzhou, China}
\email{twluo@m.scnu.edu.cn}

\address{Department of Mathematical Sciences, Tsinghua University\\ Beijing, China}
\email{yupin@mail.tsinghua.edu.cn}


\maketitle

\begin{abstract}
This is the second paper in a series studying the nonlinear stability of rarefaction waves in multi-dimensional gas dynamics. We construct initial data near singularities in the rarefaction wave region and, combined with the \emph{a priori} energy estimates from the first paper, demonstrate that any smooth perturbation of constant states on one side of the diaphragm in a shock tube can be connected to a centered rarefaction wave. We apply this analysis to study multi-dimensional perturbations of the classical Riemann problem for isentropic Euler equations. We show that the Riemann problem is structurally stable in the regime of two families of rarefaction waves.

\end{abstract}
\tableofcontents

\section{Introduction and review}
Rarefaction waves are fundamental wave patterns in gas dynamics and hyperbolic conservation laws. It provides a nonlinear expansion mechanism that can instantly resolve and smooth out the initial  discontinuities. In the final chapter of A. Majda's monograph \cite{MajdaBook} on compressible flows, ``the existence and structure of rarefaction fronts'' is presented as the first among a list of open problems: ``\textit{Discuss the rigorous existence of rarefaction fronts for the physical equations and elucidate the differences in multi-D rarefaction phenomena when compared with the 1-D case}". 
This paper is the second in a series dedicated to exploring this problem. We provide an affirmative answer for the isentropic compressible Euler equations, constructing centered rarefaction wavefronts and showing their uniqueness and stability. Additionally, we demonstrate that the multidimensional rarefaction phenomena exhibit all of the characteristics of the one-dimensional case.



\subsection{Review of 1-D Riemann problem for the isentropic Euler system}\label{section:review on p system}

The 1-D Euler equations, describing the isentropic one-dimensional motion of a gas, are given by:
\begin{equation}\label{p system}
\begin{cases}
	\partial_t \rho + v \partial_x \rho = -\rho \partial_x v,\\
	\rho(\partial_t v + v \partial_x v) = -\partial_x p,
	\end{cases}
\end{equation}
where $\rho$, $p$ and $v$ are the density, pressure, and velocity of the gas, respectively. The equation of state is given by $p(\rho) = k_0 \rho^{\gamma}$ with constants $k_0 >0$ and $\gamma\in (1,3)$, which applies to most gases. 
The sound speed $c$ is defined as  $c=\sqrt{\frac{dp}{d\rho}}=k_0^{\frac{1}{2}}\gamma^{\frac{1}{2}}\rho^{\frac{\gamma-1}{2}}$.




Riemann's seminal work \cite{Riemann} on plane waves of finite amplitude in gas dynamics  \eqref{p system} was the first to explore the phenomena of shock waves and rarefaction waves. Riemann identified and solved the Riemann problems as the primary objects of study, laying  the foundation for the theory of conservation laws in one-dimension. Riemann's idea and deep insight have proven to be of lasting significance in nonlinear waves and hyperbolic partial differential equations.

We follow the approach in Smoller's textbook \cite{Smoller} to outline the results of Riemann, see Chapter 16, 17 and 18 of \cite{Smoller} for detailed computations. In line with the standard formulations of the theory of conservation laws, \eqref{p system} is expressed as:
 \begin{equation}\label{p system in conservation laws}
\partial_t U+ A(U) \partial_x U=0, 
\end{equation}
where $U=\begin{pmatrix}
  \rho \\
  v
   \end{pmatrix}$ and $A(U)=\begin{pmatrix}
  v & \rho\\
  \rho^{-1}c^2  & \rho^{-1}v
  \end{pmatrix}$. For a given solution $U$ and a given point $(t,x)$, the matrix $A(U)$ has two distinct  eigenvalues $\lambda_1(U)=v-c$  and   $\lambda_2(U)=v+c$.  
The Riemann problem is the Cauchy problem with specific data $U(0,x)$ posed on $t=0$: 
\begin{equation}\label{p system data}
U(0,x)=\begin{cases}
&U_l, \ \ x>0, \\
&U_r, \ \ x<0,
\end{cases}
\end{equation}
where $U_l=\begin{pmatrix}
  \rho_l \\
  v_l
   \end{pmatrix}$ and $U_r=\begin{pmatrix}
  \rho_r \\
  v_r
   \end{pmatrix}$ are two constant states. 

Riemann's solution to his problem is concisely summarized in Courant-Friedrichs \cite{CourantFriedrichs} as follows:
``\textit{either the initial discontinuity is resolved immediately and the disturbance, while propagated, becomes continuous, or the initial discontinuity is propagated through one or two shock fronts, advancing not at sonic but at supersonic speed relative to the medium ahead of them}''.  The last sentence refers to the determinism condition or entropy condition for shocks.

To be more precise, there are four fundamental solution patterns for the Riemann problem, which connect constant states through either shocks or rarefaction waves:
\medskip

$\bullet$ Shock Waves.

\medskip

Shocks are piecewise smooth and discontinuous solutions that propagate the initial discontinuities, while satisfying the jump conditions and the entropy inequalities across the shock front. There are two families of shock waves: the \emph{back shocks}, corresponding to the first eigenvalue $\lambda_1(U)$, and the \emph{front shocks}, corresponding to the second eigenvalue $\lambda_2(U)$. We will focus on the back shocks, as the two families are symmetrically related.




A back shock is a solution that is piecewise constant, consisting of exactly two pieces separated by the ray $x=st$, where $s$ is the shock speed.  It satisfies the jump condition
\[(v_l - v_r)^2 = \frac{\rho_l - \rho_r}{\rho_r \rho_l}(p(\rho_l) - p(\rho_r)). \]
Combining with the determinism condition $\lambda_1(U_l) = v_l - c_l < s_1 < v_r - c_r = \lambda_1(U_r)$, we obtain $s = v_1 \pm \sqrt{\frac{\rho_2}{\rho_1}\frac{p(\rho_1) - p(\rho_2)}{\rho_1 - \rho_2}} = v_2 \pm \sqrt{\frac{\rho_1}{\rho_2}\frac{p(\rho_1) - p(\rho_2)}{\rho_1 - \rho_2}}$ and 
\begin{equation}\label{back shock curve}
	v_r - v_l = S_1(\rho_r;U_l) := -\sqrt{\frac{\rho_l - \rho_r}{\rho_r \rho_l}(p(\rho_l) - p(\rho_r))}, \ \  \rho_r > \rho_l.
\end{equation}
It can be depicted as follows:
\begin{center}
\includegraphics[width=2.5in]{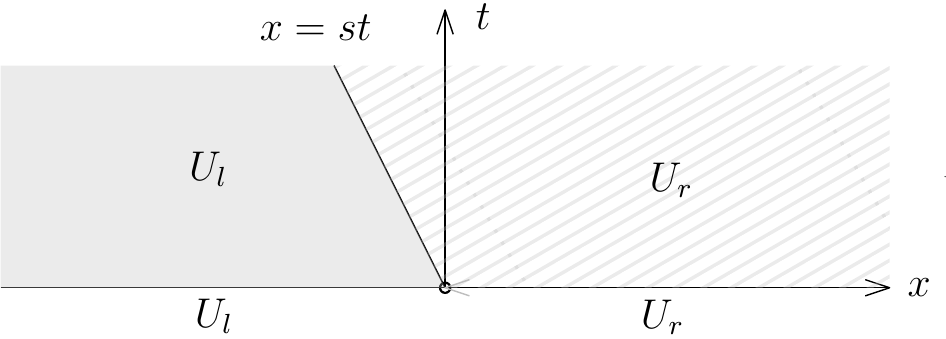}
\end{center}


\medskip

$\bullet$ Rarefaction Waves.

\medskip

A rarefaction wave is a continuous solution that immediately smooths out the initial discontinuities, in this case, a solution to \eqref{p system} of the form $U\big(\frac{x}{t}\big)$. There are two families of rarefaction waves: the \emph{back rarefaction waves}, corresponding to the first eigenvalue $\lambda_1(U)$, and the \emph{front rarefaction waves}, corresponding to the second eigenvalue $\lambda_2(U)$. Due to symmetry considerations, we will focus on the back  rarefaction waves.

For a back rarefaction wave, we use the ansatz $U(t,x)=U\left(\xi\right)$ in  \eqref{p system} where $\xi=\frac{x}{t}$, and we put $U(\xi)$ to be an eigenvector associated with the eigenvalue $\xi = \lambda_1(U) = v - c$. Therefore, we obtain 
\begin{equation}\label{back rarefaction curve}
	v_r - v_l = R_1(\rho_r;U_l) := \int_{\rho_l}^{\rho_r}\frac{c(\rho')}{\rho'} d\rho', \ \ \rho_r < \rho_l.
\end{equation}
The inequality $\rho_r > \rho_l$ arises from the fact that the characteristic speed $\lambda_1(U)$ must increase as $\xi$ increases. 
The solution  to \eqref{p system} in this case is  depicted as follows:
\begin{center}
\includegraphics[width=2.5in]{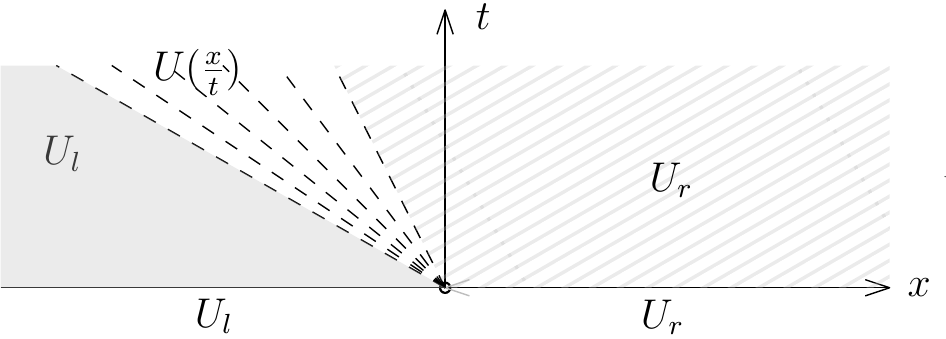}
\end{center}
The centered rarefaction wave region is filled by the dashed rays passing through the origin. The solution $U$ is constant along each ray. The left boundary is entirely determined by the data $U_l$ on $x<0$. In fact, it is the boundary of the domain of dependence of the data given on $x<0$. Since $\lambda_1(\xi)$ is an increasing function, we can increase $\frac{x}{t}$ until $\lambda_1\big(U(\frac{x}{t})\big)$ matches with $\lambda_1(U_r)$. This defines the right boundary of the rarefaction wave region.


\medskip

In the $(\rho,v)$-plane, for a given point $U_l$, the equation \eqref{back shock curve} and its counterpart for front shocks define two curves $S_1$ and $S_2$ emanating from the point $U_l$; similarly \eqref{back rarefaction curve} and its counterpart for front rarefaction waves also define two curves $R_1$ and $R_2$ emanating from the point $U_l$:
\begin{center}
\includegraphics[width=2.5in]{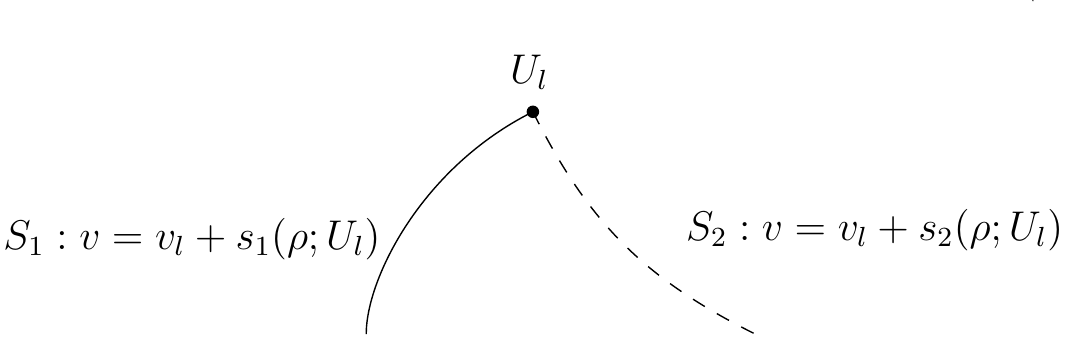} \qquad \includegraphics[width=2.5in]{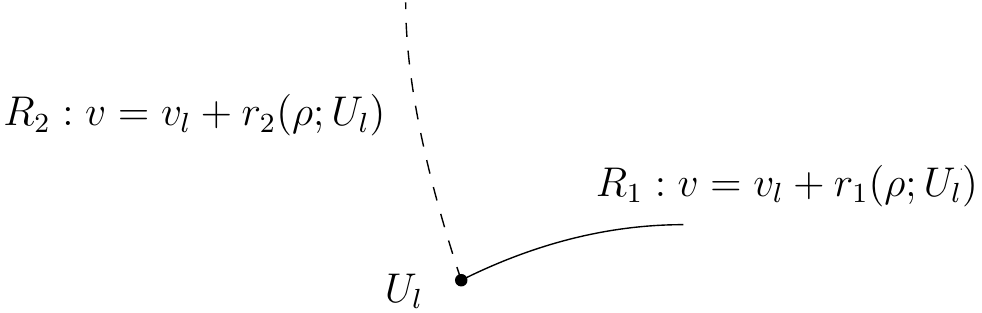}
\end{center}
In conclusion, a state $U_r$ can be connected to a given $U_l$ by a back shock if and only if $U_r\in S_1$; similarly it can be connected  by a front rarefaction wave if and only if $U_r\in R_2$.

\medskip

We will now study the four fundamental solution patterns for the Riemann problem. For any given point $U_l$ in the $(\rho,v)$-plane, for $i=1,2$,  the curves $S_i$ and $R_i$ join at $U_l$ in the $C^2$ manner and this defines the $C^2$ curve $W_i(U_l)$. The two intersecting curves $W_1(U_l)$ and $W_2(U_l)$ divide the $(\rho,v)$-plane into four parts. We list them as I, II, III and IV; see the following picture:
\begin{center}
\includegraphics[width=1.5in]{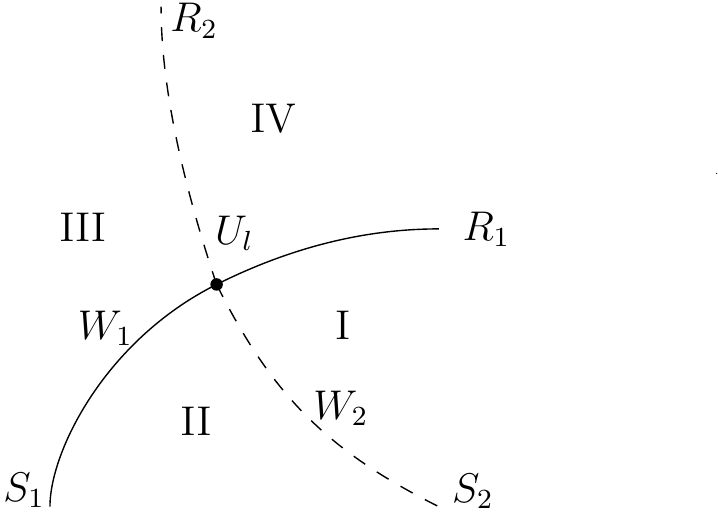} 
\end{center}
The wave curves $W_1(U)$, $W_2(U)$ can be used to solve the Riemann problem.
  For instance, if $U_r\in \text{II}$, the initial discontinuity is propagated through two shock fronts 
  ($\overline{U}$ denotes an intermediate state):
\begin{center}
	\includegraphics[width=4.5in]{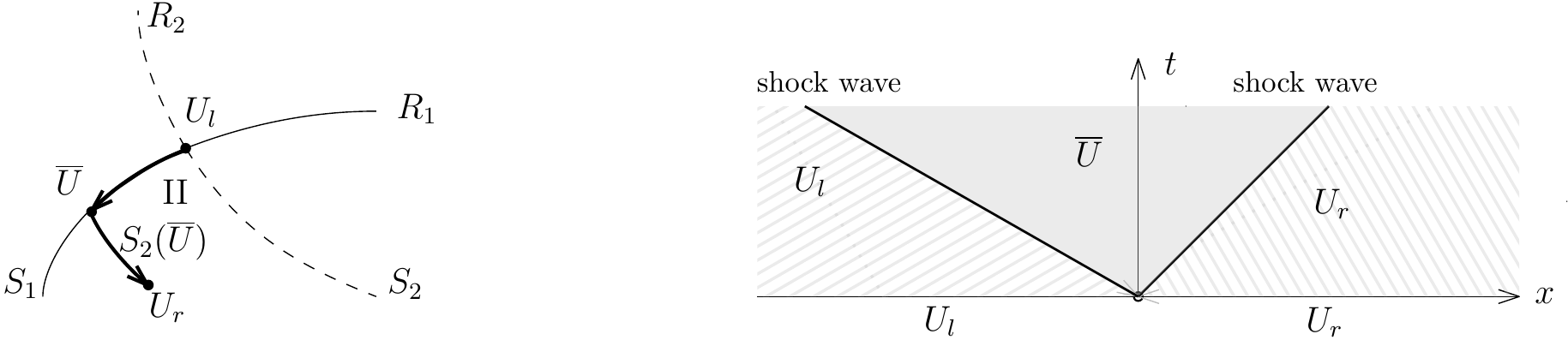}
\end{center}

If $U_r \in \text{I}$ or $U_r \in \text{III}$, $U_l$ and $U_r$ is connected by a rarefaction wave and a shock. 
The case for $U_r \in \text{I}$ is depicted as follows:
\begin{center}
\includegraphics[width=4.5in]{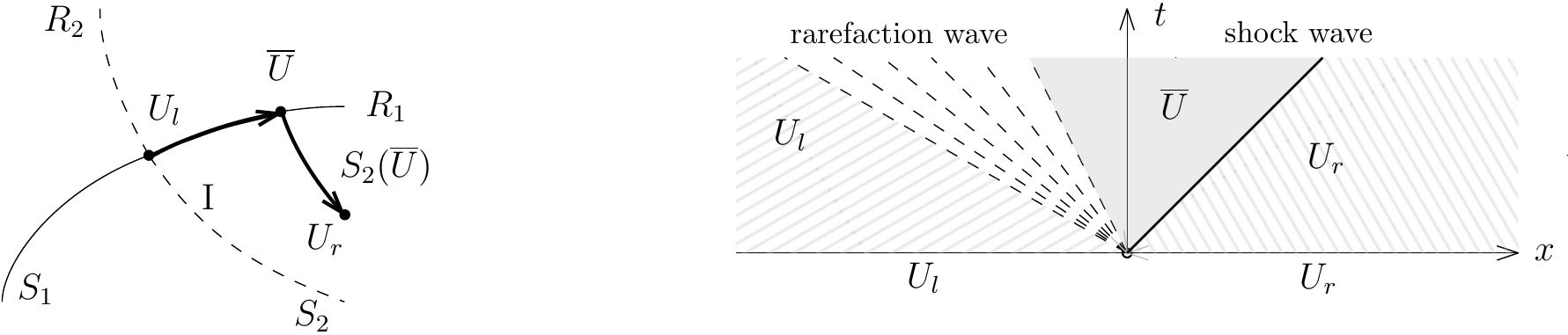}
\end{center}

The other typical case occurs when $U_r\in \text{IV}$. The behavior of solutions in this region is more intricate since vacuum may appear, as discussed in \cite{Smoller}. For simplicity, we assume that $U_r$ is close to $U_l$ in the $(\rho,v)$-plane to avoid the vacuum.
Under this assumption, $U_l$ is first connected to $\overline{U}$ by a back rarefaction wave and then connected to $U_r$ by a front rarefaction wave in a unique way. This is  illustrated as follows:
\begin{center}
	\includegraphics[width=4.5in]{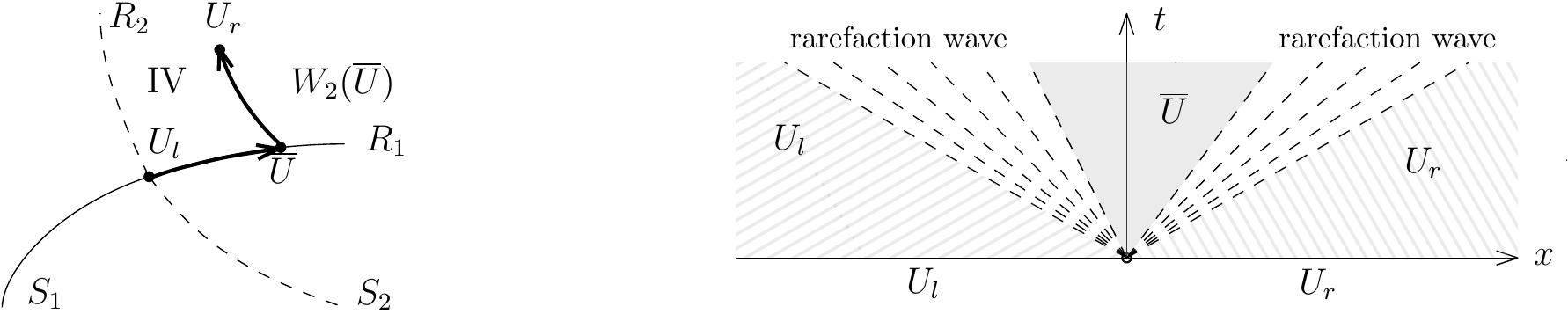}
\end{center}

\medskip

In this work, we focus on the final wave pattern, where the initial discontinuity is resolved immediately by two families of rarefaction waves. We will demonstrate that this solution pattern remains stable in multi-dimensional situations.

\subsection{The Riemann problem with two families of rarefaction waves}\label{section: calculations for 1D Riemann}
We will first provide a detailed calculation of the classical one-dimensional Riemann problem with two families of centered rarefaction waves, with the help of Riemann invariants introduced in \cite{Riemann}. This will serve as the background solution for our subsequent analysis of the multi-dimensional case.

In accordance with Riemann's notation \cite{Riemann}, we define the Riemann invariants $(\wb,w)$ relative to the solution $(v,c)$ using the following formulas:
\[ \begin{cases}
	&w=\frac{1}{2}\big(\frac{2}{\gamma-1}c-v\big),\\ 
	&\wb=\frac{1}{2}\big(\frac{2}{\gamma-1}c+v\big),
\end{cases} \ \ \Leftrightarrow \ \  \begin{cases}
	&v=\wb-w,\\ 
	&c=\frac{\gamma-1}{2}(\wb+w).
\end{cases}
\]
In addition to the Riemann invariants, we also have two characteristic spacetime vector fields:
\[\begin{cases}
	&L=\partial_t+(v+c)\partial_x,\\ 
	&\Lb=\partial_t+(v-c)\partial_x.
\end{cases}\]
These vector fields are null with respect to the acoustical metric $g = - c^2 dt^2 + (dx - vdt)^2$ defined by the solution $(v,c)$. We can rewrite equation \eqref{p system}
in two equivalent ways using the Riemann invariants and the frame $(L,\Lb)$ as follows:
\begin{equation*}
	\begin{cases}
		&L\wb=0,\\ 
		&Lw=2c\partial_x w,
	\end{cases} \ \ \Leftrightarrow \ \ \begin{cases}
		&\Lb w=0,\\ 
		&\Lb\wb=-2c\partial_x \wb.
	\end{cases}
\end{equation*}
We observe that the Riemann invariants $\wb$ and $w$ are constants along the directions of $L$ and $\Lb$, respectively. These invariant properties enable one to solve the Riemann problem explicitly:

\begin{center}
	\includegraphics[width=4in]{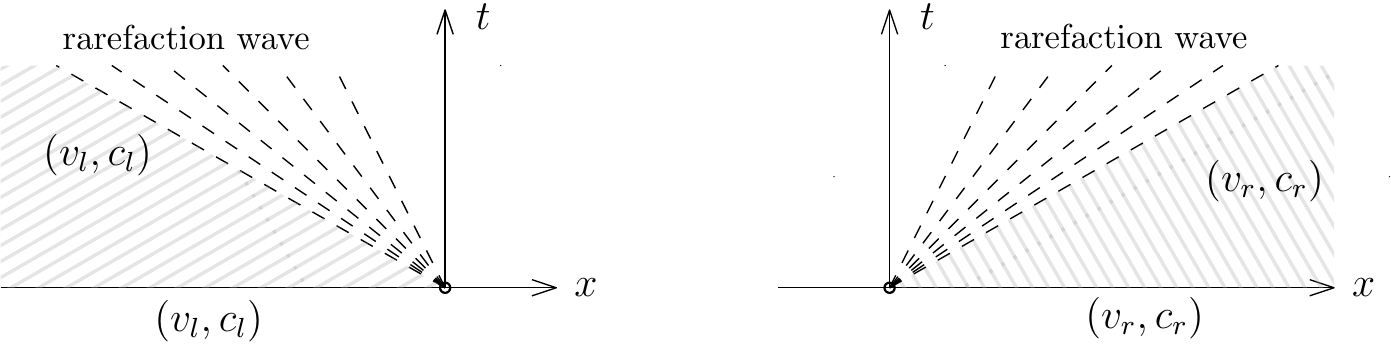}
\end{center}

\begin{itemize}
	\item By virtue of $L \wb = 0$, for a given constant state $(v_l,c_l)$ on $x\leqslant 0$, the unique centered rarefaction waves that can be connected to the maximal development of the data on the left is given by the following formula (the above figure on the left):
	\begin{equation*}
		\begin{cases}
			&\wb=\wb_l,\\
			&w=-\frac{2}{\gamma+1}\big(\frac{x}{t}+\frac{\gamma-3}{2}\wb_l\big).
		\end{cases}
	\end{equation*}
	\item By virtue of $\Lb w = 0$, for a constant state $(v_r,c_r)$ on $x\geqslant 0$, the unique centered rarefaction waves connected to the maximal development of the data on the right is given by the following formula (the above figure on the right):
	\begin{equation}\label{eq: precise solution for front rarefaction waves}
		\begin{cases}
			&w=w_r,\\
			&\wb=\frac{2}{\gamma+1}\big(\frac{x}{t}-\frac{\gamma-3}{2}w_r\big).
		\end{cases}.
	\end{equation}
\end{itemize}

Suppose we are given two constant states $(v_l,c_l)$ and $(v_r,c_r)$ on $x<0$ and $x>0$, respectively. We assume that the corresponding solution to the Riemann problem is a composite of two families of rarefaction waves, i.e., {\color{black}$w_l > w_r$} and $\wb_r > \wb_l$. 
This is depicted as follows:
\begin{center}
	\includegraphics[width=3in]{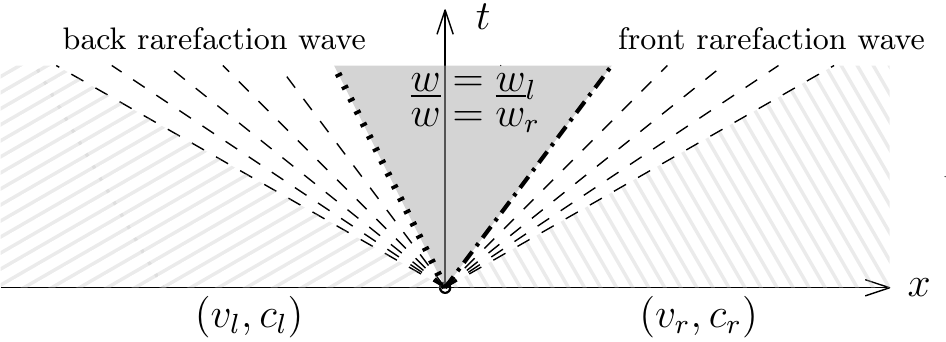}
\end{center}

The wave pattern for this case is a juxtaposition of the two rarefaction waves in the previous picture, and it corresponds to the last wave pattern described in Section \ref{section:review on p system}. The constant states $(v_l,c_l)$ and $(v_r,c_r)$ (or more precisely the corresponding $(\rho_l,v_l)$ and $(\rho_r,v_r)$) are located in region IV.

The middle state is separated from the front rarefaction wave by the dash-dotted ray in the picture. This boundary can be determined in the following way:  As we increase the {\color{black}slope $\frac{t}{x}$} by opening up the state $(v_r,c_r)$ with the front rarefaction wave,   the Riemann invariant $\wb$ decreases. The dash-dotted ray is then given by the locus of points where $\wb=\wb_l$. Similarly, we can determine the dotted ray which separates the middle state from the back rarefaction wave.

{\bf Assumptions on the background solution}. \ \ For the rest of this work, we will use the fixed background solution to the one-dimensional Riemann problem, which admits two families of rarefaction waves. We fix such constant states $(\mathring{v}_l, \mathring{c}_l)$ and $(\mathring{v}_r,\mathring{c}_r)$ (or equivalently,  $(\mathring{v}_l, \mathring{\rho}_l)$ and $(\mathring{v}_r,\mathring{\rho}_r)$). Therefore, $\mathring{\wb}_r>\mathring{\wb}_l$ (defined with respect to  $(\mathring{v}_r, \mathring{c}_r)$ and $(\mathring{v}_l,\mathring{c}_l)$) and $\mathring{w}_l>\mathring{w}_r$. In particular, the grey region in the above picture is {\bf non-degenerate}, i.e., the angle between the dotted and the dash-dotted rays is positive.

We study the two (spatial) dimensional isentropic Euler system:
\begin{equation}\label{eq: Euler in rho v}
	\begin{cases}
		(\partial_t + v \cdot \nabla) \rho &= -\rho \nabla \cdot v,\def\E{{\mathcal E}}
		\\
		(\partial_t + v \cdot \nabla)v &= -\rho^{-1} \nabla p.
	\end{cases}
\end{equation} 
We assume that the solution is defined on the 2-dimensional tube 
\[\Sigma_0=\mathbb{R}\times \mathbb{R}\slash 2\pi\mathbb{Z}=\big\{(t,x_1,x_2)\big| t=0, x_1\in \mathbb{R}, 0\leqslant x_2\leqslant 2\pi \big\}.\] 
We identify $(t,x_1,0)$ and $(t,x_1,2\pi)$  so that we only consider the problem with periodicity in $x_2$.

The equation of state is given by $p(\rho) = k_0 \rho^{\gamma}$ with $\gamma\in (1,3)$ and $k_0 >0$. In this case, the sound speed $c$ is given by $c=\sqrt{p'(\rho)}$ and the enthalpy $h$ can be computed by $h = \frac{1}{\gamma-1} c^2$. We remark that the momentum equation of the Euler equations can also be expressed as
\[(\partial_t + v \cdot \nabla)v = -\nabla h.\]
The equation of motion for irrotational flows is equivalent to the continuity equation for $h$:
\[\frac{1}{c^2}(\partial_t + v \cdot \nabla) h + \nabla \cdot v = 0.\]

\medskip

The initial data of the system \eqref{eq: Euler in rho v} are posed in terms of $v$ and $c$ on $\Sigma_0$:
\[
(v,c)\big|_{t=0}=
\begin{cases}
	\big(v_r(0,x_1,x_2),c_r(0,x_1,x_2)\big),  \ \ &x_1\geqslant 0;\\
	\big(v_l(0,x_1,x_2),c_l(0,x_1,x_2)\big),  \ \ & x_1\leqslant 0.
\end{cases}
\]
In terms of components, we have 
\[\begin{cases}
	v_l(0,x_1,x_2)=&\big(v^1_l(0,x_1,x_2),v^2_l(0,x_1,x_2)\big),\\
	v_r(0,x_1,x_2)=&\big(v^1_r(0,x_1,x_2),v^2_r(0,x_1,x_2)\big).
\end{cases}\]
\begin{center}
	\includegraphics[width=4in]{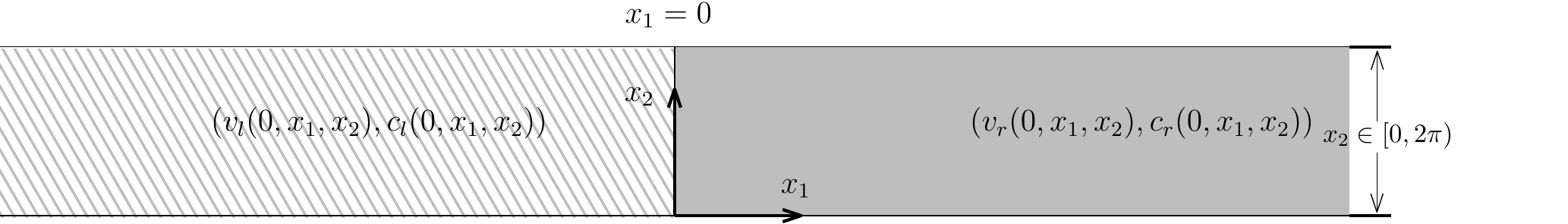}
\end{center}
We use $\varepsilon$  to quantify the size of the perturbation. The smallness of  $\varepsilon$ is understood in the following sense:
\[\varepsilon \ll w_l-w_r, \ \ \varepsilon \ll \wb_r-\wb_l.\]
We assume that the initial data $(v,c)\big|_{t=0}$ is an $\varepsilon$-perturbed Riemann data, i.e., it satisfies the following definition:
\begin{definition}\label{def:data} For a given set of initial data $(v,c)\big|_{t=0}$, we say that it is an {\bf $\varepsilon$-perturbed Riemann data} for the Euler system \eqref{eq: Euler in rho v}, if it satisfies the following assumptions:
	\begin{itemize}
		\item[1)]{\bf (smoothness)} $v^1_l(0,x_1,x_2), v^2_l(0,x_1,x_2)$ and $c_l(0,x_1,x_2)$ are smooth functions on $x_1 \leqslant 0$ (including the boundary); $v^1_r(0,x_1,x_2), v^2_r(0,x_1,x_2)$ and $c_r(0,x_1,x_2)$ are smooth functions on $x_1 \geqslant 0$.
		\item[2)]{\bf (smallness)} $(v,c)\big|_{t=0}$ is an $O(\varepsilon)$-purterbation of the fixed one dimensional data $(\mathring{v}_l, \mathring{c}_l)$ and $(\mathring{v}_r,\mathring{c}_r)$ defined as above. In other words, there exists a small constant $\varepsilon>0$ whose size will be determined in the course of the proof and a sufficiently large integer $N$ (we take $N\geqslant \Ntop+2$ where $\Ntop$ is defined in the first paper \cite{LuoYu1}\footnote{$\Ntop$ is the highest order Sobolev index appeared in the energy estimates.} and we can take $\Ntop=10$), so that
		\[\begin{cases}
			\|v^1_l(0,x)-\mathring{v}_l\|_{H^N(\Sigma_-)}+\|v^2_l(0,x)\|_{H^N(\Sigma_-)}+\|c_l(0,x)-\mathring{c}_l\|_{H^N(\Sigma_-)}<\varepsilon,\\
			\|v^1_r(0,x)-\mathring{v}_r\|_{H^N(\Sigma_+)}+\|v^2_r(0,x)\|_{H^N(\Sigma_+)}+\|c_r(0,x)-\mathring{c}_r\|_{H^N(\Sigma_+)}<\varepsilon,
		\end{cases}
		\]
		where $\Sigma_-=\mathbb{R}_{\leqslant 0}\times [0,2\pi]$ and $\Sigma_+=\mathbb{R}_{\geqslant 0}\times [0,2\pi]$.
		\item[3)]{\bf (irrotational)} The initial data $(v,c)\big|_{t=0}$ consists of locally integrable functions on $\Sigma_0$, we assume that it is irrotational in the distributional sense, i.e., 
		\begin{equation}\label{irrotational condition}
			{\rm curl}(v)\stackrel{\mathscr{D}'(\Sigma_0)}{=}0.
		\end{equation}
	\end{itemize}
\end{definition}
The data $\big(v_l(0,x_1,x_2),c_l(0,x_1,x_2)\big)$ on the left and $\big(v_r(0,x_1,x_2),c_r(0,x_1,x_2)\big)$ on the right play symmetric and exchangeable roles in the remainder of the work. Therefore, we will primarily focus on the data given on the right.

The main results of our analysis for the Riemann problem consisting of two rarefaction waves can be summarized as follows:

\medskip

\textbf{(Rough statement of the main result)} \ \ The one-dimensional wave pattern is stable under two-dimensional perturbations in the sense described above.
\medskip




It is evident that the background solution for $t>0$ is continuous and piecewise smooth. However, the transverse derivatives have a jump discontinuity on the characteristic hypersurfaces that bound the rarefaction wave regions. These hypersurfaces are known as \textbf{rarefaction fronts}, which emanate from the initial discontinuity.

As noted by Majda on page 154 of \cite{MajdaBook}, "the dominant signals in rarefaction fronts move at characteristic wave speeds." This is a fundamental difference from shock fronts, which move at non-characteristic wave speeds, and is a major technical difficulty in constructing rarefaction waves.


\begin{remark}
In general, rarefaction fronts are free boundaries that cannot be determined a priori from the data. However, the example of the one-dimensional Riemann problem with two constant states described above is special in that the two rarefaction waves are adjacent to regions of constant states and are therefore \emph{simple waves}. In particular, rarefaction fronts are straight lines, and the solutions are constant along the rarefaction fronts.	
\end{remark}

\begin{remark}\label{Remark: difference between rarefaction fronts and shock fronts}
Although the construction of rarefaction fronts and shock fronts can both be viewed as free boundary problems, they are different in nature. This is because rarefaction fronts are characteristic hypersurfaces and are subject to defining equations. More precisely, using the acoustical metric, rarefaction fronts are ruled by null geodesics and therefore satisfy the geodesic equations (ODE system). Once the rarefaction wave region has been constructed, the rarefaction fronts are no longer a free boundary problem because they can be immediately determined from the initial conditions. On the other hand, as noted by Majda in the last chapter of \cite{MajdaBook}, determining the rarefaction fronts may even be more difficult than determining the shock fronts if one does not know the solution in advance.
\end{remark}


\subsection{Prior results}
We provide a brief review of prior results. For a more detailed discussion, including singularity formation and other systems, please refer to our first paper in the series \cite{LuoYu1}. In this paper, we focus on the multi-dimensional elementary waves for the Euler system.

Riemann's fundamental idea was extended and formalized by Lax \cite{Lax1957} into a mathematical framework of hyperbolic conservation laws in one space dimension. In particular, the general Riemann problem was solved in terms of three types of elementary waves: shocks, rarefaction waves, and contact discontinuities. A significant breakthrough was achieved by Glimm's landmark work \cite{Glimm1965}, which established interaction estimates based on the Riemann problem and proved the existence of weak solutions in BV spaces. The one-dimensional theory has since developed into a mature field of mathematics. For a detailed account, we refer to the encyclopedia of Dafermos \cite{Dafermos} and Liu's recent textbook \cite{Liu2021book}.

In contrast to the success in one dimension, the multi-dimensional theory encounters many obstacles, with a major one being the breakdown of the BV space approach \cite{Rauch}. Inspired by the Riemann-Lax-Glimm program, a fundamental problem is to understand elementary waves and the Riemann problem in multi-dimension. Generally, these are free boundary problems subject to possible instabilities and coupled with initial singularities. The pioneering work of Majda \cite{MajdaShock2,MajdaShock3} proved the stability of planar shock fronts in gas dynamics and initiated the study of multi-dimensional elementary waves. For subsequent development and extensions of Majda's work on shock fronts, we refer to \cite{Metivier2001book, Benzoni-Gavage-Serre2007book} and the references therein.

Presented as an open problem at the end of Majda's book \cite{MajdaBook}, rarefaction waves are essentially different from the shock front problem: rarefaction fronts are {\it characteristic} hypersurfaces and cannot satisfy the uniform stability condition that is valid for shock fronts. This not only causes a potential loss of derivatives near rarefaction fronts but also is coupled with the initial singularity. In this direction, significant progress was achieved by Alinhac \cite{AlinhacWaveRare1,AlinhacWaveRare2}. He proved the local existence of multi-dimensional rarefaction waves for a general hyperbolic system by introducing innovative techniques, including the 'good unknown' and Nash-Moser iteration in co-normal spaces. 

However, Alinhac's work did not give a complete picture of the rarefaction front problem: not only do the estimates lose derivatives and degenerate near rarefaction fronts, but also a high-order $k$-compatibility condition on normal derivatives was required on the initial data. In particular, for a given piecewise smooth Riemann data, the rarefaction fronts bounding the rarefaction wave regions could only be described in the asymptotic limit.

Alinhac's scheme was adapted to study the Prandtl-Meyer expansion wave in \cite{WangYin} and combinations of other elementary wave patterns in \cite{Li1991,ChenLi}. There is also a different approach to study the two-dimensional Riemann problem in self-similar variables as summarized in the book \cite{Zheng2001book}. We  also mention the recent work by Q. Wang \cite{Wang}, which focuses on constructing global-in-time solutions for classical Cauchy data  to three-dimensional compressible Euler equations. In that work, the rarefaction effect is dynamically formed.

\subsection{Acoustical geometry, Riemann invariants and the second null frame}\label{section_acoustical geometry}
We set up the acoustical geometry for the rarefaction wave problem. See \cite{LuoYu1}, \cite{ChristodoulouShockFormation} or \cite{ChristodoulouMiao} for a more detailed account.

\subsubsection{Riemann invariants and the acoustical wave equation}

We only consider the case where the velocity field $v$ is {\bf irrotational} in this work. Under the irrotational assumption, there exists a potential function $\phi$ so that $v = -\nabla \phi$, i.e., $v^1=-\partial_1\phi$ and $v^2=-\partial_2\phi$. In this situation, the dynamics of $(v,c)$ are governed by the acoustical waves, i.e., the Euler equations is equivalent to the following quasi-linear wave equation in the Galilean coordinates $(t,x^1,x^2)$:
\begin{equation}\label{eq:Euler:isentropic:irrotational}
\sum_{\mu,\nu=0}^2g^{\mu \nu} \frac{\partial^2 \phi}{\partial x^\mu \partial x^\nu}  = 0,
\end{equation}
where $x^0=t$. The acoustical metric $g$ is defined by
\[g = - c^2 dt^2 + \sum_{i=1}^2 (dx^i - v^idt)^2.\]
Therefore, the restriction of $g$ on each $t$-slice of the Galilean spacetime is the standard Euclidean metric.

We also define
\[\psi_0= \frac{\partial \phi}{\partial t}, \ \psi_1= \frac{\partial \phi}{\partial x_1}=-v^1, \ \psi_2= \frac{\partial \phi}{\partial x_2}=-v^2.\]
For all $\psi \in \{\psi_0,\psi_1,\psi_2\}$,  it satisfies the following covariant wave equation
\begin{equation}\label{eq: conformal Euler}
\Box_{g} \psi = -\frac{1}{2}g(d\log(\Omega), d\psi),
\end{equation}
where $\Omega=\frac{\rho}{c}$ and $d$ is the differential of a function.

Following Riemann \cite{Riemann}, we define the Riemann invariants 
\begin{equation}\label{def: Riemann invariants}
w=\frac{1}{2}\big(\frac{2}{\gamma-1}c+\psi_1\big),  \ \ \wb=\frac{1}{2}\big(\frac{2}{\gamma-1}c-\psi_1\big),
\end{equation}
or equivalently
\[c=\frac{\gamma-1}{2}(w+\wb),  \ \ \psi_1= w-\wb.\]
The Riemann invariants satisfy the following wave equations:

  \begin{equation}\label{Main Wave equation: order 0}
\begin{cases}
\Box_{g}\wb&=-c^{-1}\left(g(D \wb,D \wb)+\frac{\gamma-3}{4}g(D \wb,D w)+\frac{\gamma+1}{4}g(D w,D w)+\frac{1}{2}g( D\psi_2, D\psi_2)\right),\\
\Box_{g} w&=-c^{-1}\left(\frac{\gamma+1}{4}g(D \wb,D \wb)+\frac{\gamma-3}{4}g(D \wb,D w)+g(D w,D w)+\frac{1}{2}g( D\psi_2, D\psi_2)\right),\\
\Box_{g}\psi_2&=-c^{-1}\left(\frac{3-\gamma}{4}g(D\wb, D \psi_2)+\frac{3-\gamma}{4}g(Dw, D \psi_2)\right).
\end{cases}
\end{equation}

\subsubsection{The acoustical geometry}\label{section:acoustical geometry}
We use $\mathcal{D}_0$ to denote the future domain of dependence determined by the data $\big(v_r(0,x_1,x_2),c_r(0,x_1,x_2)\big)$ on the right. Its boundary $C_0$ is a null hypersurface with respect to the acoustical metric $g$. We use $(v_r,c_r)=\big(v_r(t,x_1,x_2),c_r(t,x_1,x_2)\big)$ to denote the solution in $\mathcal{D}_0$.
\begin{center}
\includegraphics[width=2.5in]{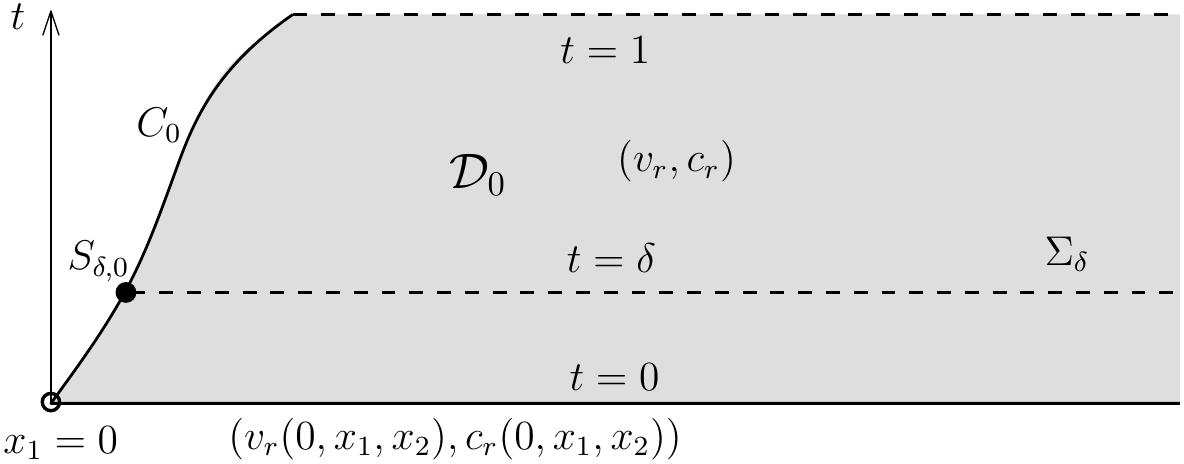}
\end{center}
Since $\varepsilon$ is small, in view of the continuous dependence of the solution on the initial data, the domain $\mathcal{D}_0$ covers up to $t=t^*=1$. For $t_0\in [0,t^*]$, we define $\Sigma_{t_0}=\{(t,x_1,x_2) |t=t_0\}$. For a small parameter $\delta$ (which will go to $0$ in a limiting process), the spatial hypersurface $\mathcal{D}_0\cap \Sigma_\delta$ is depicted as follows:
\begin{center}
\includegraphics[width=2.5in]{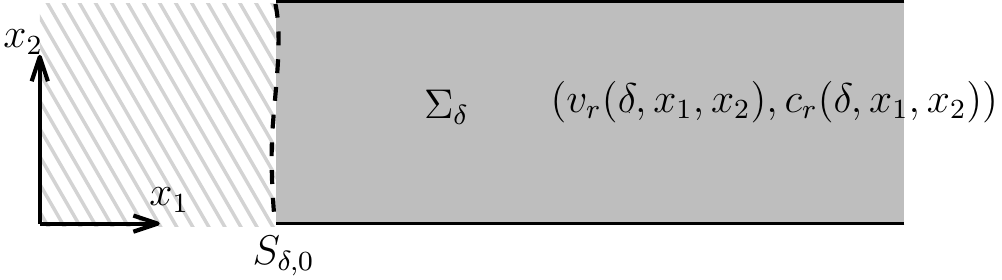}
\end{center}
We define $S_{\delta,0}=\Sigma_\delta\cap C_0$. The restriction of the solution $(v_r,c_r)$ on $\Sigma_\delta$ gives data on the right-hand side of $S_{\delta,0}$. The data for the front rarefaction waves will be given on the left-hand side of $S_{\delta,0}$ on $\Sigma_\delta$. 

We will choose a specific smooth function $u$ on $\Sigma_\delta$ so that $S_{\delta,0}$ is given by $u=0$ and the lefthand side of $S_{\delta,0}$ on $\Sigma_\delta$ is given by $u>0$, see Section \ref{section:initial foliatiion} for the construction of $u$.  The data for the front rarefaction waves will be specified for $u \in [0,u^*]$ and the parameter $u^*$, which represents the width of the rarefaction waves, is determined by $(\mathring{v}_r,\mathring{c}_r)$, see \eqref{initial size u*}. Together with the data on $C_0$, the data on $u \in [0,u^*]$ evolves to the development $\mathcal{D}(\delta)$ (the shaded region on the left of the following figure).

\begin{center}
\includegraphics[width=3.3in]{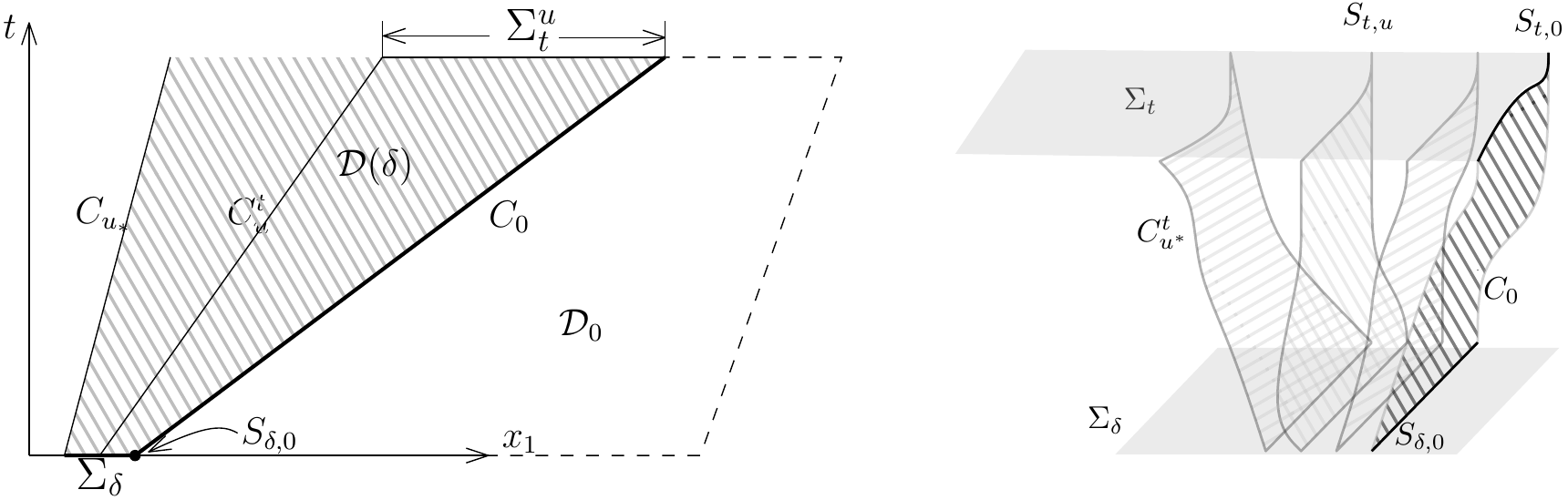}
\end{center}

We recall the definition of the acoustical coordinates $(t,u,\vartheta)$ on $\mathcal{D}(\delta)$. The function $t$ is defined as $x_0$ restricted to $\mathcal{D}(\delta)$.  We define $C_{u_0}$ to be the null hypersurfaces consisting of null (future right-going) geodesics emanating from the level set of $u=u_0$ on $\Sigma_\delta$. We define $u$ on $\mathcal{D}(\delta)$ in such a way that $C_u$'s are the level sets of $u$. Let $L$ be the generator of the null geodesics on $C_u$ subject to the normalization that $L(t) = 1$. We define $S_{t,u} = \Sigma_t \cap C_u$. Let $T$ be the vector field  tangential to $\Sigma_t$, orthogonal to $S_{t,u}$ with respect to $g$ and subject to the normalization that $Tu = 1$. We define $\kappa^2=g(T,T)$.

We also have the following notations:
\[\mathcal{D}(\delta)(t^*,u^*) =\!\!\!\!\!\!\bigcup_{(t,u) \in [\delta,t^*]\times [0,u^*]}\!\!\!\!\!\!S_{t,u}, \ \mathcal{D}(\delta)(t,u) =\!\!\!\!\!\!\bigcup_{(t',u') \in [\delta,t]\times [0,u]}\!\!\!\!\!\!S_{t',u'},\ \ \Sigma_t^u=\!\!\!\bigcup_{u' \in [0,u]}\!\!\!S_{t,u'}, \ \ C_u^{t}=\!\!\!\bigcup_{t' \in [\delta,t]}\!\!\!S_{t',u}.\]
For the sake of simplicity, we also use $\Sigma_t$ to denote $\Sigma_t^{u^*}$.

The construction of the function $\vartheta$ can be found in Section \ref{section:initial foliatiion}. Indeed, we define $\vartheta$ on $C_0$ through the following ordinary differential equation:
\begin{equation}\label{def: vartheta in acoustical}
L( \vartheta)=0, \ \ \  \vartheta\big|_{S_{0,0}}=x_2\big|_{S_{0,0}}.
\end{equation}
We then extend $\vartheta$ to $\Sigma_\delta$ by solving $T( \vartheta)=0$ with initial data on $S_{\delta,0}$ given by the solution of \eqref{def: vartheta in acoustical}. In particular, this means that  $T\big|_{\Sigma_\delta} = \frac{\partial}{\partial u}$ on $\Sigma_\delta$. Finally, we solve $L(\vartheta)=0$ to extend $\vartheta$ from $\Sigma_\delta$ to the entire spacetime $\mathcal{D}(\delta)$.

In the acoustical coordinates $(t,u,\vartheta)$, we have
\begin{equation}\label{eq: L T in terms of coordinates}
L = \frac{\partial}{\partial t}, \quad T = \frac{\partial}{\partial u} - \Xi \frac{\partial}{\partial \vartheta},
\end{equation}
where $\Xi$ is a smooth function.  We also define $X=\frac{\partial}{\partial \vartheta}$, $\slashed{g}=g(X,X)$ and the unit vector $\widehat{X}=\slashed{g}^{-\frac{1}{2}}X$. Let $\mu=c\kappa$. We then have
\[g(L,T) = -{\mu}, \quad g(L,L) = g(L,\Xh) = g(T,\Xh)=0, \quad g(\Xh,\Xh) = 1.
\]
Let $B=\frac{\partial}{\partial t}+v$ be the material vector field. We have $B(t)=1$ and $B$ is $g$-perpendicular to $\Sigma_t$.  We also define the unit vector $\widehat{T} = \kappa^{-1}T$. Thus, $L$ can be expressed as $L = \frac{\partial}{\partial t} +v-c\widehat{T}$.

We define the left-going null vector field $\underline{L} = c^{-1} \kappa L + 2T$. Therefore, we obtain {\bf the first null frame} $(L,\Lb,\Xh)$.

For the following three isometric embeddings $\Sigma_t\hookrightarrow \mathcal{D}$,  $S_{t,u} \hookrightarrow C_u $ and $S_{t,u} \hookrightarrow \Sigma_t$, the corresponding second fundamental forms are denoted as follows:
\[2 c k = \overline{\mathcal{L}}_B g, \ \ 2 \chi = \slashed{\mathcal{L}}_L g, \ \ 2 \kappa \theta = \slashed{\mathcal{L}}_T g.\]
where $\overline{\mathcal{L}}$ and $\slashed{\mathcal{L}}$ denote the projections of Lie derivatives to $\Sigma_t$ and $S_{t,u}$ respectively.
We define the torsion 1-forms $\zeta$ and $\eta$ on $S_{t,u}$ as
\[\zeta(Y) =  g(D_{Y} L, T), \ \ \eta(Y) =  -g(D_{Y} T, L),\]
where $Y$ is any vector field tangent to $S_{t,u}$. We also define the $1$-form $\slashed{\varepsilon}$ as $\kappa \slashed{\varepsilon}(Y)=k(Y,T)$.

Since the $S_{t,u}$'s are 1-dimensional circles, we can represent the tensors by functions. For the sake of simplicity, we use the same symbols to denote the following scalar functions:
\[\chi = \chi(\Xh,\Xh), \  \theta =  \theta(\Xh,\Xh), \ \slashed{k} = k(\Xh,\Xh), \ \zeta = \zeta(\Xh),\ \eta = \eta(\Xh), \ \slashed{\varepsilon}=\slashed{\varepsilon}(\Xh).\]
Since $\slashed{g}=g(X,X)$, we have $L(\slashed{g})=2 \slashed{g} \cdot \chi$. These quantities are related by
\[\chi =  c(\slashed{k} - \theta), \ \ \eta = \zeta + \Xh(\mu), \ \ \zeta=\kappa\big(c\slashed{\varepsilon}-\Xh(c)\big).\]

By using $\Lb$, we can introduce  another second fundamental form $\underline{\chi}$ which is defined by 
\[ 2 \underline{\chi} = \slashed{\mathcal{L}}_{\underline{L}} g. \]
We will also work with its scalar version $\underline{\chi} =\underline{\chi}(\Xh,\Xh)$. It can be represented by $\slashed{k}$ and  $\theta$ through $\underline{\chi} =  \kappa(\slashed{k} + \theta)$.

The wave operator $\Box_g$ can  be decomposed with respect to the null frame $(L,\Lb,\Xh)$:
\begin{equation}\label{eq:wave operator in null frame}
 \Box_{g} (f) = \Xh^2 (f) - \mu^{-1}L\big(\underline{L}(f)\big) - \mu^{-1}\big(\frac{1}{2}\chi\cdot \underline{L}(f) +\frac{1}{2}\chib\cdot L(f)\big)- 2 \mu^{-1}\zeta \cdot \Xh(f).
\end{equation}

The change of $\kappa$ along the characteristic direction $L$ is recorded in the following equation:
\begin{equation}\label{structure eq 1: L kappa}
L \kappa = m' + e' \kappa 
\end{equation}
where 
\begin{equation}\label{eq: m' e'}
m' = -\frac{\gamma+1}{\gamma-1}Tc, \ \ e'= c^{-1}\widehat{T}^i\cdot L (\psi_i).
\end{equation} 
where $\Th^i$ is the $i$-th component of $\Th$ in the Cartesian coordinates, i.e., $\Th=\sum_{i=1}^2\Th^i \frac{\partial}{\partial x_i}$. Similarly, in Cartesian coordinates, we have $\Xh=\sum_{i=1}^2\Xh^i \partial_i$ and $L=\partial_0+\sum_{i=1}^2L^i \partial_i$. We also remark that $\Th^1=-\Xh^2$ and $\Th^2=\Xh^1$. This is because $\Xh$ is perpendicular to $\Th$. 
\begin{remark}[Einstein summation convention]
The repeated Latin letter indices (say $i,j,k$) indicate the summation over $1,2$. The repeated Greek letter indices (say $\mu,\nu$) indicate the summation over $0,1,2$.  
\end{remark}


For  $k=1,2$, we have the following formulas for $L^{i}$ and $\Th^j$:
\begin{equation}\label{structure eq 3: L T on Ti Xi Li}
\begin{cases}
&L(L^{k}) =-\left(L(c) +\widehat{T}^i\cdot L(\psi_i)\right) \widehat{T}^k -\frac{\gamma+1}{2}\Xh(h) \Xh^k,\\
&L(\widehat{T}^{k})= -\kappa^{-1}\zeta \cdot \Xh^k=\left(\widehat{T}^j\cdot \Xh(\psi_j) + \Xh(c)\right)\Xh^k,\\
&T(L^{i}) = L(\kappa) \widehat{T}^i + \eta\cdot \Xh^i=L(\kappa) \widehat{T}^i + \left(-\kappa\left(\widehat{T}^j\cdot \Xh(\psi_j) + \Xh(c)\right)+\Xh(\mu)\right) \Xh^i,\\
&T(\widehat{T}^{i})= -\Xh(\kappa) \Xh^i.
\end{cases}
\end{equation}

Among all the geometric quantities, $\Th^1$, $\Th^2$ and $\kappa$ are the primitive ones. The others can be expressed explicitly using $\psi_i$, $\Th^i$ and $\kappa$. We collect those relations in the following two set of equations:
\begin{equation}\label{defining eq of theta and chi}
	\theta = \Xh^2\Xh(\Xh^1) - \Xh^1\Xh(\Xh^2), \ \ \chi  = -\Xh^i\Xh(\psi_i) - c\Xh^2\Xh(\Xh^1) + c\Xh^1\Xh(\Xh^2),
\end{equation}
and
\begin{equation}\label{structure quantities: connection coefficients}
\begin{cases}
&c k_{ij} =  - \partial_i \psi_j = - \partial_j \psi_i, \ \ \slashed{\varepsilon}=-\mu^{-1}\Xh^i T^j \partial_i \psi_j,\\
&\zeta =-\kappa\big(\widehat{T}^j\cdot \Xh(\psi_j) + \Xh(c)\big), \ \ \eta=-\kappa \widehat{T}^j\cdot \Xh(\psi_j)+c\Xh(\kappa),\\
&\chib=2\kappa \slashed{k}-\kappa\alpha^{-1}\chi=c^{-1}\kappa\big(-2\Xh^j\cdot \Xh(\psi_j)-\chi\big).
\end{cases}
\end{equation}

We also collect the following formulas of commutators:
\begin{equation}\label{eq:commutator formulas}
\begin{cases}
&[L,\Xh]=-\chi\cdot \Xh,  \ \ [L,\Lb]=- 2(\zeta + \eta)\Xh+L(c^{-1}\kappa)L, \\
&[L,T]=- (\zeta + \eta)\Xh=-\big(\kappa\big(2c\Xh^i\cdot T(\psi_i)+2\Xh(c)\big)-\Xh(\mu)\big)\Xh,\\
&[T,\Xh]=-\kappa\theta\cdot \Xh, \ \ [\Lb,\Xh]=-\chib\cdot\Xh-\Xh(c^{-2}\mu)L.
\end{cases}
\end{equation}

\subsubsection{Euler equations in the diagonal form}

The Euler equations \eqref{eq: Euler in rho v} can also be written in terms of $v$ and $c$ as follows:
\[\begin{cases}
(\partial_t + v  \cdot\nabla ) c &= -\frac{\gamma-1}{2}c \nabla \cdot v, 
\\
(\partial_t + v \cdot \nabla )v &= -\frac{2}{\gamma-1}c  \nabla  c.
\end{cases}
\]
In terms of the frame $(L,T,\Xh)$,  the Euler equations are then equivalent to
\begin{equation}\label{Euler equations:form 1}
\begin{cases}
L (\frac{2}{\gamma-1}c) &= -c \widehat{T}(\frac{2}{\gamma-1}c)+c \widehat{T}(\psi_k)\widehat{T}^k +c\Xh(\psi_k)\Xh^k,\\
L (\psi_1) &= -c \widehat{T}(\psi_1)+\frac{2}{\gamma-1}c \widehat{T}(c)\widehat{T}^1+\frac{2}{\gamma-1}c\Xh(c)\Xh^1,\\
L (\psi_2) &= -c \widehat{T}(\psi_2)+\frac{2}{\gamma-1}c \widehat{T}(c)\widehat{T}^2+\frac{2}{\gamma-1}c\Xh(c)\Xh^2.
\end{cases}
\end{equation}
Hence, we have the formulation of Euler equations in terms of Riemann invariants:
\begin{equation}\label{Euler equations:form 2}
\begin{cases}
L (\wb) &= -c \widehat{T}(\wb)(\widehat{T}^1+1)+\frac{1}{2}c \widehat{T}(\psi_2)\widehat{T}^2 +\frac{1}{2}c \Xh(\psi_2)\Xh^2-c\Xh(\wb)\Xh^1,\\
L (w) &= 	c \widehat{T}(w)(\widehat{T}^1-1)	+\frac{1}{2}c \widehat{T}(\psi_2)\widehat{T}^2 +c\Xh(w)\Xh^1+\frac{1}{2}c\Xh(\psi_2)\Xh^2,\\
L (\psi_2) &= -c \widehat{T}(\psi_2)+c\widehat{T}( w+\wb)\widehat{T}^2+c\Xh( w+\wb)
\Xh^2.
\end{cases}
\end{equation}
We now define 
\[A=\left( {\begin{array}{ccc}
-( \widehat{T}^1+1)& 0&   \frac{1}{2}\widehat{T}^2 \\
   0 &   \widehat{T}^1-1& \frac{1}{2}\widehat{T}^2\\
     \widehat{T}^2 &  \widehat{T}^2&-1
  \end{array} } \right), \ \ B=\left( {\begin{array}{ccc}
-\Xh^1& 0&   \frac{1}{2}\Xh^2 \\
   0 &   \Xh^1 & \frac{1}{2}\Xh^2\\
    \Xh^2 & \Xh^2&0
  \end{array} } \right), \ \ V=\left( {\begin{array}{c}
  \wb \\
     w\\
     \psi_2
  \end{array} } \right).\]
  The system \eqref{Euler equations:form 2} is equivalent to
  \[L(V)=c A \cdot \widehat{T}(V)+c B\cdot \Xh(V).\]
By using $(\widehat{T}^1)^2+(\widehat{T}^2)^2=1$, we can show that $A$ has three eigenvalues $0$, $-1$ and $-2$ regardless the exact values of $\widehat{T}^1$ and $\widehat{T}^2$.  We choose three eigenvectors 
\[\frac{1}{2}\left( {\begin{array}{c}
  1-\widehat{T}^1\\
   1+\widehat{T}^1 \\
    2 \widehat{T}^2
  \end{array} } \right),\ \ \frac{1}{2}\left( {\begin{array}{c}
  \widehat{T}^2\\
   -\widehat{T}^2 \\
     2\widehat{T}^1
  \end{array} } \right),\ \ \frac{1}{2} \left( {\begin{array}{c}
 1+ \widehat{T}^1\\
  1 -\widehat{T}^1 \\
     -2\widehat{T}^2
  \end{array} } \right)\] 
  corresponding to the eigenvalues $0$, $-1$ and $-2$ respectively. 
 Using these eigenvectors as columns, we can construct  
 \[P=\left( {\begin{array}{ccc}
\frac{1- \widehat{T}^1}{2}& \frac{\widehat{T}^2}{2}&  \frac{1+ \widehat{T}^1}{2} \\
    \frac{1+ \widehat{T}^1}{2}&   - \frac{\widehat{T}^2}{2}& \frac{1- \widehat{T}^1}{2}\\
  \widehat{T}^2 &\widehat{T}^1&-\widehat{T}^2
  \end{array} } \right)\]
  to diagonalize \eqref{Euler equations:form 2}. Indeed, if we define $U=P^{-1} \cdot V$, we obtain that
 \[LU=c\Lambda\cdot \widehat{T}(U)+cP^{-1} B P \cdot \Xh(U)+\left(c\Lambda P^{-1}\widehat{T}(P)-P^{-1}L(P)+cP^{-1}B \Xh(P)\right)\cdot U,\]
where 
\begin{equation}\label{def:Lambda}
\Lambda=\left( {\begin{array}{ccc}
  0 & 0   &   0 \\
  0& -1 & 0 \\
  0&  0 & -2
  \end{array} } \right).
  \end{equation}
  Since $\widehat{T}=\kappa T$, we finally obtain:
 \begin{equation}\label{eq:Euler in diagonalized form}
 LU=\frac{c}{\kappa}\Lambda\cdot {T}(U)+cP^{-1} B P \cdot \Xh(U)+\left( \frac{c}{\kappa}\Lambda P^{-1} {T}(P)-P^{-1}L(P)+cP^{-1}B \Xh(P)\right)\cdot U.
\end{equation}
The following computations may help to understand the structure of \eqref{eq:Euler in diagonalized form}:
\[P^{-1}=\left( {\begin{array}{ccc}
\frac{1-\widehat{T}^1}{2}& \frac{1+\widehat{T}^1}{2}&  \frac{\widehat{T}^2}{2} \\
   \widehat{T}^2 & -\widehat{T}^2& \widehat{T}^1\\
 \frac{1+\widehat{T}^1}{2}& \frac{1-\widehat{T}^1}{2}&-\frac{\widehat{T}^2}{2}
  \end{array} } \right),\ \ P^{-1}B=\left( {\begin{array}{ccc}
  -\frac{1}{2}\widehat{T}^2& \frac{1}{2}\widehat{T}^2 &   -\frac{1}{2}\widehat{T}^1 \\
 -1  & -1 & 0 \\
 -\frac{1}{2}\widehat{T}^2& \frac{1}{2}\widehat{T}^2  & -\frac{1}{2}\widehat{T}^1 
  \end{array} } \right),\]
  and $P^{-1}BP$ is  a constant matrix:
    \[ P^{-1}BP=\left( {\begin{array}{ccc}
  0 & -\frac{1}{2}   &   0 \\
  -1& 0 & -1  \\
  0&   -\frac{1}{2} & 0
  \end{array} } \right).  \]
In an explicit manner, we can express $U$ in terms of Riemann invariants as follows:
\begin{equation}\label{eq:explicit formula for U}
\left( {\begin{array}{c}
  U^{(0)} \\
     U^{(-1)}\\
     U^{(-2)}
  \end{array} } \right)=\left( {\begin{array}{c}
\frac{1-\widehat{T}^1}{2}\wb+\frac{1+\widehat{T}^1}{2}w+  \frac{\widehat{T}^2}{2}\psi_2 \\
   \widehat{T}^2 \wb - \widehat{T}^2 w+\widehat{T}^1\psi_2\\
  \frac{1+\widehat{T}^1}{2}\wb+ \frac{1-\widehat{T}^1}{2}w-\frac{\widehat{T}^2}{2}\psi_2
  \end{array} } \right).  \end{equation}
Conversely, we have
\begin{equation}\label{eq:explicit formula for U converse}
 \begin{cases}
   \wb&=\frac{1-\Th^1}{2}U^{(0)}+  \frac{\Th^2}{2} U^{(-1)}+ \frac{1+\Th^1}{2}U^{(-2)},\\
   w&=  \frac{1-\Th^1}{2}U^{(-2)}+\frac{1+\Th^1}{2}U^{(0)}-  \frac{\Th^2}{2} U^{(-1)},\\
   \psi_2&= \Th^1 U^{(-1)}+\Th^2U^{(0)}-\Th^2 U^{(-2)}.
   \end{cases}
  \end{equation}

  \subsubsection{The second null frame}
Following the first paper \cite{LuoYu1}, in order to derive energy estimates, we introduce
\[\Xr=\partial_2, \ \Trh=-\partial_1,  \ \Lr=\partial_t+v-c\Trh=\partial_t+(v^1+c)\partial_1+v^2\partial_2,\]
and
\[\kappar=t, \ \ \Tr=\kappar \Trh, \ \ \mur = c\kappar.\]
The vector fields satisfy the following metric relations:
\[g(\Lr,\Tr)=-\mur, \ g(\Lr,\Lr)=g(\Lr,\Xr)=0,\  g(\Xr,\Xr)=1, \ g(\Tr,\Tr)=\kappar^2, \ g(\Tr,\Xr)=0.\]
Let $\Lbr= c^{-1} \kappar \Lr + 2\Tr$. We then obtain {\bf the second null frame}
$(\Lr, \Lbr, \Xr)$. One can check that
\[g(\Lr,\Lbr)=-2\mur, \ g(\Lr,\Lr)=g(\Lbr,\Lbr)=g(\Lbr,\Xr)=g(\Lr,\Xr)=0,\  g(\Xr,\Xr)=1.\]
We introduce functions $y$, $\yr$, $z$ and $\zr$ as follows:
\begin{equation}\label{def: yring zring}
y=\Xr(v^1+c), \ \ \yr = \frac{y}{\kappar},  \ \ z=1+\Tr(v^1+c), \ \ \zr = \frac{z}{\kappar}.
\end{equation}
We list the definitions and formulas for the connection coefficients in the second null frame as follows:
\begin{equation*}
	\begin{cases}
		&\chir:=g(D_{\Xr} \Lr,\Xr)=-\Xr(\psi_2), \ \ \chibr :=  g(D_{\Xr} \Lbr, \Xr)=c^{-1}\kappar \chir=-c^{-1}\kappar\Xr(\psi_2), \\ 
		&\zetar:=  g(D_{\Xr} \Lr, \Tr)=-\kappar y, \ \ \etar:=  -g(D_{\Xr} \Tr, \Lr)=\zetar+\Xr(\mur)=ck(\Tr,\Xr)=-\Tr(\psi_2), \\
		&\deltasr:=g(D_{\Lr} \Lr,\Xr)=cy,\ \ \deltar:=g(D_{\Lr} \Lr,\Tr)=-\Lr(\mur)+cz.
	\end{cases}
\end{equation*}
The commutators for the new vector fields are collected as follows:
\begin{equation*}
\begin{cases}
&[\Tr, \Xr]=0, \ \ [\Lr, \Xr]=\yr\cdot \Tr-\chir\cdot\Xr, \ \ [\Lr, \Tr]=\zr\cdot \Tr-\etar\Xr,\\
&[\Lbr, \Xr]=-\left(\frac{1}{2}c^{-2}\kappar y+\Xr(c^{-1}\kappar)\right)\Lr-\chibr\cdot\Xr+\frac{1}{2}c^{-1}y\cdot\Lbr, \\
& [\Lr, \Lbr]=\left(\Xr(c^{-1}\kappar)-c^{-1}z\right)\Lr-2\etar\cdot \Xr+\zr\cdot \Lbr.
\end{cases}
\end{equation*}

Similar to \eqref{Euler equations:form 1}, we can rewrite the Euler equations in the following form:
\begin{equation}\label{Euler equations:form 1 ringed}
\begin{cases}
\Lr (\frac{2}{\gamma-1}c) &= -c \Trh(\frac{2}{\gamma-1}c)-c \Trh(\psi_1) +c\Xr(\psi_2),\\
\Lr (\psi_1) &= -c \Trh(\psi_1)-c \Trh\left(\frac{2}{\gamma-1}c\right),\\
\Lr (\psi_2) &= -c \Trh(\psi_2)+c\Xr\left(\frac{2}{\gamma-1}c\right).
\end{cases}
\end{equation}
In terms of the Riemann invariants, \eqref{Euler equations:form 1 ringed} reduces to a simple form
\begin{equation}\label{Euler equations:form 3}
\begin{cases}
\Lr (\wb) &= \frac{1}{2}c \Xr(\psi_2),\\
\Lr (w) &= 	-2c \Trh(w)+\frac{1}{2}c\Xr(\psi_2),\\
\Lr (\psi_2) &= -c \Trh(\psi_2)+c\Xr( w+\wb).
\end{cases}
\end{equation}

We also recall the following notations:
A \emph{multi-index} $\alpha$ is a string of numbers $\alpha=(i_1,i_2,\cdots,i_n)$ with {\color{black}$i_j=0$ or $1$} for $1\leqslant j\leqslant n$. The \emph{length} of the multi-index $\alpha$ is defined as $|\alpha|=n$. Given a multi-index $\alpha$ and a smooth function $\psi$, the shorthand notations $Z^\alpha(\psi)$ and  $\Zr^\alpha(\psi)$ denote the following functions:
\[Z^\alpha(\psi)=Z_{(i_N)}\left(Z_{(i_{N-1})}\left(\cdots \left(Z_{(i_1)}(\psi)\right)\cdots\right)\right), \ \Zr^\alpha(\psi)=\Zr_{(i_N)}\big(\cdots \big(\Zr_{(i_1)}(\psi)\big)\cdots \big),\]
where $Z_{(0)}=\Xh$, $Z_{(1)}=T$, $\Zr_{(0)}=\Xr$ and $\Zr_{(1)}=\Tr$.

\subsubsection{An effective domain in $\mathcal{D}(\delta)$}\label{section:effective domain}

For every point $(\delta,0,\vartheta)\in S_{\delta,0}$, we consider the integral curve $\underline{\ell}(\vartheta)$ of $\Lb$ emanated from this point and inside  $\mathcal{D}(\delta)$. The congruence of all such curves defines a codimension one hypersurfaces in $\mathcal{D}(\delta)$:
\[\overline{C}_\delta=\bigcup_{\vartheta\in [0,2\pi]}{\color{black}\underline{\ell}(\vartheta)}.\]
By definition, $\Lb$ is tangential to $\underline{\ell}(\vartheta)$. Therefore,  $\underline{\ell}(\vartheta)$ is a null curve with respect to the acoustical metric $g$. Hence, $\overline{C}_\delta$ is a causal hypersurface, i.e., for all $p\in \overline{C}_\delta$, $T_p\overline{C}_\delta$ is either null or timelike. Indeed, we recall that
\[D_{\underline{L}} \underline{L} = \big(\mu^{-1}\underline{L}\mu + L(c^{-1}\kappa)\big)\underline{L} - 2\mu  \Xh(c^{-1}\kappa)\cdot \Xh.\]
Therefore, $\overline{C}_\delta$ is a null hypersurface if and only if $\Xh(c^{-1}\kappa)=0$.
\begin{center}
\includegraphics[width=4in]{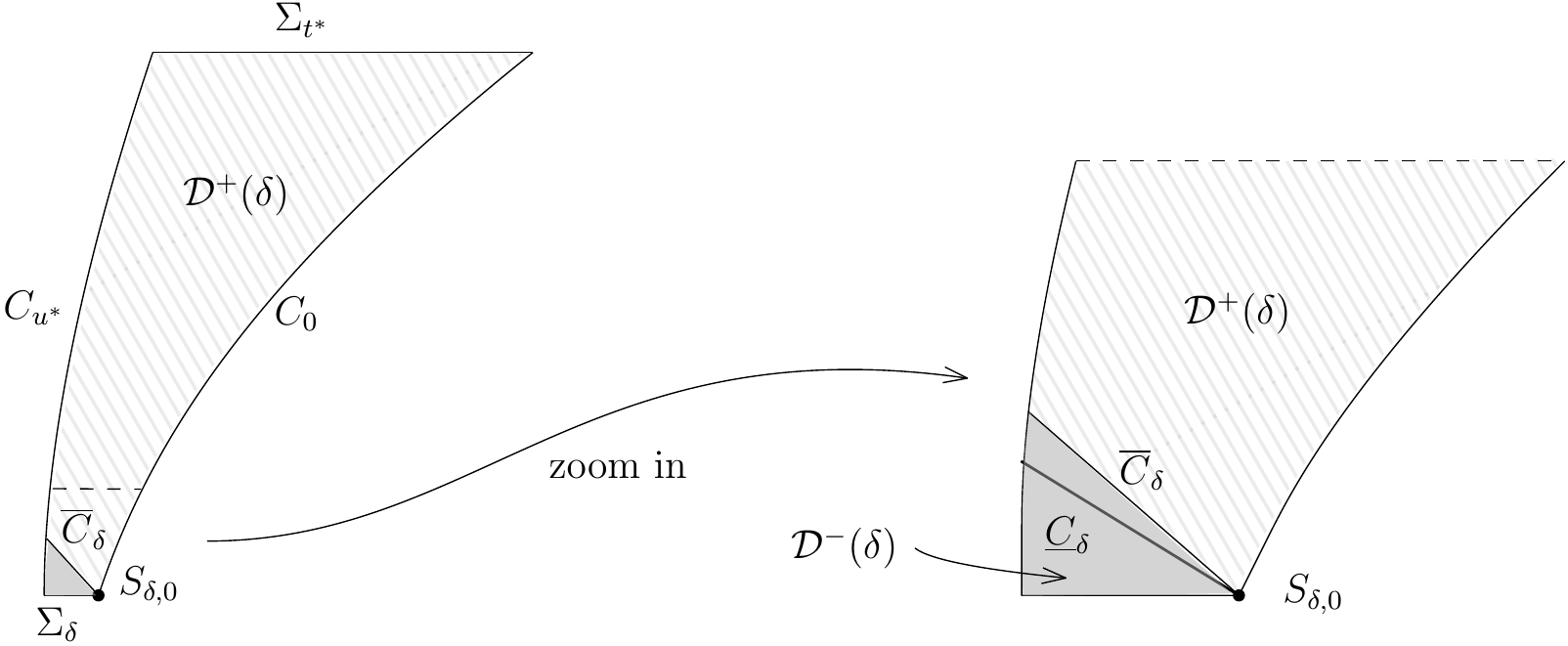}
\end{center}
The hypersurface $\overline{C}_\delta$ separates $\mathcal{D}(\delta)$ into two parts $\mathcal{D}^+(\delta)$ and $\mathcal{D}^-(\delta)$. The $\mathcal{D}^+(\delta)$ is the upper part, i.e., the one contains $C_0$. 

On the other hand, we can use $\Cb_{\delta}$ to denote the left-going null hypersurface emanated from $S_{\delta,0}$. It is the right boundary of the future domain of dependence of $\Sigma_\delta$. Since $\overline{C}_\delta$ is a causal hypersurface, it is in the future of $\Cb_\delta$. This is depicted on the right of the above picture. 

We will show that the solution $(\wb,w,\psi_2)$ associated to a given data are bounded for all derivatives  in $\mathcal{D}^+(\delta)$. We do not have effective control on solutions with multiple $L$ derivatives in $\mathcal{D}^-(\delta)$. Please see Section \ref{Section:The rough bounds on L derivatives} and \ref{Section:The precise bounds on L derivatives}.

\section{Main theorems}

\subsection{Construction of initial data and existence theorem}

\subsubsection{Formulation for the data on $\Sigma_\delta$}\label{section:compatibility}

The initial data for $(v,c)$ or equivalently $(\wb,w,\psi_2)$ is already prescribed on the acoustical null hypersurface $C_0$. This is because $C_0$ is the future boundary of the domain of dependence $\mathcal{D}_0$ of the solution $\big(v_r(t,x_1,x_2),c_r(t,x_1,x_2)\big)$ associated to the data $\big(v_r(0,x_1,x_2),c_r(0,x_1,x_2)\big)$ given on $x_1\geqslant 0$. The trace of $\big(v_r(t,x_1,x_2),c_r(t,x_1,x_2)\big)$ on $C_0$ is well-defined. It suffices to prescribe initial data on $\Sigma_\delta$. This is depicted in the following picture:
\begin{center}
\includegraphics[width=4.5in]{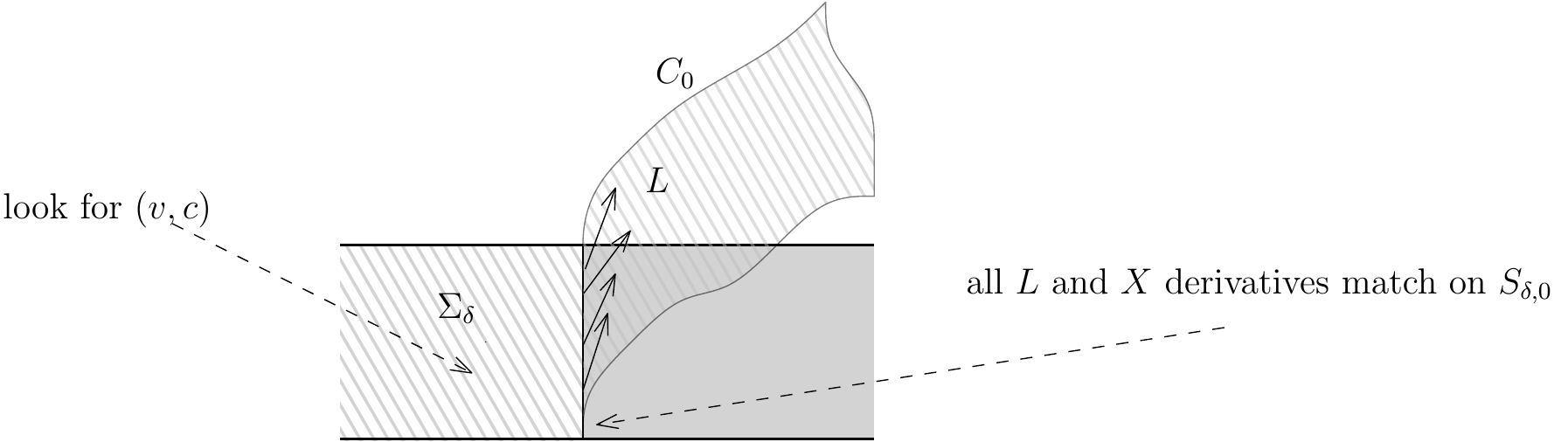}
\end{center}
In particular, since $L$ and $X$ are tangential to $C_0$, for all $m$ and $n$, $L^mX^n v$ and $L^mX^n c$ are already determined by the solution on $\mathcal{D}_0$. We remark that $X=\frac{\partial}{\partial \vartheta}$ and $[L,X]=0$. We also remark that on $S_{\delta,0}$, the vector field $L=\partial_t+v-c\widehat{T}$ is completely determined by the solution on $\mathcal{D}_0$. This is because {\color{black}$\widehat{T}$ is the unit normal of $S_{\delta,0}$ in $\Sigma_{\delta}$.}

To summarize, we will look for the data $(v,c)$ on $\Sigma_\delta$ where $u\in [0,u^*]$ so that the following compatibility conditions are satisfied:
\begin{itemize}
\item[C1)](Smoothness in acoustical coordinates) We require that $(v,c)$ are smooth functions in $(u, \vartheta)$ for $(u,\vartheta)\in [0,u^*]\times [0,2\pi]$. 
\item[C2)](Continuity across $S_{\delta,0}$) We require that  $(v,c)\big|_{S_{\delta,0}}=(v_r,c_r)\big|_{S_{\delta,0}}$. 
\item[C3)](Compatibility in higher order derivatives) We require that higher jets of $(v,c)$ along $C_0$ are compatible with those of $(v_r,c_r)$ at $S_{\delta,0}$, i.e., for a fixed positive integer $N$, for all nonnegative integers $m$ and $n$ with $m+n\leqslant N$, we have
\begin{equation}\label{compatibility condition 3}L^mX^n  {v\choose c }\big|_{S_{\delta,0}}=L^mX^n{v_r\choose c_r } \big|_{S_{\delta,0}}.
  \end{equation}
\end{itemize}

We recall that we have used a set of initial ansatz in \emph{a priori} energy estimates derived in \cite{LuoYu1}. Therefore, in addition to C1), C2) and C3),  we require that the data for rarefaction waves should verify the ansatz $\mathbf{(I_0)}, \mathbf{(I_2)}$, $\mathbf{(I_\infty)}$ and $\mathbf{(I_{irrotational})}$ in \cite{LuoYu1}.

\begin{definition}\label{def:rare data}
	Given $(v,c)$ on $\Sigma_\delta$ which satisfies the above condition C1),C2) and C3), we say that $(v,c)$ is a {\bf $C^N$ data}. We say that it is a \textbf{$C^N$ data for rarefaction waves}, or rarefaction data for short, if it satisfies in addition the rarefaction ansatz $\mathbf{(I_0)}, \mathbf{(I_2)}$, $\mathbf{(I_\infty)}$ and $\mathbf{(I_{irrotational})}$, 
\end{definition}

The first main result of the current paper is to construct initial data so that all the conditions and ansatz are realized. 
For the sake of completeness, we list the ansatz $\mathbf{(I_0)}, \mathbf{(I_2)}$ and $\mathbf{(I_\infty)}$ as follows. We  break and rearrange the ansatz $\mathbf{(I_\infty)}$ so that it suits the proof in Section \ref{section: data construction}. 

The ansatz $\mathbf{(I_0)}$ is given as follows:
\begin{equation}\label{initial size u*}
\mathbf{(I_0)} \ \ \ u^*=c_0\cdot\frac{\gamma+1}{\gamma-1}\mathring{c}_r. 
\end{equation}
We remark that the constant $c_0$ has been set to be $\frac{1}{2}$ in the first paper \cite{LuoYu1}. It can be any given positive constant less than $1$ (to exclude vacuum)  and does not affect the proof.

The ansatz $\mathbf{(I_2)}$ is given as follows:
\begin{equation}\label{initial I2}
\mathbf{(I_2)}\begin{cases}
&\mathscr{E}_{n}(\psi)(\delta,u^*)\lesssim \varepsilon^2 \delta^2, \ \mathscr{F}_{n}(\psi)(t,0)\lesssim \varepsilon^2 t^2, \ \ \psi \in \{w,\wb,\psi_2\},  \ \ 1\leqslant n\leqslant \Ntop;\\
    &\mathcal{E}(\psi)(\delta,u^*)+\underline{\mathcal{E}}(\psi)(\delta,u^*)\lesssim\varepsilon^2 \delta^2, \ \ \psi \in \{w,\psi_2\};\\
        &\mathcal{F}(\psi)(t,0)+\underline{\mathcal{F}}(\psi)(t,0)\lesssim\varepsilon^2 t^2, \ \ \psi \in \{w,\psi_2\};
\end{cases}
\end{equation}
where $t\in [\delta,t^*]$. We recall the definitions for the energy and flux:
\begin{equation}\label{def: energy L}
\mathcal{E}(\psi)(t,u)=\frac{1}{2}\int_{\Sigma_t^u}c^{-1}\kappa \left(c^{-1}\kappa (L\psi)^2+\mu (\Xh\psi)^2\right),\ \ \mathcal{F}(\psi)(t,u)=\int_{C_u^t}  c^{-1}\kappa(L\psi)^2,
\end{equation}
and
\begin{equation}\label{def: energy Lb}
\Eb(\psi)(t,u)=\frac{1}{2}\int_{\Sigma_t^u} (\Lb\psi)^2+\kappa^2 (\Xh \psi)^2, \ \ \Fb(\psi)(t,u)=\int_{C_u^t}c\kappa  (\Xh \psi)^2.
\end{equation}
{\color{black} 
For all $n\leqslant N_{_{\rm top}}$, we recall that
\[\mathscr{E}_{n}(\psi)(t,u)= \sum_{|\alpha|= n} \mathscr{E}_{\alpha}(\psi)(t,u), \ \ \mathscr{F}_{n}(\psi)(t,u)= \sum_{|\alpha|= n} \mathscr{F}_{\alpha}(\psi)(t,u),
\]
where $\mathscr{E}_{\alpha}(\psi)$ and $\mathscr{F}_{\alpha}(\psi)$ are the total energy and the total flux associated to $\Zr^\alpha(\psi)$ as follows:
\[\begin{cases}
	&\mathscr{E}_{\alpha}(\psi)(t,u)= \mathcal{E}\big(\Zr^\alpha(\psi)\big)(t,u)+\Eb\big(\Zr^\alpha(\psi)\big)(t,u),\\
	&\mathscr{F}_{\alpha}(\psi)(t,u) = \F\big(\Zr^\alpha(\psi)\big)(t,u)+\Fb\big(\Zr^\alpha(\psi)\big)(t,u).
\end{cases}
\]
}
Finally, we list the ansatz $\mathbf{(I_\infty)}$ and $\mathbf{(I_{irrotational})}$:

\begin{equation}\label{initial Iinfty1}
\mathbf{(I_{\infty,1})}
\begin{cases}
&\|\slashed{g} - 1\|_{L^\infty(\Sigma^{u}_\delta)} +  \|\frac{\kappa}{\delta}-1\|_{L^\infty(\Sigma^{u}_\delta)} + \|\Th^2\|_{L^\infty(\Sigma^{u}_\delta)} \lesssim  \varepsilon \delta, \ \|\Th^1+1\|_{L^\infty(\Sigma^{u}_\delta)} \lesssim \varepsilon^2 \delta^2;\\
	& \|Z(\slashed{g})\|_{L^\infty(\Sigma^{u}_\delta)}  \lesssim \varepsilon \delta, \ \|Z^\alpha(\kappa)\|_{L^\infty(\Sigma^{u}_\delta)} \lesssim \varepsilon \delta^2,  \ \  1 \leqslant |\alpha| \leqslant 2;\\
	& \|Z^\alpha(\Th^1)\|_{L^\infty(\Sigma^{u}_\delta)}\lesssim \varepsilon^2 \delta^2, \ \|Z^\alpha(\Th^2)\|_{L^\infty(\Sigma^{u}_\delta)}\lesssim \varepsilon \delta,  \ \  1 \leqslant |\alpha| \leqslant 2.
\end{cases}
\end{equation}

\begin{equation}\label{initial Iinfty2}
\mathbf{(I_{\infty,2})}
\begin{cases}
	&\|L(\psi)\|_{L^\infty(\Sigma^{u}_\delta)} + \|\Xh(\psi)\|_{L^\infty(\Sigma^{u}_\delta)} \lesssim \varepsilon;\\
	&\|T(w)\|_{L^\infty(\Sigma^{u}_\delta)} + \|T(\psi_2)\|_{L^\infty(\Sigma^{u}_\delta)} +  \|T\wb + \frac{2}{\gamma+1}\|_{L^\infty(\Sigma^{u}_\delta)} \lesssim \varepsilon \delta;\\
	&\|LZ^{\alpha}\psi\|_{L^\infty(\Sigma^{u}_\delta)}+\|\Xh Z^{\alpha}\psi\|_{L^\infty(\Sigma^{u}_\delta)}+\delta^{-1}\|TZ^{\alpha}\psi\|_{L^\infty(\Sigma^{u}_\delta)} \lesssim \varepsilon, \ \ 1 \leqslant |\alpha| \leqslant 2.
	\end{cases}
\end{equation}
\begin{equation}\label{initial irrotational}
	\mathbf{(I_{irrotational})} \ \ \ \frac{\partial v^2}{\partial x^1}\Big|_{\Sigma^{u}_\delta} = \frac{\partial v^1}{\partial x^2}\Big|_{\Sigma^{u}_\delta}.
\end{equation}
In the formulas of \eqref{initial Iinfty2}, $\psi\in \{\wb,w,\psi_2\}$ and $Z \in \{\Xh, T\}$.

\begin{remark}\label{rem:asymptotic condition}
	The ansatz $\mathbf{(I_\infty)}$ implies the following \textbf{asymptotic condition} for rarefaction waves:
	\[ \frac{\partial x^1}{\partial u} \approx \kappa \sim t, \quad T\wb \sim -\frac{2}{\gamma+1}, \ T\psi_2 = O(t\varepsilon), \ Tw = O(t\varepsilon), \]
	as $t\rightarrow 0$. In view of the condition C1) of the Definition \ref{def:rare data}, it implies in particular that the data is singular in $(x_1,x_2)$. 
\end{remark}

\begin{remark}
	In applications, we take $N=\Ntop+1$. We recall that $\Ntop$ is the maximal number of derivatives needed for the energy estimates in the first paper \cite{LuoYu1}. Since we are only concerned with solutions that are sufficiently smooth, the exact value of $\Ntop$ does not matter during the proof since it can be arbitrarily large.
\end{remark}

\subsubsection{Existence of data and solution for $\mathcal{D}(\delta)$}
The first theorem of the paper is as follows:
\begin{TheoremDataExistence}[Existence of data on $\Sigma_\delta$]
For all $\delta \in (0,t^*)$ and for all $N$, there exists a $C^N$ data for rarefaction waves defined on $\Sigma_\delta$, i.e., the ansatz $\mathbf{(I_0)}, \mathbf{(I_2)}$, $\mathbf{(I_{\infty,1})}$, $\mathbf{(I_{\infty,2})}$ and $\mathbf{(I_{irrotational})}$ hold.
\end{TheoremDataExistence}
The proof of the theorem is inspired by the last slice argument of Christodoulou and Klainerman in \cite{CK}, and is detailed in Section \ref{section: data construction}. By the \emph{a priori energy estimates} derived in \cite{LuoYu1} and the Sobolev inequalities, the theorem has the following immediate consequence:

\begin{CorollaryExistenceOnDdelta}[Existence of solution on $\mathcal{D}(\delta)$]For the given data constructed in {\bf Theorem 1}, there exists a unique solution $(v,c)$ to the Euler equations \eqref{eq: Euler in rho v} defined on $\mathcal{D}(\delta)$ so that 
\begin{equation*}
(\wb ,w ,\psi_2)\in C^0 \big((0,t^*]; H^k(\Sigma_t) \big),
\end{equation*}
for all $k\leqslant N$. Moreover, we have
\begin{equation*}
(\wb ,w ,\psi_2)\in C^{N-3} ([0,t^*]\times [0,u^*]\times [0,2\pi] ) \cap C^0\big([0,t^*];C^{N-2} ( [0,u^*]\times [0,2\pi])\big).
\end{equation*}
\end{CorollaryExistenceOnDdelta}
\begin{remark}
If one uses $(v,c)$ on $\mathcal{D}(\delta)$ and uses $(v_r,c_r)$ on $\mathcal{D}_0$, this defines a continuous solution to the Euler equations \eqref{eq: Euler in rho v}. It is not a $C^1$ solution: On $C_0$ from the $\mathcal{D}_0$ side,  $\frac{\partial \wb_r}{\partial x_1}$ is of size $O(\varepsilon)$; On $C_0$ from the $\mathcal{D}(\delta)$ side,  $\frac{\partial\wb}{\partial x_1}$ is of size $O(t^{-1})$.
\end{remark}

\subsubsection{Existence of rarefaction waves connected to the data on the right}
For a constant state on the right in one-dimensional case,  it can be connected by a front rarefaction wave. We show that  there is an analogue in multi-dimensional cases. In view of \eqref{eq: precise solution for front rarefaction waves} for the constant states $(\mathring{v}_r,\mathring{c}_r)$, we have
\[{\color{black}\mathring{c}_r=\frac{\gamma-1}{\gamma+1}\big(\frac{x}{t}+2\mathring{w}_r\big).}\]
Thus, the zero set of the above function defines
\[\mathring{u}_*=-2\mathring{w}_r>0.\]
Let $\varepsilon_0$ be a small universal constant so that $\varepsilon \ll \varepsilon_0$. Its definition will be given in \eqref{def: varepsilon0}. We define the following region:
\[\mathcal{W}=\big\{(t,x_1,x_2)\in \mathbb{R}^3-\mathcal{D}_0\big| 0< t\leqslant t^*,  \frac{x_1}{t}\geqslant \mathring{u}_*+\varepsilon_0\big\}.\]
\begin{center}
\includegraphics[width=3.2in]{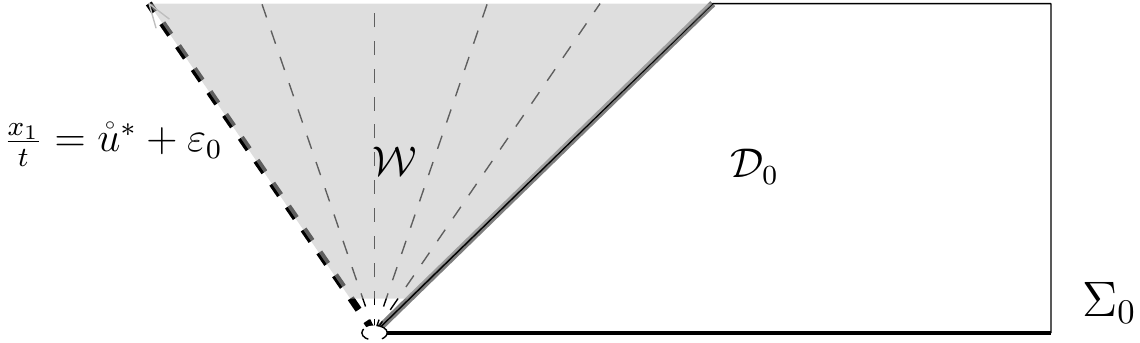}
\end{center}
\begin{TheoremExistenceRarefactionWaves}[Existence of rarefaction waves]\label{thm:existence-rarefaction}
	There exists an irrotational solution $(v,c)$ to the Euler equations \eqref{eq: Euler in rho v} defined on $\mathcal{W}$ so that 
\begin{itemize}
\item[1)] $(v,c)\big|_{C_0}=(v_r,c_r)\big|_{C_0}$;
\item[2)] $(v,c)\in C^0 \big((0,t^*]; H^k(\Sigma_t) \big)$ for all $k\leqslant N$;
\item[3)] For all $t\in (0,t^*]$, we have $\big|-t\frac{\partial\wb}{\partial x_1}+\frac{2}{\gamma-1}\big|\lesssim t\varepsilon$. Therefore,  the continuous solution to the Euler equations defined by $(v,c)$ on $\mathcal{D}_0$ and by $(v_r,c_r)$ on $\mathcal{W}$ is not a $C^1$ solution.
\end{itemize}
\end{TheoremExistenceRarefactionWaves}
{\color{black}
The theorem describes all possible local centered rarefaction waves that can be continuously connected to the given solution in $\mathcal{D}_0$, see Section \ref{section:uniqueness-centered-rare} for the uniqueness part of  {\bf Theorem 2}. 
}

\begin{remark}
The proof of {\bf Theorem 2} relies on a limiting process that may result in the loss of derivatives.  Once again, since we are only concerned with solutions that are sufficiently smooth, the exact value of $N$ does not matter during the proof since it can be taken to be arbitrarily large.
\end{remark}

\begin{remark}[Foliation]

In one-dimensional conservation laws, the manifold where we pose initial data has only one dimension, and it is foliated by sub-manifolds (or points) of co-dimension $1$ in a unique way. Consequently, the characteristic hypersurfaces (curves) emanating from the singularity are also unique.

Although the situation in multi-dimensions is much more complicated than one-dimensional cases, we show that it is still feasible to generate a canonical foliation of characteristic hypersurfaces (represented by $C_u$) in the rarefaction wave region $\mathcal{W}$, see Section  \ref{Section: app to Riemann 1}.  The proof there also shows that the foliations $C_u$'s are $O(\varepsilon)$-close to the one dimensional (or more precisely two dimensional cases with plane symmetry) rarefaction wave fronts.

In Remark \ref{Remark: difference between rarefaction fronts and shock fronts} (and proved in Section  \ref{Section: app to Riemann 1}), we discussed that a rarefaction wave front $H$ in $\mathcal{W}$ is a null hypersurface determined by its initial trace in the singularity $\mathbf{S}_*$. The singularity $\mathbf{S}_*$ can be intrinsically viewed as a two-dimensional manifold with a canonical coordinate system $(u,\vartheta)$. Instead of using the level set of $u$ as the initial data for the null hypersurfaces, we can also use other curves that are defined as graphs by $u=f(\vartheta)$ to generate rarefaction wave fronts. 
\begin{center}
\includegraphics[width=4in]{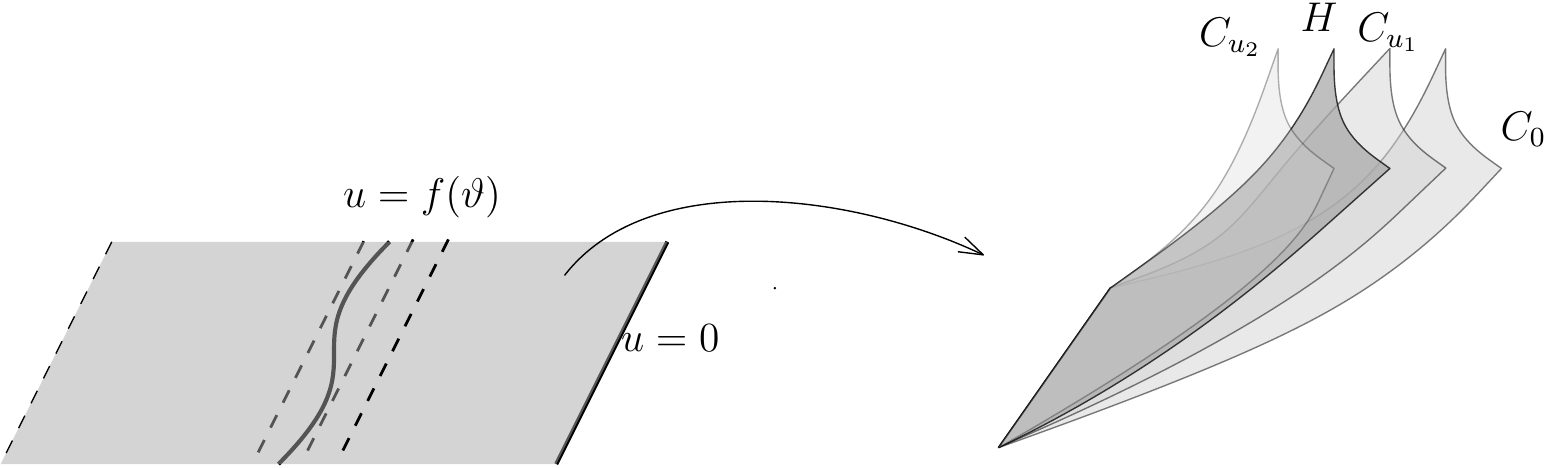}
\end{center}
\end{remark}
{\color{black}
	\begin{remark}[Singular boundary]\label{remark:singular-boundary}
		The singular boundary $\mathbf{S}_*$ is of codimension $2$ in the background spacetime. This is different from the generic singular boundary of the maximal development in the case of shock formations studied by Christodoulou in \cite{ChristodoulouShockFormation} (or in \cite{ChristodoulouMiao}). In that case, the  singular boundary is of codimension $1$. 
		
		In terms of the extrinsic geometry, the function $\kappa$ degenerate on $\mathbf{S}_*$, i.e., all $\frac{\partial}{\partial u}$ and  $\frac{\partial}{\partial \vartheta}$ derivatives of $\kappa$ vanishes on $\mathbf{S}_*$. This simplifies the causal geometry of the singular boundary: All null geodesics starting at $\mathbf{S}_*$ are outgoing, see Section \ref{section:geometric constructions} and Proposition \ref{prop:geometric constructions}, in contrast to the trichotomy of outgoing, incoming and the other null geodesics starting from the singulary boundary in the case of shock formations in \cite{ChristodoulouShockFormation,ChristodoulouMiao}. This fact will lead to a geometric construction of a canonical acoustical coordinate $(t,u,\vartheta)$, see Section \ref{Section: app to Riemann 1}.
	\end{remark}

}
\subsection{Application to the Riemann problem and uniqueness}

\subsubsection{Application to the Riemann problem}
For the classical Riemann problem in one dimension, the shape of the rarefaction wave fronts is a \emph{fan}. We show that in multi-dimension theory the shape becomes an \emph{opened book} and the structure is the same as in one dimension:
\begin{TheoremStructuralStability}[Structural stability of the Riemann problem]
We use $\mathbf{S}_*$ to denote the singularity:
\[\mathbf{S}_*:=\big\{(t,x_1,x_2)\big|t=0,x_1=0\big\}.\]
There exists a constant $\varepsilon_*>0$, for all $\varepsilon<\varepsilon_*$ and for any given data $(v,c)\big|_{t=0}$ in Definition \ref{def:data}, there exists a continuous solution to the Euler equations \eqref{eq: Euler in rho v} on $[0,t^*]\times \mathbb{R}^2-\mathbf{S}_*$.  
\begin{center}
\includegraphics[width=3in]{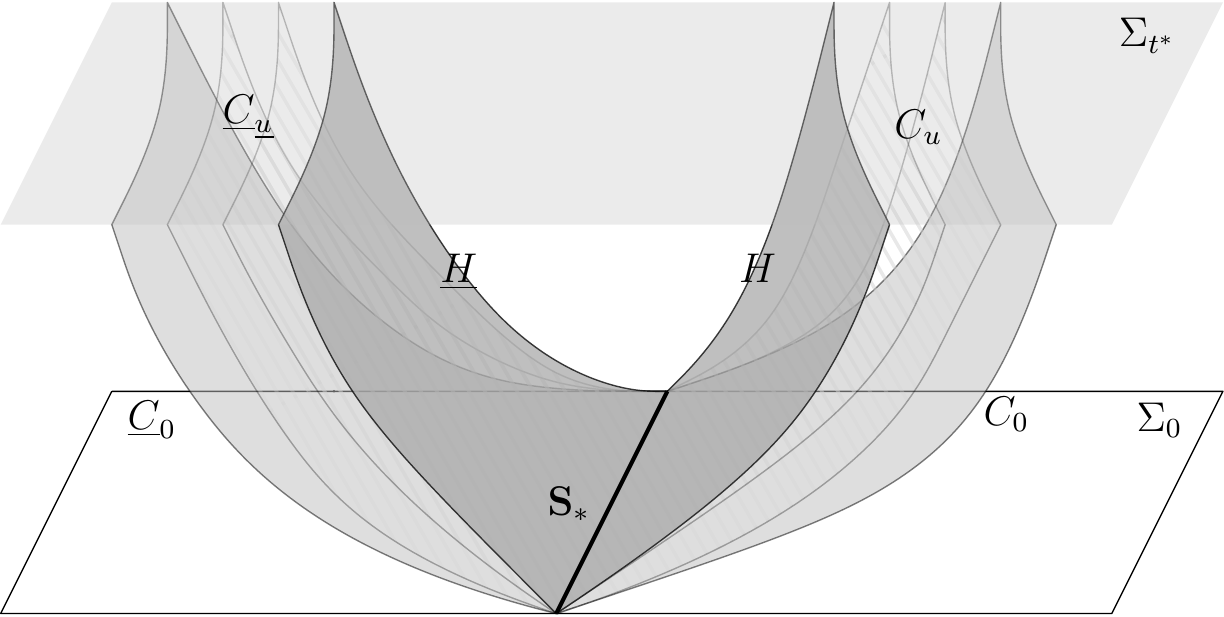}
\end{center}
The solution is piecewisely $C^N$ in the following sense:
\begin{itemize}
\item[1)] There are four characteristic hypersurfaces all emanated from the singular set $\mathbf{S}_*$ and we denote them as $\Cb_0, \Hb,H$ and $C_0$ from left to right (as $x_1$ increases). 
\begin{itemize}
\item The (gray) hypersurfaces $\Cb_0$  and $C_0$ are the characteristic boundaries of the domain of dependence of the data $\big(v_l(0,x_1,x_2),c_l(0,x_1,x_2)\big)$ posed on $x_1\leqslant 0$ and  $\big(v_r(0,x_1,x_2),c_r(0,x_1,x_2)\big)$ posed on $x_1\geqslant 0$ respectively;
\item The (dark gray) hypersurfaces $\Hb$ and $H$ are uniquely determined by the data $(v,c)\big|_{t=0}$.
\end{itemize}
\item[2)] The solution is of class $C^N$ for all points $(t,x_1,x_2)$ with $t\geqslant 0$ on the left of $\Cb_0$ or on the right of $C_0$.
\item[3)] The solution is of class $C^N$ for all points $(t,x_1,x_2)$ with $t>0$ between $\Cb_0$ and $\Hb$ or between  $H$ and $C_0$. These two regions are the back and front rarefaction wave regions respectively.
\item[4)] The solution is of class $C^N$ for all points $(t,x_1,x_2)$ with $t\geqslant 0$ (including the singularity) between $\Hb$ and $H$. 
\end{itemize}
Moreover, the solution is not of class $C^1$ on $C_0,\Cb_0,H$ and $\Hb$.
\end{TheoremStructuralStability}

\begin{remark}
The back rarefaction wave region and the front rarefaction wave region can be foliated by the characteristic hypersurfaces ${\underline{C}}_{\ub}$ and ${C}_{u}$ respectively. For each $u$ and $\ub$, ${\underline{C}}_{\ub}$ and ${C}_{u}$ are  $\varepsilon$-close to $\mathring{\underline{C}}_{\ub}$ and  $\mathring{C}_{u}$ which are the corresponding characteristic hypersurfaces arising from the one dimensional constant states $(\mathring{v}_l, \mathring{c}_l)$ and $(\mathring{v}_r,\mathring{c}_r)$ , as $\varepsilon$ approaches zero, the above picture converges to the one-dimensional picture.
\end{remark}


{\color{black}


\begin{remark}
	We have made the following assumption for the sake of simplicity: the initial discontinuity is across a \emph{straight} line (more precisely a circle) on $\Sigma_0$. To go beyond this limitation, i.e., extending the theorems to the general case when the initial discontinuity is an arbitrary smooth curve, one should make the following modifications: the Riemann invariants  should be chosen adapted to the curve of singularity:
	\[w=\frac{1}{2}\big(\frac{2}{\gamma-1}c+(\widehat{T'})^i\psi_i\big), \ \ \wb=\frac{1}{2}\big(\frac{2}{\gamma-1}c-(\widehat{T'})^i\psi_i\big), \ \ \psi_2= (\widehat{X'})^i\psi_i,\]
	where $X'$ and $\widehat{T'}$ are the unit tangential vector field and the unit normal vector field of the separating curves. We should also  choose $X'$ and $T'=\kappa  \widehat{T'}$ as commutator vector fields. The construction of the initial data can be derived in the same manner. The proof of the \emph{a priori} energy estimates in \cite{LuoYu1} will be much longer since the equations for the new Riemann invariants and the commutators of $X'$ and $T'$ will be more complicated. 
\end{remark}

\subsubsection{Uniqueness of solutions to the Riemann problem consisting of rarefaction waves}

	We show that the solution to the Cauchy problem stated in {\bf Theorem 3} is indeed unique. It is typical that the uniqueness is a separated issue from the existence for quasilinear equations and it also should be stated in a suitable class of functions. Enlightening examples to the uniqueness problem with singularities can be found in the series of papers \cite{Christodoulou 1}, \cite{Christodoulou 2},  \cite{Christodoulou 3} and \cite{Christodoulou 4} of Christodoulou and of Christodoulou and Lisibach for self-gravitating relativistic fluids. In \cite{Christodoulou 2} and \cite{Christodoulou 4}, the authors prescribed the asymptotic behaviors of the solutions and proved that the solution is unique provided the specific asymptotic behavior towards the singularities. 

We will take a different approach. In the theory of conservation laws, our uniqueness part corresponds to the  weak-strong uniqueness  in the class of ``entropy solutions''. This is among the largest possible class of functions in the 1D conservation laws that one expects uniqueness. The idea of the proof was originally from Dafermos and Diperna, see \cite{DiPerna79} or Theorem 5.2.1 of the textbook \cite{Dafermos}. It is known as \emph{the relative entropy method}. We will follow  Chen and Chen \cite{Chen-Chen} where the authors use the idea to study the uniqueness and stability for plane rarefaction waves in higher dimensions.

For a smooth (locally Lipschitz) solution $(\rho,v)$ to the Euler equations \eqref{eq: Euler in rho v}, the internal energy density $e$ is given by 
\[ e = \int \frac{p(\rho)}{\rho^2} d\rho = \frac{1}{\gamma-1}\rho^{\gamma-1}.\] 
The law of energy conservation of the system can be written as
\[ \partial_t \eta + \nabla \cdot q = 0,\] 
where $(\eta,q)$ is the mechanical energy-energy flux pair given by
\begin{equation}\label{def: eta q}
\begin{cases} 
&\eta(\rho,v) = \frac{1}{2}\rho |v|^2 + \rho e, \\
& q(\rho,v) = \big(\frac{1}{2}\rho |v|^2 + \rho e + p\big) v = \rho  (\frac{|v|^2}{2} + h) v.
\end{cases}
\end{equation}
We remark that $q(t,x^1,x^2)$ is a time-dependent vector field on $\mathbb{R}^2$.

We also recall the following definition introduced {\color{black}by Lax \cite{Lax71} and Kru\v{z}kov \cite{Kruzkov1970}}:
\begin{definition}\label{def: entropy-solution}
	An {\bf entropy solution} or an {\bf admissible bounded weak solution} to \eqref{eq: Euler in rho v} are measurable bounded functions $(\rho,v)$ satisfying  \eqref{eq: Euler in rho v} (in conservative form) in the sense of distributions and the following the entropy inequality 
	\[\partial_t \eta + \nabla \cdot q \leqslant 0.\]
\end{definition}
We remark that, a locally Lipschitz solution $(\rho,v)$ is an entropy solution since $\partial_t \eta + \nabla \cdot q = 0$. {\color{black}Also, across a weak shock front, the above entropy inequality is equivalent to the determinism condition; see \cite{Lax71,Smoller}.}
{\color{black} 
\begin{remark}
	The concept of ``entropy solution'' in this paper is introduced as a mathematical tool (essentially the $L^2$-based energy method)  for  comparing different solutions with the aim of proving uniqueness, and does not refer to the physical concept of entropy. Indeed, the present paper belongs to the isentropic context where the physical entropy is constant throughout.  Although this method is not related to entropy in physics, we still use this term out of respect for tradition, and we ask the readers not to be confused.
\end{remark}
}

\begin{proposition}[Uniqueness to Riemann problem]\label{prop: uniqueness to thm 3}
Let $(v',c')$ be an entropy solution to the Euler equations \eqref{eq: Euler in rho v} on $[0,t^*]\times \mathbb{R}^2$ with respect to the same initial data given in {\bf Theorem 3}. Then, $(v',c')\equiv (v,c)$ where $(v,c)$ is the solution constructed in {\bf Theorem 3}. 
\end{proposition}

\subsubsection{Uniqueness of centered rarefaction waves}\label{section:uniqueness-centered-rare}

We show that the single family of rarefaction waves constructed in {\bf Theorem 2} is also unique. The uniqueness in the current setup is completely different from the uniqueness for solutions to the usual Cauchy problem because we can arbitrarily extend the data to $x_1<0$ in a smooth way to provide many solutions. To specify the spaces in which we can prove uniqueness,  we first provide a definition for the word  \emph{centered} rarefaction waves. Motivated by the one-dimensional centered rarefaction waves, we regard such an object in higher dimensions as a solution to the Euler equations with a family of smooth characteristic hypersurfaces emanating from the initial discontinuity $\mathbf{S}_*$. 

\begin{definition}\label{def: centered-rarefaction-waves}
Given smooth data $U_r$ (away from vacuum) to the Euler equations  \eqref{eq: Euler in rho v}   on $\{x_1 > 0\}$, let $\mathcal{D}_0$ be its future domain of dependence with characteristic boundary $C_0$. A family of {\bf centered rarefaction waves} connected to $U_r$ is a solution  $(v',c')$ to the Euler equations \eqref{eq: Euler in rho v} defined on an open set $\mathcal{W}' \subset \big\{(t,x_1,x_2)\in \mathbb{R}^3-\mathcal{D}_0\big| 0< t\leqslant t^*\}$ so that  $C_0\subset {W}'$ with the following properties:
	\begin{enumerate}
		\item (Foliation) The rarefaction wave region $\mathcal{W}' = \mathcal{W}'_{{u^*}'} := \bigcup_{u' \in [0,{u^*}']} C_{u'}'$ is foliated by a family of smooth null hypersurfaces emanating from $\mathbf{S}_*$ determined by an eikonal function $u'$  with respect to the acoustical metric of $(v',c')$. Moreover, $C'_0=C_0$ and for all $u'_1>u'_2$, $C'_{u'_1}$ is in the future of $C'_{u'_2}$.

		\item (Smoothness) There exists an acoustical coordinates $(t,u',\vartheta')$ adapted to the foliation $\{C_{u'}'\}_{0\leqslant u'\leqslant {u^*}'}$ so that  $(v',c')$ is in $C^N(\overline{W}')$ with respect to $(t,u',\vartheta')$  and $(v',c')\big|_{C_0}=(v_r,c_r)\big|_{C_0}$.
		
		\item (Centered) The inverse density function $\kappa'(t,u',\vartheta')$ satisfies   
		\begin{equation}\label{eq:kappa-limit}
			\lim_{t \to 0}\frac{\kappa'(t,u',\vartheta')}{t} = 1.
		\end{equation} 
	\end{enumerate}
	
\end{definition}

\begin{remark}
Instead of \eqref{eq:kappa-limit}, we may assume the limit $\frac{\kappa'}{t}$ is bounded from below and above by positive constants because we can then choose a new $u'$-foliation canonical by a suitable rescaling.

We will assume $N \geqslant 3$ and we do not pursue the optimal $N$ for the uniqueness.
\end{remark}

\begin{proposition}[Uniqueness]\label{prop: uniqueness to thm 2}
	Let $(v,c)$ be the solution on $\mathcal{W}_{u^*} = \bigcup_{u \in [0,u^*]}C_u$ from {\bf Theorem 2} and $(v',c')$ be another family of (right-going) rarefaction waves on a region $\mathcal{W}'_{{u^*}'} = \bigcup_{u' \in [0,u^*]}C'_{u'}$. Then, $\mathcal{W}'_{u^*} = \mathcal{W}_{u^*}$ and $(v,c)=(v',c')$.
\end{proposition}

}

\subsection{Difficulties and ideas of the proof}

We address several major difficulties in the proof and provide a brief description of the ideas used to overcome them.
\subsubsection{Construction of rarefaction wave data and existence}

One of the main challenges is to construct $C^N$  \emph{rarefaction data} on $\Sigma_{\delta}$ (refer to Definition \ref{def:rare data}). To do this, it is essential to have a precise characterization of the asymptotic singular initial data that propagates to $\Sigma_{\delta}$ from the singularity. Moreover, the rarefaction wave data on $\Sigma_{\delta}$ must satisfy various compatibility and constraint equations.

We do this in two steps:


\smallskip


\begin{itemize}[noitemsep,wide=0pt, leftmargin=\dimexpr\labelwidth + 2\labelsep\relax]
	\item[1)] The construction of $u$ and $\vartheta$ on $\Sigma_\delta$.
	
	\smallskip

	Based on the energy estimates in the first paper \cite{LuoYu1}, which strongly suggest that rarefaction waves are smooth in the acoustical coordinate system $(t,u,\vartheta)$, it is natural to prescribe the initial data on $\Sigma_\delta$ using the functions $u$ and $\vartheta$. However, the acoustical coordinates are defined by the solutions, and we need to be extremely accurate in choosing the null hypersurface foliation $C_u$. In fact, the energy estimates in \cite{LuoYu1} depend crucially on the choice of $C_u$ to ensure that the error terms are of size $O(\varepsilon^2t^2)$. Since $t$ can be arbitrarily small as $\delta \rightarrow 0$, the choice of $u$ cannot suffer any loss of order $t$. Moreover, the rarefaction ansatz involves a number of delicate initial conditions. For instance, the function $\zr$ is of size $\varepsilon$, which is equivalent to {\color{black}$\frac{\partial \wb}{\partial u}=-\frac{2}{\gamma+1}$} up to an error of order $O(t\varepsilon)$. These conditions imply that the choice of $u$ cannot be arbitrary. In particular, an intuitively straightforward choice of initial foliation and the corresponding asymptotic data seem unlikely to satisfy the desired rarefaction ansatz.


	By formally setting $\delta=0$ and examining the asymptotic behavior of the Riemann invariants at the singularity, we can show that they must take on the following form:
	\begin{equation}\label{eq:asymptotic-Riemann-invariants}
		\begin{cases}
			&\wb(u,\vartheta)= \wb_r(0,\vartheta)-\frac{2}{\gamma+1}u,\\
			&w(u,\vartheta)=w_r(0,\vartheta),\\
			&\psi_2(u,\vartheta)=-v_r^2(0,\vartheta),
		\end{cases}
	\end{equation}
	in the normalization $\lim_{t \to 0}\frac{\kappa}{t}=1$.
	Therefore, if one assumes that the solution, or equivalently, the data is unique at $\mathbf{S}*$, then the acoustical function $u$ must also be unique. To obtain $u$ at $\Sigma_\delta$, we use ideas that are reminiscent of the last slice argument from Christodoulou-Klainerman's seminal work \cite{CK}, as we explained in the introduction of \cite{LuoYu1}. The key point is to construct the solution in a manner that accurately adapts to the geometry of the initial rarefaction front $C_0$. We formally integrate from the singularity $\mathbf{S}_*$ to $\Sigma_\delta$ to define $u$ and $\vartheta$ explicitly. See the definitions of $u$ and $\vartheta$ in Section \ref{section: definition of u on Sigmadelta} and Section \ref{section: definition of vartheta on Sigmadelta}. For technical details on the relationship between $u$ and $\yr$, we refer to Section \ref{Section: lambda on C0}.	
	
	\medskip
	
	\item[2)] The construction of rarefaction data on $\Sigma_\delta$ and the compatibility conditions.
	
	\smallskip
	

	In contrast to the simple form of the asymptotic limit at $\delta = 0$ given by equation \eqref{eq:asymptotic-Riemann-invariants}, rarefaction data at $\Sigma_{\delta}$ not only require  precise asymptotic expressions, but also must satisfy various constraint equations. 

To be more precise, rarefaction waves in acoustical coordinates must adhere to the following asymptotic condition (refer to Remark \ref{rem:asymptotic condition}):
	\begin{equation}\label{eq:asymptotic condition}
		\frac{\partial x^1}{\partial u} \approx \kappa \sim t, \quad T\wb \sim -\frac{2}{\gamma+1}, \ T\psi_2 = O(t\varepsilon), \ Tw = O(t\varepsilon).
	\end{equation}
	To connect these singular rarefaction data to a given smooth initial data across the rarefaction front $C_0$, certain conditions must be satisfied. If one is searching for a $C^N$-solution in $\mathcal{D}(\delta)$, the data on $\Sigma_\delta$ must satisfy compatible conditions for all higher-order derivatives along $C_0$ at $S_{\delta,0} = C_0 \cap \Sigma_{\delta}$, as shown in equation \eqref{compatibility condition 3}. In fact, the data must satisfy the following sequence of conditions:
	\[L^k {v\choose c }\big|_{S_{\delta,0}}=L^k{v_r\choose c_r } \big|_{S_{\delta,0}}, \ \ k\leqslant N.\]
	By using the Euler equation, the conditions can be reformulated as a set of algebraic constraints on the jets of the data on $S_{\delta,0}$. Obtaining the data that satisfies these constraints, particularly with the additional bounds imposed by the rarefaction ansatz (refer to Definition \ref{def:rare data}), is not a trivial task.
	
	By diagonalizing the Euler system in the $L$-direction, we can gain valuable insight, which is analogous to Riemann's concept of Riemann invariants. The diagonalized system \eqref{eq:Euler in diagonalized form} exposes a profound structure of the data, namely, a decoupling of normal derivatives, which arises from the characteristic nature of the rarefaction front $C_0$. By utilizing equation \eqref{eq:Euler in diagonalized form}, it is adequate to pose compatibility conditions for the components $U^{(-1)}$ and $U^{(-2)}$, while the characteristic component $U^{(0)}$ on $\Sigma_\delta$ is virtually unconstrained.
	
	Technically, we construct data on $\Sigma_\delta$ by defining a finite Taylor series expansion in $u$. We employ an induction argument to determine the Taylor coefficients, and this inductive structure can be expressed using polynomials $\mathbf{P}_n$ in the variables (including their $T$ and $X$ derivatives) from the set:
	\[\big\{\psi_2, w,\wb, h, c,c^{-1}, \slashed{g}, \kappa, \widehat{T}^1, \widehat{T}^2\big\}\cup\big\{\frac{\mathfrak{d}_1 \circ \cdots \circ \mathfrak{d}_k(\kappa)}{\kappa} \big|k\geqslant 1, \mathfrak{d}_i \in \{T,X\}\big\}.\]
	The polynomials $\mathbf{P}_n$'s encode all the quantities that have at most $n-1$ $T$-derivatives. We show that
	\begin{align}\label{eq:L^n U formal}
		 L^n (U) \approx \big(\frac{c}{\kappa}\Lambda\big)^n\cdot T^n(U) + \frac{1}{\kappa^{n-1}} \cdot \mathbf{P}_{n}
	\end{align}
	where $\Lambda = \operatorname{diag}(0,-1,-2)$; see \eqref{eq:Euler:L^n U in diagonalized form} for the precise formula. This equation defines the normal derivatives for $U^{(-1)}$ and $U^{(-2)}$. The Taylor coefficients can be determined using an induction argument. It is important to note that equation \eqref{eq:L^n U formal} is formally singular as $\delta \to 0$.
	
The construction \emph{maintains} the fundamental hierarchical structure in equation \eqref{eq:asymptotic condition} for higher and mixed derivatives, as demonstrated in equations \eqref{eq: bound on Un with X} and Proposition \ref{prop: correct data size}. This highlights the intricate structure of rarefaction waves and forms the foundation for the energy estimates presented in the initial paper \cite{LuoYu1}.

	\smallskip

\end{itemize}

In view of the rarefaction data and the energy estimates established in \cite{LuoYu1}, we can demonstrate the existence of a single family of rarefaction waves through a limiting argument; see Section \ref{section:existence}. The flat rarefaction fronts that are connected to the constant background state play an important role in setting up the region of convergence in the physical spacetime. As a result, we recover the asymptotic data \eqref{eq:asymptotic-Riemann-invariants} and a \emph{canonical} characteristic foliation in the limit.

\subsubsection{The perturbed Riemann problem of two families of rarefaction waves}
As demonstrated by the 1-D Riemann problem involving two rarefaction waves in Section \ref{section: calculations for 1D Riemann}, the normal derivatives undergo a jump discontinuity at the rarefaction fronts that bound the rarefaction wave regions. Determining these rarefaction fronts through the free boundary problem is one of the main challenges in Alinhac's work \cite{AlinhacWaveRare1}, and would potentially cause loss of derivatives as they are characteristic.

We adopt a different and more geometric approach, which is closer in spirit to the classical construction in one-dimensional Riemann problem.  
Rather than solving the Riemann problem with discontinuous data and deciding the rarefaction fronts as free boundaries through the iterations, we solve the problem in the following scheme (see the picture in {\bf Theorem 3}):
\begin{itemize}
	\item [{\bf Step 1.}] Construct a family of front rarefaction waves connected to a given smooth data $U_r$ on the right (refer to {\bf Theorem 2}), and a family of back rarefaction waves connected to $U_l$ on the left.
	
	\item [{\bf Step 2.}] Construct the rarefaction wave fronts $\Hb$ and $H$ by solving the Eikonal equation with respect to the acoustical metric.
	
	\item [{\bf Step 3.}] Solve a classical Goursat problem in the region bounded by $\Hb$ and $H$.
\end{itemize}

{\bf Theorem 2} immediately implies Step 1 and yields two family of back and front rarefactions connected to $U_l$ and $U_r$, respectively. However, identifying and gluing the free characteristic boundaries are not trivial in general. The key point is that,  
as mentioned in Remark \ref{Remark: difference between rarefaction fronts and shock fronts}, rarefaction fronts are characteristic hypersurfaces that are ruled by null geodesics; therefore, they can be determined from the initial conditions, provided that the rarefaction solution we constructed is \emph{regular} in the limit as $t \rightarrow 0$. Hence a main technical challenge in Step 2 is to obtain uniform higher regularity: 
\begin{itemize}[noitemsep,wide=0pt, leftmargin=\dimexpr\labelwidth + 2\labelsep\relax]
	\item  Regularity of rarefaction fronts and uniform bounds on $L^k(\psi)$ for $k\geqslant 2$.
	
	\smallskip

	The solution $(v,c)\in C^0 \big((0,t^*]; H^N(\Sigma_t) \big)$ obtained from {\bf Theorem 2} is regular in the spatial variables within the rarefaction region. However, for the rarefaction fronts $\Hb$ and $H$ to be regular (refer to {\bf Theorem 3}), we need  higher temporal regularity. To establish this temporal regularity, it is necessary to derive uniform bounds on $L^k(\psi)$ for the solution constructed on $\mathcal{D}(\delta)$ as presented in {\bf Corollary 1}.

	However, according to the Euler equations, we have $L^n (U) \sim \big(\frac{c}{\kappa}\Lambda\big)^n\cdot T^n(U)$, see \eqref{eq:L^n U formal}. It is likely that we could only hope to obtain the following bounds on  $\mathcal{D}(\delta)$:
	\begin{equation}\label{eq: rough bound on Lk}
		\|L^k(\wb)  \|_{L^\infty(\Sigma_t)}\lesssim t^{-k+2}\varepsilon,  \ \ \|L^k  (w)  \|_{L^\infty(\Sigma_t)} +\|L^k  (\psi_2)  \|_{L^\infty(\Sigma_t)} \lesssim t^{-k+1}\varepsilon,
	\end{equation}
	for $k\geqslant 2$. We refer to Section \ref{Section:The rough bounds on L derivatives} for detailed computations. This seems to suggest that $L^k(\psi)$ suffers from a loss in $t$ for $k \geq 2$.
	
	We believe that it is very difficult to recover the loss in $t$ on $\mathcal{D}(\delta)$. Using the diagonalized system \eqref{eq:Euler in diagonalized form}, for $\lambda_0=-1$ and $-2$, we can show that
	\begin{equation*}
		\kappa^k L^k\big(U^{(\lambda_0)}\big)=\mathbf{P}_k^{(\lambda_0)}+\mathbf{Err},
	\end{equation*}
	for all $k\geqslant 2$ where $|\mathbf{Err}|\lesssim \varepsilon t$ and $\mathbf{P}_k^{(\lambda_0)}$ is the $\lambda_0$-component of the vector $\mathbf{P}_k$ defined inductively as: 
	\[\mathbf{P}_1 =c\Lambda T(U), \ \ \mathbf{P}_{j+1}  = \big(c\Lambda T-jL(\kappa)\big) (\mathbf{P}_j) , \ \ j\geqslant 1,\]
	with $\Lambda$ given in \eqref{def:Lambda}. 	
	
In fact, the loss in $t$ for $L^k(\psi)$ would happen, unless a sequence of conditions in the form of partial differential inequalities $\mathbf{P}_{j}=O(\delta)$ for all $j\leqslant k$ are satisfied. Even assuming this can be done, we have only
	\[
	{\color{black}\|L^k \big(U^{(\lambda_0)}\big) \|_{L^\infty(\Sigma_t)}\lesssim t^{-k+2}\varepsilon,  \ \ \lambda_0=-1,-2}.
	\]
	Compared to \eqref{eq: rough bound on Lk}, the order in $\delta$ is only improved by one in the power. 
	It suggests that finding the initial data with uniform higher temporal bounds on $\Sigma_\delta$ is a difficult task. 
	
	The previous discussion shows that even for the data on $\Sigma_\delta$, we can not obtain uniforms bounds on $L^k(\psi)$. On the other hand, using the continuity across $C_0$, we know that $L^k(\psi)$ is indeed bounded on $C_0$ due to the smoothness of $(v_r,c_r)$. The question is whether we can propagate this regularity \emph{inside} the rarefaction wave region. 
	
	This turns out to be true. Rather than using the Euler system itself, we will use the acoustical wave equations for $\psi\in \{\wb,w,\psi_2\}$ in the null frame $(\Lb,L,\Xh)$ and integrate along the $\Lb$ direction to retrieve the uniform bounds on $L^k(\psi)$ from $C_0$. We recall that in Section \ref{section:effective domain}, $\mathcal{D}(\delta)$ has been decomposed into the effective domain $\mathcal{D}^+(\delta)$ and the irrelevant domain $\mathcal{D}^-(\delta)$. We will show that $L^k(\psi)$ are uniformly bounded on  $\mathcal{D}^+(\delta)$. Using the limiting argument as $\delta \rightarrow 0$, we show that the information from $\mathcal{D}^-(\delta)$ will be irrelevant to the solution in 
	{\bf Theorem 2} or {\bf Theorem 3}.

%
%
%

\end{itemize}	

Once we have the uniform bounds on $L^k(\psi)$, it remains to find the initial data for the rarefaction fronts $\Hb$ and $H$ and solve the Eikonal equations:
\begin{itemize}[noitemsep,wide=0pt, leftmargin=\dimexpr\labelwidth + 2\labelsep\relax]
	\item Determine the initial data for the two rarefaction fronts $\Hb$ and $H$. 
	
	In general $\Hb$ and $H$ are different from the null hypersurfaces $\Cb_{\ub}$ and $C_u$ of the foliation.	
	The initial data for $\Hb$ and $H$ correspond to two limiting curves $\Hb_0$ and $H_0$ in the back and front rarefaction regions, respectively. To locate $\Hb_0$ and $H_0$, the key point is that: as the solutions become singular in the limit $t \to 0$,  the data must take the simple form in \eqref{eq:asymptotic-Riemann-invariants}. 
	As a result, we can use one dimensional construction discussed in Section \ref{section: calculations for 1D Riemann} to find the initial data; see \eqref{eq:H_0} for the defining equation of $H_0$. 
	
	\item Construction of rarefaction fronts emanating from the singular boundary $\mathbf{S}_*$.
	
	 The singular boundary $\mathbf{S}_*$ of rarefaction waves is of codimension $2$ in the background spacetime. Since the rarefaction solution become singular near $\mathcal{S}_*$, solving the Eikonal equation with respect to the acoustical metric is not a trivial task. Following the ideas in \cite{ChristodoulouShockFormation, ChristodoulouMiao}, we study and classify the null geodesics starting at $\mathbf{S}_*$ (see Section \ref{section:geometric constructions} and Proposition \ref{prop:geometric constructions}), leading to a geometric construction of the canonical acoustical coordinate. The Hamiltonian approach also enables us to construct rarefaction fronts emanating from any given smooth graph in $S_*$;
	see Section \ref{section: construction of null hypersurfaces} for the detail construction.
\end{itemize}
With the two rarefaction fronts $\Hb$ and $H$ at hand, we immediately obtain a smooth solution in the region bounded by $\Hb$ and $H$ as a classical Goursat problem, completing Step 3 and the construction of the solution to the Riemann problem. For details we refer to Section \ref{Section: app to Riemann 2}.

\subsubsection{Uniqueness and some comments}
One of the main challenge in proving uniqueness is that the solutions become singular in rarefaction wave regions near the singular boundary $\mathbf{S}_*$. 
We provide two uniqueness results:
\begin{itemize}[noitemsep,wide=0pt, leftmargin=\dimexpr\labelwidth + 2\labelsep\relax]
	\item Uniqueness of the Riemann problem consisting of rarefaction waves.
	
	We use the relative entropy method from \cite{DiPerna79,Dafermos,Chen-Chen}. It provides uniqueness in the class of entropy solutions without regularity requirement (see Definition \ref{def: entropy-solution} and Proposition \ref{prop: uniqueness to thm 3}). The proof relies on the key property of rarefaction wave: the density decreases along the $T$-direction, see \eqref{eq: Wplus estimate}.
	
	\item Uniqueness of centered rarefaction waves.
	
	We provide a definition of centered rarefaction wave as a family of smooth characteristic hypersurfaces emanating from the singularity $\mathbf{S}_*$(refer to Definition \ref{def: centered-rarefaction-waves}) and shows that is also unique. The proof can be reduced to the proof of the uniqueness of the Riemann problem.
\end{itemize}

We offer some comments on the compatibility condition, data and uniqueness presented in Alinhac's work \cite{AlinhacWaveRare1, AlinhacWaveRare2} (refer to our first paper \cite{LuoYu1} for comments on the tame estimates and iteration scheme):
\begin{itemize}[noitemsep,wide=0pt, leftmargin=\dimexpr\labelwidth + 2\labelsep\relax]
	\item  In \cite{AlinhacWaveRare1}, the initial data separated by $\mathbf{S}_*$ are assumed to be compatible, in the sense that for each $\vartheta$, the state on the left and the state on the right can be connected by a one-dimensional rarefaction wave of the same family. Additionally, they must satisfy a ``$k$-compatibility condition'' on normal derivatives up to order $k$. For the Euler equations, it implies that the solution need to be \emph{smooth} across one of the rarefaction front (see Remark \ref{rem: 1 family alinhac}).
	
	In our work, we do not need any compatibility condition to establish the existence of a single family of rarefaction waves (refer to {\bf Theorem 2}) or the Riemann problems consisting of rarefaction waves (refer to {\bf Theorem 3} and Remark \ref{rem: 1 family alinhac}). Rather, we describe all rarefaction waves belonging to this family that can be connected to the given data and initial surface of discontinuity $\mathbf{S}_*$, and construct the rarefaction fronts emanating from $\mathbf{S}_*$ to solve the Riemann problem. In general the solutions are continuous but not smooth across both rarefaction fronts of the Riemann problem.

	\item In \cite{AlinhacWaveRare2}, uniqueness  is proved under the same compatible and ``$k$-compatibility condition'' on the initial data $(U_l,U_r)$ as in \cite{AlinhacWaveRare1}, assuming the existence of a `blow-up' coordinate where the solution is smooth.

\end{itemize}

\section{Construction of the initial data on $\Sigma_\delta$}\label{section: data construction}

We fix $\delta>0$. The purpose of this section is to construct a $C^N$ data for rarefaction waves (see Definition \ref{def:rare data}) so that the initial ansatz $\mathbf{(I_0)}, \mathbf{(I_2)}$,
$\mathbf{(I_{\infty,1})}$, $\mathbf{(I_{\infty,2})}$ and $\mathbf{(I_{irrotational})}$ hold.

\subsection{The initial foliation}\label{section:initial foliatiion}
In this subsection, we construct $u$ and $\vartheta$ on $\Sigma_\delta$.
\subsubsection{Functions on $C_0$}\label{section:functions on C0}
Since $\Th$ is the unit normal of $S_{t,0}$ as an embedded curve of $\Sigma_t$, $\Th$ is already determined by $C_0$. By continuity, $v$ and $c$ are also determined on $C_0$ by $v_r$ and $c_r$. As a consequence, $L$ is already fixed on $C_0$ because $L =\partial_t +v-c\Th$.

We use $\slashed{\vartheta}: C_0\rightarrow \mathbb{R}$ to denote the restriction of $\vartheta$ on $C_0$. In view of the construction of acoustical coordinates, $\slashed{\vartheta}$ is defined by the following ODE system:
\begin{equation}\label{eq:theta on C0}
\begin{cases}
&L(\slashed{\vartheta})=0,\\
&\slashed{\vartheta}\big|_{S_{0,0}}=x_2\big|_{S_{0,0}}.
\end{cases}
\end{equation}
Let $\slashed{t}$ be the restriction of $t$ on $C_0$. We then obtain a coordinate system $(\slashed{t},\slashed{\vartheta})$ (the restriction of the acoustical coordinates) on $C_0$.

Let $(\slashed{x}_1,\slashed{x}_2)$ be the restriction of the Cartesian coordinate functions $(x_1,x_2)$ on $C_0$. We now use the coordinate functions $(\slashed{t},\slashed{\vartheta})$ to represent $(\slashed{x}_1,\slashed{x}_2)$. Indeed, since $x^1(0,\slashed{\vartheta})= 0$ and $x^2(0,\slashed{\vartheta})=\vartheta$, we can integrate $L(x^i)=v^i-c\Th^i$ to derive
\begin{equation}\label{eq:x1x2 on C0 0}
\begin{cases}
\slashed{x}_1(t,\slashed{\vartheta})&=- \displaystyle\int_{0}^t \psi_1(\tau,\slashed{\vartheta})+c(\tau,\slashed{\vartheta})\Th^1(\tau,\slashed{\vartheta})d\tau,\\
\slashed{x}_2(t,\slashed{\vartheta})&=\slashed{\vartheta}-\displaystyle\int_{0}^t \psi_2(\tau,\slashed{\vartheta})+c(\tau,\slashed{\vartheta})\Th^2(\tau,\slashed{\vartheta})d\tau.
\end{cases}
\end{equation}

For a given time $\slashed{t}=t_0>0$, we have two coordinate functions $\slashed{x}_2$ and $\slashed{\vartheta}$ on $S_{t_0,0}$. The change of coordinates $\slashed{\vartheta}\mapsto \slashed{x}_2(t_0,\slashed{\vartheta})$ is given by the second formula in \eqref{eq:x1x2 on C0 0}. We compute the differential:
\begin{align*}
\frac{\partial \slashed{x}_2}{\partial \slashed{\vartheta}}(t_0,\slashed{\vartheta})=1-\int_{0}^{t_0} \frac{\partial }{\partial \slashed{\vartheta}}\big[\psi_2(\tau,\slashed{\vartheta})(\tau,\slashed{\vartheta})+c(\tau,\slashed{\vartheta})\Th^2(\tau,\slashed{\vartheta})\big]d\tau.
\end{align*}
Since $(v,c)\big|_{C_0}=(v_r,c_r)\big|_{C_0}$ is evolved from an $\varepsilon$-perturbation of the constant states $(\mathring{v}_r,\mathring{c}_r)$(see Definition \ref{def:data}),  by the standard continuous dependence on the initial conditions for hyperbolic equations, it is clear that $\big|\big(\frac{\partial}{\partial \slashed{\vartheta}}\big)^k(v,c)\big|\lesssim \varepsilon$ for $k\leqslant \Ntop$. Therefore, we obtain the following bounds
\begin{equation}\label{eq:x1x2 on C0}
\big|\frac{\partial \slashed{x}_2}{\partial \slashed{\vartheta}}(t_0,\slashed{\vartheta})-1\big|\lesssim \varepsilon t_0, \ \ \ \ \big|\frac{\partial^k \slashed{x}_2}{\partial \slashed{\vartheta}^k}(t_0,\slashed{\vartheta})\big|\lesssim \varepsilon t_0, \ \ k\leqslant \Ntop.
\end{equation}
 In particular, by inverse function theorem, we can represent $\slashed{\vartheta}$ in terms of $\slashed{x}_2$, i.e., $\slashed{\vartheta}\big|_{S_{t_0,0}}=\slashed{\vartheta}(t_0,\slashed{x}_2)$. 

\subsubsection{The function $u$ on $\Sigma_\delta$}\label{section: definition of u on Sigmadelta}

Using the standard Cartesian coordinates $(x_1,x_2)$ on $\Sigma_\delta$, to define the acoustical function $u$ on $\Sigma_\delta$, it suffices to write $u$ as $u(x_1,x_2)$.  We use $I(\tau,x_2)$ to denote the following auxiliary function:
\[I(\tau,x_2)=\psi_1\left(\tau,\slashed{\vartheta}(\delta,x_2)\right)+c\left(\tau,\slashed{\vartheta}(\delta,x_2)\right)\Th^1\left(\tau,\slashed{\vartheta}(\delta,x_2)\right).\]
The function $u$ on $\Sigma_\delta$ are defined as follows:
\begin{equation}\label{def: for u initially}
u: \Sigma_{\delta}\rightarrow \mathbb{R}, \ \  (x_1,x_2)\mapsto u(x_1,x_2)=-\frac{x_1}{\delta}-\frac{1}{\delta}\int_{0}^\delta I(\tau,x_2)d\tau. 
 \end{equation}
We check that 
\begin{equation}\label{eq: u vanishes at C0}
u\big|_{S_{\delta,0}}\equiv 0.
\end{equation}
For an arbitrary point $p=(\delta, \slashed{\vartheta}_0)\in S_{\delta,0}$. The $x_2$ coordinate of this point is given by $x_2(p)=\slashed{x}_2(\delta, \slashed{\vartheta}_0)$. In terms of the $\vartheta$-coordinate on $S_{\delta,0}$, $p$ is given by $\slashed{\vartheta}_0$ where $\slashed{\vartheta}_0=\slashed{\vartheta}(\delta,x_2(p))$. Therefore, we have
\begin{align*}
u(p)&=-\frac{x_1(\delta,\slashed{\vartheta}_0)}{\delta}-\frac{1}{\delta}\int_{0}^\delta \psi_1\left(\tau,\slashed{\vartheta}(\delta,x_2(p))\right)+c\left(\tau,\slashed{\vartheta}(\delta,x_2(p))\right)\Th^1\left(\tau,\slashed{\vartheta}(\delta,x_2(p))\right)d\tau\\
&=-\frac{x_1(\delta,\slashed{\vartheta}_0)}{\delta}-\frac{1}{\delta}\int_{0}^\delta \psi_1\left(\tau,\slashed{\vartheta}_0\right)+c\left(\tau,\slashed{\vartheta}_0\right)\Th^1\left(\tau,\slashed{\vartheta}_0\right)d\tau.
\end{align*}
In view of \eqref{eq:x1x2 on C0 0}, we obtain that $u(p)=0$. This proves \eqref{eq: u vanishes at C0}.

To study the foliation on $\Sigma_\delta$ by the level sets of $u$, it is natural to compute the gradient of $u$. We recall that $X=\frac{\partial }{\partial \slashed{\vartheta}}$ on $C_0$. Therefore, 
\begin{align*}
\frac{\partial u}{\partial x_2}= -\frac{1}{\delta}\int_{0}^\delta X \big(\psi_1+c \Th^1\big)\left(\tau,\slashed{\vartheta}(\delta,x_2)\right) \frac{\partial \slashed{\vartheta}}{\partial x_2}(\delta,x_2)d\tau,
\end{align*}
where $\psi_1+c \Th^1$ is a function defined on $C_0$. We introduce the following two auxiliary functions:
\begin{equation}\label{eq: def for a and A}
a(\tau,x_2)=  X \big(\psi_1+c \Th^1 \big) \big(\tau,\slashed{\vartheta}(\delta,x_2) \big) \frac{\partial \slashed{\vartheta}}{\partial x_2}(\delta,x_2), \ \ A(t,x_2)=\int_0^t a(\tau,x_2)d\tau.
\end{equation}
Since $\big|\big(\frac{\partial}{\partial \slashed{\vartheta}}\big)^k(v,c)\big|\lesssim \varepsilon$ on $C_0$ for $k\leqslant \Ntop$,  by regarding $a(\tau,\cdot)$ as a function on $S_{\tau,0}$, we have
\begin{equation}\label{eq: bound on a}\big|\frac{\partial^k a}{\partial \slashed{x}_2^k}(\tau,\slashed{x}_2)\big|\lesssim \varepsilon, \ \ k\leqslant \Ntop.
\end{equation}
We can represent $\nabla u$ as
\begin{equation}\label{eq: gradient of u on Sigma delta}
\nabla u=\big(\frac{\partial u}{\partial x_1},\frac{\partial u}{\partial x_2}\big)=\big(-\frac{1}{\delta},-\frac{1}{\delta}\int_{0}^\delta a(\tau,x_2)d\tau\big)=-\frac{1}{\delta}\big(1,A(\delta,x_2)\big).
\end{equation}
Since $\Th$ is the unit normal of the level sets of $u$, we obtain that
\begin{equation}\label{eq:Th on Sigma delta}
\Th =-\frac{\left(1,A(\delta,x_2)\right)}{\sqrt{1+A(\delta,x_2)^2}}.
\end{equation}
We can rotate $\Th$ by $\frac{\pi}{2}$ to obtain $\Xh$:
\begin{equation}\label{eq:Xh on Sigma delta}
\Xh =\frac{\left(-A(\delta,x_2),1\right)}{\sqrt{1+A(\delta,x_2)^2}}.
\end{equation}
The inverse density $\kappa$ is computed as
\begin{equation}\label{eq: kappa on Sigma delta}
\kappa=\frac{1}{|\nabla u|}=\frac{\delta}{\sqrt{1+A(\delta,x_2)^2}}.
\end{equation}

\begin{lemma}\label{lemma: bounds on geometry on sigma delta 1}
For all $1\leqslant n\leqslant \Ntop$, $\mathfrak{d}_1, \cdots, \mathfrak{d}_n \in \{\partial_1,\partial_2\}$ and $\Upsilon \in \{\Th^1,\Th^2, \kappa\}$, we have 
\[\big\|\mathfrak{d}_1\big(\mathfrak{d}_2\big(\cdots \big(\mathfrak{d}_n(\Upsilon)\big)\cdots \big)\big)\big\|_{L^\infty(\Sigma_\delta)}\lesssim \delta \varepsilon.\]
Moreover, we have following more precise bounds: 
\begin{equation}\label{eq:precise estimates on Th}
\|\Th^1+1\|_{L^\infty(\Sigma_\delta)}\lesssim \delta^2\varepsilon^2, \ \ \|\Th^2\|_{L^\infty(\Sigma_\delta)}\lesssim \delta\varepsilon,  \ \ \big\|\mathfrak{d}_1\big(\mathfrak{d}_2\big(\cdots \big(\mathfrak{d}_n(\Th^1)\big)\cdots \big)\big)\big\|_{L^\infty(\Sigma_\delta)}\lesssim \delta^2 \varepsilon^2. 
\end{equation}
and
\[\|\kappa\|_{L^\infty(\Sigma_\delta)} \lesssim \delta,\ \  \big\|\mathfrak{d}_1\big(\mathfrak{d}_2\big(\cdots \big(\mathfrak{d}_n(\kappa)\big)\cdots \big)\big)\big\|_{L^\infty(\Sigma_\delta)}\lesssim \delta^3 \varepsilon^2.\]
\end{lemma}
\begin{proof}
We consider the following three smooth functions:
\[F_1(x)=-\frac{1}{\sqrt{1+x^2}}, \ \ F_2(x)=-\frac{x}{\sqrt{1+x^2}}, \ \ F_3(x)=\frac{1}{\sqrt{1+x^2}}.\]
We observe that for all $k\leqslant N_{\rm top}$ and $i=1,2,3$, we have
\begin{equation}\label{eq:lamma 41 aux 1}
\big\|\big(\frac{d}{dx}\big)^k F_i\big\|_{L^{\infty}([-1,1])}\lesssim 1.
\end{equation}
In view of \eqref{eq:Th on Sigma delta}, \eqref{eq:Xh on Sigma delta} and \eqref{eq: kappa on Sigma delta}, 
\[\Th^1=F_1\big(A(\delta,x_2)\big), \ \ \Th^2=F_2\big(A(\delta,x_2)\big), \ \ \kappa=\delta F_3\big(A(\delta,x_2)\big).\]
For sufficient small $\delta$ and $\varepsilon$, it is clear that $A(\delta,x_2)\in [-1,1]$. Moreover, since
$A(\delta,x_2)=\int_0^t a(\tau,x_2)d\tau$, \eqref{eq: bound on a} yields that
\[\big|\frac{\partial^k A(\delta, x_2)}{\partial \slashed{x}_2^k}\big|\lesssim \delta\varepsilon, \ \ k\leqslant \Ntop.\]
Therefore, thanks to \eqref{eq:lamma 41 aux 1} and the Fa\`a di Bruno formula, we obtain that
\[\big\|\mathfrak{d}_1\big(\mathfrak{d}_2\big(\cdots \big(\mathfrak{d}_n(\Upsilon)\big)\cdots \big)\big)\big\|_{L^\infty(\Sigma_\delta)}\lesssim \delta \varepsilon, \ \ n\leqslant \Ntop, \ \Upsilon \in \{\Th^1,\Th^2, \kappa\}.\]
According to \eqref{eq:Th on Sigma delta}, we have
\[\Th^1+1 =1-\frac{1}{\sqrt{1+A(\delta,x_2)^2}},  \ \ \Th^2 =-\frac{ A(\delta,x_2)}{\sqrt{1+A(\delta,x_2)^2}}.
\]
Hence, the estimates in \eqref{eq:precise estimates on Th} for $\Th^1+1$ and $\Th^2$ follow from the Taylor expansion of the above functions. We also have
\[
\mathfrak{d}_1\left(\mathfrak{d}_2\left(\cdots \left(\mathfrak{d}_n(\Th^1)\right)\cdots \right)\right) =-\mathfrak{d}_1\big(\mathfrak{d}_2\big(\cdots \big(\mathfrak{d}_n\big(\frac{1}{\sqrt{1+A(\delta,x_2)^2}}\big)\big)\cdots \big)\big).
\]
The Taylor expansion of the above formula is in $A(\delta,x_2)^2$. Thus, it is bounded by $\delta^2\varepsilon^2$. According to equation \eqref{eq: kappa on Sigma delta}, we have
\[\frac{\kappa}{\delta}-1=1-\frac{1}{\sqrt{1+A(\delta,x_2)^2}}.
\]
The Taylor expansion of the above formula is in $A(\delta,x_2)^2$ and this shows that \[\big\|\frac{\kappa}{\delta}-1\big\|_{L^\infty(\Sigma^{u}_\delta)} \lesssim  \varepsilon^2 \delta^2.\]
We can apply the similar argument to derivatives of $\kappa$. This completes the proof of the lemma.
\end{proof}

\begin{proposition}\label{prop: bounds on geometry on sigma delta 1}
For any multi-index $\alpha$ with $1\leqslant |\alpha| \leqslant \Ntop$, for all $Z\in \mathscr{Z}=\{T,\Xh\}$ and $\Upsilon \in \{\Th^1,\Th^2, \kappa\}$, we have 
\[\|Z^\alpha(\Upsilon)\|_{L^\infty(\Sigma_\delta)}\lesssim \delta \varepsilon.\]
Moreover, we have following more precise bounds: 
\begin{equation}\label{eq:improved estimates on Th and kappa}
\|\Th^1+1\|_{L^\infty(\Sigma_\delta)}\lesssim \delta^2\varepsilon^2, \ \ \|\Th^2\|_{L^\infty(\Sigma_\delta)}\lesssim \delta\varepsilon, \ \  \|Z^\alpha(\Th^1)\|_{L^\infty(\Sigma_\delta)}\lesssim \delta^2 \varepsilon^2, \ \  \|Z^\alpha(\kappa)\|_{L^\infty(\Sigma_\delta)}\lesssim \delta^3 \varepsilon^2. 
\end{equation}
\end{proposition}
\begin{proof}
We first express the vector field $Z=T$ or $\Xh$ in the Cartesian frame: 
\[T=\kappa \Th =\kappa \big(\Th^1 \partial_1 + \Th^2 \partial_2\big), \ \ \Xh= \Th^2 \partial_1 -\Th^1 \partial_2.\]
By the Leibniz rule,  we then write $Z^\alpha(\Upsilon)$ as a polynomial in the variables $\partial^i \kappa, \partial^j \Th^1, \partial^k \Th^2$ and $ \partial^l \Upsilon$. It is clear that $i,j,k,l\leqslant \Ntop$. Moreover, this polynomial has no constant term. The results follow immediately from the previous lemma.
\end{proof}
\begin{remark}\label{remark: improvement with T}
Since $T=\kappa \big(\Th^1 \partial_1 + \Th^2 \partial_2\big)$, the proof indeed shows that, for any multi-index $\alpha$ with $1\leqslant |\alpha| \leqslant \Ntop-1$, for all $Z\in \mathscr{Z}=\{T,\Xh\}$ and $\Upsilon \in \{\Th^1,\Th^2, \kappa\}$, we have 
\[\|TZ^\alpha(\Upsilon)\|_{L^\infty(\Sigma_\delta)}\lesssim \delta^2 \varepsilon.\]
\end{remark}
\subsubsection{The function $\vartheta$ on $\Sigma_\delta$}\label{section: definition of vartheta on Sigmadelta}

In view of \eqref{def: vartheta in acoustical}, we have $[L,\frac{\partial}{\partial \vartheta}]=0$. On $C_0$, we express $\frac{\partial}{\partial \slashed{\vartheta}}$ in terms of the Cartesian frame: 
\[\frac{\partial}{\partial \slashed{\vartheta}}=\slashed{R}^1\frac{\partial}{\partial x_1}+\slashed{R}^2\frac{\partial}{\partial x_2}.\]
By $[L,\frac{\partial}{\partial \slashed{\vartheta}}]=0$, we obtain the defining equations for $\slashed{R}^1$ and $\slashed{R}^2$:
\begin{equation}\label{eq:partial theta on C0}
\begin{cases}
&L(\slashed{R}^k)=\sum_{j=1}^2\slashed{R}^j\partial_j L^k = X(L^k), \ \ k=1,2;\\
&(\slashed{R}^1,\slashed{R}^2)\big|_{S_{0,0}}=(0,1).
\end{cases}
\end{equation}
where $L^k=v^k-c\Th^k$. Since  $\big|\big(\frac{\partial}{\partial \slashed{\vartheta}}\big)^k(v,c)\big|\lesssim \varepsilon$ on $C_0$ for all $k\leqslant \Ntop$,  we have $|X(L^k)|\lesssim \varepsilon$. Therefore, by integrating \eqref{eq:partial theta on C0} from $0$ to $\delta$, we obtain that
\[\big|\slashed{R}^1\big|+\big|\slashed{R}^2-1\big|\lesssim \varepsilon \delta.\]
Since the restriction of the acoustical metric on $\Sigma_\delta$ is flat, this shows that 
\[\|\slashed{g} - 1\|_{L^\infty(S_{\delta,0})} =\|\sqrt{\big(\slashed{R}^1\big)^2+\big(\slashed{R}^2\big)^2}- 1\|_{L^\infty(S_{\delta,0})}\lesssim \varepsilon \delta.\]
 
 We use $T=\kappa \Th$ to extend $\vartheta$ from $S_{\delta,0}$ to $\Sigma_\delta$. Since $T=\frac{\partial}{\partial u}$ on the initial slice $\Sigma_\delta$, we have $[T,\frac{\partial}{\partial \vartheta}]=0$. We write $\frac{\partial}{\partial \slashed{\vartheta}}$ in terms of the Cartesian frame: 
\[\frac{\partial}{\partial \slashed{\vartheta}}=R^1\frac{\partial}{\partial x_1}+R^2\frac{\partial}{\partial x_2}.\]
The relation $[T,\frac{\partial}{\partial \vartheta}]=0$ gives $T(R^k)=\sum_{j=1}^2 {R}^j\partial_j (T^k)$ for $k=1,2$. Hence, we obtain the following ODE system for ${R}^1$ and ${R}^2$:
\begin{equation}\label{eq:partial theta on Sigma delta}
\begin{cases}
&T(R^k)= \sum_{j=1}^2 {R}^j\partial_j (\kappa \Th^k), \ \ k=1,2,\\
&( {R}^1, {R}^2)\big|_{S_{\delta,0}}=(\slashed{R}^1,\slashed{R}^2).
\end{cases}
\end{equation}
We integrate the above equations from $S_{\delta,0}$ to $S_{\delta,u}$ for $u \in [0,u^*]$. By Lemma \ref{lemma: bounds on geometry on sigma delta 1},  we conclude that
\[{\color{black}\big|R^1\big|+\big|R^2-1\big|\lesssim \varepsilon \delta}\]
on $\Sigma_{\delta}$. Since $g\big|_{\Sigma_\delta}$ is flat, this yields
\[\|\slashed{g} - 1\|_{L^\infty(\Sigma_{\delta})} \lesssim \varepsilon \delta.\]

\begin{proposition}\label{prop: on gslahsed}
For all $Z\in \mathscr{Z}=\{T,\Xh\}$, we have 
\[\| \slashed{g}-1\|_{L^\infty(\Sigma_\delta)}+\|Z(\slashed{g})\|_{L^\infty(\Sigma_\delta)}\lesssim \delta \varepsilon.\]
\end{proposition}
 \begin{proof}
 It remains to control $Z(\slashed{g})$. Since {\color{black}$X=\sqrt{\slashed{g}}\Xh$} and $\| \slashed{g}-1\|_{L^\infty(\Sigma_\delta)}\lesssim \delta \varepsilon$, the desired bound on $\Xh(\slashed{g})$ is equivalent to
 \[\|X(\slashed{g})\|_{L^\infty(\Sigma_\delta)}\lesssim \delta \varepsilon.\]
We first show that the above inequality holds on $S_{\delta,0}$. We commute $X$ with  \eqref{eq:partial theta on C0} to derive
\[
\begin{cases}
&L(X(\slashed{R}^k))=X^2(L^k), \ \ k=1,2,\\
&(X(\slashed{R}^1),X(\slashed{R}^2))\big|_{S_{0,0}}=(0,0).
\end{cases}
\]
Since $X^2(L^k)$  are of size $O(\varepsilon)$ on $C_0$, by integrating  from $0$ to $\delta$, we obtain that
\[\big|X(\slashed{R}^1)\big|+\big|X(\slashed{R}^2)\big|\lesssim \varepsilon \delta.\]
Next, we commute $X$ with  \eqref{eq:partial theta on Sigma delta} to derive
\begin{equation}\label{eq:partial theta on Sigma delta 1 derivative}
\begin{cases}
&T(X(R^k))= \sum_{j=1}^2 X({R}^j)\partial_j (\kappa \Th^k)+\sum_{j=1}^2 {R}^j X(\partial_j (\kappa \Th^k)), \ \ k=1,2,\\
&\big( X({R}^1),  X({R}^2)\big)\big|_{S_{\delta,0}}=\big( X(\slashed{R}^1), X(\slashed{R}^2)\big).
\end{cases}
\end{equation}
According to Lemma \ref{lemma: bounds on geometry on sigma delta 1} and Proposition \ref{prop: bounds on geometry on sigma delta 1}, we have $\partial_j (\kappa \Th^k)$ and $X(\partial_j (\kappa \Th^k))$ on the righthand side of the above equations are of size $O(\delta\varepsilon)$. By integrating equation from $S_{\delta,0}$ to $S_{\delta,u}$,  we conclude that
\begin{equation}\label{eq: construct theta aux 1}
\big|X(\slashed{R}^1)\big|+\big|X(\slashed{R}^2)\big|\lesssim \varepsilon \delta
\end{equation}
on $\Sigma_{\delta}$. Therefore,
\[|X(\slashed{g})|=|X \sqrt{\big( {R}^1\big)^2+\big({R}^2\big)^2}|\lesssim \varepsilon \delta.\]

To  estimate $T\slashed{g}$, we compute that
\[ T(\slashed{g}) =|T \sqrt{\big( {R}^1\big)^2+\big({R}^2\big)^2}|=\slashed{g}^{-1}\big({R}^1 T({R}^1)+R^2T({R}^1)\big).\]
We can use \eqref{eq:partial theta on Sigma delta} to compute the terms $T(R^1)$ and $T(R^2)$ in the expression of  $T(\slashed{g})$. In view of the estimates in Lemma \ref{lemma: bounds on geometry on sigma delta 1} and Proposition \ref{prop: bounds on geometry on sigma delta 1}, this yields $|T(\slashed{g})| \lesssim \varepsilon \delta$ which completes the proof of the proposition.
\end{proof}
 \begin{remark}
For all multi-indices $\alpha$ and for all $Z\in \mathscr{Z}=\{T,\Xh\}$, we can proceed in the same manner to show that
\begin{equation}\label{eq: higher order estimates on slashed g}
\|Z^\alpha(\slashed{g})\|_{L^\infty(\Sigma_\delta)}\lesssim \delta \varepsilon.
\end{equation}
 \end{remark}
 
\begin{remark}\label{remark: closing I infty 1}
By Lemma \ref{lemma: bounds on geometry on sigma delta 1}, Proposition \ref{prop: bounds on geometry on sigma delta 1} and Proposition \ref{prop: on gslahsed}, we have checked all the inequalities in $\mathbf{(I_{\infty,1})}$, see \eqref{initial Iinfty1}.
\end{remark}

\subsection{Algebraic preparations}\label{section: algebraic prepa}
We now introduce the proper algebraic language to describe the structure of the Euler equations. We start with a polynomial ring with includes all the quantities  in the null frame $(L,\Lb,\Xh)$ (and their derivatives) which appear in the Euler equations and the structure equations of the acoustical geometry.

We introduce the following sets of functions:
\begin{align*}
&\mathfrak{X}_{0,1}= \big\{\psi_2, w,\wb,c^{-1}\big\},\ \ \mathfrak{X}_{0,2}= \big\{ \slashed{g}, \kappa, \widehat{T}^1, \widehat{T}^2\big\},\ \ \mathfrak{X}_0= \mathfrak{X}_{0,1}\cup \mathfrak{X}_{0,2}.
\end{align*}
We also introduce the following set of differential operators
\[\mathfrak{D}=\{T,X,\Xh\}.\]
Given a positive integer $n$, we define the following set of {\bf order $n$} objects:
\[\mathfrak{Y}_n=\big\{ (\mathfrak{d}_1\circ\mathfrak{d}_2\circ\cdots \circ\mathfrak{d}_n)(x)\big|x\in \mathfrak{X}_0, \mathfrak{d}_i \in \mathfrak{D}, 0\leqslant i\leqslant n\big\}.\]
We also use $\mathfrak{Z}_{\leqslant n}$ to denote
\[\mathfrak{Z}_{\leqslant n}=  \bigcup_{k\leqslant n}\mathfrak{Y}_k \cup \big\{\frac{\mathfrak{d}_1 \circ \cdots \circ \mathfrak{d}_k(\kappa)}{\kappa} \big|1\leqslant k\leqslant n-1, \mathfrak{d}_i \in \mathfrak{D}\big\}.
\]
We emphasize that in $\mathfrak{Z}_{\leqslant n}$ the maximal number of derivatives for $\frac{\mathfrak{d}_1 \circ \cdots \circ \mathfrak{d}_k(\kappa)}{\kappa}$ are $n-1$. The definition for $\mathfrak{Z}_n$ will be clear in Remark \ref{rem:structure in LU}.

We consider the polynomial ring $\mathbb{R}[\mathfrak{Z}_{\leqslant n}]$, i.e., the set of all $\mathbb{R}$-coefficients polynomials with unknowns from $\mathfrak{Z}_{\leqslant n}$. We write a object from $\mathbb{R}[\mathfrak{Z}_{\leqslant n}]$ as ${\mathscr{P}}_n$ as a schematic expression. The following examples  help to elucidate the definition of  the symbol $\mathscr{P}_n$:
\begin{itemize}
\item For $c=\frac{\gamma-1}{2}(w+\wb)$ and $\psi_1= w-\wb$, we have $c,\psi_1\in \mathscr{P}_0$.
\item By \eqref{eq:explicit formula for U}, $U^{(\lambda)}=\mathscr{P}_0$ for $\lambda\in \{1,2,3\}$.
\item For $\mu=c\kappa$ and $h= \frac{1}{\gamma-1} c^2$, since $\kappa \in \mathfrak{Y}_0$, therefore, $\mu=\mathscr{P}_0$ and $h=\mathscr{P}_0$. 
\item For $\zeta =-\kappa\big(\widehat{T}^j\cdot \Xh(\psi_j) + \Xh(c)\big)$ and $\eta=-\kappa \widehat{T}^j\cdot \Xh(\psi_j)+c\Xh(\kappa)$, we have $\zeta=\mathscr{P}_1$ and $\eta=\mathscr{P}_1$.
\item We have  $\theta=\mathscr{P}_1$ and $\chi=\mathscr{P}_1$. This is clear from \eqref{defining eq of theta and chi}.
\item For $\chib=c^{-1}\kappa\big(-2\Xh^j\cdot \Xh(\psi_j)-\chi\big)$, we also have $\chib=\mathscr{P}_1$.

\end{itemize}
We remark that, in different situations, the polynomial ${\mathscr{P}}_n$ may change to another polynomial in $\mathbb{R}[\mathfrak{Z}_{\leqslant n}]$  but this will not affect the proof.  We also define the order of ${\mathscr{P}}_n$ as ${\rm ord}({\mathscr{P}}_n)=n$. For $n\leqslant m$, we also regard ${\mathscr{P}}_n$ as ${\mathscr{P}}_m$. 

Therefore, for all $\mathfrak{d} \in \mathfrak{D}$,  the following schematic formulas hold:
\begin{equation}\label{eq:schematic algebraic rules}
\mathfrak{d}\left({\mathscr{P}}_n\right)={\mathscr{P}}_{n+1},\ \ {\mathscr{P}}_n + {\mathscr{P}}_m ={\mathscr{P}}_{\max(m,n)},\ \ {\mathscr{P}}_n \cdot {\mathscr{P}}_m ={\mathscr{P}}_{\max(m,n)}.
\end{equation}
In fact, we only have to check for the elements from the second set in the definition of $\mathfrak{Z}_{\leqslant n}$. We make the following observation:
\[\mathfrak{d}\left(\frac{\mathfrak{d}_1 \circ \cdots \circ \mathfrak{d}_k(\kappa)}{\kappa}\right)=\frac{\mathfrak{d}\circ \mathfrak{d}_1 \circ \cdots \circ \mathfrak{d}_k(\kappa)}{\kappa}-\frac{\mathfrak{d}_1 \circ \cdots \circ \mathfrak{d}_k(\kappa)}{\kappa} \cdot \frac{\mathfrak{d}(\kappa)}{\kappa} \in \mathbb{R}[\mathfrak{Z}_{\leqslant n+1}].\]
The proof of \eqref{eq:schematic algebraic rules} is straightforward.

\subsubsection{Schematic computations}

In view of \eqref{structure eq 1: L kappa}, 
\eqref{structure eq 3: L T on Ti Xi Li},  \eqref{Euler equations:form 1} and \eqref{Euler equations:form 2}, we have 
\begin{equation}\label{eq:L x x from mathfrak X}
L(x)={\mathscr{P}}_{{\rm ord}(x)+1}, \ \ x \in \mathfrak{X}_{0,2}; \ L(y)= \frac{1}{\kappa}{\mathscr{P}}_{{\rm ord}(y)+1}, \ \  y \in \mathfrak{X}_{0,1}.
\end{equation}

For all $n\geqslant 1$, we will derive a schematic formula for $L\left({\mathscr{P}}_n\right)$. To simplify the notations, we  use $\mathfrak{d}^k$ to denote all possible differential operators $\mathfrak{d}_1\circ\mathfrak{d}_2\circ\cdots \circ\mathfrak{d}_k$ with $k\geqslant 1$, $\mathfrak{d}_i \in \mathfrak{D}$, $i\leqslant k$. 

We first derive a commutator formula for $[L,\mathfrak{d}^k]$. We observe that $[L,\Xh]$ and $[L,T]$ in \eqref{eq:commutator formulas} can be schematically written as $[L, \mathfrak{d}]={\mathscr{P}}_1 \cdot \mathfrak{d}$. We notice that the derivative $\mathfrak{d}$ on the righthand side is not $T$, i.e., there is no $T$ derivative in $[L, \mathfrak{d}]$. 
\begin{lemma}
 For all $k\geqslant 1$, we have
\begin{equation}\label{eq:commutator:schematic L commute with d^k}
[L,\mathfrak{d}^k]=\sum_{j=1}^k {\mathscr{P}}_j \cdot \mathfrak{d}^{k+1-j},
\end{equation}
where the top order operator on the righthand side $\mathfrak{d}^{k} \neq T^k$.
\end{lemma}
\begin{proof}
We prove by induction on $k$. It is clear that \eqref{eq:commutator:schematic L commute with d^k} holds for $k=1$. If it holds for $k$, the following provides a proof for the case $k+1$:
\begin{align*}
[L,\mathfrak{d}^{k+1}]x&=[L,\mathfrak{d}]\mathfrak{d}^k(x)+\mathfrak{d}\big([L,\mathfrak{d}^k]x\big)={\mathscr{P}}_1 \cdot \mathfrak{d}\big(\mathfrak{d}^k(x)\big)+\mathfrak{d}\big(\sum_{j=1}^k {\mathscr{P}}_j \cdot \mathfrak{d}^{k+1-j}\big)\\
&={\mathscr{P}}_1 \cdot \mathfrak{d}^{k+1}(x)+\sum_{j=1}^k \big(\underbrace{\mathfrak{d}({\mathscr{P}}_j)}_{{\mathscr{P}}_{j+1}}\cdot \mathfrak{d}^{k+1-j}+{\mathscr{P}}_j \cdot \mathfrak{d}^{k+2-j}\big)\\
&=\sum_{j=1}^{k+1} {\mathscr{P}}_j \cdot \mathfrak{d}^{k+2-j}.
\end{align*}
From the induction hypothesis, it is clear that $\mathfrak{d}^{k+1} \neq T^{k+1}$. This proves the formula \eqref{eq:commutator:schematic L commute with d^k}.
\end{proof}

\begin{lemma} \label{lem: L poly_n}For all $n\geqslant 1$ and ${\mathscr{P}}_n$, we have $L\left({\mathscr{P}}_n\right)=\frac{1}{\kappa}{\mathscr{P}}_{n+1}$.
\end{lemma}

\begin{proof}
By definition, each ${\mathscr{P}}_{n}$ can be written as a linear combination of monomials. Each such monomial $\mathfrak{m}$ can be written as the following form of product:
\[\mathfrak{m}=\mathfrak{d}^{i_1}(x_1)\cdot\mathfrak{d}^{i_2}(x_2) \cdots \mathfrak{d}^{i_s}(x_s) \cdot\frac{\mathfrak{d}^{j_1}(\kappa)}{\kappa}\cdot \frac{\mathfrak{d}^{j_2}(\kappa)}{\kappa} \cdots  \frac{\mathfrak{d}^{j_t}(\kappa)}{\kappa},\]
where $x_1,\cdots, x_s\in \mathfrak{X}_0$ with $\displaystyle\max_{a \leqslant s \atop \\ b\leqslant t}\big(i_a,j_b+1\big)\leqslant n$. According to the Leibniz rule, it suffices to understand each $L\big(\mathfrak{d}^{i_a}(x_a)\big)$ and $L\big(\frac{\mathfrak{d}^{j_b}(\kappa)}{\kappa}\big)$ term. Indeed, by \eqref{eq:L x x from mathfrak X} and \eqref{eq:commutator:schematic L commute with d^k}, we have
\begin{align*}
L\left(\mathfrak{d}^{i_a}x_a\right)&=\mathfrak{d}^{i_a}\left(L(x_a)\right)+\sum_{i'=1}^{i_a} {\mathscr{P}}_{i'} \cdot \mathfrak{d}^{i_a+1-{i'}}(x_a)=\mathfrak{d}^{i_a}\big(\frac{1}{\kappa}{\mathscr{P}}_{1}\big)+\sum_{i'=1}^{i_a} {\mathscr{P}}_{i_a}\\
&= \frac{1}{\kappa}{\mathscr{P}}_{i_a+1}.
\end{align*}
In the above calculations, we used the fact that $\mathfrak{d}(\kappa^{-1})=\kappa^{-1}{\mathscr{P}}_{1}$. We also have
\begin{align*}
L\big(\frac{\mathfrak{d}^{j_b}(\kappa)}{\kappa}\big)&=-\frac{L\kappa}{\kappa}\frac{\mathfrak{d}^{j_b}\kappa}{\kappa}+\frac{1}{\kappa}L\left(\mathfrak{d}^{j_b}(\kappa)\right).
\end{align*}
Similar to the computations for $L\left(\mathfrak{d}^{i_a}x_a\right)$, we have $L\kappa={\mathscr{P}}_{1}$ and $L\left(\mathfrak{d}^{j_b}(\kappa)\right)={\mathscr{P}}_{j_b+1}$. Thus,
\begin{align*}
L\big(\frac{\mathfrak{d}^{j_b}(\kappa)}{\kappa}\big)&=\frac{1}{\kappa}{\mathscr{P}}_{j_b+1}.
\end{align*}
Therefore, $L(\mathfrak{m})=\frac{1}{\kappa}{\mathscr{P}}_{n+1}$ for each monomial $\mathfrak{m}$ in the ${\mathscr{P}}_{n}$. This completes the proof of the lemma.
\end{proof}
\subsubsection{Euler equations in the diagonal schematic forms}
We follow the notations used in \eqref{eq:Euler in diagonalized form}. We rewrite the Euler equations \eqref{eq:Euler in diagonalized form} in the schematic form:
\begin{equation}\label{eq:Euler:L^n U in diagonalized form:n=1}
L(U)=\frac{c}{\kappa}\Lambda\cdot {T}(U)+\frac{c}{\kappa}\Lambda P^{-1} {T}(P)\cdot U+{\mathscr{P}}_1.
\end{equation}
\begin{remark}\label{rem:structure in LU}
According to \eqref{eq:Euler in diagonalized form}, we have
\[{\mathscr{P}}_1=cP^{-1} B P \cdot \Xh(U)-\big(P^{-1}L(P)-cP^{-1}B \Xh(P)\big)\cdot U.\]
In view of  \eqref{structure eq 3: L T on Ti Xi Li}, $\mathscr{P}_1$ does not involve  terms of the form $T(x)$ with $x\in \mathfrak{X}_{0,1}$. In particular, there are no terms of the form $T(U^{(\lambda)})$ for eigenvalues $\lambda \in \{0,-1,-2\}$. 
 \end{remark}

The next lemma provides computations for multiple $L$-derivatives on $U$:
\begin{lemma}\label{lemma:  L^n U computation}
For all integer $n\geqslant 1$, we have
\begin{equation}\label{eq:Euler:L^n U in diagonalized form}
\begin{split}
L^n (U)&=\underbrace{\big(\frac{c}{\kappa}\Lambda\big)^n\cdot {T}^n(U)}_{\mathbf{P}_{n,0}}+\underbrace{\frac{1}{\kappa^n}\sum_{k =1}^{n-1}{\mathscr{P}}_{n-k} \cdot T^k(U)}_{\mathbf{P}_{n,1}}+\underbrace{\frac{1}{\kappa^n}\sum_{k=1}^n {\mathscr{P}}_{n-k} \cdot {T}^k (P) }_{\mathbf{P}_{n,2}} +\underbrace{\frac{1}{\kappa^{n-1}}{\mathscr{P}}_n}_{\mathbf{P}_{n,3}},
\end{split}
\end{equation}
where $\mathbf{P}_{n,i}$ ($i=1,2,3,4$) denotes the four terms on the righthand side. Moreover, there is no $T^n(x)$ appearing in the last term $\mathbf{P}_{n,3}$ for $x\in \mathfrak{X}_{0,1}$.
\end{lemma}
\begin{proof}
 We do induction on $n$ to prove \eqref{eq:Euler:L^n U in diagonalized form}. In view of Remark \ref{rem:structure in LU}, the basic case for $n=1$ is clear. We now assume that \eqref{eq:Euler:L^n U in diagonalized form} holds for $n$ and we will prove it for $n+1$. 

We apply $L$ on both sides of \eqref{eq:Euler:L^n U in diagonalized form} and we treat the terms on the righthand side one by one.

First of all, we consider $L(\mathbf{P}_{n,0})$ and $L(\mathbf{P}_{n,1})$. By Leibniz rule and  Lemma \ref{lem: L poly_n}, if the derivative $L$ does not hit on $T^k(U)$,  it contributes terms into $\mathbf{P}_{n+1,1}$. Therefore, we have
\[L(\mathbf{P}_{n,0})+L(\mathbf{P}_{n,1})=\big(\frac{c}{\kappa}\Lambda\big)^n\cdot L\big({T}^n(U)\big)+\frac{1}{\kappa^n}\sum_{k =1}^{n-1}{\mathscr{P}}_{n-k} \cdot L\big(T^k(U)\big)+\mathbf{P}_{n+1,1}.\]
In view of  the commutator formula \eqref{eq:commutator:schematic L commute with d^k}, we have
\begin{align*}
L\left({T}^{n}(U)\right)&=T^n\left(L(U)\right)+\sum_{j=1}^n {\mathscr{P}}_j \cdot \underbrace{\mathfrak{d}^{n+1-j}(U)}_{={\mathscr{P}}_n}=T^n\left(L(U)\right)+{\mathscr{P}}_n.
\end{align*}
The last ${\mathscr{P}}_n$ will be assorted into $\mathbf{P}_{n+1,3}$. It remains to understand $T^{n}\left(L(U)\right)$. We then use \eqref{eq:Euler:L^n U in diagonalized form:n=1} to replace $L(U)$. By the Leibniz rule and ignoring the irrelevant constants,  we compute that
\begin{equation*}
\begin{split}
T^{n}\left(L(U)\right)&=T^n\big(\frac{c}{\kappa}\Lambda\cdot {T}(U)+\frac{c}{\kappa}\Lambda P^{-1} {T}(P)\cdot U+{\mathscr{P}}_1\big)\\
&=\underbrace{\Lambda\!\!\!\sum_{i_1+i_2+i_3=n} T^{i_1}(c)T^{i_2}\big(\frac{1}{\kappa}\big) T^{i_3+1}(U)}_{\mathbf{I_{1}}}+\underbrace{\sum_{i_1+i_2=n}T^{i_1}\big(\frac{1}{\kappa}\big) T^{i_2}\big({c}\Lambda P^{-1} {T}(P)\cdot U\big)}_{\mathbf{I_{2}}}+{\mathscr{P}}_{n+1}.
\end{split}
\end{equation*}
By the definition of the polynomial ring $\mathbb{R}[\mathfrak{Z}_{\leqslant n}]$, we have
\[T\big(\frac{1}{\kappa}\big)=-\frac{T\kappa}{\kappa^2}=-\frac{1}{\kappa}\cdot \frac{T\kappa}{\kappa}=\frac{1}{\kappa}\cdot {\mathscr{P}}_1.\]
Therefore, for all $k\geqslant 1$, we have $T^k\left(\frac{1}{\kappa}\right)=\frac{1}{\kappa}\cdot {\mathscr{P}}_k$. Hence,
\[\mathbf{I}_1=\frac{c}{\kappa}\Lambda\cdot {T}^{n+1}(U)+\frac{1}{\kappa}\sum_{j =1}^{n}{\mathscr{P}}_{n+1-j} \cdot T^j(U).\]
For $\mathbf{I}_2$, we have
\begin{align*}
\mathbf{I}_2&=\sum_{i_1+i_2=n}  \frac{{\mathscr{P}}_{i_1}}{\kappa} T^{i_2}\left({c}\Lambda P^{-1} {T}(P)\cdot U\right)=\sum_{j=1}^{n+1} \frac{{\mathscr{P}}_{n+1-j}}{\kappa} {T}^j(P).
\end{align*}
Hence,\[L(\mathbf{P}_{n,0})+L(\mathbf{P}_{n,1})=\big(\frac{c}{\kappa}\Lambda\big)^{n+1}\cdot {T}^{n+1}(U)+\mathbf{P}_{n+1,1}+\mathbf{P}_{n+1,2}+\mathbf{P}_{n+1,3}+\frac{1}{\kappa^n}{\mathscr{P}}_{n+1}.\]
We also notice that the only possible $T^{n+1}(x)$'s for $x\in \mathfrak{X}_{0,1}$ in the above computations appears in the first term of the righthand side of the above equation.

Secondly, we compute $L(\mathbf{P}_{n,2})$:
\begin{align*}
L(\mathbf{P}_{n,2})&=\sum_{k=1}^n \big(\frac{L\left({\mathscr{P}}_{n-k}\right) \cdot {T}^k (P)}{\kappa^n}+\frac{{\mathscr{P}}_{n-k} \cdot L\left({T}^k (P)\right) }{\kappa^n}-\frac{nL(\kappa){\mathscr{P}}_{n-k} \cdot {T}^k (P)}{\kappa^{n+1}}\big)\\
&=\mathbf{P}_{n+1,2}+\sum_{k=1}^n \frac{{\mathscr{P}}_{n-k} \cdot L\left({T}^k (P)\right) }{\kappa^n}
\end{align*}
where we used Lemma \ref{lem: L poly_n}. In order to compute $L\left({T}^k (P)\right)$ for $1\leqslant k \leqslant n$, we use \eqref{eq:commutator:schematic L commute with d^k} to derive
\begin{align*}
{\color{black}L\left({T}^{k}(P)\right)}&=T^k\left(L(P)\right)+\sum_{j=1}^k {\mathscr{P}}_j \cdot \underbrace{\mathfrak{d}^{k+1-j}(P)}_{={\mathscr{P}}_k}=T^k\left(L(P)\right)+{\mathscr{P}}_k.
\end{align*}
By \eqref{structure eq 3: L T on Ti Xi Li} and the definition of $P$,  $L(P)=\mathscr{P}_1$. Hence, $L\left({T}^{k}(U)\right)={\mathscr{P}}_{k+1}$. Therefore,
\begin{align*}
L(\mathbf{P}_{n,2}) =\mathbf{P}_{n+1,2}+\frac{1}{\kappa^n}{\mathscr{P}}_{n+1}.
\end{align*}

Finally, we compute $L(\mathbf{P}_{n,3})$ as follows:
\begin{align*}
L(\mathbf{P}_{n,3})&=-(n-1)\frac{L(\kappa)}{\kappa^{n}}{\mathscr{P}}_n+\frac{1}{\kappa^{n-1}}L\left({\mathscr{P}}_n\right)=\frac{1}{\kappa^{n}}{\mathscr{P}}_n+\frac{1}{\kappa^{n}} {\mathscr{P}}_{n+1}=\frac{1}{\kappa^{n}} {\mathscr{P}}_{n+1}.
\end{align*}
It is of the form $\mathbf{P}_{n+1,3}$.

It is clear that in the computations of $L(\mathbf{P}_{n,2})$ and $L(\mathbf{P}_{n,3})$, there is no $T^{n+1}(x)$ type terms where $x\in \mathfrak{X}_{0,1}$. By putting the formulas of $L(\mathbf{P}_{n,i})$'s together, this completes the proof of the lemma.
\end{proof}
\begin{remark}\label{rem: contribution to P_n1}
In the inductive process, the terms in  $\mathbf{P}_{n+1,1}$ are from $L(\mathbf{P}_{n,0})$ and $L(\mathbf{P}_{n,1})$. There is no contribution from $L(\mathbf{P}_{n,2})$ and $L(\mathbf{P}_{n,3})$ to $\mathbf{P}_{n+1,1}$.
\end{remark}

\subsection{The formal Taylor series for diagonalized variables}\label{section: determine the jets}

Since $\Th$ is already fixed on $\Sigma_\delta$, in view of \eqref{eq:explicit formula for U}, the construction of $C^N$ data $(v,c)$ for rarefaction waves (see Definition \ref{def:rare data}) is equivalent  to the construction of $U$ on $\Sigma_\delta$.

For the given integer $N\geqslant 1$, we consider a {\bf formal} finite Taylor series of $U$ on $\Sigma_{\delta}$ where we expand $U$ in the variable $u$:
\begin{equation}\label{Taylor expansions of U up to order N}
\begin{cases}
&U^{(0)}(u,\vartheta) = U^{(0)}_0(\vartheta) + U^{(0)}_1(\vartheta) u +  \cdots+ \frac{U^{(0)}_N(\vartheta)}{N!} u^N,\\
&U^{(-1)}(u,\vartheta) = U^{(-1)}_0(\vartheta) + U^{(-1)}_1(\vartheta) u +  \cdots+ \frac{U^{(-1)}_N(\vartheta)}{N!} u^N,  \\
&U^{(-2)}(u,\vartheta) = U^{(-2)}_0(\vartheta) + U^{(-2)}_1(\vartheta) u +  \cdots+ \frac{U^{(-2)}_N(\vartheta)}{N!} u^N.
\end{cases}
\end{equation}
Since $T=\frac{\partial}{\partial u}$, for $\lambda \in \{0,-1,-2\}$ and $0\leqslant n\leqslant N$, we have $T^n\big(U^{(\lambda)}\big)\big|_{u=0}=U^{(\lambda)}_n(\vartheta)$. We emphasize that the above Taylor series of $U$ are formal in the sense that it only determines the jets of $U^{(\lambda)}$'s of order at most $N$ at $S_{\delta,0}$.

Let $U^{(\lambda)}{}_r$ be the diagonal form of the solution of Euler equations with data $(v_r,c_r)$ in the domain $\mathcal{D}_0$. We observe that $\Th$ is indeed smooth across $C_0$. In view of \eqref{eq:Euler in diagonalized form}, the conditions  C2) and C3)  in Definition \ref{def:rare data} is equivalent to the following ones:
\begin{equation}\label{eq:matching condition}
\begin{cases}
U^{(\lambda)}\big|_{S_{\delta,0}}=U^{(\lambda)}{}_r\big|_{S_{\delta,0}} \ \  &\text{(continuity on boundary)},\\
L^n\big(U^{(\lambda)}\big)\big|_{S_{\delta,0}}=L^n\big(U^{(\lambda)}{}_r\big)\big|_{S_{\delta,0}}, \  1\leqslant n\leqslant N  \ \ &\text{(higher order matching)},
\end{cases}
\end{equation}
where $\lambda\in \{0,-1,-2\}$.
\begin{remark}
	We notice that, since $X$ is tangential to $\Sigma_\delta$ and $[L,X]=0$,  in order to prove \eqref{compatibility condition 3}, it suffices to prove for $n=0$ and $m\leqslant N$.
\end{remark}

\begin{remark}[Taylor coefficients for $U^{(0)}$]
We impose the Taylor coefficients of $U^{(0)}$ as follows:
\begin{equation}\label{eq:Taylor coefficients of U0}
U^{(0)}_{1}(\vartheta)\equiv -\frac{2}{\gamma+1}, \ \ 
U^{(0)}_{k}(\vartheta)\equiv 0,  \ \  2\leqslant k \leqslant N.
\end{equation}
Therefore, the formal Taylor expansion of $U^{(0)}$ is as follows:
\begin{equation}\label{eq: formula for U0}
U^{(0)}(u,\vartheta) = U^{(\lambda)}{}_r(\vartheta) -\frac{2}{\gamma+1} u .
\end{equation}
The coefficients $U^{(0)}_{k}(\vartheta)$'s for $k\geqslant 1$ indeed can be freely prescribed. Our choice is motivated by the data at the singularity $\mathbf{S}_*$  \eqref{eq:data at the singularity} where we show that higher $U^{(0)}_{k}(\vartheta)\equiv 0$ for $k\geqslant 2$. 
The choice of $U^{(0)}_{1}(\vartheta)$ is  indispensable in the nonlinear energy estimates in our first paper \cite{LuoYu1}, guaranteeing that $\mathring{z}$ is of size $O(\varepsilon)$. 
\end{remark}
The choice of $U^{(0)}_{k}(\vartheta)$ is also related to the one dimensional picture depicted as follows:
\begin{center}
	\includegraphics[width=2.5in]{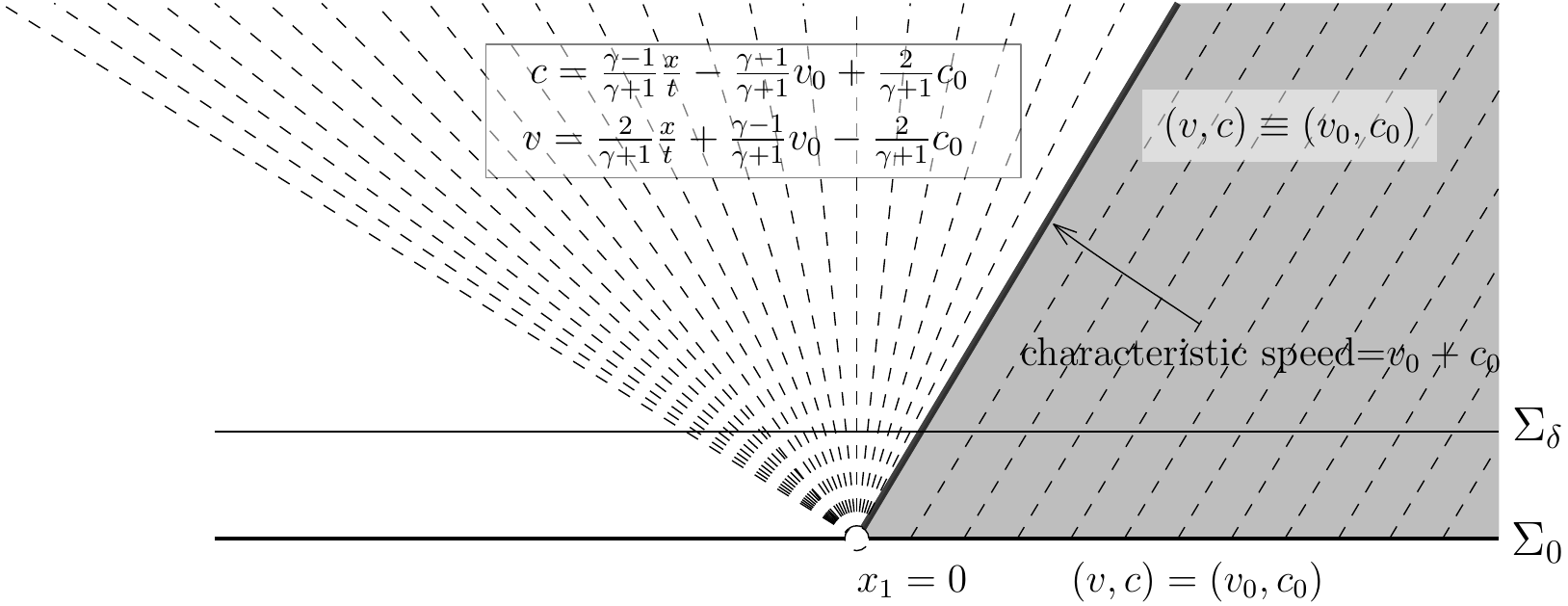}
\end{center}
The unique family of front  rarefaction waves in the dashed region connected to the constant data $(v,c)\big|_{t=0}=(v_0,c_0)$ on $x\geqslant 0$ is 
given by \eqref{eq: precise solution for front rarefaction waves}, or equivalent,
\begin{equation}\label{eq:1D-rarefaction-wave}
	\begin{cases}
		v&=\frac{2}{\gamma+1}\frac{x}{t}+\big(\frac{\gamma-1}{\gamma+1}v_0-\frac{2}{\gamma+1}c_0\big),\\
		c&=\frac{\gamma-1}{\gamma+1}\frac{x}{t}-\big(\frac{\gamma-1}{\gamma+1}v_0-\frac{2}{\gamma+1}c_0\big).
	\end{cases}
\end{equation}
The acoustical function $u$ and vector field $T$ are given by $u=-\frac{x}{t}$ and  $T=-t\partial_x$ on $\Sigma_t$. We compute that  \begin{equation}\label{eq:1D-rarefaction-wave diagonalized}
	U^{(0)}=\frac{1}{2}\big[\frac{4}{\gamma+1}\frac{x}{t}+\frac{\gamma-3}{\gamma-1}\big(\frac{\gamma-1}{\gamma+1}v_0-\frac{2}{\gamma+1}c_0\big)\big].
\end{equation}
Therefore, we have $T\big(U^{(0)}\big)=-\frac{2}{\gamma+1}$. 

\subsubsection{The inductive argument}\label{subsubsection:inductive construction of Taylor coefficients}
We use an induction argument on $n$ to construct Taylor coefficients $U^{(\lambda)}_n$ for $\lambda=-1,-2$ so that \eqref{eq:matching condition} is satisfied.  The first equation in \eqref{eq:matching condition} determines the case for $n=0$.

We make the following inductive hypothesis: for all $k\leqslant n$, $U^{(\lambda)}_k$'s have been constructed in such a way that $\slashed{\mathfrak{d}}^k\big(U^{(\lambda)}\big)=\slashed{\mathfrak{d}}^k\big(U^{(\lambda)}{}_r\big)$ holds on $S_{\delta,0}$ for all possible $\slashed{\mathfrak{d}}\in \slashed{\mathfrak{D}}=\big\{L,\Xh,X\big\}$. The notation $\slashed{\mathfrak{d}}^k$ denotes any possible differential operator $\slashed{\mathfrak{d}}_1\circ\slashed{\mathfrak{d}}_2\circ\cdots \circ\slashed{\mathfrak{d}}_k$ with $k\geqslant 1$, $\slashed{\mathfrak{d}}_i \in \mathfrak{D}$, $i\leqslant k$. In particular, we have $L^k\big(U^{(\lambda)}\big)=L^k\big(U^{(\lambda)}{}_{r}\big)$ on $S_{\delta,0}$ for all $k\leqslant n$. 

To determine the $(n+1)$-th Taylor coeffients,  we use the condition that $L^{n+1}\big(U^{(\lambda)}\big)=L^{n+1}\big(U^{(\lambda)}\big){}_{r}$  on $S_{\delta,0}$. By the formula \eqref{eq:Euler:L^n U in diagonalized form}, for each $\lambda \in \{0,-1,-2\}$, we have
\begin{equation}\label{eq:improve on mu}
L^{n+1} (U^{(\lambda)})\big|_{S_{\delta,0}}= \big(\frac{c}{\kappa}\Lambda\big)^{n+1}\cdot {T}^{n+1}(U^{(\lambda)})+ \mathbf{P}_{n+1,1}+ \mathbf{P}_{n+1,2}+ \mathbf{P}_{n+1,3}.
\end{equation}
According to the conclusion of Lemma \ref{lemma:  L^n U computation}, except for $T^{n+1}(P)$, the number of $T$ derivatives in each single term appearing in $\mathbf{P}_{n+1,1}$, $\mathbf{P}_{n+1,2}$ and $\mathbf{P}_{n+1,3}$ is at most $n$. Therefore, they are already determined by the 
the first $n$ Taylor coefficients. For $T^{n+1}(P)$, since $P$ is determined $\Th^1$ and $\Th^2$ which are already given (once the function $u$ is given) on $\Sigma_{\delta}$, this term is also known. Therefore, the last three terms on the right of \eqref{eq:improve on mu} have already been determined by the inductive hypothesis. On the other hand, the argument also shows that adding the $n+1$'s Taylor term does not change the value of $L^k\left(U^{(\lambda)}\right)$ terms on $S_{\delta,0}$ for $k\leqslant n$.

Therefore, subject to the matching conditions $L^{n+1}\big( U^{(\lambda)}\big)\big|_{S_{\delta,0}}=L^{n+1}\big(U^{(\lambda)}{}_{r}\big)\big|_{S_{\delta,0}}$, for the nonzero eigenvalues $\lambda=-1$ and $-2$, we set 
\begin{equation}\label{eq:aaa4}
U^{(\lambda)}_{n+1}(\vartheta)=\big(\frac{\kappa}{\lambda c}\big)^{n+1}\big( L^{n+1}(U^{(\lambda)}{}_{r})(\delta,0,\vartheta)- \mathbf{P}_{n+1,1}- \mathbf{P}_{n+1,2}- \mathbf{P}_{n+1,3}\big).
\end{equation}

For the eigenvalue $\lambda = 0$,  the Euler equations indeed guarantee that $L^{n+1}\big( U^{(0)}\big)\big|_{S_{\delta,0}}=L^{n+1}\big(U^{(0)}{}_{r}\big)\big|_{S_{\delta,0}}$ automatically holds. To see this, we take the component corresponding to $\lambda=0$ in \eqref{eq:Euler in diagonalized form}:
 \begin{equation}\label{eq: LU0}
 L\big(U^{(0)}\big)=\big(c P^{-1} B P \cdot \Xh(U)\big)^{(0)}+\big(\big(-P^{-1}L(P)+cP^{-1}B \Xh(P)\big)\cdot U\big)^{(0)}.
\end{equation}
where the upper index $(0)$ means that we take the first component of the vector. By applying \eqref{structure eq 3: L T on Ti Xi Li} to  $L(P)$, we see that the righthand side of the equation only involves $\Xh$-derivatives of the elements from $\mathfrak{X}_{0}$. Therefore, by applying the operator $L^{n}$ to \eqref{eq: LU0}, we obtain that 
\[L^{n+1}\big(U^{(0)}\big)={\mathscr{Q}}_{n+1}.\] 
This ${{\mathscr{Q}}}_{n+1}$ has a special structure: it is a polynomial of the unknowns from the set $\big\{L^j\widehat{X} (x)\big| x\in\mathfrak{X}_{0}, j\leqslant n\big\}$. In view of \eqref{structure eq 1: L kappa},  \eqref{structure eq 3: L T on Ti Xi Li},  \eqref{Euler equations:form 1},  for all $x\in \mathfrak{X}_{0,1}$, a term of the form $L^i\Xh^j (x)$ with total order at most $n$ only involves the Taylor coefficients (and their $\Xh$ or $X$ derivatives) up to order $n$, i.e., they are determined in an algebraic way by the $\Xh$ derivatives of  $U^{(\lambda)}_1(\vartheta),\cdots,U^{(\lambda)}_n(\vartheta)$ which are already given. Since we have $L^j(U)$ already matches with $L^j(U_{r})$ on $S_{\delta,0}$ for all $j\leqslant n$, we have $\Xh L^j(U)\big|_{S_{\delta,0}}={\color{black}\Xh L^j(U_r)}\big|_{S_{\delta,0}}$ for $j\leqslant n$. Therefore, by commuting derivatives, we obtain that $L^j\Xh (U)\big|_{S_{\delta,0}}= L^j\Xh (U_r)\big|_{S_{\delta,0}}$ for $j\leqslant n$. This shows $L^{n+1}\big( U^{(0)}\big)\big|_{S_{\delta,0}}=L^{n+1}\big(U^{(0)}{}_r\big)\big|_{S_{\delta,0}}$.

We also remark that for all possible derivatives $\slashed{\mathfrak{d}}^{n+1} U$ where $\slashed{\mathfrak{d}}=L$ or $\Xh$, we use the commutator formula \eqref{eq:commutator formulas} to move all $\Xh$ derivatives to the front of the expression. Since terms of the form $\widehat{X}^j L^i(U)$ already match with those from $\mathcal{D}_0$. Therefore,  each $\slashed{\mathfrak{d}}^{n+1} U$ match  with the corresponding one from the $\mathcal{D}_0$ side. This finishes the inductive argument.
\begin{remark}
Together with \eqref{eq:Taylor coefficients of U0}, the Taylor coefficients $U^{(\lambda)}_1(\vartheta),\cdots,U^{(\lambda)}_N(\vartheta)$ in \eqref{Taylor expansions of U up to order N} are uniquely determined by the solution $(v_r,c_r)$ on $\mathcal{D}_0$.
\end{remark}

\subsubsection{The extra power of $\kappa$}

As a byproduct of the previous construction, we show that, as long as there is a $T$-derivative, the derivatives of $U^{(-1)}$ and $U^{(-2)}$ on $S_{\delta,0}$ has an extra $\delta$ factor in $L^\infty$ norm. This is completely different from the bounds on the derivatives of $U^{(-1)}$ and $U^{(-2)}$ on $\mathcal{D}_0$ side. It manifests the nature of the centered rarefaction waves.

 For $\lambda = -1$ or $-2$, by  \eqref{eq:Euler in diagonalized form}, we have
\[
 L\big(U^{(\lambda)}\big)=\frac{c}{\kappa}\lambda\cdot {T}\big(U^{(\lambda)}\big)+\big(cP^{-1} B P \cdot \Xh(U)\big)^{(\lambda)}+\big(\big( \frac{c}{\kappa}\Lambda P^{-1} {T}(P)-P^{-1}L(P)+cP^{-1}B \Xh(P)\big)\cdot U\big)^{(\lambda)},
\]
where the upper index $(\lambda)$ indicates the corresponding components. We can use the schematic language to rewrite these equations as 
\begin{equation}\label{eq: extra kappa aux 1}
 \frac{\kappa}{\lambda c} L\big(U^{(\lambda)}\big)+\kappa \mathscr{P}_1= {T}\big(U^{(\lambda)}\big)+\big(\big( \lambda \Lambda P^{-1} {T}(P) \big)\cdot U\big)^{(\lambda)},
\end{equation}
where $\mathscr{P}_1$ only consists of $\Xh$ derivatives of $U$ (and other terms of order $0$). Since $\Xh(U)$ and $L(U)$ match with $\Xh(U_r)$ and $L(U_r)$ from $\mathcal{D}_0$, we have $\|\mathscr{P}_1\|_{L^\infty(S_{\delta,0})}\lesssim \delta \varepsilon$ and $\big\|L\big(U^{(\lambda)}\big)\big\|_{S_{\delta,0}}\lesssim   \varepsilon$. We also have $c\approx 1$ on $S_{\delta,0}$ thanks to the first condition in \eqref{eq:matching condition}. By Proposition \ref{prop: bounds on geometry on sigma delta 1}, we have $\|T(P)\|_{L^\infty(\Sigma_{\delta})}\lesssim \delta \varepsilon$, Hence, \eqref{eq: extra kappa aux 1} implies that
  \begin{equation}\label{eq: bound on T U on S delta 0}
 \| {T}\big(U^{(\lambda)}\big)\|_{L^\infty(S_{\delta,0})}\lesssim \delta \varepsilon, \ \ \lambda=-1,-2.
\end{equation}
In other words, the computation yields
\begin{equation}\label{eq: bound on U1}
\big\| U^{(\lambda)}_1(\vartheta)\big\|_{L^\infty_{\vartheta}}\lesssim \delta \varepsilon, \ \ \lambda=-1,-2.
\end{equation}
\begin{remark}
For all $l\geqslant 0$, we can apply $X^l$ to \eqref{eq: extra kappa aux 1} and use the same argument to derive:
\begin{equation}\label{eq: bound on U1 with X}
 \big\|X^l\big( U^{(\lambda)}_1(\vartheta)\big)\big\|_{L^\infty_{\vartheta}}\lesssim \delta \varepsilon, \ \ \lambda=-1,-2.
\end{equation}
\end{remark}
We now perform an induction argument on $n$ to show that for all $n\leqslant N$, the following two inequalities hold simultaneously:
\begin{equation}\label{eq: bound on Un}
 \big\| U^{(\lambda)}_n(\vartheta)\big\|_{L^\infty}\lesssim \delta \varepsilon, \ \ \lambda=-1,-2.
\end{equation}
and for all $\mathscr{P}_j$ appearing in $\mathbf{P}_{n,1},\mathbf{P}_{n,2}$ and $\mathbf{P}_{n,3}$ in \eqref{eq:Euler:L^n U in diagonalized form}, we have
\begin{equation}\label{eq: bound on Pj}
 \|\mathscr{P}_j\big\|_{L^\infty(S_{\delta,0})}\lesssim  \varepsilon.
\end{equation}

The previous analysis proves the case for $n=1$. We make the inductive hypothesis that  \eqref{eq: bound on Un} and \eqref{eq: bound on Pj} hold up to $n$. To show they hold for $n+1$, by the constructions of $\mathbf{P}_{n+1,1},\mathbf{P}_{n+1,2}$ and $\mathbf{P}_{n+1,3}$ in Section \ref{subsubsection:inductive construction of Taylor coefficients}, it is clear that the $\mathscr{P}_j$'s appearing in $\mathbf{P}_{n+1,1},\mathbf{P}_{n+1,2}$ and $\mathbf{P}_{n+1,3}$ are polynomials of the $\mathscr{P}_j$'s and their derivatives appearing in $\mathbf{P}_{n,1},\mathbf{P}_{n,2}$ and $\mathbf{P}_{n,3}$, see Remark \ref{rem: contribution to P_n1}. They are all determined by the data on the $\mathcal{D}_0$ side. On the other hand, there is no $T^{n+1}U$ appearing in those $\mathscr{P}_j$'s. We also recall Proposition \eqref{prop: bounds on geometry on sigma delta 1} that $\|{T}^k (P)\|_{L^{\infty}(\Sigma_\delta)}\lesssim\delta \varepsilon$ for all $1\leqslant k\leqslant N$. Thus, by the inductive hypothesis, \eqref{eq: bound on Pj} holds for $n+1$. To prove \eqref{eq: bound on Un} for $n+1$, we recall that
\begin{equation}\label{eq: extra kappa aux 2}
U^{(\lambda)}_{n+1}(\vartheta)=\big(\frac{\kappa}{\lambda c}\big)^{n+1}\big( L^{n+1}(U^{(\lambda)}{}_{r})(\delta,0,\vartheta)- \mathbf{P}_{n+1,1}- \mathbf{P}_{n+1,2}- \mathbf{P}_{n+1,3}\big).
\end{equation}
Since $\|{T}^k (P)\|_{L^{\infty}(\Sigma_\delta)}\lesssim\delta \varepsilon$ for all $1\leqslant k\leqslant N$, in view of the definition of $\mathbf{P}_{n+1,2}$ and $\mathbf{P}_{n+1,3}$, i.e.,
\[
\mathbf{P}_{n+1,2} =\sum_{k=1}^{n+1} \frac{{\mathscr{P}}_{n+1-k} \cdot {T}^k (P) }{\kappa^{n+1}},\ \ \mathbf{P}_{n+1,3}= \frac{1}{\kappa^{n}}{\mathscr{P}}_{n+1},\]
it is straightforward to see that
\[\big\|\big(\frac{\kappa}{\lambda c}\big)^{n+1} \mathbf{P}_{n+1,2}\big\|_{L^{\infty}(S_{\delta,0})}+\big\|\big(\frac{\kappa}{\lambda c}\big)^{n+1} \mathbf{P}_{n+1,3}\big\|_{L^{\infty}(S_{\delta,0})}\lesssim\delta \varepsilon.\]
It is also clear that, for $\lambda=-1$ and $-2$, we have
\[\big\|\big(\frac{\kappa}{\lambda c}\big)^{n+1}\big( L^{n+1}(U^{(\lambda)}{}_{r})(\delta,0,\vartheta)\big)\big\|_{L^{\infty}(S_{\delta,0})}\lesssim\delta \varepsilon.\]
It remains to bound contribution of $\mathbf{P}_{n+1,1}$ to $U^{(\lambda)}_{n+1}(\vartheta)$. In fact, for $\lambda=-1$ or $-2$, we have
\[{\color{black}\big(\frac{\kappa}{\lambda c}\big)^{n+1}}\mathbf{P}_{n+1,1}=\frac{1}{\lambda^{n+1} c^{n+1}}\sum_{k =1}^{n}{\mathscr{P}}_{n+1-k} \cdot T^k(U^{(\lambda)}).\]
Since \eqref{eq: bound on Un} and \eqref{eq: bound on Pj} hold for all $1\leqslant k\leqslant n$, we obtain that
\[\big\|\big(\frac{\kappa}{\lambda c}\big)^{n+1} \mathbf{P}_{n+1,1}\big\|_{L^{\infty}(S_{\delta,0})}\lesssim\delta \varepsilon.\]
Putting all the pieces together, we obtain that
\[ \| U^{(\lambda)}_{n+1}(\vartheta)\big\|_{L^\infty}\lesssim \delta \varepsilon.\]
This yields \eqref{eq: bound on Un} and completes the induction argument.
\begin{remark}
We can also apply $X^l$ to \eqref{eq: extra kappa aux 2} and use the same inductive argument  to derive:
\begin{equation}\label{eq: bound on Un with X}
 \big\| X^l\big(U^{(\lambda)}_n(\vartheta)\big)\big\|_{L^\infty}\lesssim \delta \varepsilon, \ \ \lambda=-1,-2,
\end{equation}
where $n\geqslant 1$.
\end{remark}

\subsection{The Taylor series for Riemann invariants}

Accroding to \eqref{eq:explicit formula for U}, for all $0\leqslant k\leqslant N$, we have
\begin{equation}\label{eq:definition for the Taylor coefficients}
 \begin{cases}
  &\wb_{;k}(\vartheta):=T^k(\wb)\big|_{u=0}=\frac{1}{2}T^k\big[(1-\Th^1)U^{(0)}+  \Th^2 U^{(-1)}+ (1+\Th^1)U^{(-2)}\big]\big|_{u=0},\\
  &w_{;k}(\vartheta):= T^k(w)\big|_{u=0}=  \frac{1}{2}T^k\big[(1-\Th^1)U^{(-2)}+(1+\Th^1)U^{(0)}-   \Th^2U^{(-1)}\big]\big|_{u=0},\\
   &{\psi_2}_{;k}(\vartheta):=T^k(\psi_2)\big|_{u=0}= T^k\big[\Th^1 U^{(-1)}+\Th^2U^{(0)}-\Th^2 U^{(-2)}\big]\big|_{u=0}.
   \end{cases}
  \end{equation}
  Therefore, $\wb_{;k}(\vartheta),w_{;k}(\vartheta)$ and ${\psi_2}_{;k}(\vartheta)$ are uniquely determined. According to \eqref{eq: bound on Un with X} and Proposition \ref{prop: bounds on geometry on sigma delta 1},  we have the following estimates for the jets of the Riemann invariants:
  \begin{equation}\label{bounds on the coefficients}
   \begin{cases}
    &\big\| X^l\big(\wb_{;0}(\vartheta)\big)\big\|_{L^\infty}+\big\| X^l\big(w_{;0}(\vartheta)\big)\big\|_{L^\infty}+\big\| X^l\big({\psi_2}_{;0}(\vartheta)\big)\big\|_{L^\infty}\lesssim \varepsilon, \ \ l\geqslant 1;\\
 &\big\| X^l\big(w_{;k}(\vartheta)\big)\big\|_{L^\infty}+\big\| X^l\big({\psi_2}_{;k}(\vartheta)\big)\big\|_{L^\infty}\lesssim \delta \varepsilon, \ \ k\geqslant 1, l\geqslant 0;\\
 &\big\|X^l\big(\wb_{;k}(\vartheta)\big)\big\|_{L^\infty}+\big\| X^{l'} \big(\wb_{;1}(\vartheta)\big)\big\|_{L^\infty}\lesssim \delta \varepsilon, \ \ l'\geqslant 1, k\geqslant 2, l\geqslant 0.
   \end{cases}
\end{equation}

We define the following Taylor series on $\Sigma_\delta$ (where $u\in [0,u^*]$):
\begin{equation}\label{Taylor expansions of Riemann invariants up to order N}
\begin{cases}
&\wb(\delta,u,\vartheta) = \wb_{;0}(\vartheta) + \wb_{;1}(\vartheta)u +  \cdots+ \frac{1}{N!} \wb_{;N}(\vartheta)u^N,\\
&w(\delta,u,\vartheta) = w_{;0}(\vartheta) + w_{;1}(\vartheta)u +  \cdots+ \frac{1}{N!} w_{;N}(\vartheta)u^N,\\
&{\psi_2}(\delta,u,\vartheta) = {\psi_2}_{;0}(\vartheta) + {\psi_2}_{;1}(\vartheta)u +  \cdots+ \frac{1}{N!} {\psi_2}_{;N}(\vartheta)u^N.
\end{cases}
\end{equation}
\begin{remark}
Rather than being formal Taylor series, $\wb(\delta,u,\vartheta),w(\delta,u,\vartheta)$ and $\psi_2(\delta,u,\vartheta)$ are functions given on $\Sigma_\delta$. They will serve as the initial data for the Euler equations.
\end{remark}
  
  \subsubsection{Preliminary pointwise bounds on the Riemann invariants on $\Sigma_\delta$}

  The bound in \eqref{bounds on the coefficients} can be interpreted as the estimate fo $\wb, w$ and $\psi_2$ on $S_{\delta,0}$. They indeed hold on the entire $\Sigma_\delta$:
\begin{proposition}\label{prop: correct data size}
For all $k+l\leqslant N$, we have the following pointwise bounds on $\Sigma_\delta$:
\begin{equation}\label{eq: bound on Riemann Invariants on Sigma delta}
\begin{cases}
 &\|{T}^kX^l\big(w\big)\|_{L^\infty(\Sigma_{\delta})}+\|{T}^kX^l\big(\psi_2\big)\|_{L^\infty(\Sigma_{\delta})}\lesssim \delta \varepsilon, \ \  k \geqslant 1;\\
 &\| X^l\big(w\big)\|_{L^\infty(\Sigma_{\delta})}+\| X^l\big(\psi_2\big)\|_{L^\infty(\Sigma_{\delta})}\lesssim \delta \varepsilon, \ \ 0\leqslant l \leqslant N. 
  \end{cases}
\end{equation}
and
\begin{equation}\label{eq: bound on wb on Sigma delta}
\begin{cases}
 &\|{T}^kX^l\big(\wb\big)\|_{L^\infty(\Sigma_{\delta})}\lesssim \delta \varepsilon, \ \  k \geqslant 1, \ k+l\geqslant 2;\\
 &\| X^l\big(\wb\big)\|_{L^\infty(\Sigma_{\delta})}\lesssim \delta \varepsilon, \ \ 2\leqslant k \leqslant N. 
  \end{cases}
\end{equation}
\end{proposition}
\begin{proof}
Since the estimates can be derived in similar manners, we only provide the proof for $w$. We apply ${T}^kX^l$ to \eqref{Taylor expansions of Riemann invariants up to order N} to derive
\[
{T}^kX^l\big(w(\delta,u,\vartheta)\big) = \sum_{k=0}^N\frac{1}{j!} X^l\big(w_{;j}(\vartheta)\big) \cdot T^k(u^j).
\]
Therefore, the conclusion follows immediately from \eqref{bounds on the coefficients}.
\end{proof}
We now show that the pointwise bounds on Riemann invariants required by $\mathbf{(I_\infty)}
$ hold:
\begin{proposition}\label{prop:pointwise bound on derivatives of psi}
Let $\psi \in \{\wb,w,\psi_2\}$. We have the following pointwise estimates:
\begin{equation*}
\begin{cases}
	&\|L\psi\|_{L^\infty(\Sigma_\delta)} + \|\Xh\psi\|_{L^\infty(\Sigma_\delta)} \lesssim \varepsilon;\\
	&\|T(w)\|_{L^\infty(\Sigma_\delta)} + \|T(\psi_2)\|_{L^\infty(\Sigma_\delta)} +  \|T\wb + \frac{2}{\gamma+1}\|_{L^\infty(\Sigma_\delta)} \lesssim \varepsilon \delta; \\
	&\|LZ^{\alpha}\psi\|_{L^\infty(\Sigma_\delta)}+\|\Xh Z^{\alpha}\psi\|_{L^\infty(\Sigma_\delta)}+\delta^{-1}\|TZ^{\alpha}\psi\|_{L^\infty(\Sigma_\delta)} \lesssim \varepsilon, \ \  Z \in \{\Xh, T\}, \ 1 \leqslant |\alpha| \leqslant N-1;
\end{cases}
\end{equation*}
\end{proposition}
\begin{proof}
The bounds on $\Xh(\wb),\Xh(w),\Xh(\psi_2), T(w)$ and $T(\psi_2)$ are already presented in \eqref{eq: bound on Riemann Invariants on Sigma delta}. 

First of all, we consider $T\wb + \frac{2}{\gamma+1}$. By \eqref{eq:definition for the Taylor coefficients}, we have
\[T\wb  =  \big(\wb_{;1}(\vartheta) +\frac{2}{\gamma+1}\big) +  \sum_{k=2}^N\frac{1}{(k-1)!} \wb_{;k}(\vartheta)u^{k-1}.
\]
By \eqref{bounds on the coefficients}, it suffices to show that $\|\wb_{;1}(\vartheta) +\frac{2}{\gamma+1}\|_{L^\infty}\lesssim \varepsilon \delta$. According to the defining equation \eqref{eq:definition for the Taylor coefficients} of $\wb_{;1}(\vartheta)$, we have
\begin{align*}
2\wb_{;1}(\vartheta)
&=T(\Th^1)\big(U^{(-2)}-U^{(0)}\big)+  T(\Th^2) U^{(-1)}+  \Th^2 T(U^{(-1)})+ (1+\Th^1)T(U^{(-2)})+(1-\Th^1)T(U^{(0)})\\
&=(1-\Th^1)T(U^{(0)})+O(\varepsilon \delta),
\end{align*}
where we have used Proposition \ref{prop: bounds on geometry on sigma delta 1} and \eqref{eq: bound on Un with X}. Since $T(U^{(0)})=-\frac{2}{\gamma+1}$(see \eqref {eq:Taylor coefficients of U0}), we have
\begin{align*}
\wb_{;1}(\vartheta) +\frac{2}{\gamma+1}
&=\frac{(1+\Th^1)}{\gamma+1} +O(\varepsilon \delta).
\end{align*}
Therefore, $\|\wb_{;1}(\vartheta) +\frac{2}{\gamma+1}\|_{L^\infty}\lesssim \varepsilon \delta$ follows from Proposition \ref{prop: bounds on geometry on sigma delta 1}.

Secondly, we derive the bounds on $L(\psi)$. For $\psi=\wb$, we use the first equation from \eqref{Euler equations:form 2}, i.e.,
\begin{equation}\label{eq: derive bounds on sigma delta aux 11}L (\wb) = -c  {T}(\wb)\frac{\widehat{T}^1+1}{\kappa}+\frac{1}{2}c\kappa^{-1} T(\psi_2)\widehat{T}^2 +\frac{1}{2}c \Xh(\psi_2)\Xh^2-c\Xh(\wb)\Xh^1.
\end{equation}
By Proposition \ref{prop: bounds on geometry on sigma delta 1} and \eqref{eq: bound on Riemann Invariants on Sigma delta}, each term on the righthand side is of size $O(\varepsilon)$. This shows that $\|L(\wb)\|_{L^\infty(\Sigma_\delta)} \lesssim \varepsilon$. We can use the other two equations in \eqref{Euler equations:form 2} to bound $L(w)$ and $L(\psi_2)$ exactly in the same manner.

Finally, we deal with the higher order derivatives. Since $\Xh=\slashed{g}^{-\frac{1}{2}}X$, it is obvious that the estimates on $\Xh Z^{\alpha}\psi$ and $TZ^{\alpha}\psi$ follow from \eqref{eq: bound on Riemann Invariants on Sigma delta}, \eqref{eq: bound on wb on Sigma delta} and \eqref{eq: higher order estimates on slashed g}. It remains to estimate $LZ^{\alpha}\psi$ and we will commute $Z^\alpha$ with \eqref{Euler equations:form 2}. We only handle the case where $\psi=\wb$ and the rest cases can be treated exactly in the same manner. We apply $Z^\alpha$ to \eqref{eq: derive bounds on sigma delta aux 11} to derive:

\begin{align*}
Z^\alpha L (\wb) &= Z^\alpha\big[-c \widehat{T}(\wb)(\widehat{T}^1+1)+\frac{1}{2}c \widehat{T}(\psi_2)\widehat{T}^2 +\frac{1}{2}c \Xh(\psi_2)\Xh^2-c\Xh(\wb)\Xh^1\big]=O(\varepsilon).
\end{align*}
By Proposition \ref{prop: bounds on geometry on sigma delta 1}, \eqref{eq: bound on Riemann Invariants on Sigma delta} and the estimates on $\Xh Z^{\alpha}\psi$ and $TZ^{\alpha}\psi$, the rightand side of the above equation is bounded by $\varepsilon$. To obtain the estimates on $LZ^\alpha(\wb)$, we use the \eqref{eq:commutator:schematic L commute with d^k} to derive 
\[ LZ^\alpha   (\wb)=Z^\alpha L (\wb)+\sum_{|\beta|\leqslant |\alpha|} \mathscr{P}_j\cdot  Z^\beta(\wb).\]
By Proposition \ref{prop: bounds on geometry on sigma delta 1}, \eqref{eq: bound on Riemann Invariants on Sigma delta} and the estimates on $\Xh Z^{\alpha}\psi$ and $TZ^{\alpha}\psi$, the coefficients functions $\mathscr{P}_j$ are all of size $O(\varepsilon)$. Therefore, $\|LZ^{\alpha}\psi\|_{L^\infty(\Sigma_\delta)}\lesssim \varepsilon$. This completes the proof of the proposition.
\end{proof}
\begin{remark}\label{remark: closing I infty 2}
By Proposition \ref{prop:pointwise bound on derivatives of psi}, we have checked all the inequalities in $\mathbf{(I_{\infty,2})}$, see \eqref{initial Iinfty2}.
\end{remark}
\begin{remark}\label{remark: closing I 2 lowest energy}
The above $L^\infty$ bounds imply the $L^2$-bounds:
\[\mathcal{E}(\psi)(\delta,u^*)+\underline{\mathcal{E}}(\psi)(\delta,u^*)\lesssim C_0 \varepsilon^2 \delta^2, \ \ \psi \in \{w,\psi_2\}.\]
On the other hand, by the continuity on $C_0$, the following estimates hold automatically
\[\mathcal{F}(\psi)(t,0)+\underline{\mathcal{F}}(\psi)(t,0)\lesssim \varepsilon^2 t^2, \ \ t\in [\delta,t^*], \ \psi \in \{w,\psi_2\}.\]
Therefore, we have also checked the lowest order energy ansatz in $\mathbf{(I_2)}$, see \eqref{initial I2}.
\end{remark}
\subsection{The irrotational condition}
The computations of the  vorticity in this subsection are inspired by those in Section 1.3 of \cite{ChristodoulouShockDevelopment}.We consider {\bf the spacetime vorticity} $2$-form $\omega = -d\beta$ where the $1$-form $\beta$ is defined by 
\begin{equation}\label{def: vorticity}
\beta = \big(h+\frac{1}{2}|v|^2\big)dt - v^idx^i.
\end{equation}
Therefore, $\omega$ can be computed explicitly as
\[
\omega = \big(\partial_t v^i + \partial_i\big(h+\frac{1}{2}|v|^2\big)\big)dt \wedge dx^i + \frac{1}{2}\omega_{ij}dx^i \wedge dx^j.
\]
with $\omega_{ij} = \partial_i v^j - \partial_j v^i$, i.e., the restriction of $\omega$ on $\Sigma_t$ is the usual vorticity. In two dimension case, the vorticity 2-form $\omega$ restricts to the classical vorticity $\overline{\omega}$ function as follows
\[\omega\big|_{\Sigma_t}=\overline{\omega} \cdot dx^1\wedge dx^2,\]
where $\overline{\omega}=\partial_1v^2-\partial_2 v^1$.  We recall that the material vector field $B$ is given by $B=\partial_t+v$. By virtue of the fact that $\partial_t v+ v^i\cdot \nabla_i v=-\nabla h$, we have 
\[\iota_{_B} \omega = 0.\]
While the classical vorticity satisfies the following transport equation:
\begin{equation}\label{eq: transport equation for omegab}
B (\overline{\omega}) + {\rm div}(v) \cdot \overline{\omega} = 0.
\end{equation}
\begin{remark}
For irrotational flow with potential $\phi$, we have $\beta = d\phi$. Hence, $\omega\equiv 0$. 
\end{remark}

We now study the equation $\iota_{_B} \omega = 0$. Since $B = L + c\kappa^{-1}T$, we have
\[\iota_{_L} \omega = - c\kappa^{-1}\iota_{_T} \omega.\]
Therefore, contracting with $T$ and $X$, we have $\omega(L,T) = -c\kappa^{-1} \omega(T,T) \equiv 0$ and $\omega(L,X) = -c\kappa^{-1} \omega(T,X)$. As a conclusion, we see that $\omega(L,X)$ determines $\omega$. In particular, to determine $\omega\big|_{C_0}$, it suffices to determine $\omega(L,X)\big|_{C_0}$.

We now use the fact that the smooth solution $(v_r,c_r)$ on $\mathcal{D}_0$ is irrotational. This means that the corresponding vorticity 2-form $\omega_r\big|_{C_0}\equiv 0$. Since $(v,c)$ is continuous across $C_0$,  in view of the definition \eqref{def: vorticity}, the vorticity $1$-forms $\beta$ and $\beta_r$ (defined from  $(v_r,c_r)$) agree on $C_0$. Therefore, since $L$ and $X$ are tangential to $C_0$, we have
\[\omega|_{C_{0}}(L,X) =d( \beta\big|_{C_{0}})(L,X) = d( \beta_r\big|_{C_{0}})(L,X)\equiv 0.
\]
Thus, $\omega\big|_{C_{0}} \equiv 0$. In particular, this implies $\overline{\omega}\big|_{C_{0}} \equiv 0$.

\begin{lemma}
We use $\mathfrak{d}^k(\overline{\omega})$ to denote all possible $\mathfrak{d}_1\circ \cdots \circ\mathfrak{d}_k (\overline{\omega})$'s where $\mathfrak{d}_1, \cdots, \mathfrak{d}_k \in \{\Xh,T,L\}$. Then, for all $1\leqslant k\leqslant N-1$, we have
\[\mathfrak{d}^k(\overline{\omega})\big|_{C_{0}}\equiv 0.\]
\end{lemma}
\begin{proof}
We remark that, since $\overline{\omega}\big|_{C_{0}} \equiv 0$, for all $l_1,l_2\geqslant 0$, we have $L^{l_1} \Xh^{l_2}\big(\overline{\omega}\big)\big|_{C_{0}}=0$.

We  prove inductively on $k$. The transport equation \eqref{eq: transport equation for omegab} plays a central role. For $k=1$, it suffices to check $T (\overline{\omega}) =0$.  We can use $B = L + c\kappa^{-1}T$ to rewrite \eqref{eq: transport equation for omegab}  as
\begin{equation}\label{eq:vorticity aux 2}
T (\overline{\omega}) =-c^{-1}\kappa L(\overline{\omega})-c^{-1}\kappa {\rm div}(v) \cdot \overline{\omega}.
\end{equation}
Since both $L(\overline{\omega})$ and $\overline{\omega}$ vanish on $C_0$,  we have $T (\overline{\omega}) =0$.

We assume that the proposition holds for all $k\leqslant n$ where $n<N-1$. To prove for $n+1$, we write $\mathfrak{d}^{n+1}(\overline{\omega})$ as 
\[\mathfrak{d}^{n+1}(\overline{\omega})=\mathfrak{d}_0\big(\mathfrak{d}_1\circ\mathfrak{d}_2\circ\cdots \circ\mathfrak{d}_k (\overline{\omega})\big)=\mathfrak{d}_0\big(\mathfrak{d}^{n}(\overline{\omega})\big).\]
We consider the following cases:
\begin{itemize}
\item $\mathfrak{d}_0=\Xh$ or $L$.  

These vector fields are tangential to $C_0$. By the inductive hypothesis, $\mathfrak{d}^{n}(\overline{\omega})\equiv 0$ on $C_0$. Thus, $\mathfrak{d}^{n+1}(\overline{\omega})=\mathfrak{d}_0\big(\mathfrak{d}^{n}(\overline{\omega})\big) \equiv 0$.
\item $\mathfrak{d}_0=T$ and $\{\mathfrak{d}_1,\cdots, \mathfrak{d}_n\} \cap \{L, \Xh\}\neq \emptyset$.

We may assume $\mathfrak{d}_{i_0} \in \{\Xh, L\}$ where $1\leqslant i_0\leqslant n$. According to \eqref{eq:commutator formulas}, $[L,T]$ and $[\Xh,T]$ are proportional to $\Xh$. Therefore, by commuting $\mathfrak{d}_{i_0}$ with $\mathfrak{d}_{i_0-1},\cdots, \mathfrak{d}_{0}$ and $T$ one by one successively, we can move the operator $\mathfrak{d}_{i_0}$ to the front as follows:
\begin{align*}
\mathfrak{d}^{n+1}(\overline{\omega})=T\mathfrak{d}^{n}(\overline{\omega})=\mathfrak{d}_{i_0}\big[T\big(\mathfrak{d}^{n-1}(\overline{\omega})\big)\big]+\sum_{k\leqslant n}f_k(t,\vartheta) \cdot \mathfrak{d}^{k}(\overline{\omega}),
\end{align*}
where the $f_k(t,\vartheta)$'s are smooth functions. By the inductive hypothesis, the sum vanishes on $C_0$; the first term also vanishes thanks to the previous case. Hence, $\mathfrak{d}^{n+1} \equiv 0$ on $C_0$ in this case.

\item $\mathfrak{d}_0=T$ and $\mathfrak{d}_1=\cdots=\mathfrak{d}_n=T$.

We want to show $T^{n+1}(\overline{\omega})=0$. According to \eqref{eq:vorticity aux 2}, we have \begin{align*}
T^{n+1}(\overline{\omega})&=-T^n\left(c^{-1} \kappa L(\overline{\omega})+c^{-1}\kappa {\rm div}(v) \cdot \overline{\omega}\right)=-c^{-1}\kappa T^n L(\overline{\omega}).
\end{align*}
In the last equality, we have used the inductive hypothesis to the lower order terms. According to the second case, we also have $T^n L(\overline{\omega})=0$. Thus, $T^{n+1} \equiv 0$ on $C_0$
\end{itemize}
This completes the proof of the lemma.
\end{proof}
We recall that the Riemann invariants hence $v$ is a finite Taylor series of order $N$ in $u$. Since $\overline{\omega}=dv$ on $\Sigma_\delta$, $\overline{\omega}$ is a finite Taylor series of order $N-1$ in $u$, i.e., we have  
\[\overline{\omega}|_{\Sigma_0} = \sum_{k\leqslant N-1}\frac{\overline{\omega}_k(\vartheta)}{k!}u^k.\]
The lemma shows that $\overline{\omega}_k(\vartheta) = T^k \overline{\omega}|_{S_{\delta,0}} \equiv 0$ for $k\leqslant N-1$. Thus,  $\overline{\omega}\equiv 0$ on $\Sigma_\delta$. On the other hand, we have already showed that $\overline{\omega}\equiv 0$ on $C_0$. Since $B$ is timelike with respect to the acoustical metric, by \eqref{eq: transport equation for omegab}, $\overline{\omega}\equiv 0$ holds in the domain of dependence of $\Sigma_\delta\cup C_0$. As a conclusion, the fluid will remain irrotational in the region of front rarefaction waves.

\begin{remark}\label{remark: closing I_{irrotational}}
	We have verified $\mathbf{(I_{irrotational})}$.
\end{remark}

\subsection{The initial energy bounds}
We have used the vector fields from $\mathring{\mathscr{Z}}=\{\mathring{T}, \mathring{X}\}$ to define the higher order energy $\mathscr{E}_{n}(\psi)(\delta,u)$ and flux $\mathscr{F}_{n}(\psi)(t,0)$.
The frames $(\Xh,T)$ and $(\Xr, \Tr)$ on $\Sigma_t$ are related by
\begin{equation}\label{eq: transformation between Xh T and Xr Tr}
 \begin{cases} \Xr &= -\Th^1 \Xh + \frac{\Th^2}{\kappa} T,\\ 
 \Tr &= -\kappar\Th^2 \Xh-\frac{\kappar}{\kappa}\Th^1 T.
 \end{cases}
\end{equation}

\begin{proposition}\label{prop:pointwise bound on circled derivatives of psi 1}
For all multi-indices $\alpha$ with $1\leqslant |\alpha|\leqslant N-1$, for all $\psi \in \{\wb,w,\psi_2\}$ and $\Zr\in \mathring{\mathscr{Z}}$, we have
\begin{equation}\label{eq:pointwise bound on circled derivatives of psi 1}
\delta\|\Xh( \Zr^\alpha \psi)\|_{L^\infty(\Sigma_\delta)}+\|T( \Zr^\alpha \psi)\|_{L^\infty(\Sigma_\delta)}\lesssim \varepsilon \delta.
\end{equation}
\end{proposition}
\begin{proof}
Each coefficient  function  $f$ from the righthand side of \eqref{eq: transformation between Xh T and Xr Tr} are from the set  (neglecting the irrelevant constants):
\[f\in \big\{\Th^1,\frac{\Th^2}{\kappa},\kappar\Th^2,\frac{\kappar}{\kappa}\Th^1\big\}.\]
We certainly have $\|f\|_{L^\infty(\Sigma_\delta)}\lesssim 1$. By Proposition \ref{prop: bounds on geometry on sigma delta 1}, for all multi-indices $\alpha$ with $1\leqslant |\alpha| \leqslant N$, for all $Z\in \mathscr{Z}=\{T,\Xh\}$, we have
\[\|Z^\alpha(f)\|_{L^\infty(\Sigma_\delta)}\lesssim \varepsilon.\]
Therefore, we can use \eqref{eq: transformation between Xh T and Xr Tr} to replace all the $\Zr$ so that $T( \Zr^\alpha \psi)$ are expressed as a linear combination of the terms of the following forms
\[(\mathbf{A}). \ Z^{\alpha_1}(f)\cdot\cdots \cdot Z^{\alpha_k}(f)\cdot Z^\beta(f)\cdot TZ^\gamma(\psi)  \ \ (\mathbf{B}). \ Z^{\alpha_1}(f)\cdot\cdots \cdot Z^{\alpha_k}(f)\cdot TZ^\beta(f)\cdot Z^\gamma(\psi)\] 
where $Z\in\mathscr{Z}=\{T,\Xh\}$ and $\sum_{i\leqslant k}\alpha_i+\beta+\gamma=\alpha$. We remark that the operator $T$ is the first operator in $T( \Zr^\alpha \psi)$.
\begin{itemize}
\item For type $(\mathbf{A})$ terms, we notice that $\gamma\neq 0$. Therefore,  \eqref{prop:pointwise bound on derivatives of psi} implies that 
\[\|Z^\beta(f)\cdot TZ^\gamma(\psi)\|_{L^\infty(\Sigma_\delta)}\lesssim \varepsilon \delta.\]
Thus, all the type $(\mathbf{A})$ terms are bounded by $O(\varepsilon \delta)$.

\item  For type $(\mathbf{B})$ terms, we first bound $TZ^\beta(f)$. Indeed, Remark \ref{remark: improvement with T} implies that for all $f$ we have 
\[\|TZ^\beta(f)\|_{L^\infty(\Sigma_\delta)}\lesssim \varepsilon\delta.\]
Thus, all the type $(\mathbf{B})$ terms are also bounded by $O(\varepsilon \delta)$.
\end{itemize}
The above argument gives the desired estimates on $T( \Zr^\alpha \psi)$ in \eqref{eq:pointwise bound on circled derivatives of psi 1}. The bounds on  $\Xh( \Zr^\alpha \psi)$ can be derived exactly in the same manner. This finishes the proof of the proposition.
\end{proof}
\begin{corollary}\label{coro: circled derivatives bounds}
For all multi-indices $\alpha$ with $1\leqslant |\alpha|\leqslant N-1$, for all $\psi \in \{\wb,w,\psi_2\}$ and $\Zr\in \mathring{\mathscr{Z}}$, we have
\begin{equation}\label{eq:pointwise bound on circled derivatives of psi 2}
\delta\|\Xr( \Zr^\alpha \psi)\|_{L^\infty(\Sigma_\delta)}+\|\Tr( \Zr^\alpha \psi)\|_{L^\infty(\Sigma_\delta)}\lesssim \varepsilon \delta.
\end{equation}
Moreover, we have
\begin{equation}\label{eq:pointwise bound on circled derivatives of psi 2 zeroth order}
\delta\|\Xr(\wb)\|_{L^\infty(\Sigma_\delta)}+\delta\|\Xr(\psi)\|_{L^\infty(\Sigma_\delta)}+\|\Tr(\psi)\|_{L^\infty(\Sigma_\delta)}\lesssim \varepsilon \delta, \ \ \psi\in \{w,\psi_2\}.
\end{equation}
\end{corollary}
\begin{proof}
This is straightforward from the proposition and the formula \eqref{eq: transformation between Xh T and Xr Tr}.
\end{proof}

\subsubsection{Auxiliary estimates on $\yr, \zr, \chir$ and $\etar$}\label{Section: lambda on C0}
We define $\Lambda = \{\yr, \zr, \chir,\etar\}$ and we use $\lambda$ to denote a generic object from $\Lambda$. The goal of the current subsection is to derive estimates for $\lambda$ on $\Sigma_\delta$.

In view of the definition that $\chir=-\Xr(\psi_2)$ and $\etar=-\Tr(\psi_2)$, according to Corollary \ref{coro: circled derivatives bounds}, for $\lambda \in \{\chir,\etar\}$, for all multi-indices $\alpha$ with $|\alpha|\leqslant N$, we have
\begin{equation}\label{pointwise bounds for chir etar}
\|\mathring{Z}^\alpha(\lambda)\|_{L^\infty(\Sigma_\delta)}\lesssim  \begin{cases}\varepsilon, \ \ & \text{if}~\Zr^\beta=\Xr^\beta ~\text{and}~\lambda =\chir;\\
\varepsilon \delta, \ \ & \text{otherwise}.
\end{cases} 
\end{equation}

For $\lambda =\zr$ (see \eqref{def: yring zring} to recall the definition), it suffices to bound $z$.  Since $v^1+c=\frac{\gamma+1}{2}\wb-\frac{3-\gamma}{2}w$, by Proposition \ref{prop:pointwise bound on derivatives of psi}, we have
\[ \|1+T(v^1+c)\|_{L^\infty(\Sigma_\delta)} \lesssim \varepsilon \delta. \]
On the other hand, using \eqref{eq: transformation between Xh T and Xr Tr}, we can derive
\begin{equation}\label{eq: a precise computation for z}
z
=1+T(v^1+c)-\kappar\Th^2 \Xh(v^1+c)-\big(\frac{\kappar}{\kappa}\Th^1+1\big) T(v^1+c).
\end{equation}
By Proposition \ref{prop:pointwise bound on derivatives of psi} and \eqref{eq:precise estimates on Th}, the last two terms are bounded by $O(\delta \varepsilon)$. This shows that $\|z\|_{L^\infty(\Sigma_\delta)} \lesssim \varepsilon \delta$. Since $\Zr^\alpha(z)=\Tr\Zr^\alpha(v^1+c)$, by Corollary \ref{coro: circled derivatives bounds}, we have $\|\Zr^\alpha(z)\|_{L^\infty(\Sigma_\delta)} \lesssim \varepsilon \delta$. Hence, for all multi-indices $\alpha$ with $|\alpha|\leqslant N$, we have
\begin{equation*}
\|\mathring{Z}^\alpha(\zr)\|_{L^\infty(\Sigma_\delta)}\lesssim  \varepsilon.
\end{equation*}

For $\lambda =\yr$, its bounds rely on the full strength of the results from Section \ref{section:initial foliatiion}. We will show that $\|\mathring{Z}^\alpha(\yr)\|_{L^\infty(\Sigma_\delta)}\lesssim  \varepsilon$ or equivalently $\|\mathring{Z}^\alpha(y)\|_{L^\infty(\Sigma_\delta)}\lesssim  \delta\varepsilon$  (see \eqref{def: yring zring} for definitions), where $\alpha$ is any multi-index with $|\alpha|\leqslant N-1$.

 According to \eqref{eq:pointwise bound on circled derivatives of psi 1}, for $y=\Xr(v^1+c)$, we have
\[\|T\mathring{Z}^\alpha(y)\|_{L^\infty(\Sigma_\delta)}\lesssim  \delta\varepsilon.\] 
Therefore, by integrating along the integral curve of $T=\frac{\partial}{\partial u}$ from $S_{\delta,0}$, it suffices to prove that  
\[\|\mathring{Z}^\alpha(y)\|_{L^\infty(S_{\delta,0})}\lesssim  \delta\varepsilon,\]
for all multi-indices $\alpha$ with $|\alpha|\leqslant N-1$. \footnote{The top derivatives on $\yr$ is at most $N-1$ since we have integrate along $T$.  This is consistent with the nonlinear energy estimates in \cite{LuoYu1}. }By Corollary \ref{coro: circled derivatives bounds}, it suffices to consider the case $\mathring{Z}^\alpha=\Xr^\alpha$. We can repeat the proof of Proposition \ref{prop:pointwise bound on circled derivatives of psi 1}, i.e., we use \eqref{eq: transformation between Xh T and Xr Tr} to replace all the $\Xr$'s, so that it suffices to show that 
\[\|\Xh^\alpha(y)\|_{L^\infty(S_{\delta,0})}\lesssim  \delta\varepsilon.\] 
Since $X=\sqrt{\slashed{g}}\Xh$, by the bounds in \eqref{eq: higher order estimates on slashed g}, it suffices to show that 
\[\|X^\alpha(y)\|_{L^\infty(S_{\delta,0})}\lesssim  \delta\varepsilon.\]
We compute that
\begin{align*}
y&=\Xr(v^1+c)=-\Th^1 \Xh(v^1+c) + \frac{\Th^2}{\kappa} T(v^1+c)\\
&=-(\Th^1+1) \Xh(v^1+c) + \frac{\Th^2}{\kappa} (T(v^1+c)+1) -\frac{\Th^2}{\kappa} + \Xh(v^1+c)\\
&=\underbrace{-\frac{\Th^2}{\kappa} + X(v^1+c)}_{y'}+ \underbrace{(1-\sqrt{\slashed{g}})\Xh(v^1+c)-(\Th^1+1) \Xh(v^1+c) + \frac{\Th^2}{\kappa} (T(v^1+c)+1)}_{y_{\rm err}}.
\end{align*}
By  \eqref{eq:precise estimates on Th},\eqref{eq: higher order estimates on slashed g}, Proposition \ref{prop:pointwise bound on derivatives of psi} and Corollary \ref{coro: circled derivatives bounds},  we have $\|X^\alpha(y_{\rm err})\|_{L^\infty(S_{\delta,0})}\lesssim  \delta\varepsilon$. Thus, 
it suffices to show that 
\[\|X^\alpha(y')\|_{L^\infty(S_{\delta,0})}\lesssim  \delta\varepsilon.\]
According to \eqref{eq:Th on Sigma delta} and \eqref{eq: kappa on Sigma delta}, we have
\[y'=\frac{A(\delta,x_2)}{\delta} + X(v^1+c).\]
We recall that  $A(\delta,\slashed{\vartheta})=\int_0^\delta a(\tau,\slashed{\vartheta})d\tau$, see \eqref{eq: def for a and A}. We decompose the function $a(\tau, \slashed{\vartheta})$ as follows:
\begin{align*}
a(\tau,\slashed{\vartheta})&=  X \big(\psi_1+c \Th^1 \big) (\tau,\slashed{\vartheta}) +X \big(\psi_1+c \Th^1 \big) (\tau,\slashed{\vartheta}) \big(\frac{\partial \slashed{\vartheta}}{\partial x_2}-1\big)\\
&=\underbrace{-X \big(v^1+c \big)(\tau,\slashed{\vartheta}) }_{a'(\tau,\vartheta)}+ \underbrace{X \big(c (\Th^1+1) \big) (\tau,\slashed{\vartheta}) +X \big(\psi_1+c \Th^1 \big) (\tau,\slashed{\vartheta}) \big(\frac{\partial \slashed{\vartheta}}{\partial x_2}-1\big)}_{a_{\rm err}}.
\end{align*}
We also decompose $y'$ accordingly as
\begin{align*}
y'&=\frac{1}{\delta}\int_0^\delta a(\tau,\slashed{\vartheta})d\tau + X(v^1+c)\\
&=\underbrace{\frac{1}{\delta}\int_0^\delta a'(\tau,\slashed{\vartheta})d\tau+ X(v^1+c)}_{y''}+\frac{1}{\delta}\int_0^\delta a_{\rm err}(\tau,\slashed{\vartheta})d\tau.
\end{align*}
By \eqref{eq:x1x2 on C0} and Proposition \ref{prop: bounds on geometry on sigma delta 1}, we have $\|X^\alpha(a_{\rm err})\|_{L^\infty(S_{\delta,0})}\lesssim  \delta\varepsilon$. Therefore, we can ignore the contribution of $a_{\rm err}$ in $y'$ and it suffices to bound
\begin{align*}
X^\alpha(y'')&=\frac{1}{\delta}\int_0^\delta \frac{\partial^\alpha a'(\tau,\slashed{\vartheta})}{\partial \slashed{\vartheta}^\alpha}d\tau+ X^{\alpha+1}(v^1+c)\\
&=\frac{1}{\delta}\int_0^\delta \big[\frac{\partial^{\alpha+1} (v^1+c)(\delta,\slashed{\vartheta})}{\partial \slashed{\vartheta}^{\alpha+1}}-\frac{\partial^{\alpha+1}(v^1+c)(\tau,\slashed{\vartheta})}{\partial \slashed{\vartheta}^{\alpha+1}}\big]d\tau.
\end{align*}
Therefore, the intermediate value theorem yields $\|X^\alpha(y'')\|_{L^\infty(S_{\delta,0})}\lesssim  \delta\varepsilon$. This completes the estimates for $\yr$. Combined with the bounds on $\zr$, we obtain that
\begin{equation}\label{pointwise bounds for yr zr}
\|\mathring{Z}^\alpha(\yr)\|_{L^\infty(\Sigma_\delta)}+\|\mathring{Z}^\alpha(\zr)\|_{L^\infty(\Sigma_\delta)}\lesssim  \varepsilon.
\end{equation}
\subsubsection{Pointwise bounds on $L( \Zr^\alpha \psi)$}
We recall the following set of formulas which has already appeared in \cite{LuoYu1}. By $[\Lr, \Xr]=\yr\cdot \Tr-\chir\cdot\Xr$ and $[\Lr, \Tr]=\zr\cdot \Tr-\etar\Xr$, for any multi-index $\alpha$, we have the following schematic commutation formula:
\begin{equation}\label{eq: commutation formular for Lr Zr}
[\Lr, \Zr^\alpha]=\sum_{\substack{\alpha_1+\alpha_2=\alpha \\ |\alpha_1|\leqslant |\alpha|-1}}\Zr^{\alpha_1}(\lambda) \Zr^{\alpha_2}, \  \ \lambda \in \{\yr, \zr, \chir,\etar\}.
\end{equation}
\begin{proposition}\label{prop: bounds on L Zalpah psi on sigma delta}For all multi-indices $\alpha$ with $|\alpha|\leqslant N-1$, for all $\psi \in \{\wb,w,\psi_2\}$ and $\Zr\in \mathring{\mathscr{Z}}$, we have
\begin{equation}\label{eq:pointwise bound on L circled derivatives of psi 1}
\|L( \Zr^\alpha \psi)\|_{L^\infty(\Sigma_\delta)}\lesssim \varepsilon.
\end{equation}
\end{proposition}
\begin{proof}
Since $\Xr$ and $\Trh$ commute with vector fields from $\mathring{\mathscr{Z}}$,  we can apply $\Zr^\alpha\in  \mathring{\mathscr{Z}}$ to \eqref{Euler equations:form 3} to derive
\begin{equation}\label{eq:Euler commutated}
\begin{cases}
\Zr^\alpha \Lr (\wb) &= \displaystyle\sum_{\alpha_1+\alpha_2=\alpha} \Zr^{\alpha_1}(c) \cdot \Xr(\Zr^{\alpha_2}(\psi_2)),\\
\Zr^\alpha \Lr (w) &= 	\displaystyle\sum_{\alpha_1+\alpha_2=\alpha}\big[\Zr^{\alpha_1}(c)  \Trh(\Zr^{\alpha_2}(w))+ \Zr^{\alpha_1}(c) \Xr(\Zr^{\alpha_2}(\psi_2))\big],\\
\Zr^\alpha\Lr (\psi_2) &= \displaystyle\sum_{\alpha_1+\alpha_2=\alpha} \big[\Zr^{\alpha_1}(c)\Trh(\Zr^{\alpha_2}(\psi_2))+ \Zr^{\alpha_1}(c) \Xr(\Zr^{\alpha_2}(w+\wb))\big].
\end{cases}
\end{equation}
where we ignore the irrelevant constants coefficients. By Corollary \ref{coro: circled derivatives bounds},  for all multi-indices $\alpha$ with $|\alpha|\leqslant N$ and for all $\psi \in \{\wb, w,\psi_2\}$, we have
\[  \|\Zr^\alpha \Lr (\psi)\|_{L^\infty(\Sigma_\delta)} \lesssim \varepsilon.\]
To bound $\Lr( \Zr^\alpha \psi)$, we use \eqref{eq: commutation formular for Lr Zr} to derive
\begin{align*}
\Lr( \Zr^\alpha \psi)=\Zr^\alpha \Lr (\psi)+ [\Lr, \Zr^\alpha]\psi=\Zr^\alpha \Lr (\psi)+\sum_{\substack{\alpha_1+\alpha_2=\alpha \\ |\alpha_1|\leqslant |\alpha|-1}}\Zr^{\alpha_1}(\lambda) \Zr^{\alpha_2}(\psi).
\end{align*}
By \eqref{pointwise bounds for chir etar} and \eqref{pointwise bounds for yr zr}, for all $\lambda \in \Lambda$, we have $\|\Zr^\alpha (\lambda)\|_{L^\infty(\Sigma_\delta)} \lesssim \varepsilon$. Hence,
\[  \|\Lr \Zr^\alpha (\psi)\|_{L^\infty(\Sigma_\delta)} \lesssim \varepsilon.\]
Finally, by $L-\Lr =c\big(\frac{\Th^1+1}{\kappar}\Tr -\Th^2\Xr\big)$ and Corollary \ref{coro: circled derivatives bounds}, the above inequality implies that
\[  \|L \Zr^\alpha (\psi)\|_{L^\infty(\Sigma_\delta)} \lesssim \varepsilon.\]
This finishes the proof of the proposition.
\end{proof}

We combine all the estimates to conclude that
\begin{proposition}\label{prop:pointwise bound on data on Sigma delta}
For all multi-indices $\alpha$ with $1\leqslant |\alpha|\leqslant N-1$, for all $\psi \in \{\wb,w,\psi_2\}$ and $\Zr\in \mathring{\mathscr{Z}}$, we have
\[
\delta\|L( \Zr^\alpha \psi)\|_{L^\infty(\Sigma_\delta)}+\delta\|\Xh( \Zr^\alpha \psi)\|_{L^\infty(\Sigma_\delta)}+{\color{black}\|\Lb( \Zr^\alpha \psi)\|_{L^\infty(\Sigma_\delta)}}\lesssim \varepsilon \delta.
\]
\end{proposition}
\begin{proof}
The bound of $\Lb \Zr^\alpha (\psi)$ is immediate from the formula $\Lb=2T+c^{-1}\kappa L$, \eqref{eq:pointwise bound on circled derivatives of psi 1} and the bound on $L \Zr^\alpha (\psi)$ in Proposition \ref{prop: bounds on L Zalpah psi on sigma delta}. 
\end{proof}
\begin{remark}\label{remark: closing I 2 higher energy}
The above $L^\infty$ implies that $L^2$ energy bounds:
\[\mathscr{E}_{n}(\psi)(\delta,u^*) \lesssim \varepsilon^2 \delta^2, \ \ \psi \in \{w,\wb,\psi_2\},  \ \ 1\leqslant n\leqslant \Ntop.\]
Therefore, we have checked  the energy ansatz for $\mathscr{E}_{n}$ in $\mathbf{(I_2)}$, see \eqref{initial I2}.
\end{remark}

\subsection{The initial flux bounds}

We still have to check the bounds on the flux through $C_0$ in $L^2$. The idea is to derive ODE systems along $L$ direction for all the $k$-jets of $\psi\in \{\wb,w,\psi_2\}$. To this purpose, we introduce another set of polynomial rings. They are based on three families of unknowns $\{\mathfrak{A}_{n}\}_{n\geqslant0}, \{\mathfrak{B}_{n}\}_{n\geqslant0}$ and $\{\mathfrak{C}_{n}\}_{n\geqslant0}$ defined inductively as follows:
\begin{itemize}
\item $n=0$. 
\[\mathfrak{A}_{0}=\big\{ \wb,w,\psi_2\big\},\ \ \mathfrak{B}_{0}=\big\{\widehat{T}^i, i=1,2\big\},\ \ \mathfrak{C}_{0}=\mathfrak{A}_{0}\cup \mathfrak{B}_{0}.\]
\item $n=1$. 
\[\mathfrak{A}_{1}=\big\{T(\psi_0),T(\psi_1),T(\psi_2),\kappa\big\},\ \  \mathfrak{B}_{1}=\big\{T(\Th^i),  i=1,2\big\},\ \ \mathfrak{C}_{1}=\mathfrak{A}_{1}\cup \mathfrak{B}_{1}.\]
\item $n\geqslant 2$.
\[\mathfrak{A}_{n}=\big\{T(x)\big| x\in \mathfrak{A}_{n-1}\big\},\ \  \mathfrak{B}_{n}=\big\{T(x)\big| x\in \mathfrak{B}_{n-1}\big\},\ \ \mathfrak{C}_{n}=\mathfrak{A}_{n}\cup \mathfrak{B}_{n}.
\]
\end{itemize}

Schematically, we use $\mathfrak{a}_n$, $\mathfrak{b}_n$ and $\mathfrak{c}_n$ to denote a generic element from $\mathfrak{A}_n$, $\mathfrak{B}_n$ and $\mathfrak{C}_n$ respectively. A key concept in the following is the {\bf \emph{prime} notation}: We use $\mathfrak{a}'_n$ to denote a tangential derivative of a given $\mathfrak{a}_n$. The word \emph{tangential} is relative to the hypersurface $C_0$. More concretely, $\mathfrak{a}'_n$ is of the following form:
\[\mathfrak{a}'_n=(\mathfrak{d}_1\circ\mathfrak{d}_2\circ\cdots \circ\mathfrak{d}_k)(\mathfrak{a}_n)\]
where $k\geqslant 0$ and $\mathfrak{d}_i$ is either $\Xh$ or $L$. We use  $\mathfrak{A}'_n$ to denote the set of all the $\mathfrak{a}'_n$'s. We can also  define $\mathfrak{b}'_n$, $\mathfrak{c}'_n$, $\mathfrak{B}'_n$  and $\mathfrak{C}'_n$ in the similar manner.

\begin{remark}
Given a function $\mathfrak{a}_n$ on $C_0$, since $\Xh$ and $L$ are tangential to $C_0$, all the corresponding $\mathfrak{a}'_n$ are determined by the value of $\mathfrak{a}_n$. Intuitively, the number $n$ for $\mathfrak{a}'_n,\mathfrak{c}'_n$ and $\mathfrak{c}'_n$ counts the number of $T$-derivatives. 
\end{remark}

We also introduce the following schematic polynomial notations through examples:
\begin{itemize}
\item We use ${\mathscr{P}}(\mathfrak{c}_{0})$ to denote a polynomial from the polynomial ring $\mathbb{R}[\mathfrak{A}_{0}]$.

For example, $c ={\mathscr{P}}(\mathfrak{c}_{0})$ and $c^2 ={\mathscr{P}}(\mathfrak{c}_{0})$.

\item We use ${\mathscr{P}}(\mathfrak{c}'_{0})$ to denote a polynomial from the polynomial ring $\mathbb{R}[\mathfrak{C}'_{0}]$.

For example, in view of \eqref{defining eq of theta and chi}, we have $\theta = {\mathscr{P}}(\mathfrak{c}'_{0})$ and $\chi  = {\mathscr{P}}(\mathfrak{c}'_{0})$.

\item For $\ell \geqslant 1$, we use ${\mathscr{P}}(\mathfrak{c}'_{\leqslant \ell-1},\mathfrak{a}'_{\leqslant \ell})$ to denote a polynomial from the polynomial ring $\mathbb{R}[c^{-1}, \mathfrak{C}'_{\leqslant \ell-1},\mathfrak{A}'_{\leqslant \ell}]$.  We emphasize the appearance of $c^{-1}$.
\end{itemize}

\subsubsection{Determining the higher jets on $C_0$}\label{Section:Determine the higher jets on C0}
We will use an inductive argument to determine all the $k$-jets of $\psi\in \{\wb,w,\psi_2\}$ along $C_0$, where $0\leqslant k\leqslant N$.  Indeed, we will show that all the $\mathfrak{c}_k$'s are determined by the initial data.

First of all, we observe that all the $\mathfrak{c}_0$'s are already given. In fact, $\wb,w$ and $\psi_2$ are determined by the continuity of the solution from $(v_r,c_r)$ on $\mathcal{D}_0$. The functions $\Th^1$ and $\Th^2$ are determined by the geometry of $C_0$.

Secondly, we will determine all the $\mathfrak{c}_1$'s. It is useful to rewrite the commutator formula and $[T,\Xh]$ in \eqref{eq:commutator formulas} as follows:
 \begin{equation}\label{eq: schematic commutator with mathscr P T d}
 [T,\mathfrak{d}]={\mathscr{P}}(\mathfrak{c}'_0,\mathfrak{a}'_1)\Xh, \ \ \mathfrak{d}\in \{T, \Xh\}.
 \end{equation}
where the polynomial ${\mathscr{P}}(\mathfrak{c}'_0,\mathfrak{a}'_1)$ has no constant term. We remark that the presence of $\mathfrak{a}'_1$ is due to the $\kappa$ appearing in $[T,\Xh]$, see \eqref{eq:commutator formulas}.

We also recall that the source terms on the righthand side of  \eqref{Main Wave equation: order 0} can be written as $\kappa^{-1}{\mathscr{P}}(\mathfrak{c}'_0)\cdot \mathfrak{a}_1$. In view of \eqref{eq:wave operator in null frame}, the equations \eqref{Main Wave equation: order 0} for $\psi\in \{\wb, w, \psi_2\}$ can be rewritten as
\[2L(T\psi)+L(c^{-1}\kappa L\psi) = c\kappa \Xh^2(\psi) - \frac{1}{2}\left(\chi \underline{L}\psi + \chib Lf\right)- 2  \zeta \Xh(\psi)  + \frac{c\kappa}{2}\Xh\left(\log(h)\right)\Xh(\psi)+{\mathscr{P}}(\mathfrak{c}'_0)\cdot \mathfrak{a}_1.\]

Together with the equation \eqref{structure eq 1: L kappa}  for $\kappa$, the conclusion is that, for all $\mathfrak{a}_1 \in \mathfrak{A}_1$, we have the schematic expression:
\begin{equation}\label{eq: L a_1}
L(\mathfrak{a}_1)={\mathscr{P}}(\mathfrak{c}'_0) \cdot \mathfrak{a}_1.
\end{equation}
\begin{remark}
We use the wave equation to obtain the transport equation for $L(T\wb)$. It is not clear if we use the Euler equations for $\wb$. This is because the corresponding eigenvalue is basically zero.
\end{remark}
The key structure for \eqref{eq: L a_1} is that ${\mathscr{P}}(\mathfrak{c}'_0)$ only involves $\mathfrak{c}'_0$ terms which have already been explicitly given on $C_0$. By integrating this equation from $S_{\delta,0}$, we  obtain the values of $\mathfrak{a}_1$ on $S_{t,0}$ for all $t\in [\delta,t^*]$.

Next, to derive an equation for $L(\mathfrak{b}_1)$, we commute $T$  with  \eqref{structure eq 3: L T on Ti Xi Li}. By \eqref{eq:commutator formulas}, we obtain that
\begin{align*}
L\big(T(\widehat{T}^{k})\big)&=[L,T]\big(T^{k}\big) +\big(\widehat{T}^j \Xh(\psi_j) + \Xh(c)\big)T\big(\Xh^k\big)+T\big(\widehat{T}^j\big) \Xh(\psi_j) \Xh^k+\big(\widehat{T}^j  T(\Xh(\psi_j)) + T(\Xh(c))\big)\Xh^k\\
&={\mathscr{P}}(\mathfrak{c}'_0) \mathfrak{b}_1+{\mathscr{P}}(\mathfrak{c}'_0,\mathfrak{a}'_1).
\end{align*}
In the above and the following computations, the formula \eqref{eq: schematic commutator with mathscr P T d} and $[L,T]={\mathscr{P}}(\mathfrak{c}'_0,\mathfrak{a}'_1)\Xh$ are useful. We notice that these commutator formulas do not involve $\mathfrak{b}_1$ and the  $\mathfrak{a}'_1$ terms have already been determined in the previous step.  Therefore, we obtain that
\begin{equation}\label{eq: L b_1}
L(\mathfrak{b}_1) ={\mathscr{P}}(\mathfrak{c}'_0) \mathfrak{b}_1+{\mathscr{P}}(\mathfrak{c}'_0,\mathfrak{a}'_1).
\end{equation}
The coefficient functions for the ODE \eqref{eq: L b_1} has been explicitly given on $C_0$ and all the coefficients except for $\mathfrak{b}_1$ have been already determined. We then integrate from $S_{\delta,0}$ and we obtain the values of $\mathfrak{b}_1$ on $C_0$.

To further proceed, we need the following lemma to handle the commutators:
\begin{lemma}For all $\ell,k,m\geqslant 1$, for all $\mathfrak{d} \in \{L,\Xh\}$, we have
\begin{equation}\label{eq: commutator formula ofr T ell d k}
[T^\ell,\mathfrak{d}^k](\mathfrak{c}'_m)={\mathscr{P}}(\mathfrak{c}'_{\leqslant \ell-1},\mathfrak{a}'_{\leqslant \ell})\cdot \mathfrak{c}'_{\leqslant m+\ell-1}.
\end{equation}
\end{lemma}
\begin{proof}
We start with the case where $\ell=1$. For $\mathfrak{d}=L$ or $\Xh$ and for any $\mathfrak{c}'_m$, we have
\begin{align*}
[T,\mathfrak{d}^k] \mathfrak{c}'_m&=\sum_{j=0}^{k-1}\mathfrak{d}^{k-1-j}[\mathfrak{d},T]\mathfrak{d}^j(\mathfrak{c}'_m)=\sum_{j=0}^{k-1}\mathfrak{d}^{k-1-j}({\mathscr{P}}(\mathfrak{c}'_0,\mathfrak{a}'_1)\Xh)\mathfrak{d}^j(\mathfrak{c}'_m).
\end{align*}
By the definition of the prime notations, we conclude that
\begin{align*}
[T,\mathfrak{d}^k] \mathfrak{c}'_m&={\mathscr{P}}(\mathfrak{c}'_0,\mathfrak{a}'_1)\cdot \mathfrak{c}'_m.
\end{align*}
We then consider the case for $\ell=2$. By the formula for $\ell=1$, we compute that
\begin{align*}
T^2\mathfrak{d}^k(\mathfrak{c}'_m)&=T\left(\mathfrak{d}^k T(\mathfrak{c}'_m) +{\mathscr{P}}(\mathfrak{c}'_0,\mathfrak{a}'_1)\cdot \mathfrak{c}'_m\right)\\
&=T\left(\mathfrak{d}^k T(\mathfrak{c}'_m)\right)+{\mathscr{P}}(\mathfrak{c}'_0,\mathfrak{a}'_1)\cdot \mathfrak{c}'_{m+1}+{\mathscr{P}}(\mathfrak{c}'_{\leqslant 1},\mathfrak{a}'_{\leqslant 2})\cdot \mathfrak{c}'_{m}.
\end{align*}
We can use again the  formula for $\ell=1$ and we obtain
\begin{align*}
T^2\mathfrak{d}^k(\mathfrak{c}'_m)&=\mathfrak{d}^k T^2(\mathfrak{c}'_m)+{\mathscr{P}}(\mathfrak{c}'_0,\mathfrak{a}'_1)\cdot \mathfrak{c}'_{m+1}+{\mathscr{P}}(\mathfrak{c}'_{\leqslant 1},\mathfrak{a}'_{\leqslant 2})\cdot \mathfrak{c}'_{m}\\
&=\mathfrak{d}^k T^2(\mathfrak{c}'_m)+{\mathscr{P}}(\mathfrak{c}'_{\leqslant 1},\mathfrak{a}'_{\leqslant 2})\cdot \mathfrak{c}'_{\leqslant m+1}.
\end{align*}
We recall that $\mathfrak{c}'_{\leqslant m+1}$ denote terms of type $\mathfrak{c}'_{\ell}$ with $\ell\leqslant m+1$. We can then repeat this process to prove the lemma. 
\end{proof}

We apply the lemma to  t\eqref{eq: L a_1} and \eqref{eq: L b_1}, i.e., the following system of equations:
\begin{equation}\label{eq: inductive equation on C0 case 1}
\begin{cases}
L(\mathfrak{a}_1)&={\mathscr{P}}(\mathfrak{c}'_0) \cdot \mathfrak{a}_1,\\
L(\mathfrak{b}_1)&={\mathscr{P}}(\mathfrak{c}'_0) \mathfrak{b}_1+{\mathscr{P}}(\mathfrak{c}'_0,\mathfrak{a}'_1).
\end{cases}
\end{equation}
To determine the values of $\mathfrak{c}_{k}$ on $C_0$ for $0\leqslant n\leqslant N-1$, we prove inductively for $n$ that
\begin{equation}\label{eq: inductive equation on C0 case n}
\begin{cases}
L\big(T^{n}\big(\mathfrak{a}_1\big)\big)&={\mathscr{P}}(\mathfrak{c}'_0) \cdot T^{n}\big(\mathfrak{a}_1\big)+{\mathscr{P}}(\mathfrak{c}'_{\leqslant n}),\\
L\big(T^{n}\big(\mathfrak{b}_1\big)\big)&={\mathscr{P}}(\mathfrak{c}'_0) \cdot T^{n}\big(\mathfrak{b}_1\big)+{\mathscr{P}}(\mathfrak{c}'_{\leqslant n},\mathfrak{a}'_{\leqslant n+1}).
\end{cases}
\end{equation}
The base cases $n=0$ is trivial. From the case $n$ to case $n+1$, it is straightforward if we  apply $[L,T]={\mathscr{P}}(\mathfrak{c}'_0,\mathfrak{a}'_1)\Xh$ and \eqref{eq: commutator formula ofr T ell d k} to \eqref{eq: inductive equation on C0 case n}.

To compute the values of  $\mathfrak{c}_{n}$ on $C_0$, we proceed inductively on $n$. The value of   $\mathfrak{c}_{0}$  and $\mathfrak{c}_{1}$ on $C_0$ are known. We make the inductive hypothesis that $n'\leqslant n$, the values of $\mathfrak{c}_{n'}$ on $C_0$ are given. Thus, we also have the values of $\mathfrak{c}'_{n'}$ on $C_0$. The values of $\mathfrak{c}_{n+1}$ are obtained by solving \eqref{eq: inductive equation on C0 case n} for $n+1$:
\[
\begin{cases}
L\big(T^{n+1}\big(\mathfrak{a}_1\big)\big)&={\mathscr{P}}(\mathfrak{c}'_0) \cdot T^{n+1}\big(\mathfrak{a}_1\big)+{\mathscr{P}}(\mathfrak{c}'_{\leqslant n+1}),\\
L\big(T^{n+1}\big(\mathfrak{b}_1\big)\big)&={\mathscr{P}}(\mathfrak{c}'_0) \cdot T^{n+1}\big(\mathfrak{b}_1\big)+{\mathscr{P}}(\mathfrak{c}'_{\leqslant n+1},\mathfrak{a}'_{\leqslant n+2}).
\end{cases}
\]
Since all the coefficient functions are known, we can first integrate the first equation from $S_{\delta,0}$ to obtain $T^{n+1}\left(\mathfrak{a}_1\right)$. As a consequence,  we can determine the values of $T^{n+1}\left(\mathfrak{a}_1\right)$ hence the values of $\mathfrak{a}'_{\leqslant n+2}$ in the second equation. Finally, we integrate the second equation  form $S_{\delta,0}$ to obtain $T^{n+1}\left(\mathfrak{b}_1\right)$. 

\subsubsection{Preliminary bounds without $L$-directions}

The inductive process gives the following proposition:
\begin{proposition}For all multi-indices $\alpha$ with $1\leqslant |\alpha|\leqslant N-1$, for all $\psi \in \{\wb, w,\psi_2\}$ and $Z\in {\mathscr{Z}}$, the following pointwise estimates hold on $C_0$:
\begin{equation}\label{eq:preliminary pointwise bounds on C0 for Z}
\big\|Z^\alpha(\psi)\big\|_{L^\infty(C_0)}+\big\|Z^\alpha(\kappa)\big\|_{L^\infty(C_0)}+\big\|Z^\alpha(\Th^i)\big\|_{L^\infty(C_0)}\lesssim \varepsilon.
\end{equation}
We also have 
\begin{equation}\label{eq:preliminary pointwise bounds on C0 for Z zeroth order}
\big\|T\wb\big\|_{L^\infty(C_0)}+\big\|\Th^1\big\|_{L^\infty(C_0)}+\big\|\kappa\big\|_{L^\infty(C_0)} \lesssim 1,  \ \ \big\|\Th^2\big\|_{L^\infty(C_0)}\lesssim \varepsilon.
\end{equation}
\end{proposition}
\begin{proof}

We prove \eqref{eq:preliminary pointwise bounds on C0 for Z} inductively on $k$ where $k$ is the index for the unknowns from $\mathfrak{C}_k$. For $|\alpha|\leqslant 1$, the bounds on $\Xh(\mathfrak{c}_0)$ are clear since they are induced from the  $\mathfrak{a}_0$'s of $(v_r,c_r)$ on $\mathcal{D}_0$. In particular, we can integrate the equations for $L(\Th^i)$ in \eqref{structure eq 3: L T on Ti Xi Li} do derive the estimates in \eqref{eq:preliminary pointwise bounds on C0 for Z zeroth order} for $\Th^i$'s. To determine $T(\wb), T(w), T(\psi_2), \kappa, T(\Th^1)$ and $T(\Th^2)$,  we integrate the first equation in \eqref{eq: inductive equation on C0 case 1} and then integrate the second equation in \eqref{eq: inductive equation on C0 case 1}. Since the data for $\mathfrak{c}_1$ on $S_{\delta,0}$ satisfy the estimates in \eqref{eq:preliminary pointwise bounds on C0 for Z} and \eqref{eq:preliminary pointwise bounds on C0 for Z zeroth order}, this gives the bounds for $\mathfrak{c}_0$ and $\mathfrak{c}_1$ appearing in \eqref{eq:preliminary pointwise bounds on C0 for Z}.  The bound on $\kappa$ in \eqref{eq:preliminary pointwise bounds on C0 for Z zeroth order} also follow immediately from integrating the equation.

We notice that the constructions in Section \ref{Section:Determine the higher jets on C0} are based on solving a given system of ordinary differential equations. In fact, we can differentiate \eqref{eq: inductive equation on C0 case n} along the $\Xh$ direction to derive
\[
\begin{cases}
L\big(\Xh^\ell T^{n}\big(\mathfrak{a}_1\big)\big)&={\mathscr{P}}(\mathfrak{c}'_0) \sum_{j=0}^\ell \Xh^j T^{n}\big(\mathfrak{a}_1\big)+{\mathscr{P}}(\mathfrak{c}'_{\leqslant n}),\\
L\big(\Xh^\ell T^{n}\big(\mathfrak{b}_1\big)\big)&={\mathscr{P}}(\mathfrak{c}'_0) \sum_{j=0}^\ell \Xh^j T^{n}\big(\mathfrak{b}_1\big)+{\mathscr{P}}(\mathfrak{c}'_{\leqslant n},\mathfrak{a}'_{\leqslant n+1}).
\end{cases}
\]
These equations have smooth coefficient functions on $C_0$ and have given initial data on $S_{\delta,0}$. Therefore, all the $\mathfrak{c}'_k$'s depend in a continuous way on the $k$-jets (with  $0\leqslant k\leqslant N$) of the data given by on $S_{\delta,0}$.  Since the data are $\varepsilon$-close to the constant states, see Proposition \eqref{prop: bounds on geometry on sigma delta 1}, the property of continuous dependence on initial conditions for ODE system  proves \eqref{eq:preliminary pointwise bounds on C0 for Z}. This completes the proof of the proposition.
\end{proof}

The bounds in the above proposition can be improved as follows:
\begin{corollary}\label{coro:bounds on C0} On  $C_0$, for all $\delta \leqslant t \leqslant t^*$, the initial data satisfy the following estimates:
\begin{equation}\label{eq: pointwise bounds on C0 order 0}
\big\|T\wb+\frac{2}{\gamma+1}\big\|_{L^\infty(S_{t,0})}+\big\|\Th^1+1\big\|_{L^\infty(S_{t,0})}+\big\|\kappa-t\big\|_{L^\infty(S_{t,0})}+\big\|\Th^2\big\|_{L^\infty(S_{t,0})} \lesssim \varepsilon t.
\end{equation}
For all multi-indices $\alpha,\beta$ with $1\leqslant |\alpha|\leqslant N-1$ and $0\leqslant |\beta|\leqslant N-2$, for all $\psi \in \{\wb, w,\psi_2\}$ and $Z\in {\mathscr{Z}}$, 
\begin{equation}\label{eq: pointwise bounds on C0 order high}
\begin{cases}
&\|Z^\alpha(\psi)\|_{L^\infty(S_{t,0})} \lesssim \varepsilon,\\
&\|TZ^\beta(\psi)\|_{L^\infty(S_{t,0})} \lesssim \varepsilon t, \ \ \text{except for} \ \ TZ^\beta(\psi)=T\wb.
\end{cases}
\end{equation}
For all multi-indices $\alpha$ with $1\leqslant |\alpha|\leqslant N-1$ and $Z\in {\mathscr{Z}}$, 
\begin{equation}\label{eq: pointwise bounds on C0 on kappa Thi}
\|Z^\alpha(\Th^1)\|_{L^\infty(S_{t,0})}+\|Z^\alpha(\Th^2)\|_{L^\infty(S_{t,0})}+\|Z^\alpha(\kappa)\|_{L^\infty(S_{t,0})}\lesssim \varepsilon t.
\end{equation}
\end{corollary}
\begin{proof}It is obvious that the bounds \eqref{eq: pointwise bounds on C0 order 0}, \eqref{eq: pointwise bounds on C0 order high} and \eqref{eq: pointwise bounds on C0 on kappa Thi} hold on $S_{\delta,0}$. Therefore, compared to the estiamtes in \eqref{eq:preliminary pointwise bounds on C0 for Z} and \eqref{eq:preliminary pointwise bounds on C0 for Z zeroth order}, the improvement of a $t$-factor comes from the initial data and the integration between $\delta$ to $t$.
\end{proof}
\begin{remark}
The estimate of $\|\Th^1+1\|$ can be further improved as follows:
\begin{equation}\label{eq: pointwise bounds improved on Th1}
\big\|\Th^1+1\big\|_{L^\infty(S_{t,0})} \lesssim \varepsilon t^2.
\end{equation}
It holds for $t=\delta$. Since the terms on the righthand side for $L\big( \Th^1+1 \big)$ in \eqref{structure eq 3: L T on Ti Xi Li} are all of size $O(t\varepsilon)$, we can the integrate $L\big( \Th^1+1 \big)$ to obtain \eqref{eq: pointwise bounds improved on Th1}. 

Similarly, by writing $L(\kappa)$ as
\[L \kappa = 1-\big(1+\frac{\gamma+1}{\gamma-1}Tc\big) + e' \kappa,\]
the estimate  of $\kappa$ can be improved as follows:
\begin{equation}\label{eq: improved kappa on C0}
\big\|\kappa-t\big\|_{L^\infty(S_{t,0})} \lesssim \varepsilon t^2.
\end{equation}
For all multi-indices $\alpha$ with $1\leqslant |\alpha|\leqslant N-1$, the same argument gives
\begin{equation}\label{eq: improved kappa on C0 derivatives}
\big\|Z^\alpha(\kappa)\big\|_{L^\infty(S_{t,0})} \lesssim \varepsilon t^2.
\end{equation}
Similarly, by Remark \ref{remark: improvement with T}, we have
\begin{equation}\label{eq: improved T Ti on C0 derivatives}
\big\|TZ^\alpha(\Th^1)\big\|_{L^\infty(S_{t,0})}+\big\|TZ^\alpha(\Th^2)\big\|_{L^\infty(S_{t,0})} \lesssim \varepsilon t^2.
\end{equation}
Finally, we repeat the argument for 
$L\slashed{g}=L(\slashed{g})=2 \slashed{g} \cdot \chi$ to derive
\begin{equation}\label{eq:slashed g on C0}
\|\slashed{g}-1\|_{L^\infty(S_{0,t})}+\|Z^\alpha(\slashed{g})\|_{L^\infty(S_{0,t})}\lesssim t \varepsilon, \ \ 1\leqslant |\alpha|\leqslant N-1.
\end{equation}
\end{remark}

\begin{corollary}\label{coro: pointwise bounds on C0 with circles}
For all multi-indices $\alpha$ with $1\leqslant |\alpha|\leqslant N-1$, for all $\delta \leqslant t \leqslant t^*$, for all $\psi \in \{\wb,w,\psi_2\}$ and $\Zr\in \mathring{\mathscr{Z}}$, we have
\begin{equation}\label{eq: pointwise bounds on C0 with circles}
t\|\Xh( \Zr^\alpha \psi)\|_{L^\infty(S_{t,0})}+\|T( \Zr^\alpha \psi)\|_{L^\infty(S_{t,0})}\lesssim \varepsilon t.
\end{equation}
\end{corollary}

\begin{proof}
We repeat the proof of Proposition \ref{prop:pointwise bound on circled derivatives of psi 1}. For all $f\in \big\{\Th^1,\frac{\Th^2}{\kappa},\kappar\Th^2,\frac{\kappar}{\kappa}\Th^1\big\}$ (appeared in \eqref{eq: transformation between Xh T and Xr Tr}), for all multi-indices $\alpha$ with $1\leqslant |\alpha| \leqslant N$ with $Z\in \mathscr{Z}=\{T,\Xh\}$, by Corollary \ref{coro:bounds on C0},  we have 
\[\|f\|_{L^\infty(S_{t,0})}\lesssim 1, \ \ \|Z^\alpha(f)\|_{L^\infty(S_{t,0})}\lesssim \varepsilon.\]

To derive the bounds on $T( \Zr^\alpha \psi)$,  we use \eqref{eq: transformation between Xh T and Xr Tr} to replace all the $\Zr$ so that it is  expressed as a linear combination of the terms of the following forms
\[(\mathbf{A}). \ Z^{\alpha_1}(f)\cdot\cdots \cdot Z^{\alpha_k}(f)\cdot Z^\beta(f)\cdot TZ^\gamma(\psi)  \ \ (\mathbf{B}). \ Z^{\alpha_1}(f)\cdot\cdots \cdot Z^{\alpha_k}(f)\cdot TZ^\beta(f)\cdot Z^\gamma(\psi)\] 
where $Z\in\mathscr{Z}=\{T,\Xh\}$ and $\sum_{i\leqslant k}\alpha_i+\beta+\gamma=\alpha$. We remark that the operator $T$ in $(\mathbf{A})$ or $(\mathbf{B})$ is form the first operator in $T( \Zr^\alpha \psi)$. For type $(\mathbf{A})$ terms, we notice that $\gamma\neq 0$ so that $\|TZ^\gamma(\psi)\|_{L^\infty(S_{t,0})}\lesssim \varepsilon t$; For type $(\mathbf{B})$ terms, by \eqref{eq: improved T Ti on C0 derivatives}, we have $\|TZ^\beta(f)\|_{L^\infty(S_{t,0})}\lesssim \varepsilon t$. Therefore, for both type of terms, by Corollary \ref{coro:bounds on C0},  they are bounded by $O(\varepsilon t)$. Hence, we obtain the desired estimates on $T( \Zr^\alpha \psi)$ in \eqref{eq: pointwise bounds on C0 with circles}. 

The bounds on  $\Xh( \Zr^\alpha \psi)$ can be derived exactly in the same manner. This completes the proof of the corollary.
\end{proof}
\begin{corollary}\label{coro: pointwise bounds on C0 with circles}
For all multi-indices $\alpha$ with $1\leqslant |\alpha|\leqslant N-1$,  for all $\delta \leqslant t \leqslant t^*$, for all $\psi \in \{\wb,w,\psi_2\}$ and $\Zr\in \mathring{\mathscr{Z}}$, we have
\[t\|\Xr( \Zr^\alpha \psi)\|_{L^\infty(S_{t,0})}+\|\Tr( \Zr^\alpha \psi)\|_{L^\infty(S_{t,0})}\lesssim \varepsilon t.
\]Moreover, we have
\[
t\|\Xr(\wb)\|_{L^\infty(S_{t,0})}+t\|\Xr(\psi)\|_{L^\infty(S_{t,0})}+\|\Tr(\psi)\|_{L^\infty(S_{t,0})}\lesssim \varepsilon t, \ \ \psi\in \{w,\psi_2\}.
\]
\end{corollary}
\begin{proof}
This is straightforward from   \eqref{eq: pointwise bounds on C0 with circles}: indeed, it suffices to use the   \eqref{eq: transformation between Xh T and Xr Tr} replace the first $\Zr$'s in the above two inequalities by the $Z$'s.\end{proof}

\subsubsection{Auxiliary estimates on $\yr, \zr, \chir$ and $\etar$}
We derive estimates for $\lambda\in  \{\yr, \zr, \chir,\etar\}$ on $C_0$.

For $ \chir=-\Xr(\psi_2)$ or $ \etar=-\Tr(\psi_2)$, Corollary \ref{coro: pointwise bounds on C0 with circles} shows that, for all multi-indices $\alpha$ with $|\alpha|\leqslant N$, we have
\begin{equation}\label{eq: bounds for lambda on C0}
\|\mathring{Z}^\alpha(\lambda)\|_{L^\infty(S_{t,0})}\lesssim  \begin{cases}\varepsilon, \ \ & \text{if}~\Zr^\beta=\Xr^\beta ~\text{and}~\lambda =\chir;\\
\varepsilon t, \ \ & \text{otherwise}.
\end{cases} 
\end{equation}

For $\lambda =\zr$,   it suffices to bound $z$. Therefore, we can use Corollary \ref{coro:bounds on C0} to bound each term on the righthand side of \eqref{eq: a precise computation for z}. This shows that $\|z\|_{L^\infty(S_{t,0})} \lesssim \varepsilon t$. Next, by applying $\Zr^\alpha$ to \eqref{eq: a precise computation for z} and using the fact that $\Zr^\alpha(z)=\Tr\Zr^\alpha(v^1+c)$ as well as  Corollary \ref{coro: pointwise bounds on C0 with circles}, we have $\|\Zr^\alpha(z)\|_{L^\infty(S_{t,0})} \lesssim \varepsilon t$. Therefore, we obtain the estimate for $\zr$:
\begin{equation}\label{eq: bounds for zr on C0}
\|\mathring{Z}^\alpha(\zr)\|_{L^\infty(S_{t,0})}\lesssim   \varepsilon.
\end{equation}

For $\lambda=\yr$, it suffices  to  show that $\|\mathring{Z}^\alpha(y)\|_{L^\infty(S_{t,0})}\lesssim  t\varepsilon$,  for all multi-indices $\alpha$ with $|\alpha|\leqslant N-1$ and $t\in [\delta,1]$. By Corollary \ref{coro: pointwise bounds on C0 with circles}, we may assume $\mathring{Z}^\alpha=\Xr^\alpha$. By the proof of Corollary \ref{coro: pointwise bounds on C0 with circles}, we may use \eqref{eq: transformation between Xh T and Xr Tr} to replace all the $\Xr$'s, so that it suffices to show that $\|\Xh^\alpha(y)\|_{L^\infty(S_{t,0})}\lesssim  t\varepsilon$. Since $X=\sqrt{\slashed{g}}\Xh$, by the bounds in \eqref{eq:slashed g on C0}, it suffices to show that $\|X^\alpha(y)\|_{L^\infty(S_{t,0})}\lesssim t\varepsilon$.
According to the computations in Section \ref{Section: lambda on C0}, we have
\begin{align*}
y&=\underbrace{-\frac{\Th^2}{\kappa} + X(v^1+c)}_{y'}+ \underbrace{(1-\sqrt{\slashed{g}})\Xh(v^1+c)-(\Th^1+1) \Xh(v^1+c) + \frac{\Th^2}{\kappa} (T(v^1+c)+1)}_{y_{\rm err}}.
\end{align*}
By Corollary \ref{coro:bounds on C0},  Corollary \ref{coro: pointwise bounds on C0 with circles} and \eqref{eq:slashed g on C0},  we have $\|X^\alpha(y_{\rm err})\|_{L^\infty(S_{t,0})}\lesssim  t\varepsilon$. Thus,  it suffices to show 
\[\|X^\alpha(y')\|_{L^\infty(S_{t,0})}\lesssim  t\varepsilon.\]
By \eqref{eq:Th on Sigma delta} and \eqref{eq: kappa on Sigma delta}, we have
\[y'(t)=\frac{A(t,x_2)}{\kappa} + X(v^1+c).\]
where $A(t,\slashed{\vartheta})=\int_0^t a(t,\slashed{\vartheta})d\tau$. As in Section \ref{Section: lambda on C0}, we also have 
\begin{align*}
a(\tau,\slashed{\vartheta})&=\underbrace{-X \big(v^1+c \big)(\tau,\slashed{\vartheta}) }_{a'(\tau,\vartheta)}+ \underbrace{X \big(c (\Th^1+1) \big) (\tau,\slashed{\vartheta}) +X \big(\psi_1+c \Th^1 \big) (\tau,\slashed{\vartheta}) \big(\frac{\partial \slashed{\vartheta}}{\partial x_2}-1\big)}_{a_{\rm err}},
\end{align*}
which leads to
\begin{align*}
y'&=\underbrace{\frac{1}{t}\int_0^t a'(\tau,\slashed{\vartheta})d\tau+ X(v^1+c)}_{y''}+\underbrace{\big(\frac{1}{\kappa}-\frac{1}{t}\big)\int_0^t a'(\tau,\slashed{\vartheta})d\tau}_{y'_{\rm err}}+\frac{1}{\kappa}\int_0^t a_{\rm err}(\tau,\slashed{\vartheta})d\tau.
\end{align*}

By \eqref{eq:x1x2 on C0} and Corollary \ref{coro:bounds on C0},  we have $\|X^\alpha(a_{\rm err})\|_{L^\infty(S_{t,0})}\lesssim  t\varepsilon$. By \eqref{eq: improved kappa on C0} and\eqref{eq: improved kappa on C0 derivatives}, we have $\|X^\alpha(y'_{\rm err})\|_{L^\infty(S_{t,0})}\lesssim  t\varepsilon$. Hence, it suffices to bound
\begin{align*}
X^\alpha(y'')&=\frac{1}{\delta}\int_0^\delta \big[\frac{\partial^{\alpha+1} (v^1+c)(\delta,\slashed{\vartheta})}{\partial \slashed{\vartheta}^{\alpha+1}}-\frac{\partial^{\alpha+1} (v^1+c)(\tau,\slashed{\vartheta})}{\partial \slashed{\vartheta}^{\alpha+1}}\big]d\tau.
\end{align*}
We now can use the intermediate value theorem to conclude that $\|X^\alpha(y'')\|_{L^\infty(S_{\delta,t})}\lesssim  t\varepsilon$. Therefore,
\begin{equation}\label{eq: bounds for yr on C0}
\|\mathring{Z}^\alpha(\yr)\|_{L^\infty(S_{t,0})}\lesssim   \varepsilon.
\end{equation}

\subsubsection{The initial energy flux}

\begin{proposition}For all multi-indices $\alpha$ with $|\alpha|\leqslant N-1$, for all $\psi \in \{\wb,w,\psi_2\}$ and $\Zr\in \mathring{\mathscr{Z}}$, we have
\begin{equation}\label{eq:pointwise bound on L circled derivatives of psi 1}
\|L( \Zr^\alpha \psi)\|_{L^\infty(C_0)}\lesssim \varepsilon.
\end{equation}
\end{proposition}
\begin{proof}
According to \eqref{eq:Euler commutated}  and Corollary \ref{coro: pointwise bounds on C0 with circles}, for all multi-indices $\alpha$ with $|\alpha|\leqslant N$, for all $\psi \in \{\wb, w,\psi_2\}$, we have
\[  \|\Zr^\alpha \Lr (\psi)\|_{L^\infty(C_0)} \lesssim \varepsilon.\]
Since we have $\|\Zr^\alpha \lambda\|_{L^\infty(C_0)} \lesssim \varepsilon$ for all $\lambda\in \{\etar,\chir,\yr,\zr\}$,  according to \eqref{eq: commutation formular for Lr Zr}, we have
\begin{align*}
\|\Lr( \Zr^\alpha \psi)\|_{L^\infty(C_0)}\leqslant \|\Zr^\alpha \Lr (\psi)\|_{L^\infty(C_0)}+\sum_{\substack{\alpha_1+\alpha_2=\alpha \\ |\alpha_1|\leqslant |\alpha|-1}}\|\Zr^{\alpha_1}(\lambda) \Zr^{\alpha_2}(\psi)\|_{L^\infty(C_0)}\lesssim \varepsilon.
\end{align*}
Finally, by $L-\Lr =c\big(\frac{\Th^1+1}{\kappar}\Tr -\Th^2\Xr\big)$, we conclude that
\[  \|L \Zr^\alpha (\psi)\|_{L^\infty(C_0)} \lesssim \varepsilon.\]
This finishes the proof of the proposition.
\end{proof}
As a summary, we have
\begin{proposition}\label{prop:pointwise bound on data on Sigma delta}
For all multi-indices $\alpha$ with $1\leqslant |\alpha|\leqslant N-1$, for all $\psi \in \{\wb,w,\psi_2\}$ and $\Zr\in \mathring{\mathscr{Z}}$, we have
\[
t\|L( \Zr^\alpha \psi)\|_{L^\infty(S_{t,0})}+t\|\Xh( \Zr^\alpha \psi)\|_{L^\infty(S_{t,0})}+\|\Lb( \Zr^\alpha \psi)\|_{L^\infty(S_{t,0})}\lesssim t\varepsilon.
\]
\end{proposition}
\begin{proof}
The bound of $\Lb \Zr^\alpha (\psi)$ is immediate from the formula $\Lb=2T+c^{-1}\kappa L$, \eqref{eq:pointwise bound on circled derivatives of psi 1} and the bounds on $L \Zr^\alpha (\psi)$. 
\end{proof}

\begin{remark}\label{remark: closing I 2 flux bounds}
The above $L^\infty$ bounds imply the following $L^2$ flux bounds:
\[\mathscr{F}_{n}(\psi)(t,0) \lesssim \varepsilon^2 t^2, \ \ \psi \in \{w,\wb,\psi_2\},  \ \ 1\leqslant n\leqslant \Ntop.\]
Therefore, we have checked  the ansatz for $\mathscr{F}_{n}$ in $\mathbf{(I_2)}$, see \eqref{initial I2}.
\end{remark}

By  Remark \ref{remark: closing I infty 1},  Remark \ref{remark: closing I infty 2},  Remark \ref{remark: closing I 2 lowest energy},  Remark \ref{remark: closing I_{irrotational}}, Remark \ref{remark: closing I 2 higher energy} and Remark \ref{remark: closing I 2 flux bounds}, the data constructed in \eqref{Taylor expansions of Riemann invariants up to order N} satisfies all the required ansatz. Therefore, we have completed the proof of {\bf Theorem 1} as well as {\bf Corollary 1}.

\section{Existence of solutions}\label{section:existence}

For all $\delta>0$ and the data constructed on $\Sigma_\delta$, we have a unique solution $(v_\delta,c_\delta)$ defined on $\mathcal{D}(\delta)$. For the sake of simplicity, we drop the dependence on $\delta$ and write the solution as $(v,c)$.

\subsection{The region of convergence}\label{section:region of convergence}

For the constant state {\color{black}$(\mathring{v}_r,\mathring{c}_r)$} on $x_1\geqslant 0$, the corresponding characteristic boundary of the future development is a ray in $(t,x_1)$ plane. It is hence a flat hypersurface ${C}^{\rm cst}_{0}$ inside $\mathbb{R}^3$. The slope of the ray can be computed as 
\[ {k}^{\rm cst}=\mathring{v}_r+\mathring{c}_r.\]

We compute the null vector field $L^{\rm cst}$ and the acoustical function $u^{\rm cst}$ associated to the constant state {\color{black}$(\mathring{v}_r,\mathring{c}_r)$}. In view of the associated 1-dimensional rarefaction wave computed in \eqref{eq:1D-rarefaction-wave},
we have
\[L^{\rm cst}=\partial_t+\frac{x_1}{t}\partial_1, \quad u^{\rm cst}=k^{\rm cst}-\frac{x_1}{t}. \] 


We can rewrite $L$ as
\[ L=\partial_t+\big(v^1+c\big)\partial_1+\big(v^2-c\Th^2\big)\partial_2-c\big(\Th^1+1\big)\partial_1.\]
Therefore, we can compare $L$ and $L^{\rm cst}$ as follows
\[L-L^{\rm cst}=\big(v^1+c-\frac{x_1}{t}\big)\partial_1+\big(v^2-c\Th^2\big)\partial_2-c\big(\Th^1+1\big)\partial_1.\]
In view of the expression \eqref{eq:x1x2 on C0 0} on $C_0$, we have
\[ \big(v^1+c - \frac{\slashed{x}_1}{t}\big)(t,\vartheta) = \frac{1}{t}\int_0^t (v^1+c)(t,\vartheta) - (v^1+c)(\tau,\vartheta)d\tau + {\color{black}\frac{1}{t}\int_0^t \big(c\cdot (\Th^1+1)\big)(\tau,\vartheta)d\tau}. \]
Therefore, by \eqref{ineq: L infty bound for limiting Thi kappa} and the intermediate value theorem  we have 
\[ \|v^1+c-\frac{x_1}{t}\|_{L^\infty(C_0 \cap \Sigma_t)}\lesssim \varepsilon t. \]
Since $\Tr\big(v^1+c-\frac{x_1}{t}\big)=\Tr(v^1+c)+1=z$,  in view of the estimates derived in Section 7.1 of \cite{LuoYu1}, for all $t\in [\delta,t^*]$, we have 
\[\|\Tr\big(v^1+c-\frac{x_1}{t}\big)\|_{L^\infty(\Sigma_t)}\lesssim \varepsilon t.\]
We can integrate the above inequality for $\Tr\big(v^1+c-\frac{x_1}{t}\big)$ form $C_0$ and this leads to
\[\|v^1+c-\frac{x_1}{t}\|_{L^\infty(\Sigma_t)}\lesssim \varepsilon t.\]
In view of the pointwise bounds on $\psi_2, \Th^2$ and $\Th^1+1$ in \eqref{eq:basic L estimates} and \eqref{ineq: L infty bound for limiting Thi kappa},
we conclude that
\begin{equation}\label{eq: compare L Lcst}
	\|(L-L^{\rm cst})(x_1)\|_{L^\infty(\Sigma_t)} \lesssim \varepsilon t, \quad \|(L-L^{\rm cst})(x_2)\|_{L^\infty(\Sigma_t)} \lesssim \varepsilon, \quad \big|L-L^{\rm cst}\big|\lesssim \varepsilon,
\end{equation}
on the entire $\mathcal{D}(\delta)$, where the norm $|\cdot |$ is relative to the Euclidean metric on $\Sigma_t$.

Now we compare the function $u$ with its counterpart $u^{\rm cst}$. First, since $u\big|_{S_{\delta,0}}\equiv 0$ by \eqref{eq: u vanishes at C0} and the data is $O(\varepsilon)$-close to the constant state {\color{black}$(\mathring{v}_r,\mathring{c}_r)$},  we have
\[ u^{\rm cst}|_{S_{\delta,0}} = u - u^{\rm cst}|_{S_{\delta,0}} = -k^{\rm cst}-\frac{1}{\delta}\int_{0}^\delta I(\tau,x_2)d\tau = -k^{\rm cst} + (v^1+c)|_{S_{\delta,0}} + O(\varepsilon t) \lesssim O(\varepsilon). \]
By \eqref{def: for u initially}, \eqref{eq: gradient of u on Sigma delta} and \eqref{eq: bound on a}, we have on $\Sigma_{\delta}$ that
\begin{equation}
\nabla\big[u-\big(k^{\rm cst}-\frac{x_1}{\delta}\big)\big] = -\frac{1}{\delta}\big(0,A(\delta,x_2)\big) \lesssim O(\varepsilon).
\end{equation}
for all $u\in [0, u^*]$.  
Furthermore, since $L^{\rm cst}(u^{\rm cst}) = 0$, by \eqref{eq: compare L Lcst}  we have 
\begin{equation}
	L(u^{\rm cst}) = (L-L^{\rm cst})(k^{\rm cst}-\frac{x_1}{t}) \lesssim O(\varepsilon).
\end{equation}
In view of these estimates, We conclude that
\begin{equation}\label{eq: compare u ucst}
	\|u - u^{\rm cst}\|_{L^\infty(\Sigma_t)} \lesssim \varepsilon.
\end{equation}


\begin{center}
\includegraphics[width=3.2in]{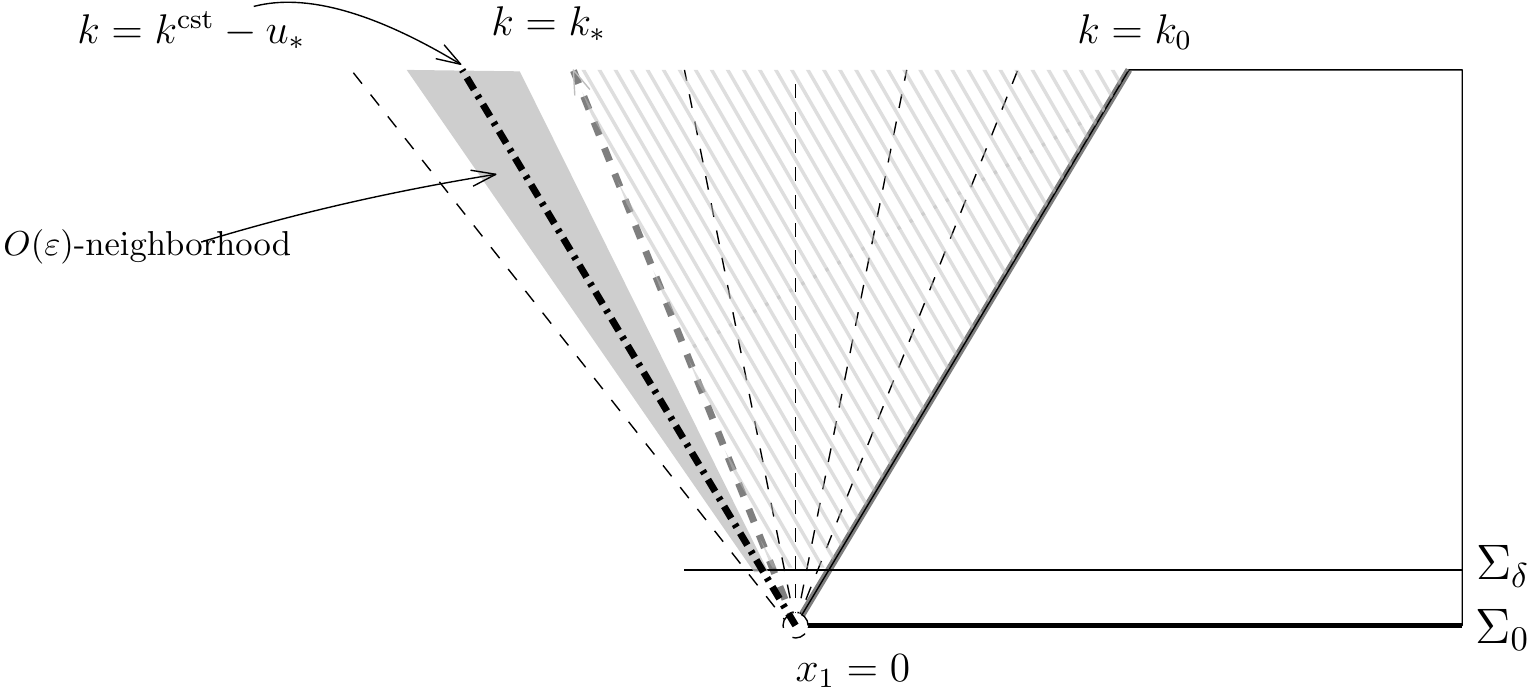}
\end{center}

By definition, $C_{u^*}$ is the characteristic hypersurface emanated from the level set  $u=u^*$ inside $\Sigma_\delta$ along the direction of $L$. {\color{black}More precisely, we consider the
	vector field $L$ along $S_{\delta,u^*}$ and extend each vector to a geodesic to generate $C_{u^*}$. Similarly, the hypersurface $C^{\rm cst}_{u^*}$ is generated by extending each vector $L^{\rm cst}$ along the surface $u^{\rm cst}=u^*$, i.e., the circle $\frac{x^1}{\delta} = {k}^{\rm cst} - u^*$, to a geodesic of the background solution which corresponds to the constant state $(\mathring{v}_r,\mathring{c}_r)$ on the right.}
In the above picture, it is  depicted as the dash-dotted line and it is the center in the grey region. It can be written down explicitly as follows:
 \[C^{\rm cst}_{u^*}=\big\{(t,x_1,x_2)\big|\frac{x_1}{t}+u^*-k^{\rm cst}=0, t\in [\delta, t^*], x_2\in [0,2\pi]\big\}.\]

By \eqref{eq: compare L Lcst} and \eqref{eq: compare u ucst}, the continuous dependence on the initial data of the ODE implies that $C_{u^*}$ is in the $O(\varepsilon t)$-neighborhood (depicted as the grey region in the picture) of $C^{\rm cst}_{u^*}$, i.e., on each $\Sigma_t$, the distance between $C^{\rm cst}_{u^*}$ and {\color{black}$C_{u^*}$} is bounded above by {\color{black}$A_0 \varepsilon t$} where $A_0$ is a universal constant.
We define 
\begin{equation}\label{def: varepsilon0}
\varepsilon_0= A_0 \varepsilon. 
\end{equation} 
We also define
\[k_*=k^{\rm cst}-u^*+\varepsilon_0.\]
For all $\delta>0$, we then consider the following region:
\[\mathcal{W}_\delta=\big\{(t,x_1,x_2)\in \mathcal{D}(\delta)\big| \delta\leqslant t\leqslant t^*,  \frac{x_1}{t}\geqslant k_*\big\}.\]
This is the shaded region in the above picture. It is clear that $\mathcal{W}_\delta$ is a compact domain. It depends only on the data $(v_r,c_r)$ on $x_1\geqslant 0$ and it is independent of $\delta$. We also have $\mathcal{W}_\delta\subset \mathcal{D}_{\delta'}$ for all $\delta\geqslant \delta'$. 
We also define
\[\mathcal{W}=\bigcup_{0<\delta<\frac{1}{2}}\mathcal{W}_\delta.\]
 In the rest of the section, we will construct centered rarefaction waves on $\mathcal{W}$ associated to the solution $(v_r,c_r)$ on $\mathcal{D}_0$.

\subsection{Uniform bounds for solutions on $\mathcal{K}_k$}
We fix the parameter $\delta$ and we consider the rarefaction wave already constructed on $\mathcal{D}(\delta)$. The estimates derived in the section will be independent of $\delta$. We recall that, according to the \emph{a priori} energy estimates derived in the first paper \cite{LuoYu1}, for all $\delta \leqslant t\leqslant t^*$, we have
{\color{black}
\[\begin{cases}
&\sum_{1\leqslant|\alpha|\leqslant \Ntop}\mathcal{E}\big(\Zr^\alpha(\psi)\big)(t,u^*)+\Eb\big(\Zr^\alpha(\psi)\big)(t,u^*)\lesssim  \varepsilon^2 t^2, \ \ \psi \in \{w,\wb,\psi_2\};\\
    &\mathcal{E}(\psi)(t,u^*)+\underline{\mathcal{E}}(\psi)(t,u^*)\lesssim  \varepsilon^2 t^2, \ \ \psi \in \{w,\psi_2\}.
\end{cases}
\]
}
We also have $\|\wb\|_{L^\infty}\lesssim 1$.

In the following convergence argument, the smallness of $\varepsilon$ does not play a role so that we can use the rough bound $\varepsilon \lesssim 1$. Therefore, by \eqref{def: energy L}, \eqref{def: energy Lb} and $\kappa \approx t$, we can rewrite the above estimates in the following rough form:
\begin{equation}\label{eq: space rough bound on Ddelta}
\int_{0}^{u^*}\!\!\!\int_{0}^{2\pi}\!\!\! |\Zr^\alpha(\psi)|^2 \sqrt{\slashed{g}}du'd\vartheta'\lesssim 1,
\end{equation}
where $\psi \in \{w,\wb,\psi_2\}$ and an arbitrary multi-index $\alpha$ with $ 0\leqslant |\alpha| \leqslant \Ntop+1$. In fact, according to $T=\frac{1}{2}\big(\Lb-c^{-1} \kappa L\big)$, we first obtain the following estimates:
\[\int_{0}^{u}\!\!\!\int_{0}^{2\pi}\!\!\! |\Xh\Zr^\alpha(\psi)|^2+|T\Zr^\alpha(\psi)|^2  \sqrt{\slashed{g}}du'd\vartheta'\lesssim 1,\]
for $\psi \in \{w,\wb,\psi_2\}$ and $\alpha$ is an arbitrary multi-index with $ 0\leqslant |\alpha| \leqslant \Ntop$. We then apply the inverse of \eqref{eq: transformation between Xh T and Xr Tr} to convert the $\Xh$ and $T$ derivatives to $\Xr$ and $\Tr$ derivatives. This proves \eqref{eq: space rough bound on Ddelta}.

By integrating over time, for all $\psi \in \{w,\wb,\psi_2\}$ and for all multi-indices $\alpha$ with $ 0\leqslant |\alpha| \leqslant \Ntop+1$, we have the following spacetime estimates:
\begin{equation}\label{eq: spacetime rough bound on Ddelta}
\int_{\delta}^{t^*}\!\!\!\int_{0}^{u}\!\!\!\int_{0}^{2\pi}\!\!\! |\Zr^\alpha(\psi)|^2 \sqrt{\slashed{g}}{\color{black}dt'}du'd\vartheta'\lesssim 1.
\end{equation}

We consider the following compact regions 
\[\mathcal{K}_k=\mathcal{W}_{2^{-k}}=\big\{(t,x_1,x_2)\in \mathcal{W}\big|2^{-k} \leqslant  t  \leqslant t^*\big\},\]
where $k\geqslant 1$ is a positive integer. It provides an exhaustion of $\mathcal{W}$ by compact sets. In fact, we have $\mathcal{K}_l\subset \mathcal{K}_{l+1}$ for all $l$ and $\bigcup_{k\geqslant 1}\mathcal{K}_k=\mathcal{W}$,
\begin{center}
\includegraphics[width=2.3in]{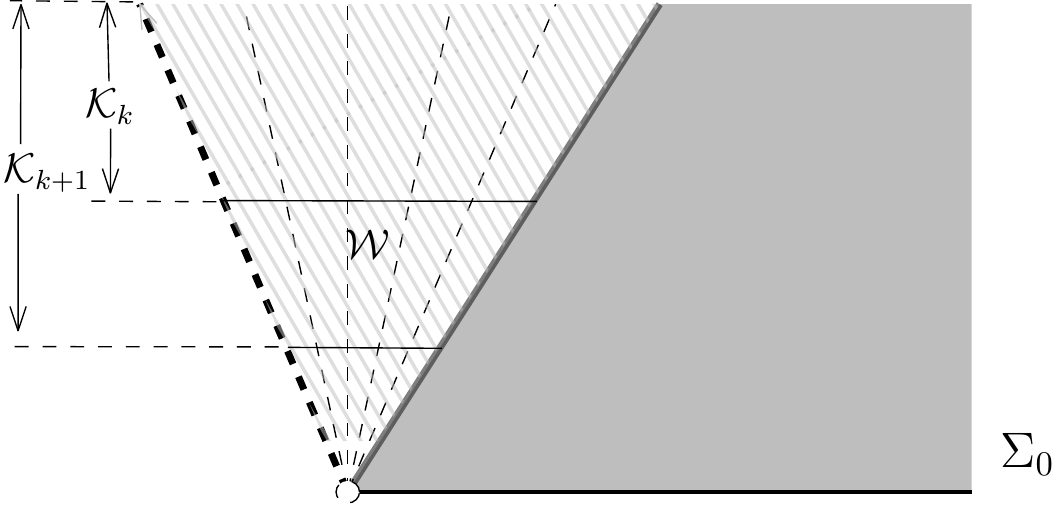}
\end{center}
We will rewrite \eqref{eq: spacetime rough bound on Ddelta} in terms of the Cartesian coordinates. We recall that the Jacobian of the coordinates transformation $\Phi:(t,u,\vartheta)\mapsto (t,x_1,x_2)$ between the acoustical coordinates $(u,t,\vartheta)$ and the Cartesian coordinates $(x_0,x_1,x_2)$ can be computed as
\begin{equation}\label{eq: jacobian}
\big|\det \big(d\Phi\big)\big|=\kappa \sqrt{\slashed{g}}.
\end{equation}
For a given $k$ with $\mathcal{K}_k\subset \mathcal{D}_{\delta}$, since $t\geqslant 2^{-k}$ in the region $\mathcal{K}_k$, we have {\color{black}$2^{-k} \lesssim \kappa \lesssim 1$.} Therefore, by \eqref{eq: spacetime rough bound on Ddelta} and \eqref{eq: jacobian}, we have
{\color{black}
\[\int_{\mathcal{K}_k}|\Zr^\alpha(\psi)|^2 dtdx_1dx_2 =  \int_{2^{-k}}^{t^*}\!\!\int_{0}^{u}\!\!\!\int_{0}^{2\pi}\!\!\! |\Zr^\alpha(\psi)|^2  \kappa\sqrt{\slashed{g}}dtdud\vartheta \lesssim 1.\]
}
We then use $\partial_1$ and $\partial_2$ to replace the $\Zr$'s in the above inequality. Indeed, because $\Tr=-t\partial_1$ and $2^{-k} \lesssim t$ in $\mathcal{K}_k$, for all $\psi \in \{w,\wb,\psi_2\}$ and multi-indices $\alpha$ with $ 0\leqslant |\alpha| \leqslant \Ntop+1$, we obtain that
\[
\int_{\mathcal{K}_k}|\partial^\alpha(\psi)|^2 dtdx_1dx_2 \lesssim {\color{black}2^{2|\alpha|k}},
\]
where $\partial \in \{\partial_1,\partial_2\}$. In view of the definition of $w$ and $\wb$, we conclude that
\begin{equation}\label{eq:bounds in Kk}
\int_{\mathcal{K}_k}|\partial^\alpha(v)|^2+|\partial^\alpha(c)|^2 dtdx_1dx_2 \lesssim {\color{black}2^{2|\alpha|k}},
\end{equation}
for all multi-indices $\alpha$ with $ 0\leqslant |\alpha| \leqslant \Ntop+1$.
\begin{remark}
We have also derived pointwise bounds on  $\Zr^\alpha(\psi)$ for $\psi \in \{w,\wb,\psi_2\}$ and multi-index $\alpha$ with $ 0\leqslant |\alpha| \leqslant \Ninf$ in \cite{LuoYu1}. Therefore, we can repeat the above argument exactly in the same manner and we obtain the following pointwise bounds on $v$ and $c$
\begin{equation}\label{eq:infty bounds in Kk}
\|\partial^\alpha(v)\|_{L^\infty(\mathcal{K}_k)}+\|\partial^\alpha(c)\|_{L^\infty(\mathcal{K}_k)}   \lesssim {\color{black}2^{2|\alpha|k}}.
\end{equation}
\end{remark}

We now turn to the bounds on $\partial_t$ derivatives of $v$ and $c$. 
Using the Euler equations, the time derivatives of $v$ and $c$ can be expressed in their spatial derivatives:
\begin{equation}\label{Euler in v and c another form}
\begin{cases}
\partial_t c &=-v \cdot \nabla  c-\frac{\gamma-1}{2}c\nabla \cdot v,\def\E{{\mathcal E}}
\\
 \partial_t v  &= -v \cdot \nabla v-\frac{2}{\gamma-1}c \nabla c.
\end{cases}
\end{equation} 
Using \eqref{Euler in v and c another form} to compute $\partial_t^{\alpha}\partial^\beta_1\partial_2^\gamma \psi$ for  $\psi \in \{v,c\}$ and multi-indices $\alpha,\beta$ and $\gamma$ with $0\leqslant |\alpha|+|\beta|+|\gamma| \leqslant \Ntop$
we have following rough estimates for $(v,c)$ in $\mathcal{K}_k$:
\begin{proposition}
For all $k\geqslant 1$ and for all $\delta<2^{-k}$, for the solution $(v,c)$ constructed in $\mathcal{D}(\delta)$, 
we have the following spacetime $L^2$-bounds:
\begin{equation}\label{eq:final bounds in Kk}
\big(\sum_{|\alpha|\leqslant \Ntop+1}\int_{\mathcal{K}_k}|\partial^\alpha(c)|^2 +|\partial^\alpha(v)|^2dtdx_1dx_2\big)^{\frac{1}{2}} \lesssim {\color{black}2^{k(\Ntop+1)}},
\end{equation}
where $\partial \in \{\partial_t,\partial_1,\partial_2\}$. \end{proposition}

\subsection{Convergence argument and existence}\label{section: convergence and existence}
We define a sequence of parameters $\{\delta_l\}_{l\geqslant 1}$, where $\delta_1=\frac{1}{2^{10}}$ and $\delta_{l+1}=\frac{1}{2}\delta_l$. For a given $l$,  we have already proved the existence of the initial data on $\Sigma_{\delta_k}$ for $0\leqslant u \leqslant u^*$ and we can solve the Euler equations to obtain $(v^{(l)},c^{(l)})$ on $\mathcal{D}(\delta_l)$. We use an upper index with parenthesis, i.e., $\,^{(l)}$, to indicate a quantity constructed with respect to the given initial data on  $\Sigma_{\delta_l}$. We also recall that, in the previous subsection, we have constructed the compact regions $\mathcal{K}_k$ with $\mathcal{K}_k\subset \mathcal{D}(\delta_l)$ for all $l\geqslant k$. 

We now use a limiting argument combined with the diagonal process to construct solutions on $\mathcal{W}$.
\begin{itemize}
\item For $k=1$, we consider the restrictions of all $(v^{(l)},c^{(l)})$ on $\mathcal{K}_1$. By \eqref{eq:final bounds in Kk}, we have
\[
\sum_{|\alpha|\leqslant \Ntop+1}\int_{\mathcal{K}_1}|\partial^\alpha\big(c^{(l)}\big)|^2 +|\partial^\alpha\big(v^{(l)}\big)|^2dtdx_1dx_2 \lesssim 2^{2(\Ntop+1)},
\]
i.e., $\big\|\big(c^{(l)},v^{(l)}\big)\big\|_{H^{\Ntop+1}(\mathcal{K}_1)}\lesssim 2^{2(\Ntop+1)}$. Hence, we obtain a uniformly bounded sequence of (vector valued) functions in $H^{\Ntop+1}(\mathcal{K}_1)$. Accroding to the Rellich–Kondrachov theorem, the embedding $H^{\Ntop+1}(\mathcal{K}_1) \hookrightarrow H^{\Ntop}(\mathcal{K}_1)$ is compact. We then can extract a subsequence $\big\{\big(c^{(l_i)},v^{(l_i)}\big)\big\}_{i\geqslant 1}$ so that it converges weakly in $H^{\Ntop+1}(\mathcal{K}_1)$ and strongly in $H^{\Ntop}(\mathcal{K}_1)$. Since $\mathcal{K}_1\subset \mathcal{D}_{\delta_{l_i}}$ for all $i\geqslant k$, we may relabel $\big\{\big(c^{(l_i)},v^{(l_i)}\big)\big\}_{i\geqslant 1}$ as $\big\{\big(c^{(l)},v^{(l)}\big)\big\}_{i\geqslant 1}$.

\item For  $k=2$, we repeat the above process so that we obtain a sequence of solutions (still use $l$ as labels) $\big\{\big(c^{(l)},v^{(l)}\big)\big\}_{i\geqslant 1}$ so that it converges weakly in $H^{\Ntop+1}(\mathcal{K}_2)$ and strongly in $H^{\Ntop}(\mathcal{K}_2)$. 

\item We now apply the standard diagonal process and we repeat the above operations at each step. This yields a sequence of solutions $\big\{\big(c^{(l)},v^{(l)}\big)\big\}_{l\geqslant 1}$ so that for all $k\geqslant 1$, it converges (eventually) weakly in $H^{\Ntop+1}(\mathcal{K}_k)$ and strongly in $H^{\Ntop}(\mathcal{K}_k)$. 
\end{itemize}
Since $\mathcal{K}_k\subset \mathcal{K}_{k+1}$ for all $k$ and $\bigcup_{k\geqslant 1}\mathcal{K}_k=\mathcal{W}$, the sequence $\big\{\big(c^{(l)},v^{(l)}\big)\big\}_{l\geqslant 1}$ defines the limit functions $(v,c)$ on $\mathcal{W}$ and $\big\{\big(c^{(l)},v^{(l)}\big)\big\}_{l\geqslant 1}$ converges to $(v,c)$ in the $H^{\Ntop}$-topology on any compact set of $\mathcal{W}$. Moreover, {\color{black}recalling the notation $\Ninf:=\Ntop-1$ from the first paper \cite{LuoYu1}} and using Sobolev embeddings for $H^{\Ntop}$-functions into the $C^{\Ninf}$-functions, $(v,c)$ is a classical solution to the Euler equations defined on $\mathcal{W}$. This proves the existence part of {\bf Theorem 2}.
\begin{remark}[Convergence of $\Th$ and $\kappa$] \label{remark: existence of Th}We consider the solution defined on $\mathcal{D}(\delta)$. The components $\Th^i$ of $\Th$ and $\kappa$ satisfy the transport equations
\[\begin{cases}
L(\widehat{T}^{i})&= \big(\widehat{T}^j\cdot \Xh(\psi_j) + \Xh(c)\big)\Xh^i,\\
L(\kappa)&=-\frac{\gamma+1}{\gamma-1}Tc+c^{-1}\widehat{T}^i\cdot L (\psi_i)\kappa.
\end{cases}
\]
We can commute $Z^\alpha$ with the equation to derive
\[L\big(Z^\alpha(f)\big)= \sum_{|\beta|\leqslant |\alpha|}\mathbf{P}_{\beta}\cdot Z^\beta(f),\]
where $f\in \{\Th^1,\Th^2,\kappa\}$ and $\mathbf{P}_{\beta}$ is a polynomial in terms of unkowns from the following set:
\[\big\{Z^{\beta'}(\psi), Z^{\beta''}(\Th^i), Z^{\beta'''}(\kappa) \big| \psi\in\{\wb,w,\psi_2\}, |\beta'|\leqslant |\alpha|+1, |\beta''|\leqslant |\beta|\big\}.\]
Therefore, by integrating these equations from $\Sigma_\delta$ we can derive the bound on $f\in \{\Th^1,\Th^2,\kappa\}$ and their derivatives. We can then apply the convergence argument in this section to show that the sequence of $f$'s from $\mathcal{D}_{\delta_l}$ indeed gives a limit (from a subsequence). Therefore, we have $\widehat{T}^{1}$, $\widehat{T}^{2}$ and $\kappa$ as well defined (at least $C^{\Ntop-1}$-)functions on $\mathcal{W}$. Moreover, since the convergence keeps the pointwise estimates, they enjoy the following estimates on $\mathcal{W}$:
\begin{equation}\label{ineq: L infty bound for limiting Thi kappa}
\|\Th^1+1\|_{L^\infty(\Sigma_t)}+\|\Th^2\|_{L^\infty(\Sigma_t)}+\|\frac{\kappa}{t}-1\|_{L^\infty(\Sigma_t)}\lesssim \varepsilon t.
\end{equation}
The proof of the above statements is routine and omitted.
\end{remark}

\subsection{Description of the solution}\label{section: description of the solution}

We have proved in Section 4.3.2 of \cite{LuoYu1} that, for all $\delta_l>0$, for the solution $\big(c^{(l)},v^{(l)}\big)$ constructed on $\mathcal{D}(\delta_l)$, for all multi-indices $\alpha$ with $|\alpha|\leqslant \Ninf$ and for all $t\in[\delta_l,t^*]$, for all $\psi^{(l)} \in \{w^{(l)},\wb^{(l)},\psi^{(l)}_2\}$, except for the case $\Zr^\alpha (\psi^{(l)})=\Tr(\wb^{(l)})$, we have
	\begin{equation*}\label{ineq: L infty bound}
		\|\mathring{Z}^\alpha(\psi^{(l)})\|_{L^\infty(\Sigma_t)}\lesssim  \begin{cases}\varepsilon, \ \ & {\color{black}\text{if}~\Zr^\alpha=\Xr^{|\alpha|};}\\
			\varepsilon t, \ \ & \text{otherwise}.
		\end{cases}
\end{equation*}
 By Lemma 4.2 of the first paper \cite{LuoYu1}, we also have
\[\|\Tr(\wb^{(l)})+\frac{2}{\gamma-1}\|_{L^\infty(\Sigma_t)}\lesssim \varepsilon t.\]
Since $\big\{\big(c^{(l)},v^{(l)}\big)\big\}_{l\geqslant 1}$ converges to $(v,c)$ in the $H^{\Ntop}$-topology, by the continuity of the Sobolev embedding theorem, we obtain the following estimates for the solution $(v,c)$ defined on $\mathcal{W}$:
\begin{proposition}\label{prop: L infty bound for limiting solution}
For all multi-indices $\alpha$ with $|\alpha|\leqslant \Ninf$, for all $\psi \in \{w ,\wb ,\psi _2\}$, for all $t\in (0,t^*]$, except for the case $\Zr^\alpha (\psi)=\Tr(\wb)$, we have
	\begin{equation}\label{ineq: L infty bound for limiting solution}
		\|\mathring{Z}^\alpha(\psi)\|_{L^\infty(\Sigma_t)}\lesssim  \begin{cases} \varepsilon, \ \ & {\color{black}\text{if}~\Zr^\alpha=\Xr^{|\alpha|};}\\
			\varepsilon t, \ \ & \text{otherwise}.
		\end{cases}
\end{equation}
and
\begin{equation}\label{ineq: L infty bound for limiting solution wb}
\|\Tr(\wb)+\frac{2}{\gamma-1}\|_{L^\infty(\Sigma_t)}\lesssim \varepsilon t.
\end{equation}
\end{proposition}

In view of Remark \ref{remark: existence of Th}, we also define $\Th=\Th^1 \partial_1+\Th^2\partial_2$ on $\mathcal{W}$. Therefore, we can define $T=\kappa\cdot \Th$ on $\mathcal{W}$. {\color{black}Using the definition of $T$, we define $u\equiv 0$ on $C_0$ and solve $Tu=1$ to define $u$ on $\mathcal{W}$. This provides a characteristic foliations $C_u$ on $\mathcal{W}$.} Since $C_0$ is in the $O(\varepsilon t)$-neighborhood of $C^{\rm cst}_0$ where $C^{\rm cst}_0$ is the counterpart of $C_0$ for the constant states $(\mathring{v}_r,\mathring{c}_r)$ (see Section \ref{section:region of convergence}), by \eqref{ineq: L infty bound for limiting Thi kappa} and the property of continuous dependence on the initial data for ODEs, we obtain that for each $t>0$, $u$ is $O(\varepsilon)$-close to $u^{\rm cst}=k^{\rm cst}-\frac{x_1}{t}$. This shows that the rarefaction wave fronts $C_u$ are in the $O(\varepsilon t)$-neighborhood  of its counterparts $C^{\rm cst}_u$ for the constant states.

\section{Applications to the stability of Riemann problem}
 
To apply the constructions in {\bf Corollary 1} and {\bf Theorem 2} to the Riemann problem, we have to derive uniform estimates on $L^n(\psi)$ for sufficiently large $n$ and for $\psi \in \{\wb,w,\psi_2\}$.

\subsection{The rough bounds on $L$ derivatives}\label{Section:The rough bounds on L derivatives}

We consider the solution constructed on $\mathcal{D}(\delta)$. We use the diagonalized Euler equations \eqref{eq:Euler in diagonalized form} to bound all the $L^k \Zr^\alpha (U^{(\lambda)})$'s where $\lambda \in \{0,-1,-2\}$. These bounds will suffer a loss in $t$. 


Recall that in Proposition 7.5 of \cite{LuoYu1}, we have derived the following estimates: for all multi-indices $\alpha$ with $|\alpha|\leqslant \Ninf$, for all $t\in [\delta,t^*]$, for all $\psi,\psi'\in \{\wb,w,\psi_2\},  \Zr\in \{\Xr, \Tr\}$ except for the case $(\alpha, \psi')=(0,\wb)$, 
\begin{equation}\label{eq:basic L estimates with and without L}
\|L \Zr^\alpha \psi \|_{L^\infty(\Sigma_t)}+\|\Xh\Zr^\alpha \psi \|_{L^\infty(\Sigma_t)}+t^{-1}\|T\Zr^\alpha \psi' \|_{L^\infty(\Sigma_t)}\lesssim \varepsilon.
\end{equation}

\begin{lemma}
For any multi-index $\alpha$ with $|\alpha|\leqslant \Ninf - 1$, for all $t\in [\delta,t^*]$, for all $\psi,\psi'\in \{\wb,w,\psi_2\},  \Zr\in \{\Xr, \Tr\}$ except for the case $(\alpha, \psi')=(0,\wb)$, we have (independent of $\delta$)
\begin{align}\label{eq:basic L estimates}
	\|L Z^\alpha \psi \|_{L^\infty(\Sigma_t)}+\|\Xh Z^\alpha \psi \|_{L^\infty(\Sigma_t)}+t^{-1}\|T Z^\alpha \psi' \|_{L^\infty(\Sigma_t)}\lesssim \varepsilon.
\end{align}

\end{lemma}
\begin{proof}
	The proof is based on the transformation \eqref{eq: transformation between Xh T and Xr Tr} and the commutator estimates in Section 7.1 of \cite{LuoYu1}. We refer to Section 9 of \cite{LuoYu1} for details. 
\end{proof}

We can derive the following rough estimates on multiple $L$-derivatives:
\begin{corollary}
	For any positive integer $k\geqslant 1$ and all multi-indices $\alpha$ with $k+|\alpha|\leqslant \Ninf$, the following bounds hold on  $\mathcal{D}(\delta)$:
	\begin{equation}
		\|L^k Z^\alpha(\wb)  \|_{L^\infty(\Sigma_t)}\lesssim t^{-k+2}\varepsilon,  \ \ \|L^k   Z^\alpha(w)  \|_{L^\infty(\Sigma_t)} +\|L^k   Z^\alpha(\psi_2)  \|_{L^\infty(\Sigma_t)} \lesssim t^{-k+1}\varepsilon.
	\end{equation}
\end{corollary}
\begin{proof}
	These bounds are direct consequences of \eqref{eq:basic L estimates} and the equation \eqref{eq:Euler:L^n U in diagonalized form} which is valid on $\Sigma_t$. In fact we have $\|\mathbf{P}_{k,i}\|_{L^\infty(\Sigma_t)}\lesssim  t^{-k+1}\varepsilon$ for each term $\mathbf{P}_{k,i}$ in \eqref{eq:Euler:L^n U in diagonalized form}. We commute $Z^\alpha$ to \eqref{eq:Euler:L^n U in diagonalized form} to derive the bounds for arbitrary $\alpha$. 
	
	To obtain the improved bounds for $\wb$, we take the $0$-eigenvalue component of \eqref{eq:Euler:L^n U in diagonalized form}, namely,
	\begin{align*}L^{k+1}\big(U^{(0)}\big)=\sum_{k_1+k_2=k}L^{k_1}(c)  L^{k_2}(\Xh(U))+\big(L^{k_1}\big(\frac{\zeta}{\kappa}\big) +L^{k_1}(c) \big)L^{k_2}U,
	\end{align*}
	and make use of the following computations
	\begin{equation}\label{equation: recall of zeta over kappa}
		\frac{\zeta}{\kappa}=-\frac{\gamma+1}{2}\Xh(U^{(0)})-\frac{\gamma-3}{2}\Xh(U^{(-2)})+\theta \cdot U^{(-1)}, \ \theta= \Th^2\Xh(\Th^1)-\Th^1\Xh(\Th^2).
	\end{equation}
	Since these bounds are not used in other parts of the paper, we omit the details.
\end{proof}

\subsection{Retrieving the uniform bounds on $L$-derivatives}\label{Section:The precise bounds on L derivatives}

\subsubsection{An ODE systems for $L^k(\psi)$}

 The main idea in this section is to use the acoustical wave equations for $\psi\in \{\wb,w,\psi_2\}$ in the null frame $(\Lb,L,\Xh)$. We will integrate along the $\Lb$ direction to retrieve the uniform bounds on $L^k(\psi)$ from $C_0$ where  $L^k(\psi)$ are uniformly bounded from the continuity across $C_0$. 
 
Formally, we consider an acoustical wave equation for a smooth function $\psi$ with source term $\varrho$, i.e.,  $ \Box_{g} (\psi) =\varrho$. It can be decomposed in the null frame $(L,\Lb,\Xh)$ as follows:
\[
 \Box_{g} (\psi) = \Xh^2 (\psi) - \mu^{-1}L\big(\underline{L}(\psi)\big) - \mu^{-1}\big(\frac{1}{2}\chi\cdot \underline{L}(\psi) +\frac{1}{2}\chib\cdot L(\psi)\big)- 2 \mu^{-1}\zeta \cdot \Xh(\psi)=\varrho.
\] 
Since $ [L,\Lb]=- 2(\zeta + \eta)\Xh+L(c^{-1}\kappa)L$, we can rewrite the equation as
\begin{equation*}
\Lb (L \psi)+\big(\frac{1}{2}\chib +L(c^{-1}\kappa)\big)L\psi=-\mu\rho+\mu\Xh^2(\psi)-\frac{1}{2}\chi\Lb(\psi)+2\eta\Xh(\psi).
\end{equation*}
By  $\Lb=2T+c^{-1}\kappa L$ and $\chib=-2c^{-1}\kappa\Xh^j\Xh(\psi_j)-c^{-1}\kappa \chi$, we move the terms involving $L\psi$ to the lefthand side and we obtain
\begin{equation}\label{eq: Lb L equation aux 1}
\Lb (L \psi)+\big( L(c^{-1}\kappa)-c^{-1}\kappa\Xh^j\Xh(\psi_j)\big)L\psi=-\mu\rho+\mu\Xh^2(\psi)-\chi T(\psi)+2\eta\Xh(\psi).
\end{equation}
In applications, we will take $\psi\in \{\wb,w,\psi_2\}$ so that the corresponding source terms $\varrho$ are given by \eqref{Main Wave equation: order 0}. In these circumstances, $\varrho$ can be written as a linear combination of the following terms:
\[\big\{c^{-1}g(D f_1,D f_2)\big| f_1,f_2\in \{\wb,w,\psi_2 \}\big\}.\]
We compute that
\[\mu g(D f_1,D f_2)=-\frac{1}{2}\Lb(f_1) L(f_2)-\frac{1}{2}\Lb(f_2)L(f_1)+ \mu\Xh(f_1)\Xh(f_2).\]
We then move all the terms involving $L\psi$ to the lefthand side of the equation \eqref{eq: Lb L equation aux 1}. Ignoring the irrelevant coefficients, this leads to the following type of schematic expression:
\begin{align*}
&\Lb (L \psi)+\big(L(c^{-1}\kappa)-c^{-1}\kappa\Xh^j\Xh(\psi_j)\big)L\psi+\sum_{\psi',\psi''\in \{\wb,w,\psi_2\}}c^{-1}\Lb(\psi'')L(\psi')\\
=& \sum_{\psi',\psi''\in \{\wb,w,\psi_2\}}c^{-1}\mu\Xh(\psi')\Xh(\psi'')+\mu\Xh^2(\psi)+\chi T(\psi)+\eta\Xh(\psi).
\end{align*}
If one regards $\psi$ as the vector valued function $(\wb,w,\psi_2)$, $\psi$ then satisfies the following schematic equation:
\begin{equation*}
\Lb (L \psi)+\big(c^{-1}\kappa\Xh^j\Xh(\psi_j) +L(c^{-1}\kappa)+c^{-1}\Lb(\psi)\big)L\psi= c^{-1}\mu\Xh(\psi)\Xh(\psi)+\mu\Xh^2(\psi)+\chi T(\psi)+\eta\Xh(\psi).
\end{equation*}
We then use  $\Lb(\psi)=c^{-1}\kappa L\psi+2T\psi$, $L(c^{-1}\kappa)=-c^{-2}\kappa Lc+ c^{-1}m'+c^{-1}\kappa e'$, \eqref{Euler equations:form 1}, \eqref{eq: m' e'}, \eqref{defining eq of theta and chi} and \eqref{structure quantities: connection coefficients} to conclude that
\begin{equation}\label{eq: Lb L equation order 0}
\Lb (L \psi)+\mathbf{P}_0 \cdot L\psi+\mathbf{P}_0 \cdot (L\psi)^2= \mathbf{P_0}.
\end{equation}
Here for $n \geq 0$, $\mathbf{P}_n$'s are  polynomials with $\mathbb{R}$-coefficients in the unknowns from the following set
\[\mathbf{V}_n:=\big\{c^{-1}, L^kZ^\alpha(f_1),Z^\beta(f_2)\big|f_1\in \{\wb,w,\psi_2\}, f_2\in \{\kappa,\Th^1,\Th^2\}, Z\in \{T,\Xh\}, k \leqslant n, k+|\alpha|\leqslant n+2, |\beta|\leqslant 1\big\}.\]
We take the $L$-derivative of \eqref{eq: Lb L equation order 0} to derive the following schematic formula:
\begin{equation}\label{eq: Lb L equation order 1}
\Lb (L^2 \psi)+\mathbf{P}_1 \cdot L^2\psi = \mathbf{P_1}.
\end{equation}
In fact, by applying $L$ to \eqref{eq: Lb L equation order 0}, we have
\begin{align*}
\Lb (L^2 \psi)+[L,\Lb]L\psi+\mathbf{P}_0 \cdot L^2\psi + L(\mathbf{P}_0) \cdot (L\psi) +\mathbf{P}_0 \cdot L\psi \cdot L^2\psi+L(\mathbf{P}_0) \cdot (L\psi)^2= L(\mathbf{P_0}).
\end{align*}
By definition, $\mathbf{P}_0$ and $\mathbf{P}_0 \cdot L\psi$ can be written as $\mathbf{P}_1$. By $[L,\Lb]=- 2(\zeta + \eta)\Xh+L(c^{-1}\kappa)L$, we see that $[L,\Lb](L\psi)=\mathbf{P}_1\cdot L^2\psi+\mathbf{P}_1$. Thus, we can rewrite the above equation as
\begin{align*}
\Lb (L^2 \psi)+\mathbf{P}_1 \cdot L^2\psi= L(\mathbf{P_0}) + L(\mathbf{P}_0) \cdot (L\psi) +L(\mathbf{P}_0) \cdot (L\psi)^2+\mathbf{P}_1.
\end{align*}
In view of the following schematic formula, we have proved \eqref{eq: Lb L equation order 1}.
\begin{lemma}
	For $n \geq 1$ we have the following schematic formula
	\begin{equation}\label{eq:LP_n}
		L(\mathbf{P}_{n-1}) = \mathbf{P}_{n}.
	\end{equation}
\end{lemma}

\begin{proof}
	Since $\mathbf{P}_{n-1}$ is an  polynomials with variables from the set $\mathbf{V}_{n-1}$, it suffices to compute  $L(c^{-1})$, $L\big(L^k\big(Z^\alpha(f_1)\big)\big)$ and $L\big(Z^\beta(f_2)\big)$ where $f_1\in \{\wb,w,\psi_2\}$, $f_2\in \{\kappa,\Th^1,\Th^2\}$, $Z\in \{T,\Xh\}$ and $k\leqslant n-1$. It is straightforward that $L(c^{-1})$ and $L\big(L^k\big(Z^\alpha(f_1)\big)\big)$ can be represented as $\mathbf{P}_{n}$. 
	
	We now compute $L\big(Z^\beta(f_2)\big)$ with $|\beta|\leqslant n+1$. First of all, since $[L,\Xh]=-\chi\cdot \Xh$ and $[L,T]=- (\zeta + \eta)\Xh$, according to the commutation formula $
	[L,Z^\beta] =\sum_{|\beta_1|+|\beta_2|=|\beta|-1}Z^{\beta_1}[L,Z] Z^{\beta_2}$,
	we obtain that
	\begin{align*}
		L\big(Z^\beta(f_2)\big)=Z^\beta\big(L(f_2)\big)+[L,Z^\beta](f_2)=Z^\beta\big(L(f_2)\big)+\mathbf{P}_{n}.
	\end{align*}
	We the use \eqref{structure eq 1: L kappa} and \eqref{structure eq 3: L T on Ti Xi Li} to replace $L(f_2)$ on the righthand side.  Because $L(f_2)$ can be written as a polynomial in  the variables from the following set
	\[\big\{c^{-1}, Z^\alpha(f_1), f_2\big|f\in \{\wb,w,\psi_2\}, f_2\in \{\kappa,\Th^1,\Th^2\}, Z\in \{T,\Xh\}, |\alpha|\leqslant 1\big\},\]
	it is straightforward to see that $Z^\beta\big(L(f_2)\big)=\mathbf{P}_{n}$. Hence, $
	L\big(Z^\beta(f_2)\big)= \mathbf{P}_{n}$ and we are done.
\end{proof}

\begin{lemma}\label{lemma:5.3}
For all $n\geqslant 2$, we have the following schematic formula:
\begin{equation}\label{eq: Lb L equation order n}
\Lb (L^{n} \psi)+\mathbf{P}_{1} \cdot L^n\psi = \mathbf{P}_{n-1}.
\end{equation}
\end{lemma}
\begin{proof}
We use an induction argument on $n$. We have already proved the case for $n=2$. Assume that \eqref{eq: Lb L equation order n} holds for $n$. To obtain the case for $n+1$, we take the $L$-derivative of \eqref{eq: Lb L equation order n} to derive
\[\Lb (L^{n+1} \psi)+\mathbf{P}_{1} \cdot L^{n+1}\psi+[L,\Lb]\big(L^{n} \psi\big)+L(\mathbf{P}_{1})\cdot L^{n}\psi=L(\mathbf{P}_{n-1}).\]
Since $[L,\Lb]=- 2(\zeta + \eta)\Xh+L(c^{-1}\kappa)L$, we have $[L,\Lb](L^n\psi)=\mathbf{P}_1\cdot L^{n+1}\psi+\mathbf{P}_n$. {\color{black}By writing $L\mathbf{P}_{n-1}$ as $\mathbf{P}_{n}$, the above analysis closes the induction argument hence the proof of the lemma.} 
%
\end{proof}


  
 \begin{corollary}
For all $n\geqslant 2$ and for all multi-indices $\alpha_0$, we have the following schematic formula:
\begin{equation}\label{eq: Lb L equation order n Z alpha}
\Lb (L^{n} Z^{\alpha_0}\psi)+ \mathbf{Q}_{1} \sum_{{\color{black}|\alpha_1|\leqslant |\alpha_0|}}L^n Z^{\alpha_1} \psi=\mathbf{Q}_{n-1}.
\end{equation}
The function $\mathbf{Q}_{n-1}$ is a $\mathbb{R}$-coefficients polynomials with unknowns from the following set
\[ \mathbf{\widetilde{V}}_{n-1} = \big\{c^{-1}, L^kZ^\alpha(f_1),Z^\beta(f_2),\frac{Z^\gamma(c^{-1}\kappa)}{c^{-1}\kappa}\big|f_1\in \{\wb,w,\psi_2\}, f_2\in \{\kappa,\Th^1,\Th^2\}, Z\in \{T,\Xh\}, k \leqslant n-1\big\}.\]
\end{corollary}
 \begin{proof}
We use an induction argument on $|\alpha_0|$. The case for $|\alpha_0|=0$ and arbitrary $n$ has been already proved  by the above lemma. 
We make the following the induction hypothesis: the identity \eqref{eq: Lb L equation order n Z alpha} holds for all multi-indices $\alpha$ with $|\alpha|\leqslant m-1$.  It suffices to show that,  for an arbitrary $\alpha_0$ with $|\alpha_0|=m$, for all $Z_0\in \{T,\Xh\}$ and $n\geqslant 1$, we have
\begin{equation}\label{eq: Lb L equation order n Z alpha aux 1}
\Lb (L^{n} Z_0Z^{\alpha_0}\psi)+\mathbf{Q}_{1} \sum_{|\alpha_1|\leqslant |\alpha_0|+1}L^n Z^{\alpha_1} \psi=\mathbf{Q}_{n-1}.
\end{equation}
We apply $Z_0$-derivative to \eqref{eq: Lb L equation order n Z alpha} and we obtain
\begin{equation}\label{eq: Lb L equation order n Z alpha aux 2}
\begin{split}
\Lb (L^{n} Z_0 Z^{\alpha_0}\psi) &+\underbrace{\mathbf{Q}_1 \sum_{|\alpha_1|\leqslant|\alpha|} L^{n} Z_0 Z^{\alpha_1}\psi}_{A_0}=\underbrace{[Z_0,\Lb] (L^{n} Z^{\alpha_0}\psi)}_{A_1}+\underbrace{ \Lb ([Z_0,L^{n}] Z^{\alpha_0}\psi)}_{A_2}\\
&+\underbrace{Z_0(\mathbf{Q}_{1}) \sum_{|\alpha_1|\leqslant|\alpha|}L^n Z^{\alpha_1} \psi}_{A_3}+\underbrace{\mathbf{Q}_1\sum_{|\alpha_1|\leqslant|\alpha|} [L^{n}, Z_0] Z^{\alpha_1}\psi}_{A_4} +\underbrace{ Z_0( \mathbf{Q}_{n-1}) }_{A_{5}}.
\end{split}
\end{equation}
We will use the following two schematic formulas:
\begin{itemize}

\item[a)] For all $\psi \in \{\wb,w,\psi_2\}$, 
we have 
\[[L^n,Z] (Z^\beta\psi) =\mathbf{Q}_{n-1}.\]

To prove the above commutator formulas, we rewrite the commutators $[L,\Xh]=-\chi\cdot \Xh$ and $[L,T]=- (\zeta + \eta)\Xh$ as $[L,T]=\mathbf{Q}_0\cdot \Xh$. For all  $\psi\in \{\wb,w,\psi_2\}$ and all multi-indices $\beta$ with $|\beta|\leqslant |\alpha|$, we  consider the following schematic commutator formula:
\begin{align*}
[L^n,Z] (Z^\beta\psi)&=\sum_{n_1+n_2=n-1}L^{n_1}\big([L,Z] L^{n_2}(Z^\beta\psi)\big) =\sum_{n_1+n_2=n-1}L^{n_1}\big(\mathbf{Q}_0 \cdot \Xh L^{n_2}(Z^\beta\psi)\big) \\
&=\sum_{n_1+n_2=n-1}L^{n_1}\big(\mathbf{Q}_0 \cdot \mathbf{Q}_{n_2}\big) =\mathbf{Q}_{n-1}.
\end{align*}
In the last step,  we have used $L(\mathbf{P}_{k-1})=\mathbf{P}_k$ in \eqref{eq:LP_n}.

\item[b)] For all $Z\in \{\Xh,T\}$, we have $Z(\mathbf{Q}_{n-1})=\mathbf{Q}_{n-1}$. 

Since $\mathbf{Q}_{n-1}$ is a linear combination of monomials with factors from $\mathbf{\widetilde{V}}_{n-1}$, by Leibniz rule, it suffices to check for each single term from  the set $\mathbf{\widetilde{V}}_{n-1}$.  We have the following three cases:
\begin{itemize}
\item 
Consider  $Z(c^{-1})$ and $Z\big(L^k\big(Z^\alpha(f_1)\big)\big), f_1\in \{\wb,w,\psi_2\}$. It is obvious that $Z(c^{-1})=\mathbf{Q}_{n-1}$. 
According to the above commutator formula $[L^n,Z] (Z^\beta(\psi))=\mathbf{Q}_{n-1}$, $Z\big(L^k\big(Z^\alpha(f_1)\big)\big)$ can also be represented as $\mathbf{Q}_{n-1}$. 
\item  For $Z^{\beta}(f_2), f_2\in \{\kappa,\Th^1,\Th^2\}$, it is obvious that $ZZ^{\beta}(f_2) = \mathbf{Q}_{n-1}$.
\item  For $f_3 = \frac{Z^\gamma(c^{-1}\kappa)}{c^{-1}\kappa}$, it is clear that $Z(f_3)=\mathbf{Q}_{n-1}$.
\end{itemize}

\end{itemize}

We now deal with the $A_i$'s in \eqref{eq: Lb L equation order n Z alpha aux 2} one by one. By a) and b), it is obvious that $A_3,A_4$ and $A_5$ are of type $\mathbf{Q}_n$. It remains to calculate $A_1$ and $A_2$.

For $A_1$, we first study the following commutator formulas
\[  [\Lb,\Xh]=-\chib\cdot\Xh-\Xh(c^{-2}\mu)L, \ \ [\Lb,T]=-c^{-1}\kappa(\zeta+\eta)\cdot\Xh-T(c^{-1}\kappa)L.\]
By expressing $L$ in terms of $\Lb$ and $T$, they can be all written in the following schematic form:
\[[\Lb,Z]=\mathbf{Q}_1\cdot \Xh +2\frac{Z(c^{-1}\kappa)}{c^{-1}\kappa}T  -\frac{Z(c^{-1}\kappa)}{c^{-1}\kappa}\Lb=\mathbf{Q}_1\cdot \Xh +\mathbf{Q}_1\cdot T  +\mathbf{Q}_1\cdot\Lb.\]
Therefore, by the formula in a), we have
\begin{align*}
A_1&=[Z_0,\Lb] (L^{n} Z^{\alpha_0}\psi)=\mathbf{Q}_1\big(\Xh L^{n} Z^{\alpha_0}\psi + TL^{n} Z^{\alpha_0}\psi+\Lb L^{n} Z^{\alpha_0}\psi\big)\\
&=\mathbf{Q}_1\big(L^{n} \Xh Z^{\alpha_0}\psi + L^{n} TZ^{\alpha_0}\psi+\mathbf{Q}_{n-1}+\Lb L^{n} Z^{\alpha_0}\psi\big).
\end{align*}
The first two terms can be classified into $A_0$. The third one is of type $\mathbf{Q}_{n-1}$. For the last term, we use the inductive hypothesis that
\[\Lb (L^{n} Z^{\alpha_0}\psi)=\mathbf{Q}_{1} \sum_{|\alpha_1|\leqslant\alpha_0}L^n Z^{\alpha_1} \psi+\mathbf{Q}_{n-1}.\]
Therefore, schematically, we have
\[A_1=A_0+\mathbf{Q}_{n-1}.\]

For $A_2$, we rewrite $[L,Z_0]$ as $[L,Z_0]=f\cdot \Xh$, where $f=-\chi$ for $Z_0=\Xh$ and $f=-(\zeta+\eta)$ for $Z_0=T$. We remark that $f$ can be expressed explicitly by \eqref{defining eq of theta and chi} and \eqref{structure quantities: connection coefficients}. Hence, 
\begin{align*}
[L^n,Z_0] (Z^{\alpha_0}\psi)&=\sum_{n_1+n_2=n-1}L^{n_1}\big(f\cdot \Xh L^{n_2}(Z^\beta\psi)\big) =\sum_{n_1+n_2+n_3=n-1}L^{n_1}(f)  \cdot L^{n_2} \Xh L^{n_3}Z^\beta(\psi).
\end{align*}
Therefore,
\begin{align*}
A_2&=\sum_{n_1+n_2+n_3=n-1}\Lb L^{n_1}(f)  \cdot L^{n_2} \Xh L^{n_3}Z^\beta(\psi)+L^{n_1}(f) \cdot \Lb L^{n_2} \Xh L^{n_3}Z^\beta(\psi).\end{align*}
By writing $\Lb=2T+c^{-1}\kappa L$ and the inductive hypothesis as well as the explicit expression of $f$ and the equations \eqref{structure eq 1: L kappa} \eqref{structure eq 3: L T on Ti Xi Li}, all the terms in $A_2$ are of type $\mathbf{Q}_{n-1}$.

In view of \eqref{eq: Lb L equation order n Z alpha aux 2}, the above analysis completes the inductive argument. Hence, we have proved the  corollary.
\end{proof}
\subsubsection{The uniform bounds on $L$-derivatives}

We now consider the effective domain in $\mathcal{D}^+(\delta)\subset \mathcal{D}(\delta)$, see Section \ref{section:effective domain} for definitions and pictures.
\begin{proposition}\label{prop: estimates on Lk}
For any positive integer $k\geqslant 1$ and all multi-indices $\alpha$ with $k+|\alpha|\leqslant \Ninf-1$, for all $\psi \in \{\wb,w,\psi_2\}$ and for all $Z\in \{\Xh, T\}$, we have
\begin{equation}\label{eq: Lk estimates  on effective domain}
\|L^k Z^\alpha (\psi) \|_{L^\infty(\mathcal{D}^+(\delta))} \lesssim \varepsilon.
\end{equation}
\end{proposition}
\begin{proof}
We prove inductively on the double indices $k$ and $\alpha$. For all $k=1$ and arbitrary $\alpha$, \eqref{eq: Lk estimates  on effective domain} hold on the entire $\mathcal{D}(\delta)$ hence on $\mathcal{D}^+(\delta)$. We make the inductive hypothesis that  the estimate \eqref{eq: Lk estimates  on effective domain} holds for all $k$ and $\alpha$ with $k<n$. We now prove that  \eqref{eq: Lk estimates  on effective domain} holds for $k=n$ and all $\alpha$ where $k+|\alpha|\leqslant \Ninf-1$.

We use \eqref{eq: Lb L equation order n Z alpha} for $n=k$ and $\alpha_0$ runs over for all $\alpha$ with $k+|\alpha|\leqslant \Ninf-1$. Therefore, schematically, we have the following system of equations for $L^{n} Z^{\alpha}\psi$ along integral curves of $\Lb$:
\[
\Lb (L^{k} Z^{\alpha}\psi)+ \mathbf{Q}_{1} \sum_{|\alpha'|\leqslant \alpha}L^k Z^{\alpha'}\psi =\mathbf{Q}_{k-1}.
\]
By the inductive hypothesis and the definition of $\mathbf{Q}_{k-1}$, we have $|\mathbf{Q}_{k-1}|\lesssim \varepsilon$ and $|\mathbf{Q}_{1}|\lesssim \varepsilon$. We then integrate the above equations.  By the Gronwall's inequality and the fact that $\|L^{k} Z^{\alpha}\psi\|_{L^\infty(C_0)}\lesssim \varepsilon$, this gives the desired estimates for all $L^{k} Z^{\alpha}\psi$. Therefore, the inductive argument is finished and we complete the proof of the proposition.
\end{proof}

We now give an estimate of the size of the irrelevant domain $\mathcal{D}^-(\delta)$. By definition, $\mathcal{D}^-(\delta)$ is the union of integral curves of $\Lb$ emanated from $\Sigma_\delta$. By $\Lb=2T+c^{-1}\kappa L$, we know that $\Lb(u)=2$ and we can use $u$ to parameterize such an integral curve. 
We now compute the maximal possible time for such a curve. Since $\Lb(t)=c^{-1}\kappa$, the maximal time is bounded above by
\[\int_0^{u^*}c^{-1}\kappa du'\lesssim \frac{1}{\mathring{c}_r}\int_0^{u^*}\kappa du'.\] 
On the other hand, we have $|\Lb(\kappa)(u,\vartheta,t)|\lesssim \kappa$. Therefore, since $|\kappa-\delta|\lesssim \varepsilon \delta$ on $\Sigma_{\delta}$, by regarding $\kappa$ as a function in $u$, we have
\[\kappa(u)\lesssim \delta+\int_0^{u^*} \Lb\kappa(u')du' \ \ \Rightarrow \ \ |\kappa(u)|\lesssim \delta. \]
Hence, there exists a universal constant $C_1$ depending only on $(\mathring{v}_r,\mathring{c}_r)$ so that the maximal time on $\mathcal{D}^-(\delta)$ is at most $t_\delta=C_1\cdot\delta$.
\begin{center}
\includegraphics[width=1.5in]{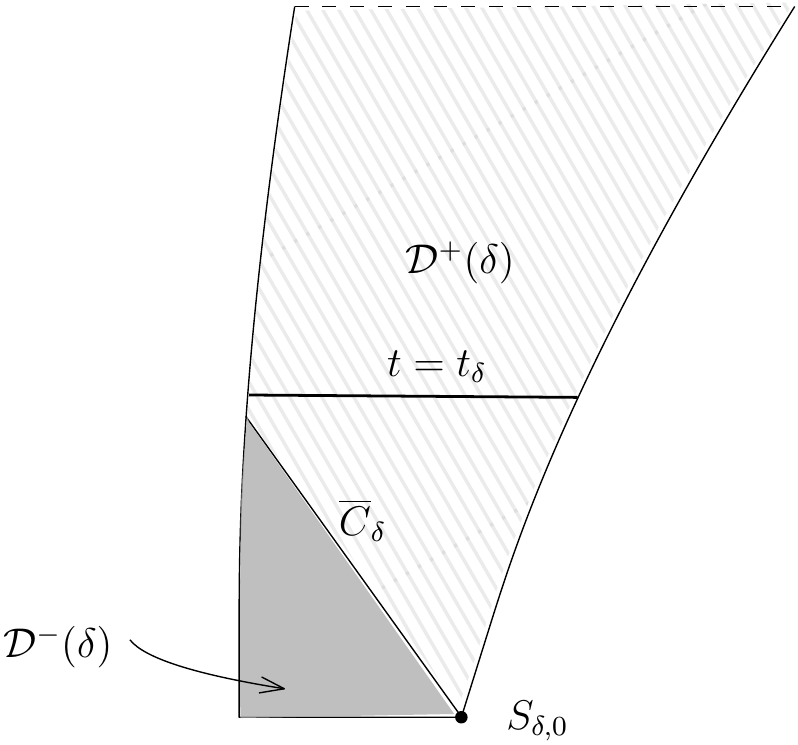}
\end{center}

Therefore, by the argument in Section \ref{section: convergence and existence}, the limit of  the effective domains $\mathcal{D}^+(\delta_l)$ also gives $\mathcal{W}$. By passing to the limit, Proposition \ref{prop: estimates on Lk}
gives the following proposition:
\begin{proposition}\label{prop: estimates on Lk final}
For the solution $(v,c)$ on $\mathcal{W}$, for any positive integer $k\geqslant 1$ and all multi-indices $\alpha$ with $k+|\alpha|\leqslant \Ninf-1$, for all $\psi \in \{\wb,w,\psi_2\}$ and for all $Z\in \{\Xh, T\}$, we have
\begin{equation}
\|L^k Z^\alpha (\psi) \|_{L^\infty(\mathcal{W})} \lesssim \varepsilon.
\end{equation}
\end{proposition}

{\color{black}

\subsection{The singular boundary and canonical construction of acoustical coordinates}\label{Section: app to Riemann 1}

\subsubsection{The construction using a coordinate system}
The solution $(\wb,w, \psi_2)$ or $(v,\rho)$ is now on $\mathcal{W}$. In view of the domains defined in Section \ref{section:region of convergence} and the definition of $u$ (see the discussion after Proposition \ref{prop: L infty bound for limiting solution}), there exists a $\overline{u}$ so that $0\leqslant u^*-\overline{u}\leqslant \varepsilon_0$ and for each $u \in [0,\overline{u}]$, the characteristic hypersurface $C_u$ is complete in the following sense: for all $t>0$, $C_u\cap \Sigma_t$ is a complete circle. We then define 
\[\overline{\mathcal{W}}=\bigcup_{u\in [0,\overline{u}]}C_u.\]
In particular, the functions $t$ and $u$ are defined on $\overline{\mathcal{W}}$. To construct the acoustical coordinates (not canonical at the moment),  we can take an arbitrary $\delta$ and use the constructions in Section \ref{section: definition of vartheta on Sigmadelta} to define $\vartheta$ on $\Sigma_\delta$ and then extend it to the entire $\overline{\mathcal{W}}$. Hence, the Riemann invariants $(\wb,w, \psi_2)$ can be viewed as functions defined on $(t,u,\vartheta)\in (0,t^*]\times [0,\overline{u}]\times [0,2\pi]$. We now extend $(\wb,w, \psi_2)$ to the singularity $\mathbf{S}_*$ or intuitively to $\Sigma_0$ which corresponds to the limiting initial hypersurface as $t \to 0$. We can take $(u,\vartheta)$ as the coordinate system on $\mathbf{S}_*$. We will show that the solution on $\mathbf{S}_*$ is in a much simpler form in $(u,\vartheta)$.

According to Proposition \ref{prop: L infty bound for limiting solution} and Proposition \ref{prop: estimates on Lk final}, for all $\psi \in \{\wb,w,\psi_2\}$ for all $(t,u,\vartheta)\in (0,t^*]\times [0,\overline{u}]\times [0,2\pi]$, for all multi-indices $\alpha, \beta$ and for all $k\geqslant 1$, if $|\alpha|+|\beta|+k\leqslant \Ninf-1$, we have
\[
\big|\big(L^k T^\alpha \Xh^\beta \psi\big)(t,u,\vartheta)\big|\lesssim \varepsilon.
\]
Hence, for all $t>t'>0$, if $|\alpha|+|\beta|+k\leqslant \Ninf-2$, we have
\[
\big|\big(L^k T^\alpha \Xh^\beta \psi\big)(t,u,\vartheta)-\big(L^k T^\alpha \Xh^\beta \psi\big)(t',u,\vartheta)\big|\leqslant \int_{t'}^{t}\big|\big(L^{k+1} T^\alpha \Xh^\beta \psi\big)(\tau,u,\vartheta)\big|d\tau \lesssim |t-t'|\varepsilon.
\]
Therefore, for any decreasing sequence $\{t_i\}_{i\geqslant 1}$ where $t_i\in(0,t^*)$ and $\lim_{i\rightarrow \infty} t_i=0$, the sequence of functions  $\big\{\big(\wb(t_i,u,\vartheta),w(t_i,u,\vartheta),\psi_2(t_i,u,\vartheta)\big)\big\}_{i\geqslant 1}$ is a Cauchy sequence in the space of $C^{\Ntop-2}$ functions defined for $(u,\vartheta)\in [0,\overline{u}] \times [0,2\pi]$. Therefore, by taking the limit as $i\rightarrow \infty$, we obtain the limiting data defined on the singularity:
\begin{equation}
\big(\wb(0,u,\vartheta),w(0,u,\vartheta),\psi_2(0,u,\vartheta)\big)=\lim_{t_i\rightarrow 0}\big(\wb(t_i,u,\vartheta),w(t_i,u,\vartheta),\psi_2(t_i,u,\vartheta)\big).
\end{equation}
Therefore, we have extended $ (\wb,w,\psi_2)$ to the entire region $[0,t^*]\times [0,\overline{u}]\times [0,2\pi]$ in such way that for all fixed $t\in [0,t^*]$, $\big(\wb(t ,u,\vartheta),w(t,u,\vartheta),\psi_2(t,u,\vartheta) \big)\in C^{\Ntop-2}\big([0,\overline{u}]\times [0,2\pi]\big)$. 

On the other hand, we have
\begin{align*}
&\lim_{t'\rightarrow 0}\frac{1}{t-t'}\big[\big(L^{k-1} T^\alpha \Xh^\beta \psi\big)(t,u,\vartheta)-\big(L^{k-1} T^\alpha \Xh^\beta \psi\big)(t',u,\vartheta)\big]\\
=& \lim_{t'\rightarrow 0}\frac{1}{t-t'}\int_{t'}^{t} \big(L^{k} T^\alpha \Xh^\beta \psi\big)(\tau,u,\vartheta) d\tau=\frac{1}{t}\int_{0}^{t} \big(L^{k} T^\alpha \Xh^\beta \psi\big)(\tau,u,\vartheta) d\tau.
\end{align*}
Since for $|\alpha|+|\beta|+k\leqslant \Ninf-2$, $\lim_{\tau \rightarrow 0}\big(L^{k} T^\alpha \Xh^\beta \psi\big)(\tau,u,\vartheta)$ exists,  we can further let $t\rightarrow 0$ in the above equation. This shows that $\big(\wb(t ,u,\vartheta),w(t,u,\vartheta),\psi_2(t,u,\vartheta) \big)\in C^{\Ntop-3}\big([0,t^*]\times [0,\overline{u}]\times [0,2\pi]\big)$. Hence,
\begin{equation*}
(\wb ,w ,\psi_2)\in C^{\Ntop-3} ([0,t^*]\times [0,\overline{u}]\times [0,2\pi] ) \cap C^0\big([0,t^*];C^{\Ntop-2} ( [0,\overline{u}]\times [0,2\pi])\big).
\end{equation*}

The above extension allows us to construct the canonical $\vartheta$ on $\overline{\mathcal{W}}$. In fact, for all $u\in [0,\overline{u}]$, the characteristic vector field  $L = \frac{\partial}{\partial t} +v-c\widehat{T}$ can be extended to $\mathbf{S}_*$. Therefore, it is a smooth vector field on $\overline{C_u}=C_u\cup \mathbf{S}_*$. We define $\vartheta = x_2$ on $\mathbf{S}_*$ and we extend it by $L$ to $\overline{C_u}$. This construction from the singularity defines the canonical $\vartheta$ on  $\overline{\mathcal{W}}$.

We now study the property of the limiting function $\big(\wb(0,u,\vartheta),w(0,u,\vartheta),\psi_2(0,u,\vartheta)\big)$. According to Proposition \ref{prop: L infty bound for limiting solution}, for all $\psi=w$ or $\psi _2$, we can let $t\rightarrow 0$ in  \eqref{ineq: L infty bound for limiting solution}. Since $\|T \psi \|_{L^\infty(\Sigma_t)}\lesssim \varepsilon t$, we obtain that
\[ T (\psi)(0,u,\vartheta)=\lim_{t\rightarrow 0}T(\psi)(t,u,\vartheta)=0.  \ \ \psi \in \{w,\psi_2\}\]
On the other hand, we have $T =\frac{\partial}{\partial u}$ and $X =\frac{\partial}{\partial \vartheta}$ on $\mathbf{S}_*$. Therefore, for all $u\in [0,\overline{u}]$, we obtain that, 
\[
w(u,\vartheta)=w(0,\theta)=w_r(0,\vartheta), \ \
\psi_2(u,\vartheta)=\psi_2(0,\theta)=-v^2_r(0,\vartheta).
\]
Similarly, by  Proposition \ref{prop: L infty bound for limiting solution}, we have $\|T^2 \wb \|_{L^\infty(\Sigma_t)}\lesssim \varepsilon t$. Hence,
\[ T^2(\wb)(0,u,\vartheta)=\lim_{t\rightarrow 0}T^2  (\wb)(t,u,\vartheta)=0.\]
This implies that $\partial_u \wb(u,\vartheta)=\partial_u \wb(0,\theta)=-\frac{2}{\gamma+1}$, and hence
\[\wb(u,\vartheta)= (\wb)_r(0,\vartheta)-\frac{2}{\gamma+1}u.
\]
\begin{proposition}\label{prop:data at the singularity}
On the singularity $\mathbf{S}_*$, we have the following limiting data:
\begin{equation}\label{eq:data at the singularity}
\begin{cases}
&\wb(u,\vartheta)= \wb_r(0,\vartheta)-\frac{2}{\gamma+1}u,\\
&w(u,\vartheta)=w_r(0,\vartheta),\\
&\psi_2(u,\vartheta)=-v^2_r(0,\vartheta),
\end{cases}
\end{equation}
where $\vartheta=x_2$.
\end{proposition}

\begin{remark}
In the above construction, we use a given acoustical coordinate system $(t,u,\vartheta)$ to describe the geometry (and the construction) of the singularity $\mathbf{S}_*$. In the rest of this subsection, we will provide a geometric construction and this will lead to a canonical acoustical coordinate system on the rarefaction wave region.
\end{remark}

\subsubsection{Geometric constructions}\label{section:geometric constructions}
 We consider the set $\mathscr{L}$ of past-pointed null geodesics in the rarefaction wave region  $\overline{\mathcal{W}}$ that end at the singularity $\mathbf{S}_*$.  We remark that such geodesics always exist: given any point $(t_0,u_0,\vartheta_0) \in \overline{\mathcal{W}}$, the curve $\gamma_{u_0,\vartheta_0}:t\mapsto (t_0-t,u_0,\vartheta_0)$ is such a null geodesic. 

We recall that 
\begin{align*}
	g
	&=-\mu (dt \otimes du+du \otimes dt)+\kappa^2du\otimes du+\slashed{g}(d\vartheta+\Xi du)\otimes(d\vartheta+\Xi du)
\end{align*}
where $\mu=c\kappa$. For  the canonical choice of $\vartheta$, since $\Xi$ vanishes at $\mathbf{S}_*$, we can rewrite $\Xi$ as
\[\Xi=\mu\cdot \widehat{\Xi}.\]
To study the null geodesics in $\mathscr{L}$, we follow the ideas of Christodoulou in the last section of \cite{ChristodoulouShockFormation} (more precisely, we follow the approach in \cite{ChristodoulouMiao} which studies the classical rather than the relativistic Euler equations). Since conformal changes of metrics neither affect the causality nor the set of null geodesics, we study $\mathscr{L}$ under the metric $\mu^{-1}g$ on the cotangent bundle of $\overline{\mathcal{W}}$ using the Hamiltonian formulation for geodesics.

Let $(t,u,\vartheta,p_t,p_u,p_\vartheta)$ be the canonical coordinate system on the cotangent bundle. The corresponding Hamiltonian function for geodesics in terms of $\mu^{-1}g$ can be computed as follows:
\begin{align*}
	H(t,u,\vartheta,p_t,p_u,p_\vartheta)&=\frac{1}{2}\mu \big(g^{-1}\big)^{\alpha\beta}p_\alpha p_\beta
	=-p_tp_u+\mu\left(-\frac{1}{2}c^{-2} p_t^2+\frac{1}{2}  \slashed{g}^{-1}p_\theta^2+\Xih p_tp_\vartheta\right).
\end{align*}
The canonical equation associated to the above Hamiltonian is given by
\begin{equation}\label{eq:Hamiltonian for geodesics}
	\begin{cases}
		\frac{dt}{d\tau}&=-p_u+\mu\left(- c^{-2}p_t+\Xih p_\vartheta\right),\\	
		\frac{du}{d\tau}&=-p_t,\\
		\frac{d\vartheta}{d\tau}&=  \mu\left(\slashed{g}^{-1}p_\vartheta +\Xih p_t\right),\\
		\frac{dp_t}{d\tau}&= -\frac{\partial \mu}{\partial t}\left(-\frac{1}{2}c^{-2} p_t^2+\frac{1}{2}  \slashed{g}^{-1}p_\vartheta^2+\Xih p_tp_\vartheta\right)-\mu\left(c^{-3} \frac{\partial c}{\partial t} p_t^2+\frac{1}{2}  \frac{\partial (\slashed{g}^{-1})}{\partial t}p_\vartheta^2+\frac{\partial\Xih}{\partial t} p_tp_\vartheta\right),\\
		\frac{dp_u}{d\tau}&= -\frac{\partial \mu}{\partial u}\left(-\frac{1}{2}c^{-2} p_t^2+\frac{1}{2}  \slashed{g}^{-1}p_\vartheta^2+\Xih p_tp_\vartheta\right)-\mu\left(c^{-3} \frac{\partial c}{\partial u} p_t^2+\frac{1}{2}  \frac{\partial (\slashed{g}^{-1})}{\partial u}p_\vartheta^2+\frac{\partial\Xih}{\partial u} p_tp_\vartheta\right),\\
		\frac{dp_\vartheta}{d\tau}&= -\frac{\partial \mu}{\partial \vartheta}\left(-\frac{1}{2}c^{-2} p_t^2+\frac{1}{2}  \slashed{g}^{-1}p_\vartheta^2+\Xih p_tp_\vartheta\right)-\mu\left(c^{-3} \frac{\partial c}{\partial \vartheta} p_t^2+\frac{1}{2}  \frac{\partial (\slashed{g}^{-1})}{\partial \vartheta}p_\vartheta^2+\frac{\partial\Xih}{\partial \vartheta} p_tp_\vartheta\right).
	\end{cases}
\end{equation}
For $\gamma \in \mathscr{L}$, we will regard it as a future-pointed null geodesics starting at a point $q_0\in \mathbf{S}_*$ or equivalently, we assume that $t(0)=0$. At the initial point $q_0\in \mathbf{S}_*$, in view of the expression of the Hamiltonian $H$, the null cone degenerates to two hyperplanes defined by
\[(p_t)_0 (p_u)_0=0,\]
where $(p_t)_0=p_t(0), (p_u)_0=p_u(0)$ and $(p_\vartheta)_0=p_\vartheta(0)$. Following \cite{ChristodoulouShockFormation} or \cite{ChristodoulouMiao}, we study the following trichotomy of null geodesics starting at $q_0\in \mathbf{S}_*$ (up to a scaling constant for the parametrization):
\begin{itemize}
	\item Outgoing null geodesics: $(p_t)_0=0, (p_u)_0=-1$;
	\item Incoming null geodesics: $(p_t)_0=-1, (p_u)=0$;
	\item Other null geodesics: $(p_t)_0=(p_u)_0=0$ and $|(p_\vartheta)_0|=1$.
\end{itemize}
The inverse density $\mu$ and all its $u$ or $\vartheta$ derivatives vanish at $\mathbf{S}_*$. This makes the geometry of the singular boundary in the rarefaction wave region very special. In contrast to the singular boundary studied in \cite{ChristodoulouShockFormation} or \cite{ChristodoulouMiao}, we have 
\begin{proposition}\label{prop:geometric constructions}
	All curves in $\mathscr{L}$ are outgoing null geodesics (up to the parametrization).
\end{proposition}
\begin{proof}
	Let $\gamma$ be an ingoing or other null geodesics. By definition, we have $p_u=0$ at $\tau =0$. We consider the following ansatz for solutions of the Hamiltonian systems \eqref{eq:Hamiltonian for geodesics}:
	\[ t(\tau)\equiv 0, \ \vartheta(\tau)\equiv \vartheta_0, \ p_u(\tau)\equiv 0, \ p_\vartheta(\tau)\equiv (p_\vartheta)_0,\]
	and the functions $u(\tau)$ and $p_t(\tau)$ may depend on $\tau$.
	
	Since $t\equiv 0$, we have $\mu\equiv 0, \frac{\partial \mu}{\partial u}\equiv 0$ and $\frac{\partial \mu}{\partial \vartheta}\equiv 0$. Therefore, except for the second and the fourth equations, the other equations in \eqref{eq:Hamiltonian for geodesics} are satisfied automatically. To solve \eqref{eq:Hamiltonian for geodesics}, it suffices to consider the following system
	\begin{equation}\label{eq:Hamiltonian for geodesics reduced}
		\begin{cases}
			\frac{du}{d\tau}&=-p_t,\\
			\frac{dp_t}{d\tau}&= \frac{1}{2}\frac{\partial \mu}{\partial t}(0, u(\tau),\vartheta_0)\left(c^{-2}(0,u(\tau),\vartheta_0) p_t^2-2\Xih(u(\tau),\vartheta_0)(p_\vartheta)_0 p_t-  (p_\vartheta)_0^2\right).
		\end{cases}
	\end{equation}
	This is a closed system in $u$ and $p_t$. Therefore, by the classical uniqueness results for ordinary differential equations, for any initial data of \eqref{eq:Hamiltonian for geodesics} in the following form:
	\[ \big(t, u, \vartheta, p_t, p_u , p_\vartheta\big)\big|_{\tau=0} =\big(0,u_0,\vartheta_0,(p_t)_0,0,(p_\vartheta)_0\big),\]
	the corresponding solution to \eqref{eq:Hamiltonian for geodesics} will be of the form
	\[ t(\tau)\equiv 0, \  u(\tau), \ \vartheta(\tau)\equiv \vartheta_0, \ \ p_t(\tau), \ \ p_u(\tau)\equiv 0, \ p_\vartheta(\tau)\equiv (p_\vartheta)_0.\]
	In particular, we have $t(\tau)\equiv 0$. Therefore, the null geodesic $\gamma$ stays completely in the singularity $\mathbf{S}_*$. On the other hand, by the definition of $\mathscr{L}$, the curve $\gamma$ indeed passes through the rarefaction wave region before it ends at $\mathbf{S}_*$. This leads to a contradiction. We then conclude that there are no incoming or other null geodesics in $\mathscr{L}$.
\end{proof}

For any $\gamma \in \mathscr{L}$, it is outgoing  and we use the parametrization that $(p_t)_0=0, (p_u)_0=-1$. By the first three equations of \eqref{eq:Hamiltonian for geodesics}, we have
\begin{equation}\label{eq:outcoming 1st order}
	\frac{dt}{d\tau}\big|_{\tau=0}=1, \ 	\frac{du}{d\tau}\big|_{\tau=0}=0,\ 
	\frac{d\vartheta}{d\tau}\big|_{\tau=0} =  0.
\end{equation}
Therefore, the last three equations of \eqref{eq:Hamiltonian for geodesics} lead to
\[
\frac{dp_t}{d\tau}\big|_{\tau=0}= -\frac{1}{2}\frac{\partial \mu}{\partial t}  p_\vartheta^2, \ 
\frac{dp_u}{d\tau}\big|_{\tau=0}= 0, \ 
\frac{dp_\vartheta}{d\tau}\big|_{\tau=0}= 0.
\]
By \eqref{eq:outcoming 1st order}, we have
\[\frac{d\mu}{d\tau}\big|_{\tau=0}=\frac{\partial\mu}{\partial t}(q_*).\]
Thus, by differentiating the first three equations of \eqref{eq:Hamiltonian for geodesics}, we have
\[
\frac{d^2t}{d\tau^2}\big|_{\tau=0}=\frac{\partial \mu}{\partial t} \Xih p_\vartheta, \ 	\frac{d^2 u}{d\tau^2}\big|_{\tau=0}=\frac{1}{2}\frac{\partial \mu}{\partial t}  p_\vartheta^2,\ 
\frac{d^2\vartheta}{d\tau^2}\big|_{\tau=0} =  \frac{\partial \mu}{\partial t}  p_\vartheta.
\]
This shows that, as $\tau \rightarrow 0$, we have
\begin{equation}\label{eq: outgoing null geodesics} 
	\begin{cases}
		&t=\tau+\frac{1}{2}\frac{\partial \mu}{\partial t} \Xih p_\vartheta \cdot \tau^2 +O(\tau^3), \\
		&u=u_0+\frac{1}{4}\frac{\partial \mu}{\partial t} p^2_\vartheta \cdot \tau^2 +O(\tau^3), \\
		&\theta=\vartheta_0+\frac{1}{2}\frac{\partial \mu}{\partial t}  p_\vartheta \cdot \tau^2 +O(\tau^3).
	\end{cases}
\end{equation}	
In particular, we have
\[\gamma'(\tau)=\frac{\partial}{\partial t}+\tau\big[\frac{1}{2}\frac{\partial \mu}{\partial t} p^2_\vartheta \frac{\partial}{\partial u}+\frac{\partial \mu}{\partial t}  p_\vartheta \frac{\partial}{\partial \vartheta}\big]+O(\tau^2).\]
This shows that for $\gamma \in \mathscr{L}$ ending at $q_0\in \mathbf{S}_*$, its tangent vector  converges asymptotically to the null generator of $C_{u_0}$ at $q_0$. 

We remark that the choice of the acoustical coordinate system is not unique. The above construction allows us to define the singularity $\mathbf{S}_*$ in a geometric way which is independent of the choice of the acoustical coordinate system $(t,u,\vartheta)$. It will eventually lead to a canonical construction of the acoustical coordinates in the rarefaction region. 
\begin{itemize}
\item The canonical construction of the singular set $\mathbf{S}_*$.

For two curves $\gamma_1, \gamma_2\in \mathscr{L}$, if $\gamma_1(\tau)-\gamma_2(\tau)\rightarrow 0$ and $\gamma'_1(\tau)-\gamma'_2(\tau)=O(\tau)$ as $\tau\rightarrow 0$, we say that $\gamma_1$ and $\gamma_2$ are equivalent. This defines an equivalent relation $\sim$ on $\mathscr{L}$. We define $\mathbf{S}_*$ as the set of  equivalent classes, i.e., $\mathbf{S}_*=\mathscr{L}/\sim$. Therefore, for each equivalent class, we may choose a unique curve $\gamma_{u_0,\vartheta_0}:t\mapsto (t_0-t,u_0,\vartheta_0)$ as the representative of this class $[\gamma_{u_0,\vartheta_0}] \in \mathbf{S}_*$. 

\item The trace of $\wb,w$ and $\psi_2$ on $\mathbf{S}_*$.

For all $\psi \in \{w,\wb, \psi_2\}$, if $\gamma_1 \sim \gamma_2$ are null geodesics in $\mathcal{L}$, then 
\[\lim_{\tau\rightarrow 0}\psi\big|_{\gamma_1}(\tau)=\lim_{\tau\rightarrow 0}\psi\big|_{\gamma_2}(\tau).\]
Therefore, for any $q=[\gamma]\in \mathbf{S_*}$, we can define the limiting value of $\psi \in \{w,\wb, \psi_2\}$ as
\[\psi(q)=\lim_{\tau\rightarrow 0}\psi\big|_{\gamma}(\tau).\]
This gives a coordinate-independent definition of $w,\wb$ and $\psi_2$ on $\mathbf{S}_*$.

\item The canonical definition of $u$ and $\vartheta$ on $\mathbf{S}_*$.

For any point $q=[\gamma]\in \mathbf{S_*}$, the limiting point $\lim_{\tau\rightarrow 0}\gamma(\tau)$ is a real point in the physical spacetime equipped with the Cartesian coordinates $(t,x_1,x_2)$. In fact, $\lim_{\tau\rightarrow 0}\gamma(\tau)=(0,0,x_2)$. We (re)define the canonical coordinate function $\vartheta$ on $\mathbf{S}_*$ as 
\[\vartheta(q)=x_2.\]

Finally, since we can define $\wb$ on $\mathbf{S}_*$, in view of \eqref{eq:data at the singularity}, we (re)define the canonical coordinate function $u$ on $\mathbf{S}_*$ as
\[u(q)= \frac{\gamma+1}{2}\big(\wb_r(0,\vartheta(q))-\wb(q)\big).\]

\end{itemize}

\subsubsection{Constructions of characteristic hypersurfaces from singularity and canonical acoustical coordinates}\label{section: construction of null hypersurfaces}

Let $\Gamma_f=\big\{(u,\vartheta)=(f(\vartheta),\vartheta)\big|\vartheta\in [0,2\pi]\big\}$ be a smooth graph in $\mathbf{S}_*$. We will construct a characteristic hypersurface emanated from $\Gamma_f$. Indeed, this requires the initial data of the Hamiltonian system to be
\[p_tdt+p_u du +p_\vartheta d\vartheta\big|_{\tau=0}=d(-u+f(\vartheta))=-du+f'(\vartheta)d\vartheta,\]
i.e., the initial data must be
\[(t,u,\vartheta, p_t,p_u,p_\vartheta)\big|_{\tau=0}=\big(0,f(\alpha),\alpha, 0,-1,f'(\alpha)\big), \ \ \alpha\in [0,2\pi].\]

For a given $\alpha\in [0,\pi]$,  the solution provides a specific outgoing null geodesic $\{L_{\alpha}(\tau)= (t(\tau),u(\tau),\vartheta(\tau))\}$ in $\mathcal{L}$. We then define 
\[C_{\Gamma_f}=\bigcup_{\alpha \in [0,2\pi]} L_\alpha.\]
This is the characteristic hypersurface emanated from $\Gamma_f$.

\begin{remark}
	All rarefaction wave fronts emanated from the singularity $\mathbf{S}_*$ must be of the form $C_{\Gamma_f}$, because they are ruled by the outgoing null geodesics.
\end{remark}

\begin{itemize}
	\item To illustrate this construction, we assume that  $f(\vartheta)\equiv u_0$ where $u_0\in [0,u^*)$. Therefore, $\Gamma_f=\big\{(u_0,\vartheta)\big|\vartheta\in [0,2\pi]\big\}$. 
	 For a given $\alpha\in [0,\pi]$, the initial data of \eqref{eq:Hamiltonian for geodesics} is given by 
	\[(t,u,\vartheta, p_t,p_u,p_\vartheta)\big|_{\tau=0}=\big(0,u_0,\alpha, 0,-1, 0\big).\] 
	It is straightforward to check that the corresponding solution of \eqref{eq:Hamiltonian for geodesics} is
	\[\big(t(\tau),u(\tau),\vartheta(\tau), p_t(\tau),p_u(\tau),p_\vartheta(\tau)\big)=\big(\tau,u_0,\alpha, 0,-1, 0\big).\] 
	This is the geodesic defined by the vector field $L$. Hence, $C_{\Gamma_f}=C_{u_0}$.
	
	\item Take a smooth function $h(\vartheta)$ so that $h\neq 0=f'(\vartheta)$. For a given $\alpha\in [0,2\pi]$, we consider the following the initial data of \eqref{eq:Hamiltonian for geodesics}:
	\[(t,u,\vartheta, p_t,p_u,p_\vartheta)\big|_{\tau=0}=\big(0,u_0,\alpha, 0,-1, h(\alpha)\big).\] 
	By \eqref{eq: outgoing null geodesics}, for sufficiently small $\tau$, the corresponding solution of \eqref{eq:Hamiltonian for geodesics} satisfies
	\[u(\tau)=u_0+\frac{1}{4}\frac{\partial \mu}{\partial t} h(\alpha)^2\cdot \tau^2 +O(\tau^3).\]
	The quadratic term in the above expression encodes the causality of the system. Indeed, if $h(\alpha)\neq 0$, for sufficiently small $\tau$, we have $u(\tau)>u_0$. This means that the curve lies in the future of $C_{u_0}$. In particular, the union of the outgoing null geodesics (before the caustics form) defined by the above set of initial data is generically a time-like hypersurface with respect to the acoustical metric.
\end{itemize}

We can use the above construction to extend $u$ and $\vartheta$ form $\mathbf{S}_*$ to the entire rarefaction wave region (by requring $\vartheta$ is constant along each outgoing null geodesic). This provides a canonical acoustical coordinate system on the rarefaction wave region.

}

\subsection{Applications to the Riemann problem}\label{Section: app to Riemann 2}
We now consider the Cauchy problem with the $\varepsilon$-perturbed Riemann data, see Definition \ref{def:data}. We have already proved that,  for the data $(v_r,c_r)$ given on $x_1>0$, we can construct a family of rarefaction waves connecting to it. It corresponds to {\bf the front rarefaction waves} and is depicted on the right of the following picture.
\begin{center}
\includegraphics[width=3.5in]{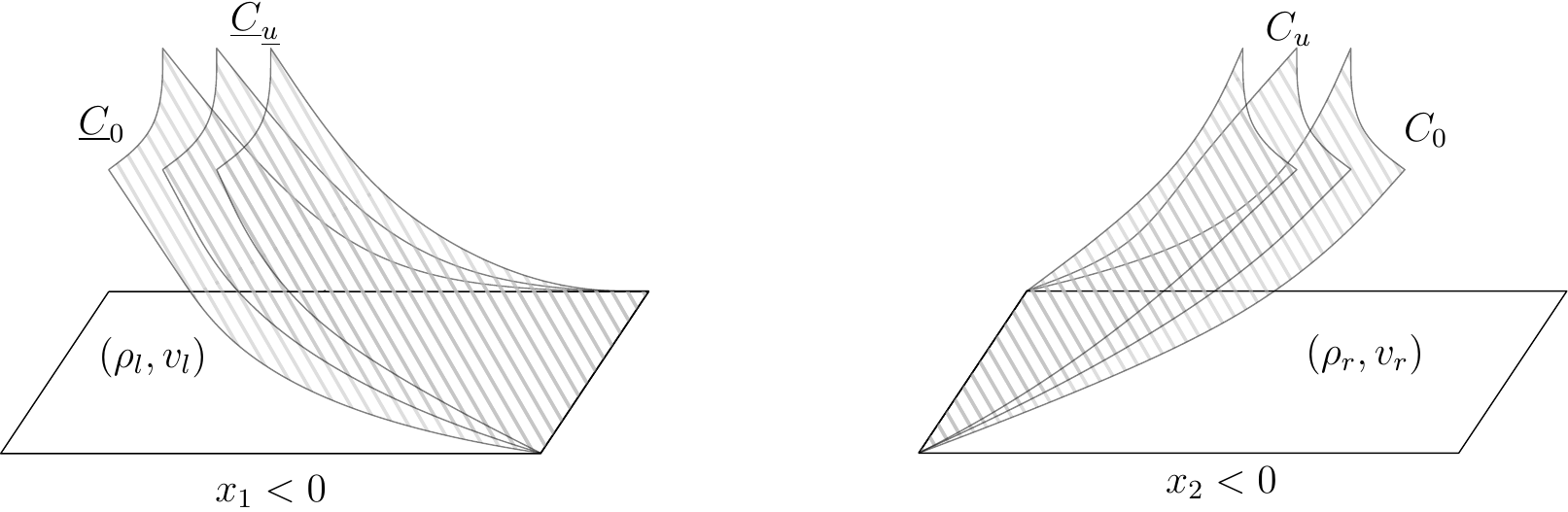}
\end{center}
For the data $(v_l,c_l)$ given on $x_1<0$, we can construct also a family of rarefaction waves connecting to it. This is {\bf the back rarefaction waves} and is depicted on the left of the above picture. These two families of rarefactions are associated to different families of characteristic hypersurfaces. We use $\Cb_{\ub}$ and $C_u$ to denote the back rarefaction wave fronts and  the front rarefaction wave fronts respectively. By definition, they are the characteristic hypersurfaces for the back  and front rarefaction waves respectively. The families of the front and back rarefaction wave fronts are universal. It is important to observe that certain part of them are not physical (with respect to the given data) so that they will not appear in the solution for the Cauchy problem. To illustrate this point, we consider the case of front rarefaction waves.
For a given front rarefaction wave front $C_u$, it cuts the singularity or equivalently the limiting surface $\mathbf{S}_*$ at $S_{0,u}$. We have shown that we can at least open up the front rarefaction waves up to $u=\overline{u}$, i.e., the solution exists for $u\in [0,\overline{u}]$. This is depicted in the right part of the following picture.
\begin{center}
\includegraphics[width=3.5in]{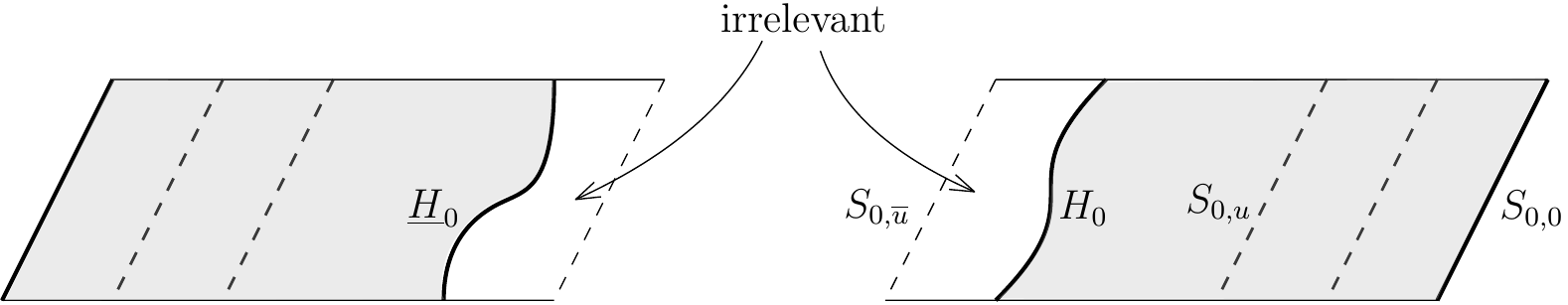}
\end{center}

We will show that there exists a curve $H_0$ between $S_{0,0}$ and $S_{0,\overline{u}}$ so that the region bounded by $H_0$ and $S_{0,\overline{u}}$ is not relevant to the perturbed Riemann problem. Let $H$ be the union of all the null geodesics emanated from $H_0$ along the $L$ direction. The  physically relevant front rarefaction wave region is the spacetime domain bounding $H$ and $C_0$. We also have a similar picture for back rarefaction waves.

We now define the inner boundary $H$ for the front rarefaction waves. According to \eqref{eq:data at the singularity}, we have $\wb(0,u,\vartheta)= \wb_r(0,\vartheta)-\frac{2}{\gamma+1}u$ at the singularity. In particular, $\wb$ decreases as $u$ increases. The curve $H_0$ consists of those points where $\wb$ decreases to $\wb_l$, $\wb_l$ being the Riemann invariants defined by the data on $x_1<0$, i.e., $\wb=\wb_l$. More precisely, we define
\begin{equation}\label{eq:H_0}
	H_0:=\big\{(u,\vartheta)\big|u=\frac{\gamma+1}{2}\big(\wb_r(0,\vartheta)-\wb_l(0,\vartheta)\big)\big\},
\end{equation}
where $\vartheta=x_2$.   In fact, for the one dimensional Riemann data $(\mathring{v}_l,\mathring{c}_l)$ and $(\mathring{v}_r,\mathring{c}_r)$ which leads to two families of centered rarefaction waves, there exists a $u'$ so that $u'=\frac{\gamma+1}{2}\big(\mathring{\wb}_r-\mathring{\wb}_l\big)$; see Section \ref{section: calculations for 1D Riemann}. Since the solution on $\mathbf{S}_*$ is also  $O(\varepsilon)$-close to the one dimensional picture, this shows the existence of $H_0$.

{\color{black}The rarefaction front $H$ is defined as the null hypersurface starting from $H_0$ with respect to the acoustical metric, see previous sections for the precise construction using the Hamiltonian formulation.}
\begin{center}
\includegraphics[width=2.5in]{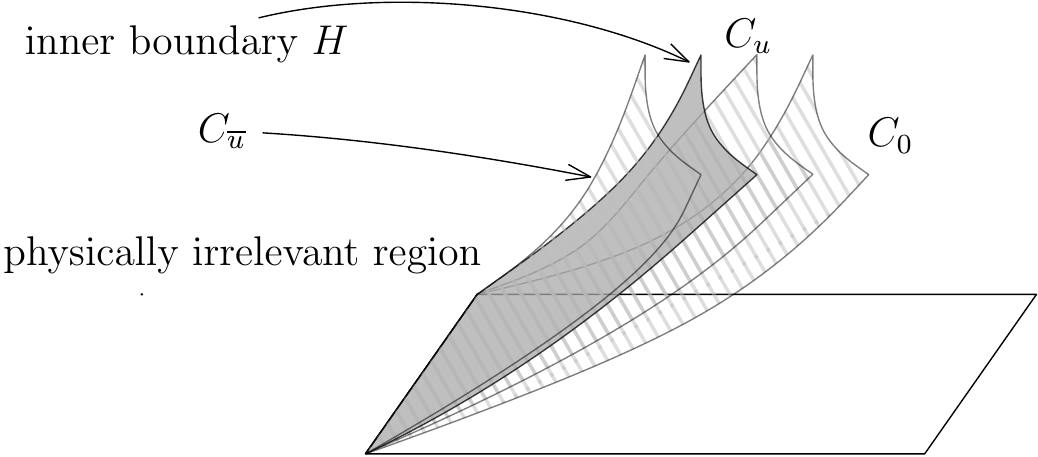}
\end{center}

For the back rarefaction waves, we can define $\Hb$ in a similar manner. The initial curve $\Hb_0$ is given by
\[\Hb_0:=\big\{(\ub,\vartheta)\big|\ub=\frac{\gamma+1}{2}\big(w_l(0,\vartheta)-w_r(0,\vartheta)\big)\big\},\]
where $\ub$ is the acoustical functions defined in the back rarefaction wave region. The $\Hb$ is defined as the null hypersurface emanated from $\Hb_0$. 

The above constructions give two characteristic hypersurfaces $H$ and $\Hb$. They are all emanated from the singularity
\[\mathbf{S}_*:=\big\{(t,x_1,x_2)\big|t=0,x_1=0\big\}\]
It remains to construct the solution to the Euler equations in the regions which is bounded by $H$, $\Hb$ and $\Sigma_{t^*}$. These three hypersurfaces are depicted in grey colors in the following picture.
\begin{center}
\includegraphics[width=2.2in]{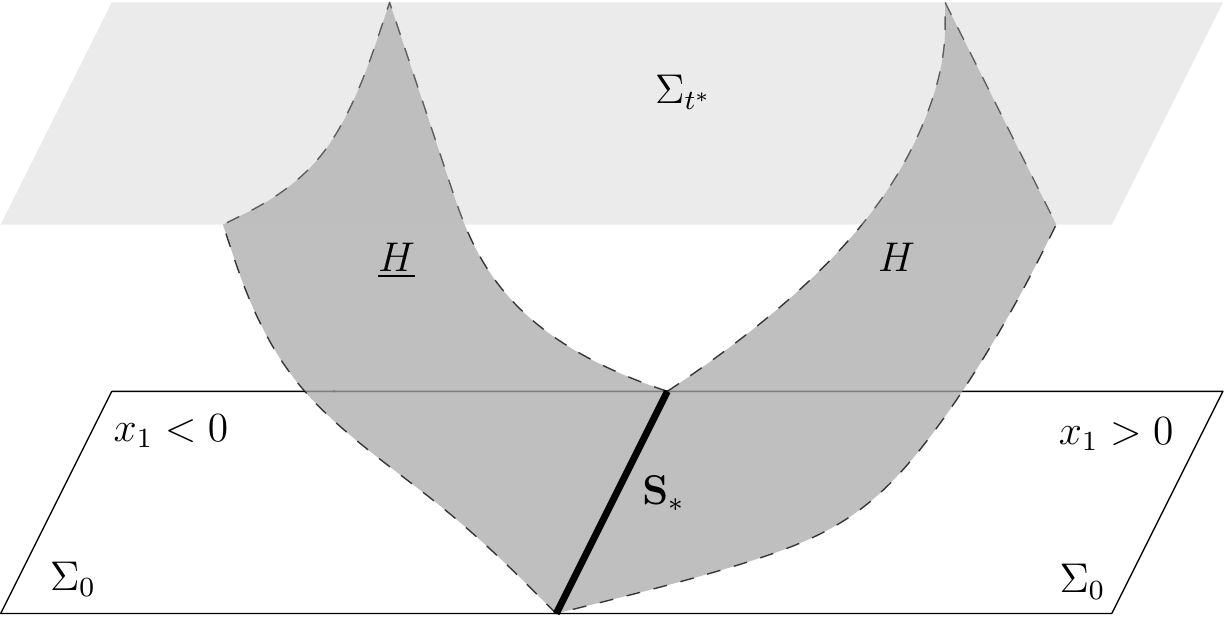}
\end{center}
The solutions have already been constructed on the right of $H$ and on the left of $\Hb$. Since $H$ is smooth in the front rarefaction wave region, the restriction of the solutions already constructed in this region gives $C^{\Ntop-3}$ data $(\wb^{(H)},w^{(H)},\psi_2^{(H)})$ on $H$ up to its back boundary $H_0$; similarly, the restriction of the solutions  constructed in the back rarefaction wave region gives $C^{\Ntop-2}$ data $(\wb^{(\Hb)},w^{(\Hb)},\psi_2^{(\Hb)})$ on $\Hb$ up to its back boundary $\Hb_0$. By the construction of $H_0$ and $\Hb_0$, we have 
\[\begin{cases}
&\wb^{(\Hb)}=\wb^{(H)},\\
&w^{(\Hb)}=w^{(H)},
\end{cases} \ \ \ \ \text{on}~\mathbf{S}_{*}=H\cap \Hb.
\]
We recall the irrotational condition \eqref{irrotational condition} ${\rm curl}(v)\stackrel{\mathscr{D}'}{=}0$ which holds in the distributional sense on $\Sigma_0$. Therefore, for all test function $\varphi(x_1,x_2)\in \mathscr{D}(\Sigma_0)$, we have
\[\langle {\rm curl}(v),\varphi\rangle_{\mathscr{D}'\times\mathscr{D}}=0 \ \ \Rightarrow \ \ \langle v_1,\partial_2\varphi\rangle_{\mathscr{D}'\times\mathscr{D}}=\langle v_2,\partial_1\varphi\rangle_{\mathscr{D}'\times\mathscr{D}}.\]
Therefore,
\[\int_{x_1\leqslant 0} (v_l)_1 \partial_2\varphi dx_1dx_2+\int_{x_1\geqslant 0} (v_r)_1 \partial_2\varphi dx_1dx_2=\int_{x_1\leqslant 0} (v_l)_2 \partial_1\varphi dx_1dx_2+\int_{x_1\geqslant 0} (v_r)_2 \partial_1\varphi dx_1dx_2.\]
Since $v_l$ and $v_r$ are smooth on $x_1<0$ and $x_1>0$ respectively, we can integrate by parts to remove the derivatives on the test function $\varphi$. Therefore, by the facts that $v_l$ and $v_r$ are irrotational on $x_1<0$ and $x_1>0$ respectively, we have
\[\int_{0}^{2\pi} (v_l)_2(0,x_2) \varphi(0,x_2) dx_2=\int_{0}^{2\pi} (v_r)_2(0,x_2) \varphi(0,x_2)dx_2.\]
Hence, $(v_l)_2$ and $(v_r)_2$ are equal on $\mathbf{S}_*$, i.e., $\psi_2^{(\Hb)}=\psi_2^{(H)}$ on $\mathbf{S}_{*}=H\cap \Hb$.

Therefore, we have characteristic initial data $(\wb^{(H)},w^{(H)},\psi_2^{(H)})$ on $H$ and $(\wb^{(\Hb)},w^{(\Hb)},\psi_2^{(\Hb)})$ on $\Hb$ in such a way that they coincide on $H\cap \Hb$. We can then solve a classical Goursat problem with such set of initial data for the Euler equations.  According to the estimates already derived in this work, the data on $H\cup \Hb$ is $O(\varepsilon)$-close to the standard one dimensional problem (see Section \ref{section: calculations for 1D Riemann}
). By the continuous dependence of the data for the Goursat problem of the Euler equations, we can solve the equation in the region bounded by $H, \Hb$ and $\Sigma_{t^*}$.

This completes the construction of the solution to Riemann problem hence the proof of {\bf Theorem 3}.

{\color{black}

\begin{remark}\label{rem: 1 family alinhac}
 Given specific initial data, we can also solve the Riemann problem with one family of rarefaction waves. This is similar to the result of  Alinhac  in  \cite{AlinhacWaveRare1}.	
	\begin{center}
\includegraphics[width=3in]{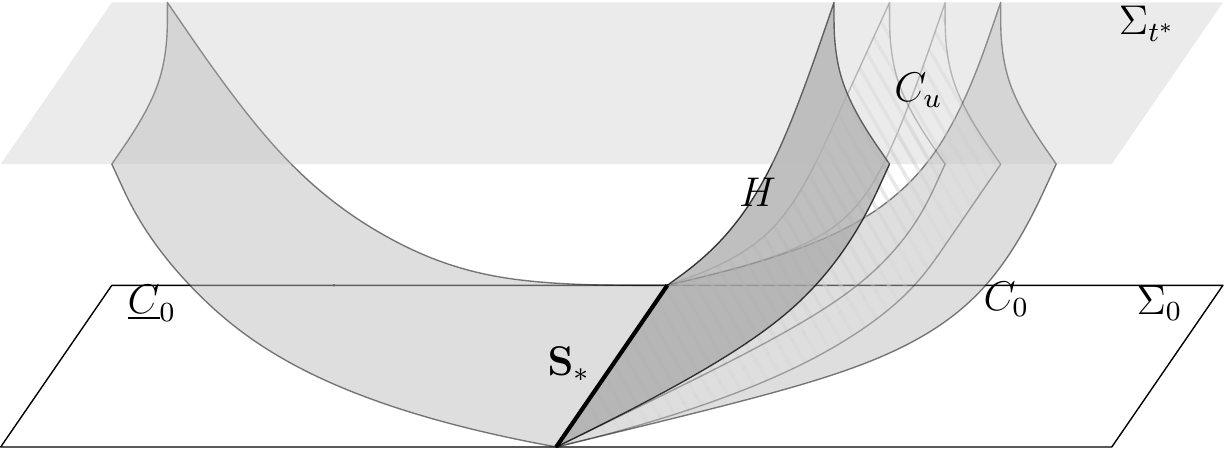}
\end{center}
We make the following extra assumptions on the initial data $(U_l, U_r)$:
\begin{itemize}
\item $
	 w_l \big|_{\mathbf{S}_*} = w_r\big|_{\mathbf{S}_*}$ and $ v^2_l\big|_{\mathbf{S}_*} = v^2_r\big|_{\mathbf{S}_*}$;
\item $\wb_l|_{\mathbf{S}_*} < \wb_r\big|_{\mathbf{S}_*}$ and there exists $H_0$ (as defined in \eqref{eq:H_0}) so that $\wb\big|_{H_0} = \wb_l\big|_{\mathbf{S}_*}$ (this can be achieved if we assume that $\wb_l|_{\mathbf{S}_*}$ is sufficiently close to  $\wb_r|_{\mathbf{S}_*}$).
\end{itemize}
Then, there exists a solution connecting $U_l$ and $U_r$ with only one family of front rarefaction waves. In this case the rarefaction fronts $\Cb_0$ and $\Hb$ in the picture of {\bf Theorem 3} coincide, see the above picture. This can be viewed as the limit of the Riemann problem of two rarefaction waves, see Section \ref{Section: uniqueness Thm 2} for details.

We remark that the solution in the above picture is not smooth across $\underline{C}_0$. If we use Alinhac's ''$k$-compatibility condition'' (see \cite{AlinhacWaveRare1}) on normal derivatives of order $k$, the solution will be $C^k$ across $\underline{C}_0$. In contrast, the above construction does not require	compatibility conditions on the initial data. 
\end{remark}

\subsection{Uniqueness}
Based on the method of relative entropy, the precise estimates obtained in the first paper \cite{LuoYu1} can also be applied to prove the uniqueness of the rarefaction waves constructed in  {\bf Theorem 2} and  {\bf Theorem 3}.

\subsubsection{The general set-up of the relative entropy method}
We follow Diperna's notations in \cite{DiPerna79} to review the relative entropy method. The method was originally introduced  to study the stability and uniqueness of classical solutions (so called weak-strong uniqueness) by Dafermos and Diperna. We consider the system of conservation laws in the following form
\begin{equation}\label{eq: general system}
\frac{\partial }{\partial t}U^a +\sum_{i=1}^n \frac{\partial }{\partial x^i}F^{ia}(U)=0, \ a=1,\cdots, m,
\end{equation}
where $U(t,x)=\begin{pmatrix}
  U^1 \\
  U^2\\
  \cdots\\
  U^m
   \end{pmatrix}: \mathbb{R}\times \mathbb{R}^n\rightarrow \mathbb{R}^m$. We use $\{x^i\}_{1\leqslant i\leqslant n}$ as a coordinate system for $x\in \mathbb{R}^n$ and use $\{y^a\}_{1\leqslant \alpha \leqslant m}$ as a coordinate system on the target $\mathbb{R}^m$. The $\{F^{ia}\}_{ 1\leqslant i  \leqslant n, 1\leqslant a \leqslant m}$ are smooth functions on $\mathbb{R}^m$. Let $F^i=\begin{pmatrix}
  F^{i1} \\
  F^{i2}\\
  \cdots\\
  F^{im}
   \end{pmatrix}$. Thus, \eqref{eq: general system} can be written in the following vector form:
   \begin{equation}\label{eq: general system vector form}
\frac{\partial }{\partial t}U  +\sum_{i=1}^n \frac{\partial }{\partial x_i}F^i(U)=0.
\end{equation}

We recall the following definitions of entropy solution for the above system, see \cite{Dafermos}.
\begin{itemize}
\item An {\bf entropy-entropy flux pair} consists of a smooth convex function $\eta:\mathbb{R}^m\rightarrow \mathbb{R}$ and a smooth function  $q:\mathbb{R}^m\rightarrow \mathbb{R}^n$ so that
\begin{equation}\label{eq: symmetry condition for eta q}
\frac{\partial q^i}{\partial y^b}(y)=\sum_{a=1}^m\frac{\partial \eta}{\partial y^a}(y)\frac{\partial F^{ia}}{\partial y^b}(y),
\end{equation}
for all $1\leqslant i\leqslant n$ and $1\leqslant b\leqslant m$.
\item We say that $U$ is an {\bf entropy solution} of \eqref{eq: general system}, if there exists an entropy-entropy flux pair $(\eta,q)$ so that 
\[\partial_t \eta(U) + \nabla_x\cdot q(U) \leqslant 0,\]
in the distributional sense. 

\item For a classical solution (locally Lipschitz)  $\overline{U}$  of the system,  we have
\[\partial_t \eta(\overline{U}) + \nabla_x \cdot q(\overline{U}) = 0\]
as an immediate consequence of the condition \eqref{eq: symmetry condition for eta q}.\end{itemize}
If we take $\frac{\partial}{\partial y^c}$ derivative in \eqref{eq: symmetry condition for eta q},  we obtain the following symmetry condition:
\begin{equation}\label{eq: symmetry condition}
\sum_{a=1}^m\frac{\partial^2\eta}{\partial y^c\partial y^a} \frac{\partial F^{ia}}{\partial y^b}=\sum_{a=1}^m\frac{\partial^2\eta}{\partial y^b\partial y^a} \frac{\partial F^{ia}}{\partial y^c},
\end{equation}
for all $1\leqslant b,c\leqslant m$ and $1\leqslant i\leqslant n$.

\begin{remark}
We can use the matrix $\left(\frac{\partial^2\eta}{\partial y_a\partial y_b}\right)$ as a Friedrichs symmetrizer of the hyperbolic system. In fact, if we assume that $U$ is smooth in a spacetime region, the hyperbolic system \eqref{eq: general system} can be written as
\[\frac{\partial }{\partial t}U^a +\sum_{i=1}^n \sum_{b=1}^m \frac{\partial F^{ia}}{\partial y^b}(U)\frac{\partial U^b}{\partial x^i}=0, \ a=1,\cdots, m.
\]
Let $A^{ia}{}_{b}=\frac{\partial F^{ia}}{\partial y^b}$, the system becomes
\[\frac{\partial }{\partial t}U^a +\sum_{i=1}^n \sum_{b=1}^m A^{ia}{}_{b}(U)\frac{\partial U^b}{\partial x^i}=0, \ a=1,\cdots, m.
\]
The Friedichs symmetrizer is a positive symmetric $2$-tensor $(g_{ab})$on the target $\mathbb{R}^m$ so that
\[A^{i}{}_{ab}=A^{i}{}_{ba}\]
where $A^{i}{}_{ab}=\sum_{c=1}^m g_{ac}A^{ic}{}_{b}$. Thus, the condition \eqref{eq: symmetry condition} implies that $\left(\frac{\partial^2\eta}{\partial y_a\partial y_b}\right)$ is a Friedrichs symmetrizer for \eqref{eq: general system}.
\end{remark}

The {\bf relative entropy} of an entropy solution $U$ with respect to a classical solution $\overline{U}$ is defined as
\[
	\alpha(U,\overline{U}) = \eta(U) - \eta(\overline{U}) - \sum_{a=1}^m\frac{\partial \eta}{\partial y^a}(\overline{U})(U^a - \overline{U}^a).
\]
It can also be written as
\begin{equation}\label{eq: def relative entropy alpha} \alpha(U,\overline{U}) =\sum_{a=1}^m\sum_{b=1}^m (U^a-\overline{U}^a) \big(\int_0^1\int_{0}^\tau \frac{\partial^2\eta}{\partial y^a\partial y^b}\big(\overline{U} + s(U - \overline{U})\big)  ds d\tau\big)  (U^b-\overline{U}^b). 
\end{equation}
In applications, $U$ and $\overline{U}$ are all bounded and $\left(\frac{\partial^2\eta}{\partial y^a\partial y^b}\right)$ has a positive lower bound. Therefore, we will have
\[\alpha(U,\overline{U}) \approx |U-\overline{U}|^2. \]

The {\bf relative entropy flux} $\beta(U,\overline{U})=\begin{pmatrix}
  \beta^1(U,\overline{U}) \\
  \beta^2(U,\overline{U})\\
  \cdots\\
\beta^n(U,\overline{U})
   \end{pmatrix}$ takes values in $\mathbb{R}^n$ and it is defined as
\begin{equation}\label{eq: def relative entropy beta}	
\beta^i(U,\overline{U}) = q^i(U) - q^i(\overline{U}) - \sum_{a=1}^m \frac{\partial \eta}{\partial y^a}(\overline{U}) \big(F^{ia}(U) - F^{ia}(\overline{U})\big), \ \ i =1,\cdots, n.
\end{equation}
The relative entropy $\beta^i(U,\overline{U})$ is indeed quadratic in $|U-\overline{U}|$. This is based on the following computations:
\begin{align*}
\beta^i(U,\overline{U}) =& q^i(U) - q^i(\overline{U}) -\sum_{b=1}^m \frac{\partial q^i}{\partial y^b}(\overline{U}) ( U^b - \overline{U}^{b})\\
&+\sum_{b=1}^m \frac{\partial q^i}{\partial y^b}(\overline{U}) ( U^b - \overline{U}^{b})- \sum_{a=1}^m \frac{\partial \eta}{\partial y^a}(\overline{U}) \big(F^{ia}(U) - F^{ia}(\overline{U})\big)\\
=&\sum_{a=1}^m\sum_{b=1}^m (U^a-\overline{U}^a) \big(\int_0^1\int_{0}^\tau \frac{\partial^2 q^i}{\partial y^a\partial y^b}\big(\overline{U} + s(U - \overline{U})\big)  ds d\tau\big)  (U^b-\overline{U}^b)\\
&+\sum_{b=1}^m \sum_{a=1}^m\frac{\partial \eta}{\partial y^a}( \overline{U})\frac{\partial F^{ia}}{\partial y^b}( \overline{U})( U^b - \overline{U}^{b})- \sum_{a=1}^m \frac{\partial \eta}{\partial y^a}(\overline{U}) \big(F^{ia}(U) - F^{ia}(\overline{U})\big).
\end{align*}
where we have used \eqref{eq: symmetry condition for eta q} in the last step. Hence,
\begin{equation}\label{eq:relative entropy beta quadratic}
\begin{split}
\beta^i(U,\overline{U}) =&\sum_{a=1}^m\sum_{b=1}^m (U^a-\overline{U}^a) \big(\int_0^1\int_{0}^\tau \frac{\partial^2 q^i}{\partial y^a\partial y^b}\big(\overline{U} + s(U - \overline{U})\big)  ds d\tau\big)  (U^b-\overline{U}^b)\\
&- \sum_{a=1}^m \frac{\partial \eta}{\partial y^a}(\overline{U}) \left(F^{ia}(U) - F^{ia}(\overline{U})-\sum_{b=1}^m\frac{\partial F^{ia}}{\partial y^b}( \overline{U})( U^b - \overline{U}^{b})\right).
\end{split}
\end{equation}

 In applications, $U$ and $\overline{U}$ are all bounded. The first term is clearly bounded  above by $ |U-\overline{U}|^2$. The second term can also be bounded in the similar manner. Therefore, we will have
\[\beta^i(U,\overline{U}) \lesssim |U-\overline{U}|^2, \ i=1,\cdots,n. \]

We consider the following distributional function
\begin{align*}
	\mu(U,\overline{U}) = \partial_t \alpha(U,\overline{U}) + \nabla_x \cdot \beta(U,\overline{U}).\end{align*}
Since $\overline{U}$ is a classical solution, we have $\partial_t \eta(\overline{U}) + \nabla_x \cdot q(\overline{U}) = 0$. Combined with the fact that both $U$ and $\overline{U}$ are solutions to \eqref{eq: general system}, we have
\begin{align*}
	\mu(U,\overline{U}) 	=& \partial_t \eta(U) + \nabla_x \cdot q(U) -\sum_{a=1}^m\sum_{b=1}^m\frac{\partial \overline{U}^{b}}{\partial t}\frac{\partial^2\eta}{\partial y^a\partial y^b}( \overline{U})\left(U^a-\overline{U}^a\right)\\
	&-\sum_{a=1}^m\sum_{b=1}^m\sum_{i=1}^n\frac{\partial \overline{U}^{b}}{\partial x^i}\frac{\partial^2\eta}{\partial y^a\partial y^b}( \overline{U})\left(F^{ia}(U)-F^{ib}(\overline{U})\right).
\end{align*}

By using the condition \eqref{eq: symmetry condition} and the equation \eqref{eq: general system}, it is straightforward to compute that
\begin{align*}
	\mu(U,\overline{U}) 	=& \partial_t \eta(U) + \nabla_x \cdot q(U) \\
	&-\sum_{a=1}^m\sum_{b=1}^m\sum_{i=1}^n\frac{\partial \overline{U}^{a}}{\partial x^i}\frac{\partial^2\eta}{\partial y^a\partial y^b}( \overline{U})\left[F^{ib}(U)-F^{ib}(\overline{U})-\sum_{c=1}^m\frac{\partial F^{ib}}{\partial y^c}(\overline{U})\cdot(U^c- \overline{U}^c)\right].
	 \end{align*}
To simplify the expression, we define functions $QF^{ib}$'s as follows:
\[QF^{ib}(U,\overline{U})=F^{ib}(U)-F^{ib}(\overline{U})-  \sum_{c=1}^m\frac{\partial F^{ib}}{\partial y^c}(\overline{U})\cdot(U^c- \overline{U}^c),\]
where $1\leqslant i\leqslant n, 1\leqslant b\leqslant m$. The letter $Q$ in $QF^{ib}$ stands for quadratic expressions, i.e.,  we may regard every $QF^{ib}(U,\overline{U})$ term as quadratic in $U-\overline{U}$. Since $\partial_t \eta(U) + \nabla_x\cdot q(U) \leqslant 0$, we conclude that
\begin{equation}\label{eq: entropy density inequality}
\mu(U,\overline{U}) 	\leqslant -\sum_{a=1}^m\sum_{b=1}^m\sum_{i=1}^n\frac{\partial \overline{U}^{a}}{\partial x^i}\frac{\partial^2\eta}{\partial y^a\partial y^b}( \overline{U})QF^{ib}(U,\overline{U})
\end{equation}
in the distributional sense.

We now apply the above theory to the Euler equations \eqref{eq: Euler in rho v} and we refer to \cite{Chen-Chen} for the detailed computations.  

We first use the density-momentum variables $U=(\rho, P)$ to rewrite \eqref{eq: Euler in rho v} in the form of \eqref{eq: general system}. We first define 
\[P=\rho v=\rho v^1 \frac{\partial}{\partial x^1}+\rho v^2 \frac{\partial}{\partial x^1}.\]
Thus, \eqref{eq: Euler in rho v}  becomes
\[\begin{cases}
&\partial_t \rho+\frac{\partial P^1}{\partial x^1}+\frac{\partial P^2}{\partial x^2}=0,\\
&\partial_t P^1+\frac{\partial}{\partial x^1}\big(\frac{(P^1)^2}{\rho}+p(\rho)\big)+\frac{\partial}{\partial x^2} \big(\frac{ P^1P^2}{\rho}\big)=0,\\
&\partial_t P^2+\frac{\partial}{\partial x^1} \big(\frac{ P^1P^2}{\rho}\big)+\frac{\partial}{\partial x^2}\big(\frac{(P^2)^2}{\rho}+p(\rho)\big)=0.
\end{cases}
\]
We then define $U$, $F^1$ and $F^2$ as in \eqref{eq: general system vector form} in the following form
\[U=\begin{pmatrix}
  \rho \\
 P^1\\
 P^2
\end{pmatrix}, \ \ F^1(U)=\begin{pmatrix}
 P^1\\
\frac{(P^1)^2}{\rho}+p(\rho)\\
\frac{ P^1P^2}{\rho}
\end{pmatrix}, \ \  F^2(U)=\begin{pmatrix}
 P^2\\
\frac{ P^1P^2}{\rho}\\
\frac{(P^2)^2}{\rho}+p(\rho)
\end{pmatrix}.\]
The matrix expressions for $\frac{\partial F^{ia}}{\partial y^b}(U)$ (where the index $a$ stands for the rows) are given by
\[\left(\frac{\partial F^{1a}}{\partial y^b}(U)\right)=\begin{pmatrix}
 0& 1 &0\\
-(v^1)^2+c^2& 2v^1 &0\\
-v^1v^2& v^2 &v^1
\end{pmatrix}, \ \ \left(\frac{\partial F^{2a}}{\partial y^b}(U)\right)=\begin{pmatrix}
 0& 0 &1\\
-v^1v^2& v^2 &v^1\\
-(v^2)^2+c^2& 0 &2v^2
\end{pmatrix}.\]

In view of \eqref{def: eta q}, we have
\begin{equation}
\begin{cases} 
&\eta(\rho,v) = \frac{1}{2}\frac{|P|^2}{\rho} + \frac{1}{\gamma-1}p(\rho), \\
& q(\rho,v) = \left(\frac{1}{2}\frac{|P|^2}{\rho} + \frac{\gamma}{\gamma-1}p(\rho)\right)\frac{P}{\rho}.\end{cases}
\end{equation}
We then compute
\[\left(\frac{\partial\eta}{\partial y^a}(U)\right)=\begin{pmatrix}
 \frac{1}{\gamma-1}c^2-\frac{1}{2}|v|^2 \\
v^1 \\
v^2
\end{pmatrix},
\]
and
\begin{equation}\label{eq: hassian eta}\left(\frac{\partial^2\eta}{\partial y^a\partial y^b}(U)\right)=\frac{1}{\rho}\begin{pmatrix}
 |v|^2+c^2& -v^1 & -v^2\\
-v^1& 1 &0\\
-v^2& 0 &1
\end{pmatrix}.
\end{equation}
In particular, the three eigenvalues of the above matrix are
\[\frac{1}{\rho}, \ \frac{1}{2\rho}(|v|^2+c^2+1\pm \sqrt{|v|^2+(c^2-1)^2+2|v|^2(c^2+1)}).\] 
They are roots (multiplied by $\rho^{-1}$) of the following cubic polynomial:
\[\big(\lambda^2-(|v|^2+c^2+1)\lambda+c^2\big)(\lambda-1).\]
As along as $\rho$ is bounded below by a positive constant, the above matrix is positive definite and bound below. 
For the symmetry condition \eqref{eq: symmetry condition for eta q}, we notice that
\[\left(\frac{\partial q^{1}}{\partial y^b}(U)\right)=\begin{pmatrix}
 \big(c^2-|v|^2\big)v^1,& \frac{1}{\gamma-1}c^2+\frac{1}{2}|v|^2+(v^1)^2, &v^1v^2
 \end{pmatrix}\]
 and
\[\left(\frac{\partial q^{2}}{\partial y^b}(U)\right)=\begin{pmatrix}
\big(c^2-|v|^2\big)v^2,& v^1v^2, & \frac{1}{\gamma-1}c^2+\frac{1}{2}|v|^2+(v^1)^2 
\end{pmatrix}.\]
Therefore,  condition \eqref{eq: symmetry condition for eta q} can  be checked in a straightforward way by the above matrices.

Finally, we also compute $QF^{ib}(U,\overline{U})$ in the vector form:
\begin{equation}\label{eq: QFone}
QF^1(U,\overline{U})=\begin{pmatrix}
 0\\
[p(\rho)-p(\overline{\rho})-p'(\overline{\rho})(\rho-\overline{\rho})]+\rho(v^1-\overline{v}^1)^2\\
\rho(v^1-\overline{v}^1)(v^2-\overline{v}^2)\end{pmatrix}
\end{equation}
and
\begin{equation}\label{eq: QFtwo}
QF^2(U,\overline{U})=\begin{pmatrix}
 0\\
\rho(v^1-\overline{v}^1)(v^2-\overline{v}^2)\\
[p(\rho)-p(\overline{\rho})-p'(\overline{\rho})(\rho-\overline{\rho})]+\rho(v^2-\overline{v}^2)^2\end{pmatrix}.
\end{equation}


\subsubsection{The uniqueness in {\bf Theorem 3}}

We consider two solutions $\overline{U}$ and $U$ to the Euler equations \eqref{eq: Euler in rho v}. We make the following assumptions:
\begin{itemize}
\item We fix $\overline{U}$ to be the solution constructed in {\bf Theorem 3}. It is defined on the following spacetime region:
\[\Omega=[0,t^*]\times\mathbb{R}\times \mathbb{R}\slash 2\pi\mathbb{Z}=\big\{(t,x_1,x_2)\big| 0\leqslant t\leqslant t^*, x_1\in \mathbb{R}, 0\leqslant x_2\leqslant 2\pi \big\}.\] 
The initial data $\overline{U}_0$ is the restriction of $\overline{U}$ on $\Sigma_0$.

\item Let $U$ be an entropy solution to the Euler equations \eqref{eq: Euler in rho v}. Without loss of generality, we assume that $U$ is also defined on $\Omega$. The initial data $U_0$ is the restriction of $U$ on $\Sigma_0$.

\item $\overline{U}_0 = {U}_0$.

\end{itemize}

Both $U$ and $\overline{U}$ are bounded on $\Omega$.
Since $\overline{\rho}$ is uniformly bounded above and below by a positive constant, $\left(\frac{\partial^2\eta}{\partial y^a\partial y^b}(\overline{U})\right)$ is a positive definite matrix and all its eigenvalues are bounded and away from $0$. In particular, in this case, in view of \eqref{eq: def relative entropy alpha}, the quantity $\alpha(U,\overline{U})$ can be used as a pointwise norm:
\begin{equation}\label{eq: alpha as a norm} 
\alpha(U,\overline{U}) \approx |U-\overline{U}|^2. 
\end{equation}
Similarly, according to \eqref{eq:relative entropy beta quadratic} and the fact that $q(U)$, $\eta(U)$ and $F_a^i(U)$ are smooth in $U$, we have
\[\beta(U,\overline{U}) \lesssim |U-\overline{U}|^2. \]
In particular, there exists a constant $s_0$ so that
\begin{equation}\label{eq: bound on beta} 
\beta(U,\overline{U}) \leqslant s_0\alpha(U,\overline{U}) .
\end{equation}
Similarly, according to \eqref{eq: QFone} and \eqref{eq: QFtwo}, we also have
\begin{equation}\label{eq: bound on QFi} 
|QF^1(U,\overline{U})|+|QF^2(U,\overline{U})|\lesssim \alpha(U,\overline{U}) .
\end{equation}

For a sufficiently large $L>0$ and a fixed $t\leqslant t^*$, we consider the region 
\[\Omega_L=\big\{(t,x_1,x_2)\in \Omega\big| 0\leqslant \tau \leqslant t,  |x_1|\leqslant L+s_0(t-\tau) \big\}.\]
\begin{center}
\includegraphics[width=2.5in]{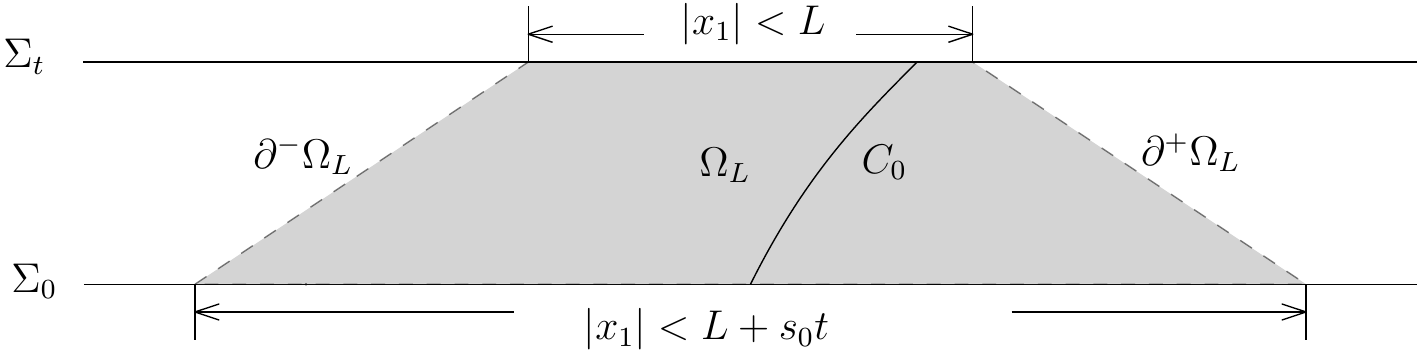}
\end{center}
The boundary of $\Omega_L$ consists of four components: the right one $\partial^+ \Omega_L$, the left one $\partial^- \Omega_L$, the top one $\Sigma_t$ (with $|x_1|\leqslant L$) and the bottom one $\Sigma_0$ (with $|x_1|\leqslant L+s_0 t$). They are depicted in the above picture.

We integrate the inequality \eqref{eq: entropy density inequality} on $\Omega_L$. According to the Stokes formula and the definition of $\mu(U,\overline{U})$, the lefthand side of  \eqref{eq: entropy density inequality} contributes four boundary integrals on  $\partial^\pm \Omega_L$, $\Sigma_t$  and $\Sigma_0$.  Since $U_0=\overline{U}_0$, the boundary integral of $\Sigma_0$ vanishes. Therefore, by choosing $L$ sufficiently large, we can assume that $\partial^+\Omega_L$ is on the righthand side of $C_0$ (the characteristic boundary of the future development of the data defined on $x_1\geqslant 0$). 
Since the integrands of boundary integrals on $\partial^\pm\Omega$ are proportional to $s_0\alpha(U,\overline{U})-\beta^1(U,\overline{U})$, in view of \eqref{eq: bound on beta}, the contribution of the boundary integrals on $\partial^+\Omega$ and $\partial^-\Omega$ are both negative. 
These discussions lead to
\begin{equation}\label{eq: entropy inequality}
\int_{\Sigma_t} \alpha(U,\overline{U})	\leqslant -\int_{\Omega_L}\sum_{i=1}^2\sum_{a=1}^3\sum_{b=1}^3\frac{\partial \overline{U}^{a}}{\partial x^i}\frac{\partial^2\eta}{\partial y^a\partial y^b}( \overline{U})QF^{ib}(U,\overline{U}).
\end{equation}

We recall that in the first paper \cite{LuoYu1},  for $\psi\in \{\wb,w,\psi_2\}$, we have proved that $|\Xr(\psi)|=|\frac{\partial \psi}{\partial x_2}|\lesssim \varepsilon$. Since $\overline{U}$ is a smooth function of $ \wb,w$ and $\psi_2$ (provided $\rho$ is bounded below by a positive constant), we have
\[|\frac{\partial \overline{U}^a}{\partial x^2}|\lesssim \varepsilon, \  a=1,2,3.\] 
By \eqref{eq: bound on QFi} and the fact that $\frac{\partial^2\eta}{\partial y_a\partial y_b}( \overline{U})$ is also bounded above by a universal constant, we can control the terms for $i=2$ on the righthand side of \eqref{eq: entropy inequality}. Hence,
\begin{equation}\label{eq: entropy inequality 1}
\int_{\Sigma_t} \alpha(U,\overline{U})	\leqslant -\int_{\Omega_L} \sum_{a=1}^3\sum_{b=1}^3\frac{\partial \overline{U}^{a}}{\partial x^1}\frac{\partial^2\eta}{\partial y^a \partial y^b}( \overline{U})QF^{1b}(U,\overline{U})+C_0\varepsilon\int_{0}^t \left(\int_{\Sigma_\tau} \alpha(U,\overline{U})\right)d\tau,
\end{equation}
where $C_0$ is a universal constant, and we use $\Sigma_\tau$ to denote the part $\Sigma_\tau$ with $|x_1|\leqslant L+s_0(t-\tau)$. 

The characteristic  hypersurfaces $C_0,H_0,\Hb_0$ and $\Cb_0$ divide $\Omega_L$ into five regions $\mathcal{O}_+$, $\mathcal{W}_+$, $\mathcal{W}_m$, $\mathcal{W}_-$ and $\mathcal{O}_-$, as depicted below:
\begin{center}
\includegraphics[width=3in]{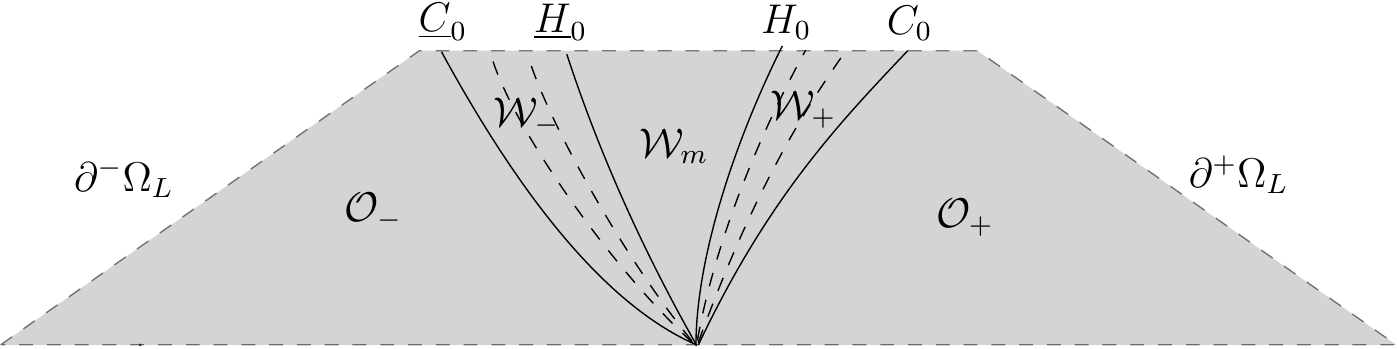}
\end{center}

By the construction in Section \ref{Section: app to Riemann 1} and  Section \ref{Section: app to Riemann 1}, we know that, on $\mathcal{W}_m$ and $\mathcal{O}_\pm$, 
\[|\frac{\partial \overline{U}^a}{\partial x^1}|\lesssim \varepsilon, \  a=1,2,3.\] 
{\color{black}Therefore, we can bound the integrals on $\mathcal{W}_m$ and $\mathcal{O}_\pm$ by the second term of  \eqref{eq: entropy inequality 1}.} This leads to
\begin{equation}\label{eq: entropy inequality 2}
\int_{\Sigma_t} \alpha(U,\overline{U})	\leqslant -\left(\int_{\mathcal{W}_-} +\int_{\mathcal{W}_+}\right)\sum_{a=1}^3\sum_{b=1}^3\frac{\partial \overline{U}^{a}}{\partial x^1}\frac{\partial^2\eta}{\partial y^a\partial y^b}( \overline{U})QF^{1b}(U,\overline{U})+C_0\varepsilon\int_{0}^t \left(\int_{\Sigma_\tau} \alpha(U,\overline{U})\right)d\tau,
\end{equation}
where $C_0$ is a universal constant and it may be different from the one in \eqref{eq: entropy inequality 1}. It remains to consider the integrals of \eqref{eq: entropy inequality 1} on the rarefaction wave zone  $\mathcal{W}_+$, and $\mathcal{W}_-$.

We consider the integral on $\mathcal{W}_+$ and the key is to compute $\frac{\partial \overline{U}^{a}}{\partial x^1}$ for $a=1,2,3$. We first compute the differential of the transformation  from  $V=\begin{pmatrix}
  \wb \\
 w\\
 \psi_2
\end{pmatrix}$ to $U=\begin{pmatrix}
  \rho \\
 P^1\\
 P^2
\end{pmatrix}$. We use $\frac{\partial U}{\partial V}$ to denote the differential. In view of \eqref{def: Riemann invariants},  it can be written as the following matrix
\begin{equation}\label{eq: partial U partial V}
\left(\frac{\partial U^a}{\partial V^b}\right)=\begin{pmatrix}
 \frac{\gamma-1}{2}\frac{1}{c'(\rho)} & \frac{\gamma-1}{2}\frac{1}{c'(\rho)} & 0\\
 \frac{\gamma-1}{2}\frac{1}{c'(\rho)} v^1+\rho & \frac{\gamma-1}{2}\frac{1}{c'(\rho)} v^1-\rho&0\\
 \frac{\gamma-1}{2}\frac{1}{c'(\rho)} v^2 &  \frac{\gamma-1}{2}\frac{1}{c'(\rho)} v^2 &-\rho
\end{pmatrix}.
\end{equation}
Moreover, all the entries are uniformly bounded. We recall that in \cite{LuoYu1},  for $\psi\in \{w,\psi_2\}$, we have proved that $|\Tr(\psi)|=|t\frac{\partial \psi}{\partial x^1}|\lesssim \varepsilon t$. Hence,
\[|\frac{\partial \psi}{\partial x^1}|\lesssim \varepsilon, \  \psi\in \{w,\psi_2\}.\] 
By expanding $\frac{\partial \overline{U}^{a}}{\partial x^1}$ as follows:
\begin{equation}\label{eq: compute partial1 of Ub}
\frac{\partial \overline{U}^{a}}{\partial x^1}=\frac{\partial \overline{U}^{a}}{\partial \wb}\frac{\partial \wb}{\partial x^1}+\frac{\partial \overline{U}^{a}}{\partial w}\frac{\partial w}{\partial x^1}
+\frac{\partial \overline{U}^{a}}{\partial \psi_2}\frac{\partial \psi_2}{\partial x^1},
\end{equation}
we see that the contributions in \eqref{eq: entropy inequality 2} of the last two terms are also bounded by a universal constant times $\varepsilon \alpha(U,\overline{U})$. Therefore,
\begin{equation}\label{eq: Wplus estimate}
\begin{split}
&- \int_{\mathcal{W}_+} \sum_{a=1}^3\sum_{b=1}^3\frac{\partial \overline{U}^{a}}{\partial x^1}\frac{\partial^2\eta}{\partial y^a\partial y^b}( \overline{U})QF^{1b}(U,\overline{U})\\
\leqslant& - \frac{\partial \wb}{\partial x^1}\int_{\mathcal{W}_+} \sum_{a=1}^3\sum_{b=1}^3 \frac{\partial \overline{U}^{a}}{\partial \wb}\frac{\partial^2\eta}{\partial y^a\partial y^b}( \overline{U})QF^{1b}(U,\overline{U})+C_0\varepsilon\int_{0}^t \left(\int_{\Sigma_\tau} \alpha(U,\overline{U})\right)d\tau\\
=&- \frac{\partial \wb}{\partial x^1}\big\{[p(\rho)-p(\overline{\rho})-p'(\overline{\rho})(\rho-\overline{\rho})]+\rho(v^1-\overline{v}^1)^2\big\}+C_0\varepsilon\int_{0}^t \left(\int_{\Sigma_\tau} \alpha(U,\overline{U})\right)d\tau.
\end{split}
\end{equation}
The last line is from a direct computation by using \eqref{eq: hassian eta}, \eqref{eq: QFone} and \eqref{eq: partial U partial V}. We recall that in \cite{LuoYu1}, we have proved that $|\Tr(\wb)+\frac{2}{\gamma+1}|\lesssim \varepsilon t$. Thus, although $|\frac{\partial \wb}{\partial x^1}|\rightarrow \infty$ as $t\rightarrow 0$, we have $\frac{\partial \wb}{\partial x_1}>0$. Since $\gamma>1$, we also have 
\[p(\rho)-p(\overline{\rho})-p'(\overline{\rho})(\rho-\overline{\rho})\geqslant 0.\]
Thus, we can drop the first term on the last line of \eqref{eq: Wplus estimate} and we obtain that
\[
- \int_{\mathcal{W}_+} \sum_{a=1}^3\sum_{b=1}^3\frac{\partial \overline{U}^{a}}{\partial x^1}\frac{\partial^2\eta}{\partial y^a\partial y^b}( \overline{U})QF^{1b}(U,\overline{U})
\leqslant C_0\varepsilon\int_{0}^t \left(\int_{\Sigma_\tau} \alpha(U,\overline{U})\right)d\tau.
\]
Similar argument also works for the integral on $\mathcal{W}_-$. Thus, \eqref{eq: entropy inequality} finally leads to
\begin{equation*}\label{eq: entropy inequality final}
\int_{\Sigma_t} \alpha(U,\overline{U})	\leqslant C_0\varepsilon\int_{0}^t \left(\int_{\Sigma_\tau} \alpha(U,\overline{U})\right)d\tau.
\end{equation*}
Then, by the Gronwall's inequality, $\alpha(U,\overline{U})\equiv 0$. Hence, $U=\overline{U}$. This completes the proof of Proposition \ref{prop: uniqueness to thm 3} hence the uniqueness of the solution in {\bf Theorem 3}.

\subsubsection{The uniqueness of a single family of centered rarefaction waves}\label{Section: uniqueness Thm 2}
We now provide a proof for Proposition \ref{prop: uniqueness to thm 2}. We start by showing that the solution $(v',c')$ must satisfy the same limiting data on the singularity $\mathbf{S}_*$ in Proposition \ref{prop:data at the singularity}, i.e., 
\begin{equation}\label{eq:data at the singularity'}
	\begin{cases}
		&\lim_{t \rightarrow 0}\wb'(t,u',\vartheta')= \wb_r(0,\vartheta')-\frac{2}{\gamma+1}u',\\
		&\lim_{t \rightarrow 0}w'(t,u',\vartheta')=w_r(0,\vartheta'),\\
		&\lim_{t \rightarrow 0}\psi_2'(t,u',\vartheta')=-v^2_r(0,\vartheta').		
	\end{cases}
\end{equation}
Let $(L',T',\Xh')$ be the null frame associated with the acoustical coordinate $(t,u',\vartheta')$. Since $\Th' := \frac{1}{\kappa'}T'$ and $\Xh'$ are normal and tangential vectors to  $S_{t,u'}' \subset \Sigma_t$ and that $S_{t,u'}'$ converges to $\mathbf{S}_*$, we have
\[ \lim_{t \rightarrow 0}\Th^1 =-1,  \ \lim_{t \rightarrow 0}\Th^2 =0, \ \ \lim_{t \rightarrow 0}\Xh^1 = 0, \ \lim_{t \rightarrow 0}\Xh^2 = 1. \]
As $t\rightarrow 0$, according to \eqref{eq:kappa-limit} and the Euler equations \eqref{Euler equations:form 2}, we have
\[ T'(w) = O(\kappa'), \ \ T'(\psi_2) = O(\kappa'), \ \ t \rightarrow 0. \]
This proved the last two identities in \eqref{eq:data at the singularity'} because $T' =\frac{\partial}{\partial u'}$ and $X' =\frac{\partial}{\partial \vartheta'}$ on $\mathbf{S}_*$. On the other hand, by \eqref{structure eq 1: L kappa},  we have  
\[ \lim_{t \rightarrow 0}m' = \lim_{t \rightarrow 0}\frac{\kappa'}{t} = 1. \] 
Hence, $T'\wb'\big|_{\mathbf{S}_*} \equiv -\frac{2}{\gamma+1}$ and this completes the proof of \eqref{eq:data at the singularity'}.

We now choose suitable data $U_l$ on $x_1<0$, see Remark \ref{rem: 1 family alinhac}. We only require the trace of $U_l$ on $\mathbf{S}_*$ is given by
\begin{equation}
	\wb_l(0,x^2) = \wb_r(0,x^2) - \frac{2}{\gamma+1}u^*, \ \ w_l(0,x^2) = w_r(0,\vartheta), \ \ v^2_l(0,x^2) = v^2_r(0,x^2),
\end{equation}
and $U_l$ is smooth on $\{x^1 \leqslant 0\}$. 

We repeat the construction in Section \ref{Section: app to Riemann 2}. Let $(v_l,c_l)$ be the smooth solution developed from $U_l$ on $\{x^1 \leqslant 0\}$ and $\Cb_0$ be the characteristic boundaries of its domain of dependence. Let $H' := C_{u^*}'$ be the left boundary of $(v',c')$ in $\mathcal{W}'$. We then solve the Goursat problem with smooth characteristic data on $\Cb_0$ and $H'$ to obtain $(v_m',c_m')$. Hence, the combination of $(v_l,c_l)$, $(v_m',c_m')$, $(v',c')$ and $(v_r,c_r)$ is a solution to the Riemann problem with given data $U_l$ and $U_r$ on $t=0$.

We also apply the above construction to $(v,c)$ constructed in {\bf Theorem 2} in the same manner. We then obtain another solution $(v_l,c_l)$, $(v_m,c_m)$, $(v,c)$ and $(v_r,c_r)$  to the Riemann problem with given data $U_l$ and $U_r$ on $t=0$. 

We then apply the proof of \ref{prop: uniqueness to thm 3} exactly in the same way and we prove that  the above two solutions must be the same. In particular, we have $\mathcal{W}' = \mathcal{W}$ and $(v,c) = (v',c')$. This completes the proof of Proposition \ref{prop: uniqueness to thm 2}.

\begin{remark}
There is an alternative proof for the uniqueness of centered rarefaction wave near $C_0$. We choose the following one dimensional Riemann data $(\mathring{v}'_l,\mathring{c}'_l)$ and $(\mathring{v}_r,\mathring{c}_r)$. The data $(\mathring{v}_r,\mathring{c}_r)$ is exactly the one used in Definition \ref{def:data} and {\bf Theorem 3}. To prescribe $(\mathring{v}'_l,\mathring{c}'_l)$, we consider the following picture which was discussed in Section \ref{section:review on p system}:
\begin{center}
\includegraphics[width=1.5in]{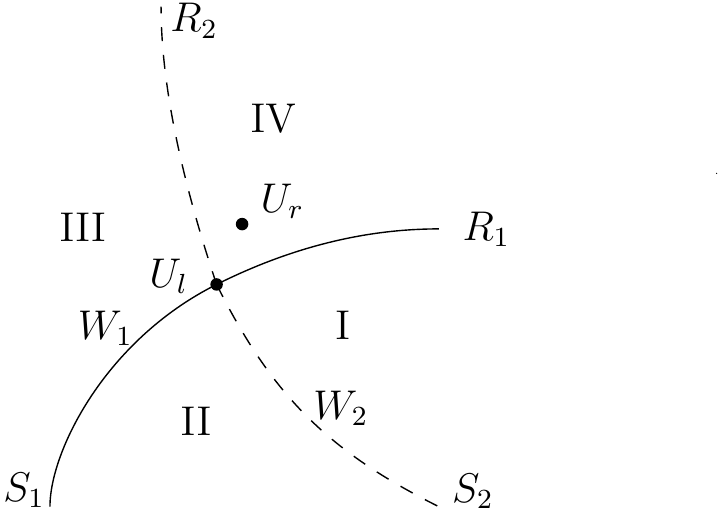} 
\end{center}
The point $U_l$ denotes the data $(\mathring{v}'_l,\mathring{c}'_l)$ and the point  $U_r$ denotes the data $(\mathring{v}_r,\mathring{c}_r)$. We can choose $U_l$ in the following way:
\begin{itemize}
\item[a)] We make sure that $U_r$ is in the region $IV$ so that there are  two families of centered rarefaction waves connecting $U_l$ and $U_r$. 
\item[b)] We make sure that $U_r$ is sufficiently close to $U_l$ so that the corresponding $
u^*$ for $U_r$ is much less than $u'^*$. 
\end{itemize}
We then take $(u_r,v_r)$ and $(\mathring{v}'_l,\mathring{c}'_l)$ as initial data to the Euler equations. We use the standard one dimensional centered rarefaction waves $(v_l,c_l)$ connecting $(\mathring{v}'_l,\mathring{c}'_l)$.  We have two solutions connected to $(u_r,v_r)$: the solution $(v,c)$ and the solution $(v',c')$. The constructions in  Section \ref{Section: app to Riemann 2} yield two solutions to the Riemann problem with data $(u_r,v_r)$ and $(\mathring{v}'_l,\mathring{c}'_l)$ on $\Sigma_0$. By  Proposition \ref{prop: uniqueness to thm 3}, they must be the same. We can restrict to the family on the righthand side and this provides an alternative proof of Proposition \ref{prop: uniqueness to thm 2}.
\end{remark}
}

\section*{Acknowledgment}
The authors are grateful to the anonymous referees, who suggested many valuable improvements and corrections. PY is supported by NSFC11825103, NSFC12141102, New Cornerstone Investigator Program and Xiao-Mi Professorship. TWL is supported by NSFC 11971464.

\end{document}